\documentclass[11pt]{article}

\usepackage{graphicx}
\usepackage{amsmath,amssymb,amsfonts}
\usepackage{amsthm}
\usepackage[T1]{fontenc}
\usepackage{xcolor}
\usepackage[colorlinks]{hyperref} 
\AtBeginDocument{\hypersetup{
    citecolor=magenta,
    colorlinks=true,
    linkcolor=blue,
    filecolor=magenta,      
    urlcolor=cyan,    
    }}

\usepackage{subcaption}
\usepackage[subrefformat=parens]{caption}
\usepackage{csquotes} 
\usepackage{afterpage} 
\usepackage[all,cmtip]{xy}
\usepackage{enumerate}

\DeclareCaptionSubType*{figure}
\newtheorem{theorem}{Theorem}[section]
\newtheorem{lemma}[theorem]{Lemma}
\newtheorem{proposition}[theorem]{Proposition} 
\newtheorem{corp}{Corollary}[theorem]
\newtheorem*{theorem*}{Theorem} 
\theoremstyle{remark}
\newtheorem{remark}[theorem]{Remark}
\theoremstyle{definition}\newtheorem{example}[theorem]{Example}
\numberwithin{equation}{section}

\newcommand{\h}{\mathfrak h}

\newcommand{\loge}{\overset{\log}{\sim}}
\newcommand{\gc}{\mathcal{GC}}
\newcommand{\bg}{\mathbf{\Gamma}}
\newcommand{\ronum}[1]{%
  \textup{\uppercase\expandafter{\romannumeral#1}}%
}
\newcommand{\lgrl}[1]{(\mathcal{L}^{#1} \nabla^{#1}_{\mathcal{L}})}
\newcommand{\grll}[1]{(\nabla^{#1}_{\mathcal{L}}\mathcal{L}^{#1})}
\newcommand{\grl}[1]{(\nabla^{#1}_{\mathcal{L}})}
\newcommand{\lm}[1]{(\mathcal{L}^{#1})}

\definecolor{mygray}{gray}{0.6}
\newcommand{\cdts}{\textcolor{mygray}{\cdots}}
\newcommand{\vdts}{\textcolor{mygray}{\vdots}}

\usepackage{doi}

\DeclareMathOperator{\rank}{rank}
\DeclareMathOperator{\id}{id}
\DeclareMathOperator{\diag}{diag}

\DeclareMathOperator{\ad}{ad}
\DeclareMathOperator{\GL}{GL}
\DeclareMathOperator{\gl}{gl}
\DeclareMathOperator{\SL}{SL}
\DeclareMathOperator{\sll}{sl}
\DeclareMathOperator{\tr}{tr}

\begin{document}
\SetBlockThreshold{1} 
\title{Multiple generalized cluster structures on $D(\GL_n)$}
\date{}
\author{Dmitriy Voloshyn}

\maketitle

\makeatletter
\def\blfootnote{\xdef\@thefnmark{}\@footnotetext}
\makeatother

\begin{abstract}We produce a large class of generalized cluster structures on the Drinfeld double of $\GL_n$ that are compatible with Poisson brackets given by Belavin-Drinfeld classification. The resulting construction is compatible with the previous results on cluster structures on $\GL_n$.
\end{abstract}

\blfootnote{Published in Forum of Mathematics, Sigma \textbf{11}(46) (2023), 1--78. \doi{10.1017/fms.2023.44} }

\tableofcontents

\section{Introduction}
The present article is a continuation in a series of papers by Misha Gekhtman, Misha Shapiro and Alek Vainshtein that aim at proving the following conjecture:
\blockquote{\emph{Any simple complex Poisson-Lie group endowed with a Poisson bracket from the Belavin-Drinfeld classification possesses a compatible generalized cluster structure.}
}
For conciseness, we refer to the above conjecture as the \emph{GSV conjecture} and to Belavin–Drinfeld triples as \emph{BD triples}. The conjecture was first formulated in~\cite{conj} assuming ordinary cluster structures of geometric type, and later it was realized in~\cite{double} that a more general notion of cluster algebras is needed. Cluster algebras were invented by Fomin and Zelevinsky in~\cite{fathers} as an algebraic framework for studying dual canonical bases and total positivity. The notion of generalized cluster algebra suitable for the GSV conjecture was first introduced in~\cite{double} as an adjustment of an earlier definition given in~\cite{earlier}.

\paragraph{Progress on GSV conjecture.}
As the recent progress shows, the GSV conjecture might be extended beyond simple groups and brackets compatible with the group structure. At present, we know that
\begin{itemize}
\item Any simple complex Poisson-Lie group endowed with the standard Poisson bracket possesses a compatible cluster structure; see~\cite{conj};
\item For any Belavin-Drinfeld data, the existence of a compatible cluster structure for $\SL_n$ for all $n < 5$ was shown in~\cite{conj}; for $\SL_5$, the conjecture was proved by Eisner in~\cite{slfive};
\item For a large class of the so-called aperiodic oriented\footnote{As of October 2023, M. Gekhtman, M. Shapiro and A. Vainshtein have pushed the conjecture further to aperiodic BD triples which are not necessarily oriented.} BD triples, the conjecture was proved for $\SL_n$ in~\cite{plethora}; 
\item For other Poisson-Lie groups, the conjecture was established for the Drinfeld double of $\SL_n$ in~\cite{double} for the standard Poisson structure, as well as for\footnote{As of October 2023, for any BD triple $\bg$ of type $A_{n-1}$, we have constructed a  generalized cluster structure on $\SL_n^{\dagger}$ compatible with $\{\cdot,\cdot\}_{\bg}^{\dagger}$.} $\SL_n^\dagger$, which is an image of the dual group $\SL_n^*$ in $\SL_n$. An alternative construction on the Drinfeld double of $\SL_n$ was also given in~\cite{periodic,doublerel}.
\end{itemize}
The above results naturally extend to $\GL_n$. The present paper combines the cluster structures from~\cite{plethora} for aperiodic oriented BD triples with the generalized cluster structure from~\cite{double} for the Drinfeld double endowed with the standard bracket. As a result, we derive generalized cluster structures on the Drinfeld doubles of $\GL_n$ and $\SL_n$ compatible with Poisson brackets from the aperiodic oriented class of Belavin-Drinfeld triples.

\paragraph{Belavin-Drinfeld triples.} Let $\Pi:=[1,n-1]$ be a set of simple roots of type $A_n$ identified with an interval $[1,n-1]$. Recall that a Belavin-Drinfeld triple is a triple $(\Gamma_1,\Gamma_2,\gamma)$ such that $\Gamma_1,\Gamma_2\subseteq \Pi$ and $\gamma:\Gamma_1 \rightarrow \Gamma_2$ a nilpotent isometry; we say that the triple is \emph{trivial} if $\Gamma_1=\Gamma_2=\emptyset$. As Belavin and Drinfeld showed in~\cite{bd,bd2}, such triples (together with some additional data) parametrize factorizable quasitriangular Poisson structures on connected simple complex Poisson-Lie groups (for details, see Section~\ref{s:plgrps}). As in~\cite{plethora}, however, we consider even more general Poisson brackets that depend on a pair $\bg:=(\bg^r,\bg^c)$ of Belavin-Drinfeld triples $\bg^r:=(\Gamma_1^r,\Gamma_2^r,\gamma_r)$ and $\bg^c:=(\Gamma_1^c,\Gamma_2^c,\gamma_c)$. A Belavin-Drinfeld triple $(\Gamma_1,\Gamma_2,\gamma)$ is called \emph{oriented} if for any $i,i+1 \in \Gamma_1$, $\gamma(i+1) = \gamma(i)+1$; a pair of Belavin-Drinfeld triples is called \emph{oriented} if both $\bg^r$ and $\bg^c$ are oriented. The pair is called \emph{aperiodic} if the map $\gamma_c^{-1}w_0\gamma_rw_0$ is nilpotent, where $w_0$ is the longest Weyl group element. Given a Cartan subalgebra $\mathfrak{h}$ of $\sll_n(\mathbb{C})$ and a Belavin-Drinfeld triple $(\Gamma_1,\Gamma_2,\gamma)$, set
\[
\mathfrak{h}_{\bg}:= \{h \in \mathfrak{h} \ | \ \alpha(h) = \beta(h), \ \gamma^j(\alpha)=\beta \ \text{for some} \ j\},
\]
and let $\mathcal{H}_{\bg}$ be the connected subgroup of $\SL_n(\mathbb{C})$ with Lie algebra $\mathfrak{h}_{\bg}$. The dimension of $\mathcal{H}_{\bg}$ is given by $k_{\bg}:=|\Pi\setminus\Gamma_1|$.

\paragraph{Main results and the outline of the paper.} In this paper, we consider generalized cluster structures in the rings of regular functions of $\GL_n\times\GL_n$ and $\SL_n \times \SL_n$ (for the precise definition, see Section~\ref{s:cluster_str}). Roughly, the difference between generalized cluster structures and ordinary cluster structures of geometric type (in the sense of Fomin and Zelevinsky) is that the former allows more than two monomials in exchange relations. In fact, there is only one generalized exchange relation in the initial seeds that we study in this paper (more generalized exchange relations appear in the case of nonaperiodic BD pairs; see~\cite{periodic}). Recall that an extended cluster $(x_1,\ldots,x_{N+M})$ is called \emph{log-canonical} (relative some Poisson bracket $\{\cdot,\cdot\}$) if $\{x_i,x_j\} = \omega_{ij} x_ix_j$ for some constants $\omega_{ij}$ and all $1 \leq i, j\leq N+M$; a generalized cluster structure is called \emph{compatible} with the Poisson bracket if all extended clusters are log-canonical. An extended cluster $(x_1,\ldots,x_{N+M})$ is called \emph{regular} if all $x_i$'s are represented as regular functions on the given variety ($\GL_n \times \GL_n$ or $\SL_n \times \SL_n$ in our case); the generalized cluster structure is called \emph{regular} if all extended clusters are regular. The main part of the paper is devoted to proving the following theorem:

\begin{theorem}\label{t:main}
Let $\bg = (\bg^r, \bg^c)$ be a pair of aperiodic oriented Belavin-Drinfeld triples. There exists a generalized cluster structure $\gc(\bg)$ on $D(\GL_n) = \GL_n \times \GL_n$ such that
\begin{enumerate}[(i)]
\item The number of stable variables is $k_{\bg^r}+k_{\bg^c} + (n+1)$, and the exchange matrix has full rank;
\item The generalized cluster structure $\gc(\bg)$ is regular, and the ring of regular functions $\mathcal{O}(D(\GL_n))$ is naturally isomorphic to the upper cluster algebra $\bar{\mathcal{A}}_{\mathbb{C}}(\gc(\bg))$;\label{tm:iso}
\item The global toric action of $(\mathbb{C}^*)^{k_{\bg^r}+k_{\bg^c}+2}$ on $\gc(\bg)$ is induced by the left action of $\mathcal{H}_{\bg^r}$, the right action of $\mathcal{H}_{\bg^c}$ and the action by scalar matrices on each component of $\GL_n \times \GL_n$;\label{tm:tor}
\item Any Poisson bracket defined by the pair $\bg$ on $D(\GL_n)$ is compatible with $\gc(\bg)$.\label{tm:compb}
\end{enumerate}
\end{theorem}
For the trivial $\bg^r$ and $\bg^c$, the theorem was proved in~\cite{double} (we refer to the corresponding generalized cluster structure $\gc(\bg^r,\bg^c)$ as the \emph{standard} one). When $\bg^r = \bg^c$, the group $D(\GL_n)$ together with its Poisson structure is the Drinfeld double of $\GL_n$. By default, we work over the field of complex numbers $\mathbb{C}$ (however, the results hold over $\mathbb{R}$ for the same class of Poisson brackets). The initial seed is described in Section~\ref{s:descr} (a rough description is available below). The proof of Theorem~\ref{t:main} is contained in Sections~\ref{s:regular}-\ref{s:compb}. In Section~\ref{s:regular}, we prove that all cluster variables in the seeds adjacent to the initial one are regular functions. In Section~\ref{s:complet}, we prove Part~\ref{tm:iso} by induction on the size $|\Gamma_1^r| + |\Gamma_1^c|$. The step of the induction employs the construction of certain birational quasi-isomorphisms introduced in~\cite{plethora} (i.e., quasi-isomorphisms in the sense of~\cite{fraser} that are also birational isomorphisms of the underlying varieties). Section~\ref{s:torpr} is devoted to Part~\ref{tm:tor}, and in Sections~\ref{s:logcan} and \ref{s:compb} we prove Part~\ref{tm:compb} via a direct computation. A similar result holds in the case of $D(\SL_n)$:
\begin{theorem}\label{t:mainsln}
Let $\bg = (\bg^r, \bg^c)$ be a pair of aperiodic oriented Belavin-Drinfeld triples. There exists a generalized cluster structure $\gc(\bg)$ on $D(\SL_n) = \SL_n \times \SL_n$ such that
\begin{enumerate}[(i)]
\item The number of stable variables is $k_{\bg^r}+k_{\bg^c} + (n-1)$, and the exchange matrix has full rank;
\item The generalized cluster structure $\gc(\bg)$ is regular, and the ring of regular functions $\mathcal{O}(D(\SL_n))$ is naturally isomorphic to the upper cluster algebra $\bar{\mathcal{A}}_{\mathbb{C}}(\gc(\bg))$;
\item The global toric action of $(\mathbb{C}^*)^{k_{\bg^r}+k_{\bg^c}}$ on $\gc(\bg)$ is induced by the left action of $\mathcal{H}_{\bg^r}$ and the right action of $\mathcal{H}_{\bg^c}$ on $D(\SL_n)$;
\item Any Poisson bracket defined by the pair $\bg$ on $D(\SL_n)$ is compatible with $\gc(\bg)$.
\end{enumerate}
\end{theorem}
In Section~\ref{s:adj_sln} we show how to derive Theorem~\ref{t:mainsln} from Theorem~\ref{t:main}, and in Section~\ref{s:exs} we provide a few examples of the generalized cluster structures studied in this paper.

\paragraph{A rough description of the initial extended seed.} The construction of the initial quiver consists of two parts: First, we construct the initial quiver for the case of the trivial $\bg$ (this was described in~\cite{double}); second, for each root in $\Gamma^r_1$ and $\Gamma^c_2$, we add three additional arrows (see Figure~\ref{f:nbd_extra}). The initial extended cluster consists of five types of regular functions: $c$-functions, $\varphi$-functions, $f$-functions, $g$-functions and $h$-functions. The $c$-, $\varphi$- and $f$-functions were constructed in~\cite{double} and are the same for any choice of $\bg$. More specifically, the $c$-functions comprise $n-1$ Casimirs\footnote{To be more precise, the $c$-functions are Casimirs on $D(\GL_n)$ if an only if $R_0(I) = (1/2)I$; see a discussion in Section~\ref{s:iniclust}.} of the given Poisson bracket which also serve as isolated frozen variables, and the $\varphi$- and $f$-functions are $(n-1)n/2$ and $(n-1)(n-2)/2$ cluster variables that satisfy the following invariance properties:
\[
f(X,Y) = f(N_+XN_-,N_+YN_-^\prime), \ \ \tilde{\varphi}(X,Y) = \tilde{\varphi}(AXN_-,AYN_-)
\]
where $(X,Y)$ are the standard coordinates on $D(\GL_n)$, $N_+$ is any unipotent upper triangular matrix, $N_-$ and $N_-^\prime$ are unipotent lower triangular matrices, $A$ is any invertible matrix and $\varphi(X,Y) = (\det X)^m\tilde{\varphi}(X,Y)$ for some number $m$ that depends on $\varphi$. Furthermore, for any BD data, the initial seed also contains the $g$-functions $\det X^{[i,n]}_{[i,n]}$ and the $h$-functions $\det Y^{[i,n]}_{[i,n]}$, $1 \leq i \leq n$ ($\det X$ and $\det Y$ are both frozen variables). All the other $g$- and $h$-functions are constructed via a combinatorial procedure based on the given root data $\bg$ (there are $n(n+1)/2$ $g$-functions and $n(n+1)/2$ $h$-functions). As in~\cite{plethora}, we construct a list of so-called $\mathcal{L}$-matrices, and then we set the $g$- and $h$-functions to be the trailing minors of the $\mathcal{L}$-matrices. The determinants of $\mathcal{L}$-matrices are declared to be frozen variables. If we let $\psi$ to be any $g$- or $h$-variable, then it satisfies the following invariance properties:
\[
\psi(N_+X,\tilde{\gamma}_r(N_+)Y) = \psi(X\tilde{\gamma}_c^*(N_-),YN_-) = \psi(X,Y)
\]
where $\tilde{\gamma}_r$ and $\tilde{\gamma}_c^*$ are group lifts of the Belavin-Drinfeld maps $\gamma_r$ and $\gamma_c^*$ associated with $\bg^r$ and $\bg^c$, respectively. The combinatorial construction relies on the nilpotency of the map $\gamma_c^{-1}w_0\gamma_rw_0$, where $w_0$ is the longest Weyl group element. When $\gamma_c^{-1} w_0 \gamma_r w_0$ is not nilpotent, the $\mathcal{L}$-matrices become infinite, so a different procedure has to be applied (one such example was studied in~\cite{periodic}). See Section~\ref{s:exs} for some examples of $\mathcal{L}$-matrices and initial quivers.

\paragraph{Future work.} As explained in Remark 3.7 in~\cite{plethora}, the authors of the conjecture have already identified generalized cluster structures in type $A$ for any Belavin-Drinfeld data. However, there is a lack of tools for producing proofs. One new tool was introduced in~\cite{plethora}, which consists in considering birational quasi-isomorphisms between different cluster structures and which significantly reduced labor in showing that the upper cluster algebra is naturally isomorphic to the algebra of regular functions. However, there's yet no better tool for proving the compatibility with a Poisson bracket except a tedious and direct computation. We hope that the constructed birational quasi-isomorphisms might be used for proving log-canonicity as well, but this idea is still under development. Furthermore, Schrader and Shapiro recently in~\cite{schrader2019cluster} have embedded the quantum group $U_q(\sll_n)$ into a quantum cluster $\mathcal{X}$-algebra introduced by Fock and Goncharov in~\cite{fock2006cluster}. As noted in~\cite{schrader2019cluster}, one should be able to embed $U_q(\sll_n)$ into a quantum cluster $\mathcal{A}$-algebra in the sense of Berenstein and Zelevinsky~\cite{berenstein2005quantum}, which is suggested by the existence of a generalized cluster structure on the dual group $\SL_n^*$ from~\cite{double}. We plan to address the question about $\mathcal{A}$-cluster realization of $U_q(\sll_n)$, as well as the question of existence of other generalized cluster structures on $\SL_n^*$ in our future work. 

\paragraph{Software.} During the course of working on this paper, we have developed a Matlab application that is able to produce the initial seed of any generalized cluster structure presented in the paper and which provides various tools for manipulating the quiver and the associated functions. It also presents some tools for working with Poisson brackets. The software is freely available under MIT license on the author's GitHub repository: \href{https://github.com/Grabovskii/GenClustGLn}{https://github.com/Grabovskii/GenClustGLn}

\paragraph{Acknowledgments.} The author would like to thank his advisor, Misha Gekhtman, for giving this problem to him and for numerous discussions and valuable suggestions. The author would also like to thank the anonymous referee, whose valuable comments have improved the exposition. The first version of the paper was written at University of Notre Dame and was partially supported by the NSF grant 2100785. The final version was written at Institute for Basic Science and supported by the grant IBS-R003-D1.

\section{Background}\label{s:back}
\subsection{Generalized cluster structures}\label{s:cluster_str}
In this section, we briefly recall the main definitions and propositions of the generalized cluster algebras theory from~\cite{double}, which constitute a generalization of cluster algebras of geometric type invented by Fomin and Zelevinsky in~\cite{fathers}. Throughout this section, let $\mathcal{F}$ be a field of rational functions in $N+M$ independent variables with coefficients in $\mathbb{Q}$. Fix an algebraically independent set $x_{N+1},\ldots, x_{N+M} \in \mathcal{F}$ over $\mathbb{Q}$ and call its elements \emph{stable} (or \emph{frozen})\emph{ variables}.

\paragraph{Seeds.} To define a seed, we first define the following data:
\begin{itemize}
 \item Let $\tilde{B} = (b_{ij})$ be an $N\times (N+M)$ integer matrix whose principal part $B$ is skew-symmetrizable (recall that the principal part of a matrix is its leading square submatrix). The matrices $B$ and $\tilde{B}$ are called the \emph{exchange matrix} and the \emph{extended exchange matrix}, respectively;
\item Let $x_1,\ldots, x_N$ be an algebraically independent subset of $\mathcal{F}$ over $\mathbb{Q}$ such that the elements $x_1,\ldots,x_N,\ldots,x_{N+M}$ freely generate the field $\mathcal{F}$. The elements $x_1,\ldots,x_N$ are called \emph{cluster variables}, and the tuples $\mathbf{x} := (x_1,\ldots,x_N)$ and $\tilde{\mathbf{x}} := (x_1,\ldots,x_{N+M})$ are called a \emph{cluster} and an \emph{extended cluster}, respectively;
\item For every $1 \leq i \leq N$, let $d_i$ be a factor of $\gcd(b_{ij} \ | \ 1 \leq j \leq N)$. The \emph{$i$th string $p_i$} is a tuple $p_i := (p_{ir})_{0 \leq r \leq d_i}$, where each $p_{ir}$ is a monomial in the stable variables with an integer coefficient and such that $p_{i0} = p_{id_i} = 1$. The $i$th string is called \emph{trivial} if $d_i = 1$. Set $\mathcal{P} := \{ p_i \ | \ 1 \leq i \leq N\}$.
\end{itemize}
Now, a \emph{seed} is $\Sigma := (\mathbf{x}, \tilde{B}, \mathcal{P})$ and an \emph{extended seed} is $\tilde{\Sigma} := (\tilde{\mathbf{x}}, \tilde{B}, \mathcal{P})$. In practice, one additionally names one of the seeds as the \emph{initial seed}.

\paragraph{Generalized cluster mutations.} Let $\Sigma = (\mathbf{x}, \tilde{B},\mathcal{P})$ be a seed constructed via the recipe from the previous paragraph. A \emph{generalized cluster mutation in direction $k$} produces a seed $\Sigma^\prime = (\mathbf{x}^\prime, \tilde{B}^\prime, \mathcal{P}^\prime)$ that is constructed as follows.
\begin{itemize}
\item Define \emph{cluster $\tau$-monomials} $u_{k;>}$ and $u_{k;<}$, $1 \leq k \leq N$, via
\[
u_{k;>} := \prod_{\substack{1 \leq i \leq N, \\ b_{ki} > 0}} x_{i}^{b_{ki}/d_k}, \ \ u_{k;<} := \prod_{\substack{1 \leq i \leq N, \\ b_{ki} < 0}} x_{i}^{-b_{ki}/d_k},
\]
and \emph{stable $\tau$-monomials} $v_{k;>}^{[r]}$ and $v_{k;<}^{[r]}$, $1 \leq k \leq N$, $0 \leq r \leq d_k$, as
\[
v_{k;>}^{[r]} := \prod_{\substack{N+1\leq i \leq N+M,\\ b_{ki} > 0}} x_i^{\lfloor rb_{ki}/d_k \rfloor}, \ \ v_{k;<}^{[r]} := \prod_{\substack{N+1 \leq i \leq N+M,\\ b_{ki} < 0}} x_i^{\lfloor-rb_{ki}/d_k\rfloor},
\]
where the product over an empty set by definition equals $1$ and $\lfloor m \rfloor$ denotes the floor of a number $m \in \mathbb{Z}$. Define $x_k^\prime$ via the \emph{generalized exchange relation}
\begin{equation}\label{eq:exchr}
x_k x_k^\prime := \sum_{r=0}^{d_k} p_{kr} u_{k;>}^{r} v_{k;>}^{[r]} u_{k;<}^{d_k-r} v_{k;<}^{[d_k-r]},
\end{equation}
and set $\mathbf{x}^\prime := (\mathbf{x} \setminus \{x_k\}) \cup \{x_k^\prime\}$.
\item The matrix entries $b_{ij}^\prime$ of $\tilde{B}^\prime$ are defined as
\[
b_{ij}^\prime := \begin{cases}
-b_{ij} \ \ &\text{if } i=k \text{ or }j=k;\\
b_{ij} + \dfrac{|b_{ik}|b_{kj} + b_{ik} |b_{kj}|}{2} \ \ &\text{otherwise.}
\end{cases}
\]
\item The strings $p_i^\prime \in \mathcal{P}^\prime$ are given by the \emph{exchange coefficient mutation}
\[
p_{ir}^\prime := \begin{cases} p_{i,d_i-r} \ \ &\text{if }i = k; \\ p_{ir} \ \ &\text{otherwise.}\end{cases}
\]
\end{itemize}

The seeds $\Sigma$ and $\Sigma^\prime$ are also called \emph{adjacent}. A few comments on the definition:
\begin{enumerate}[1)]
\item Call a frozen variable $x_i$ \emph{isolated} if $b_{ji} = 0$ for all $1 \leq j \leq N$. The definition of $b_{ij}^\prime$ implies that a mutation preserves the property of being isolated;
\item Since $\gcd\{b_{ij} \ | \ 1 \leq j \leq N\} = \gcd\{b_{ij}^\prime \ | \ 1 \leq j \leq N\}$, the numbers $d_1,\ldots, d_N$ retain their defining property after a mutation is performed;
\item If a string $p_k$ is trivial, then the generalized exchange relation in equation \eqref{eq:exchr} becomes the exchange relation from the ordinary cluster theory of geometric type:
\begin{equation}\label{eq:ordexchrel}
x_k x_k^\prime = \prod_{\substack{1 \leq  i \leq N+M\\ b_{ki>0}}} x_i^{b_{ki}} + \prod_{\substack{1 \leq  i \leq N+M\\ b_{ki<0}}} x_i^{-b_{ki}}.
\end{equation}
In fact, the generalized cluster structures studied in this paper have only one nontrivial string, hence all exchange relations except one are ordinary.
\item The generalized exchange relation can also be written in the following form. For any $i$, denote $v_{i;>} := v_{i;>}^{[d_i]}$, $v_{i;<} := v_{i;<}^{[d_i]}$; set
\[
q_{ir} := \frac{v_{i;>}^r v_{i; < }^{d_i-r}}{(v_{i;>}^{[r]} v_{i;<}^{[d_i-r]})^{d_i}}, \ \ \hat{p}_{ir} := \frac{p_{ir}^{d_i}}{q_{ir}}, \ \ 1 \leq i \leq N,\ 0 \leq r \leq d_i.
\]
Note that the mutation rule for $\hat{p}_{ir}$ is the same as for $p_{ir}$. Now, equation~\eqref{eq:exchr} becomes
\[
x_k x_k^\prime = \sum_{r=0}^{d_k} (\hat{p}_{kr} v_{k;>}^r v_{k;<}^{d_k-r})^{1/d_k} u_{k;>}^r u_{k;<}^{d_k-r}.
\]
The expression $(\hat{p}_{kr} v_{k;>}^r v_{k;<}^{d_k-r})^{1/d_k}$ is a monomial in the stable variables.

\end{enumerate}

\paragraph{Generalized cluster structure.} Two seeds $\Sigma$ and $\Sigma^\prime$ are called \emph{mutation equivalent} if there's a sequence $\Sigma_1, \ldots, \Sigma_m$ such that $\Sigma_1 = \Sigma$, $\Sigma_m = \Sigma^\prime$, and such that $\Sigma_{i+1}$ and $\Sigma_i$ are adjacent for each~$i$. For a fixed seed $\Sigma$, the set of all seeds that are mutation equivalent to $\Sigma$ is called the \emph{generalized cluster structure} and is denoted as $\gc(\Sigma)$ or simply $\gc$.

\paragraph{Generalized cluster algebra.} Let $\gc$ be a generalized cluster structure constructed as above. Define $\mathbb{A} := \mathbb{Z}[x_{N+1},\ldots,x_{N+M}]$ and $\bar{\mathbb{A}} := \mathbb{Z}[x_{N+1}^{\pm 1},\ldots, x_{N+M}^{\pm 1}]$. Choose a \emph{ground ring} $\hat{\mathbb{A}}$, which is a subring of $\bar{\mathbb{A}}$ that contains $\mathbb{A}$. The $\hat{\mathbb{A}}$-subalgebra of $\mathcal{F}$ given by
\begin{equation}
\mathcal{A} := \mathcal{A(\gc)} := \hat{\mathbb{A}}[\ \text{cluster variables from all seeds in }\gc\ ]
\end{equation}
is called the \emph{generalized cluster algebra}. For any seed $\Sigma := ((x_1,\ldots,x_{N}),\tilde{B},\mathcal{P})$, set
\begin{equation}\label{eq:laurring}
\mathcal{L}(\Sigma) := \hat{\mathbb{A}}[x_1^{\pm 1},\ldots,x_N^{\pm 1}]
\end{equation}
to be \emph{the ring of Laurent polynomials} associated with $\Sigma$, and define
\begin{equation}
\bar{\mathcal{A}} :=\bar{\mathcal{A}}(\gc):=
\bigcap_{\Sigma \in \gc}\mathcal{L}(\Sigma). 
\end{equation}
The algebra $\bar{\mathcal{A}}$ is called the \emph{generalized upper cluster algebra}. The \emph{generalized Laurent phenomenon} states that $\mathcal{A} \subseteq \bar{\mathcal{A}}$. 

\paragraph{Upper bounds.} Let $\mathbb{T}_N$ be a labeled $N$-regular tree. Associate with each vertex a seed so that adjacent seeds are adjacent in the tree\footnote{Multiple vertices might receive the same seed, and for this reason the tree is considered labeled. Identifying the vertices with the same seeds (up to permutations of cluster variables), one obtains an unlabeled $N$-regular graph, which encodes all mutations between distinct seeds.}, and if a seed $\Sigma^\prime$ is adjacent to $\Sigma$ in direction $k$, label the corresponding edge in the tree with number $k$. A \emph{nerve} $\mathcal{N}$ in $\mathbb{T}_N$ is a subtree on $N+1$ vertices such that all its edges have different labels (for instance, a star is a nerve). An \emph{upper bound} $\bar{\mathcal{A}}(\mathcal{N})$ is defined as the algebra
\begin{equation}
\bar{\mathcal{A}}(\mathcal{N}) := \bigcap_{\Sigma \in V(\mathcal{N})}\mathcal{L}(\Sigma) 
\end{equation}
where $V(\mathcal{N})$ stands for the vertex set of $\mathcal{N}$. Upper bounds were first defined and studied in~\cite{upper_bounds}. Let $L$ be the number of isolated variables in $\gc$. For the $i$th nontrivial string in $\mathcal{P}$, let $\tilde{B}(i)$ be a $(d_{i}-1) \times L$ matrix such that the $r$th row consists of the exponents of the isolated variables in $p_{ir}$ (recall that $p_{ir}$ is a monomial in the stable variables). The following result was proved in \cite{double}:

\begin{proposition}\label{p:upperb}
Assume that the extended exchange matrix has full rank and let $\rank \tilde{B}(i) = d_i - 1$ for any nontrivial string in $\mathcal{P}$. Then the upper bounds $\bar{\mathcal{A}}(\mathcal{N})$ do not depend on the choice of $\mathcal{N}$ and hence coincide with the generalized upper cluster algebra $\bar{\mathcal{A}}$.
\end{proposition}

\paragraph{Generalized cluster structures on varieties.} 
Let $V$ be a Zariski open subset of $\mathbb{C}^{N+M}$, $\mathcal{O}(V)$ be the ring of regular functions, and let $\mathbb{C}(V)$ be the field of rational functions on $V$. As before, let $\gc$ be a generalized cluster structure, and assume that $f_1,\ldots, f_{N+M}$ is a transcendence basis of $\mathbb{C}(V)$ over $\mathbb{C}$. Pick an extended cluster $(x_1,\ldots,x_{N+M})$ in $\gc$ and define a field isomorphism $\theta: \mathcal{F}_{\mathbb{C}} \rightarrow \mathbb C(V)$ via $\theta: x_i \mapsto f_i$, $1 \leq i \leq N+M$, where $\mathcal{F}_{\mathbb{C}} := \mathcal{F} \otimes \mathbb{C}$ is the extension by complex scalars of $\mathcal{F}$. The pair $(\gc, \theta)$ (or sometimes just $\gc$) is called a \emph{generalized cluster structure on $V$}. It's called \emph{regular} if $\theta(x)$ is a regular function for every variable $x$. Choose a ground ring as
\[
\hat{\mathbb{A}}:= \mathbb Z[x_{N+1}^{\pm 1},\ldots, x_{N+M^\prime}^{\pm 1}, x_{N+M^\prime+1},\ldots, x_{N+M}],
\]
where $\theta(x_{N+i})$ does not vanish on $V$ if and only if $1 \leq i \leq M^\prime$. Set $\mathcal{A}_{\mathbb C} := \mathcal{A} \otimes \mathbb{C}$ and $\bar{\mathcal{A}}_{\mathbb C}:= \bar{\mathcal{A}}\otimes \mathbb C$. 

\begin{proposition}\label{p:starfish}
Let $V$ be a Zariski open subset of $\mathbb C^{N+M}$ and $(\gc,\theta)$ be  a generalized cluster structure on $V$ with $N$ cluster and $M$ stable variables. Suppose there exists an extended cluster $\tilde{\mathbf{x}} = (x_1,\ldots,x_{N+M})$ that satisfies the following properties:
\begin{enumerate}[(i)]
\item For each $1 \leq i \leq N+M$, $\theta(x_i)$ is regular on $V$, and for each $1 \leq i \neq j \leq N+M$, $\theta(x_i)$ is coprime with $\theta(x_j)$ in $\mathcal{O}(V)$;\label{p:sfcopri}
\item For any cluster variable $x_k^\prime$ obtained via the generalized exchange relation~\eqref{eq:exchr} applied to $\tilde{\mathbf{x}}$ in direction $k$, $\theta(x_k^\prime)$ is regular on $V$ and coprime with $\theta(x_k)$ in $\mathcal{O}(V)$.\label{p:sfregi}
\end{enumerate}
Then $(\gc,\theta)$ is a regular generalized cluster structure on $V$. If additionally
\begin{enumerate}[(i)]
\setcounter{enumi}{2}
\item each regular function on $V$ belongs to $\theta(\bar{\mathcal{A}}_{\mathbb{C}}(\gc))$,\label{c:natiso}
\end{enumerate}
then $\theta$ is an isomorphism between $\bar{\mathcal{A}}_{\mathbb{C}}(\gc)$ and $\mathcal{O}(V)$.
\end{proposition}
In the case of ordinary cluster structures, the proof of Proposition \ref{p:starfish} is available in \cite{dasbuch} (\mbox{Proposition 3.37}) and in a more general setup in~\cite{fomin6} (\mbox{Proposition 6.4.1}). As explained in \cite{double}, Proposition \ref{p:starfish} is a direct corollary of a natural extension of \mbox{Proposition 3.6} in \cite{tensordiags} to the case of generalized cluster structures. When $\theta$ is an isomorphism between $\bar{\mathcal{A}}_{\mathbb{C}}(\gc)$ and $\mathcal{O}(V)$, these algebras are also said to be \emph{naturally isomorphic}.  A practical way of verifying Condition~\ref{c:natiso} of Proposition~\ref{p:starfish} is based on Proposition~\ref{p:upperb}.

\paragraph{Poisson structures in $\gc$.} Let $\{\cdot, \cdot\}$ be a Poisson bracket on $\mathcal{F}$ (or on $\mathcal{F}_{\mathbb{C}}$) and let $\tilde{\mathbf{x}}$ be any extended cluster in $\gc$. We say that $\tilde{\mathbf{x}}$ is \emph{log-canonical} if $\{x_i,x_j\} = \omega_{ij} x_i x_j$ for all $1 \leq i,j \leq N+M$, where $\omega_{ij} \in \mathbb{Q}$ (or $\omega_{ij} \in \mathbb{C}$ for $\mathcal{F}_{\mathbb C}$). We call the generalized cluster structure \emph{compatible} with the bracket if any extended cluster in $\gc$ is log-canonical. Let $\Omega:=(\omega_{ij})_{i,j=1}^{N+M}$ be the \emph{coefficient matrix} of the bracket with respect to the extended cluster $\tilde{\mathbf{x}}$. The following proposition is a natural generalization of Theorem 4.5 from \cite{dasbuch}.

\begin{proposition}\label{p:compb}
Let $\Sigma = (\tilde{\mathbf{x}}, \tilde{B}, \mathcal{P})$ be an extended seed in $\mathcal{F}$ that satisfies the following properties:
\begin{enumerate}[(i)]
\item The extended cluster $\tilde{\mathbf{x}}$ is log-canonical with respect to the bracket;\label{pr:c1}
\item For a diagonal matrix with positive entries $D$ such that $DB$ is skew-symmetric, there exists a diagonal $N \times N$ matrix $\Delta$ such that $\tilde{B}\Omega = \begin{bmatrix}\Delta & 0\end{bmatrix}$ and such that $D\Delta$ is a multiple of the identity matrix;\label{pr:c2}
\item The Laurent polynomials $\hat{p}_{ir}$ are Casimirs of the bracket.\label{pr:c3}
\end{enumerate}
Then any other seed in $\gc$ satisfies properties \ref{pr:c1}, \ref{pr:c2} (with the same $\Delta$) and \ref{pr:c3}. In particular, $\gc$ is compatible with $\{\cdot, \cdot \}$.
\end{proposition}
\noindent Condition~\ref{pr:c2} has the following interpretation, which is used in practice. For each $1\leq i \leq N$, define $y_i := \prod_{j=1}^{N+M} x_j^{b_{ij}}$. Then~\ref{pr:c2} is equivalent to $\{\log y_i, \log x_j\} = \delta_{ij}\Delta_{ii}$, where $\delta_{ij}$ is the Kronecker symbol. The variable $y_i$ is called the \emph{$y$-coordinate} of the cluster variable $x_i$. Note that Condition~\ref{pr:c2} implies that $\tilde{B}$ has full rank.

\paragraph{Toric actions.} Given an extended cluster $(x_1,\ldots,x_{N+M})$ in $\gc$, a \emph{local toric action} (of rank $s$) is an action $\mathcal{F}_{\mathbb{C}} \curvearrowleft (\mathbb{C}^*)^s$ by field automorphisms given on the variables $x_i$'s as
\[
x_i.(t_1,\ldots,t_s) \mapsto x_i \prod_{j=1}^{s} t_j^{\omega_{ij}}, \ \ \ t_j \in \mathbb{C}^*,
\]
where $W:= (\omega_{ij})$ is an integer-valued $(N+M) \times s$ matrix of rank $s$ called the \emph{weight matrix} of the action. We say that two local toric actions of \mbox{rank $s$} defined on some extended clusters $\tilde{\mathbf{x}}$ and $\tilde{\mathbf{x}}^\prime$ are \emph{compatible} 
if the composition of mutations that takes $\tilde{\mathbf{x}}$ to $\tilde{\mathbf{x}}^\prime$ intertwines the actions. A collection of pairwise compatible local toric actions of rank $s$ defined for every extended cluster is called a \emph{global toric action}. We also say that a local toric action is \emph{$\gc$-extendable} if it belongs to some global toric action.

\begin{proposition}\label{p:toric}
A local toric action with a weight matrix $W$ is uniquely $\gc$-extendable to a global toric action if $\tilde{B}W = 0$ and the Laurent polynomials $\hat{p}_{ir}$ are invariant with respect to the action.
\end{proposition}
As noted in \cite{double}, this proposition is a natural extension of Lemma 5.3 in \cite{dasbuch}. For the purposes of this paper, it suffices to assume that $\hat{p}_{ir}$ are invariant with respect to the action; however, in the case of ordinary cluster structures of geometric type, the statement of the proposition is \mbox{\emph{if and only if}.}

\paragraph{Quasi-isomorphisms that arise from global toric actions.} Let $\gc_1(\Sigma_1)$ and $\gc_2(\Sigma_2)$ be generalized cluster structures with initial extended seeds $\Sigma_1:=(\tilde{\mathbf{x}}, \tilde{B}_1, \mathcal{P}_1)$ and $\Sigma_2 := (\tilde{\mathbf{f}}, \tilde{B}_2, \mathcal{P}_2)$, and let $\mathcal{F}_1$ and $\mathcal{F}_2$ be the corresponding ambient fields. Assume the following:
\begin{itemize}
\item There is the same number of cluster and stable variables in $\tilde{\mathbf{x}}$ and $\tilde{\mathbf{f}}$;
\item The numbers $d_1,\ldots,d_N$ from the definition of the generalized cluster structure are equal for both $\gc_1$ and $\gc_2$;
\item The strings $\mathcal{P}_1$ and $\mathcal{P}_2$ are the same in the following sense: If one picks $p_{ir}$ and substitutes all $x_i$'s with $f_i$'s, one obtains the $r$th component of the $i$th string from $\mathcal{P}_2$, and vice versa;
\item The extended exchange matrices $\tilde{B}_1$ and $\tilde{B}_2$ are the same in all but the last column, which corresponds to a stable variable;
\item There are integer-valued vectors $u = (u_1,\ldots,u_{n+m})^t$ and $v = (v_1,\ldots,v_{n+m})^t$ that define local toric actions (of rank $1$) on $\tilde{\mathbf{x}}$ and $\tilde{\mathbf{f}}$, respectively, and they are $\gc$-extendable.
\end{itemize}
\begin{proposition}\label{p:compar}

Assume that $\frac{v_i-u_i}{u_{N+M}}$ is an integer for each $1 \leq i \leq N+M$ and $u_{N+M} = v_{N+M}$. Define a field isomorphism $\theta :\mathcal{F}_2 \rightarrow \mathcal{F}_1$ on the generators as $\theta(f_i):=x_i x_{N+M}^{\left(\frac{v_i-u_i}{u_{N+M}}\right)}$, $1 \leq i \leq N+M$. If $\tilde{\mathbf{x}}^\prime := (x_1^\prime,\ldots,x_{N+M}^\prime)$ and $\tilde{\mathbf{f}}^\prime := (f_1^\prime,\ldots,f_{N+M}^\prime)$ are two extended clusters obtained from $\tilde{\mathbf{x}}$ and $\tilde{\mathbf{f}}$ via the same sequence of mutations (i.e., the mutations followed the same indices), and $(u_1^\prime,\ldots, u^\prime_{N+M})^t$ and $(v_1^\prime,\ldots,v^\prime_{N+M})^t$ are the weight vectors of the global toric actions in the extended clusters $\tilde{\mathbf{x}}^{\prime}$ and $\tilde{\mathbf{f}}^\prime$, then
\[
\theta(f_i^\prime) = x_i^\prime x_{N+M}^{\left(\frac{v_i^\prime - u_i^\prime}{u_{N+M}}\right)}, \ \ 1 \leq i \leq N+M.
\]
\end{proposition}
The above proposition\footnote{As of October 2023, we have posted a paper on arXiv on birational quasi-isomorphisms, where a more clear and laconic formulation of this proposition is provided.} is a generalization of Lemma~8.4 from~\cite{exotic} to the case of generalized cluster structures. The map $\theta$ is an instance of a \emph{quasi-isomorphism} defined by Fraser in~\cite{fraser}. 

\paragraph{Quiver.} 
A \emph{quiver} is a directed multigraph with no $1$- and $2$-cycles. Pick an extended seed $(\tilde{\mathbf{x}},\tilde{B},\mathcal{P})$ and let $D:=\diag(d_1^{-1},\ldots, d_N^{-1})$ be a diagonal matrix with $d_i$'s defined as above. Assume that $DB$ is skew-symmetric, where $B$ is the principal part of $\tilde{B}$. Then the matrix 
\[
\hat{B}:= \begin{bmatrix}
DB & \tilde{B}^{[N+1,N+M]} \\
-(\tilde{B}^{[N+1,N+M]})^T & 0
\end{bmatrix}
\]
is the adjacency matrix of a quiver $Q$, in which each vertex $i$ corresponds to a variable $x_i \in \tilde{\mathbf{x}}$. The vertices that correspond to cluster variables are called \emph{mutable}, the vertices that correspond to stable variables are called \emph{frozen}, and the vertices that correspond to isolated variables are called \emph{isolated}. For each $i$, the number $d_i$ is called the \emph{multiplicity} of the $i$th vertex. If one mutates the extended seed $(\tilde{x},\tilde{B},\mathcal{P})$, then the quiver of the new seed can be obtained from the initial quiver via the following steps:
\begin{enumerate}[1)]
\item For each path $i \rightarrow k \rightarrow j$, add an arrow $i \rightarrow j$;
\item If there is a pair of arrows $i \rightarrow j$ and $j \rightarrow i$, remove both;
\item Flip the orientation of all arrows going in and out of the vertex $k$.
\end{enumerate} 
The above process is also called a \emph{quiver mutation in direction $k$} (or \emph{at vertex $k$}). Instead of describing the matrix $\tilde{B}$, we describe the corresponding quiver and multiplicities $d_i$'s.
\subsection{Poisson-Lie groups}\label{s:plgrps}
In this section, we briefly recall relevant concepts from Poisson geometry. A more detailed account can be found in \cite{chari}, \cite{etingof} and \cite{rs}.

\paragraph{Poisson-Lie groups.} A \emph{Poisson bracket $\{\cdot, \cdot\}$} on a commutative algebra is a Lie bracket that satisfies the Leibniz rule in each slot. Given a manifold $M$, a Poisson bivector field on $M$ is a section $\pi \in \Gamma(M,\bigwedge^2 TM)$ such that $\{f,g\}:=\pi(df\wedge dg)$ is a Poisson bracket on the space of smooth functions on $M$. A Lie group $G$ endowed with a Poisson bivector field $\pi$ is called a \emph{Poisson-Lie group} if for any $g,h \in G$, $\pi_{gh} = (dL_g \otimes dL_g)\pi_h + (dR_h \otimes dR_h)\pi_g$, where $L_g$ and $R_h$ are the left and right translations by $g$ and $h$, respectively. Let $\mathfrak{g}$ be the Lie algebra of $G$ and $r \in \mathfrak g \otimes \mathfrak g$. If $G$ is a connected Lie group, then the bivector field\footnote{If $G$ is simple and complex, then any bivector field $\pi$ that yields the structure of a Poisson-Lie group on $G$ is of this form for some $r \in \mathfrak{g}\otimes\mathfrak{g}$.} $\pi_g := (dL_g\otimes dL_g) r - (dR_g \otimes dR_g)r$ defines the structure of a Poisson-Lie group on $G$ if and only if the following conditions are satisfied:
\begin{enumerate}[1)]
\item The symmetric part of $r$ is $\ad$-invariant;\label{i:rc1}
\item The $3$-tensor $[r,r]:=[r_{12},r_{13}] + [r_{12},r_{23}] + [r_{13},r_{23}]$ is $\ad$-invariant, where $(a\otimes b)_{12} = a \otimes b \otimes 1$, $(a \otimes b)_{13} = a \otimes 1 \otimes b$ and $(a\otimes b)_{23} = 1 \otimes a \otimes b$, $a,b \in \mathfrak{g}$. \label{i:rc2}
\end{enumerate}
The \emph{Classical Yang-Baxter equation (CYBE)} is the equation $[r,r] = 0$. For simple complex Lie algebras~$\mathfrak{g}$, Belavin and Drinfeld in~\cite{bd,bd2} classified solutions of the CYBE that have a nondegenerate symmetric part. The classification was partially extended by Hodges in~\cite{hodges} to the case of reductive complex Lie algebras (however, Hodges required the symmetric part of $r$ to be a multiple of the Casimir element). A full classification of solutions of the CYBE with an arbitrary nondegenerate $\ad$-invariant symmetric part in the case of reductive complex Lie algebras was obtained by Delorme in~\cite{bdr}.

\paragraph{The Belavin-Drinfeld classification.} Let $\mathfrak{g}$ be a reductive complex Lie algebra endowed with a nondegenerate symmetric invariant bilinear form $\langle\,,\,\rangle$, and let $\Pi$ be a set of simple roots of $\mathfrak{g}$. 
A \emph{Belavin-Drinfeld triple} (for conciseness, a \emph{BD triple}) is a triple $(\Gamma_1, \Gamma_2, \gamma)$ with $\Gamma_1, \Gamma_2 \subset \Pi$ and $\gamma: \Gamma_1 \rightarrow \Gamma_2$ a nilpotent isometry. The nilpotency condition means that for any $\alpha \in \Gamma_1$ there exists a number $j$ such that $\gamma^j(\alpha) \notin \Gamma_1$. Decompose $\mathfrak{g}$ as $\mathfrak{g} = \mathfrak{n}_+ \oplus \mathfrak{h} \oplus \mathfrak{n}_-$, where $\mathfrak{n}_+ = \bigoplus_{\alpha > 0} \mathfrak{g}_{\alpha}$ and $\mathfrak{n}_- = \bigoplus_{\alpha > 0} \mathfrak{g}_{-\alpha}$ are nilpotent subalgebras, $\mathfrak{g}_\alpha$ are root subspaces and $\mathfrak{h}$ is a Cartan subalgebra. For every positive root $\alpha$, choose $e_{\alpha} \in \mathfrak{g}_\alpha$ and $e_{-\alpha} \in \mathfrak{g}_{-\alpha}$ such that $\langle e_\alpha,e_{-\alpha}\rangle = 1$, and set $h_{\alpha}:=[e_{\alpha},e_{-\alpha}]$. 
Let $\mathfrak{g}_{\Gamma_1}$ and $\mathfrak{g}_{\Gamma_2}$ be the simple Lie subalgebras of $\mathfrak{g}$ generated by $\Gamma_1$ and $\Gamma_2$. Extend $\gamma$ to an isomorphism $\mathbb{Z}\Gamma_1 \xrightarrow{\sim} \mathbb{Z}\Gamma_2$, and then define $\gamma : \mathfrak{g}_{\Gamma_1} \rightarrow \mathfrak{g}_{\Gamma_2}$ via $\gamma(e_{\alpha}) = e_{\gamma(\alpha)}$ and $\gamma(h_{\alpha}) = h_{\gamma(\alpha)}$. Let $\gamma^* : \mathfrak{g}_{\Gamma_2} \rightarrow \mathfrak{g}_{\Gamma_1}$ be the conjugate of $\gamma$ with respect to the form on $\mathfrak{g}$. Extend both $\gamma$ and $\gamma^*$ to $[\mathfrak{g},\mathfrak{g}]$ via setting $\gamma|_{\mathfrak{g}_{\Gamma_1}^{\perp}} = 0$ and $\gamma^*|_{\mathfrak{g}_{\Gamma_2}^{\perp}} = 0$. For an element $r \in \mathfrak{g}\otimes \mathfrak{g}$, set $R_+, R_- :\mathfrak{g}\rightarrow\mathfrak{g}$ to be the linear transformations determined by $\langle R_+(x), y\rangle = \langle r, x\otimes y\rangle$ and $\langle R_-(y), x\rangle = -\langle r, x\otimes y\rangle$, $x, y \in \mathfrak{g}$. Let $\pi_{>}$, $\pi_{<}$ and $\pi_0$ be the projections onto $\mathfrak{n}_+$, $\mathfrak{n}_-$ and $\mathfrak{h}$, respectively. In terms of $R_+$ and $R_-$, the CYBE assumes the form
\begin{equation}\label{eq:rendcyb}
[R_+(x), R_+(y)] = R_+([R_+(x),y] + [x, R_-(y)]), \ \ x, y \in \mathfrak{g}.
\end{equation}
Let $R_0 : \mathfrak{h} \rightarrow \mathfrak{h}$ be a linear transformation that satisfies the following conditions:
\begin{equation}\label{eq:r0}
R_0 + R_0^* = \id_{\mathfrak{h}};
\end{equation}
\begin{equation}\label{eq:ralg}
R_0(\alpha-\gamma(\alpha)) = \alpha, \ \ \alpha \in \Gamma_1,
\end{equation}
where $\id_{\mathfrak{h}} : \mathfrak{h} \rightarrow \mathfrak{h}$ is the identity and $R_0^*$ is the adjoint of $R_0$. If $\mathfrak{g}$ is simple, then the solutions $R_0$ of equations~\eqref{eq:r0}-\eqref{eq:ralg} form an affine subspace of $\hom(\mathfrak{h},\mathfrak{h})$ (linear maps) of dimension $k_{\bg}(k_{\bg}-1)/2$.

\begin{theorem}\label{t:bd}(Belavin, Drinfeld)
Under the above setup, if
\begin{equation}\label{eq:rplus}
R_+ = \frac{1}{1-\gamma} \pi_{>} - \frac{\gamma^*}{1-\gamma^*}\pi_{<} + R_0 \pi_0,
\end{equation}
where $R_0$ is any solution of the system \eqref{eq:r0}-\eqref{eq:ralg}, then $R_+$ satisfies the CYBE~\eqref{eq:rendcyb}. Moreover, 
\begin{equation}\label{eq:plusmin}
R_+ - R_- = \id_{\mathfrak{g}}.
\end{equation}
Conversely, if $R_+ : \mathfrak{g} \rightarrow \mathfrak{g}$ is any linear transformation that satisfies equation \eqref{eq:plusmin}, then $R_+$ assumes the form~\eqref{eq:rplus} for a suitable decomposition of $\mathfrak{g}$, for some Belavin-Drinfeld triple and some choice of root vectors $e_{\alpha}$.
\end{theorem}
The matrix $R_+$ from the theorem is called a \emph{classical $R$-matrix}. In this form, the theorem follows from Theorem~6.3 in~\cite{hodges}. It is important that the form on $\mathfrak{g}$ is fixed; however, if $\mathfrak{g}$ is simple, then all nondegenerate symmetric invariant bilinear forms are multiples of one another, so the theorem yields a full classification of solutions $r \in \mathfrak{g}\otimes \mathfrak{g}$ of the CYBE with nondegenerate $\ad$-invariant symmetric parts.

\paragraph{The Drinfeld double.}
Let $G$ be a reductive complex connected Poisson-Lie group endowed with a nondegenerate symmetric invariant bilinear form on $\mathfrak{g}$ and with a Poisson bivector field defined as \[\pi_g := (dL_g \otimes dL_g)r - (dR_g\otimes dR_g)r\] for some $r \in \mathfrak{g}\otimes \mathfrak{g}$ that satisfies the conditions of Theorem~\ref{t:bd}. Let $R_+$ and $R_-$ be defined from $r$ as in the previous paragraph, and set $\mathfrak{d} := \mathfrak g \oplus \mathfrak g$ to be the direct sum of Lie algebras. Define a nondegenerate symmetric invariant bilinear form on $\mathfrak{d}$ as
\[
\langle (x_1,y_1), (x_2,y_2) \rangle = \langle x_1, x_2 \rangle - \langle y_1, y_2 \rangle, \ \ x_1, \,x_2,\, y_1,\, y_2 \in \mathfrak{g}.
\]
As a vector space, $\mathfrak{d}$ splits into the direct sum of the following isotropic Lie subalgebras:
\[
\mathfrak{g}^{\delta} := \{(x, x) \ | \ x \in \mathfrak{g}\}, \ \ \mathfrak{g}^* := \{(R_+(x), R_-(x)) \ | \ x \in \mathfrak{g}\}.
\]
Set $R^{\mathfrak{d}}_+ := P_{\mathfrak{g}^*}$, where $P_{\mathfrak{g}^*}$ is the projection of $\mathfrak{d}$ onto $\mathfrak{g}^*$, and let $r^{\mathfrak{d}} \in \mathfrak{d}\otimes \mathfrak{d}$ be the $2$-tensor that corresponds to $R^{\mathfrak{d}}_+$. Then $R_+^{\mathfrak{d}}$ yields the structure of a Poisson-Lie group on the Lie group $D(G):=G \times G$ via the Poisson bivector field $\pi^{\mathfrak{d}}_{(g,h)} := (dL_{(g,h)}\otimes dL_{(g,h)}) r^{\mathfrak{d}} - (dR_{(g,h)}\otimes dR_{(g,h)})r^{\mathfrak{d}}$, ${(g,h)} \in D(G)$. The Poisson-Lie group $D(G)$ is called the \emph{Drinfeld double} of $G$.

The Poisson bracket on $D(G)$ can be written in the form
\[
\{f_1,f_2\} = \langle R_+(E_L f_1), E_Lf_2 \rangle - \langle R_+(E_R f_1), E_Rf_2 \rangle + \langle \nabla_X^R f_1, \nabla^R_Y f_2\rangle - \langle \nabla_X^L f_1, \nabla_Y^L f_2\rangle,
\]
where $\nabla^L f_i = (\nabla_X^L f_i, -\nabla_Y^L f_i)$ and $\nabla^R f_i = (\nabla_X^R f_i, -\nabla_Y^R f_i)$ are the left and the right gradients, respectively, $E_Lf_i = \nabla^L_X f_i+ \nabla^L_Yf_i$ and $E_Rf_i = \nabla^R_Xf_i + \nabla^R_Yf_i$. We define the gradients on $G$ as\footnote{This convention is opposite to the one in~\cite{plethora} and~\cite{rs}, but in this way the left gradient is the gradient in the left trivialization, and the right gradient is the gradient in the right trivialization of the group.}
\[
\langle \nabla^L f|_g, x\rangle = \left. \frac{d}{dt}\right|_{t=0} f(g\exp (tx)), \ \ \ \langle \nabla^R f|_g, x\rangle = \left. \frac{d}{dt}\right|_{t=0} f(\exp(tx)g), \ \ g \in G, \ x \in \mathfrak{g}.
\]
 The group $G$ can be identified with the connected Poisson-Lie subgroup $G^{\delta}$ of $D(G)$ that corresponds to the Lie subalgebra $\mathfrak{g}^\delta$. The Poisson bracket $\{\cdot, \cdot \}_G$ on $G$ can be expressed as
\[
\{f_1,f_2\}_G = \langle R_+(\nabla^L f_1), \nabla^L f_2 \rangle - \langle R_+(\nabla^R f_1), \nabla^R f_2 \rangle.
\]
Additionally, the connected Poisson-Lie subgroup $G^*$ of $D(G)$ that corresponds to $\mathfrak{g}^*$ is called the \emph{dual Poisson-Lie group of $G$}. The Poisson structure on $G^*$ (which is induced from $D(G)$) can be modeled locally in the group $G$ via the map 
\[
G^* \ni (g,h) \mapsto gh^{-1} \in G.
\]
The image of this map is an open dense subset of $G$  denoted as $G^\dagger$ (however, the map is not injective in general).

Following \cite{plethora}, we consider a more general Poisson bracket on $G\times G$ that is defined by a pair of classical $R$-matrices $R_+^c$ and $R_+^r$ (the meaning of the upper indices is unveiled later in the text). For such a pair, the Poisson bracket is defined as
\begin{equation}\label{eq:brackgen}
\{f_1,f_2\} = \langle R_+^c(E_L f_1), E_Lf_2 \rangle - \langle R_+^r(E_R f_1), E_Rf_2 \rangle + \langle \nabla_X^R f_1, \nabla^R_Y f_2\rangle - \langle \nabla_X^L f_1, \nabla_Y^L f_2\rangle.
\end{equation}
We will frequently abuse the terminology and call $G\times G$ endowed with bracket~\eqref{eq:brackgen} the Drinfeld double of $G$. However, the bracket~\eqref{eq:brackgen} yields the structure of a Poisson-Lie group on $G\times G$ if and only if $R^c_+ = R^r_+$. 

\paragraph{Symplectic foliation and Poisson submanifolds.} Let $(M,\pi)$ be a Poisson manifold. An immersed submanfiold $S \subseteq M$ is called a \emph{Poisson submanifold} if $\pi|_S \in \Gamma(S,\bigwedge^2TS)$. Examples of Poisson submanifolds include nonsingular parts of the zero loci of frozen variables (see Section~\ref{s:pgfroz}). Let $\pi^\# : TM^* \rightarrow TM$ be a morphism of vector bundles defined as $\langle\pi^\#(\xi),\eta\rangle := \langle \pi,\xi\wedge \eta\rangle$, $\xi,\eta \in T_p^*M$, $p \in M$. 
The Poisson bivector $\pi$ is called \emph{nondegenerate} if $\pi^\#$ is an isomorphism of vector bundles. A \emph{symplectic leaf} is a maximal (by inclusion) connected Poisson submanifold $S$ of $M$ for which $\pi|_S$ is nondegenerate. It is a theorem that any Poisson manifold $M$ is a union of its symplectic leaves.
\subsection{Desnanot-Jacobi identities}
We will frequently use the following Desnanot-Jacobi identities, which can be easily derived from short Pl\"ucker relations:

\begin{proposition}\label{p:dj13} Let $A$ be an $m \times (m+1)$ matrix with entries in an arbitrary field. Then for any $1\leq i < j < k\leq (m+1)$ and $1\leq \alpha \leq m$, the following identity holds:
\[
\det A^{\hat{i}} \det A_{\hat{\alpha}}^{\hat{j}\hat{k}} + \det A^{\hat{k}} \det A_{\hat{\alpha}}^{\hat{i}\hat{j}} = \det A^{\hat{j}} \det A_{\hat{\alpha}}^{\hat{i}\hat{k}},
\]
where the hatted upper (lower) indices indicate that the corresponding column (row) is removed.
\end{proposition}

\begin{proposition}\label{p:dj22}
Let $A$ be an $m\times m$ matrix with entries in an arbitrary field. If $1 \leq i < j\leq m$ and $1\leq k < l\leq m$, then the following identity holds:
\[
\det A \det A^{\hat{i}\hat{j}}_{\hat{k}\hat{l}} = \det A^{\hat{i}}_{\hat{k}} \det A_{\hat{l}}^{\hat{j}} - \det A_{\hat{l}}^{\hat{i}} \det A_{\hat{k}}^{\hat{j}}.
\]
\end{proposition}

\section{Description of $D(\GL_n)$ and $\gc(\bg^r,\bg^c)$}\label{s:descr}
\subsection{BD graphs in type $A$}\label{s:bdgraphs}

In this section, we describe BD graphs that are attached to pairs of BD triples; the material is drawn from~\cite{plethora}. Let us identify the positive simple roots of $\sll_n(\mathbb C)$ with an interval $[1, n-1]$. We define $\bg^r := (\Gamma_1^r, \Gamma_2^r, \gamma_r)$ and $\bg^c := (\Gamma_1^c, \Gamma_2^c, \gamma_c)$ to be a pair of BD triples for $\sll_n(\mathbb C)$, and we name the first triple a \emph{row BD triple} and the other one a \emph{column BD triple}. Furthermore, if $\Gamma_1^r = \Gamma_1^c = \emptyset$, we call $(\bg^r, \bg^c)$ the \emph{standard} or \emph{trivial BD pair}.

\paragraph{BD graph for a pair of BD triples.}  The graph $G_{(\bg^r, \bg^c)}$ is defined in the following way. The vertex set of the graph consists of two copies of $[1,n-1]$, one of each is called the \emph{upper part} and the other one is the \emph{lower part}. We draw an edge between vertices $i$ and $n-i$ if they belong to the same part (if $i = n-i$, we draw a loop). If $\gamma_r(i) = j$, draw a directed edge from $i$ in the upper part to $j$ in the lower part; if $\gamma_c(i) = j$, draw a directed edge from $j$ in the lower part to $i$ in the upper part. The edges between vertices of the same part are called \emph{horizontal}, between different parts \emph{vertical}. Figure \ref{f:bdgraphs} provides two examples of BD graphs.
\begin{figure}[htb]
\begin{center}
\hspace{0.6in}\begin{subfigure}[t]{2.4in}
\includegraphics[scale=0.45]{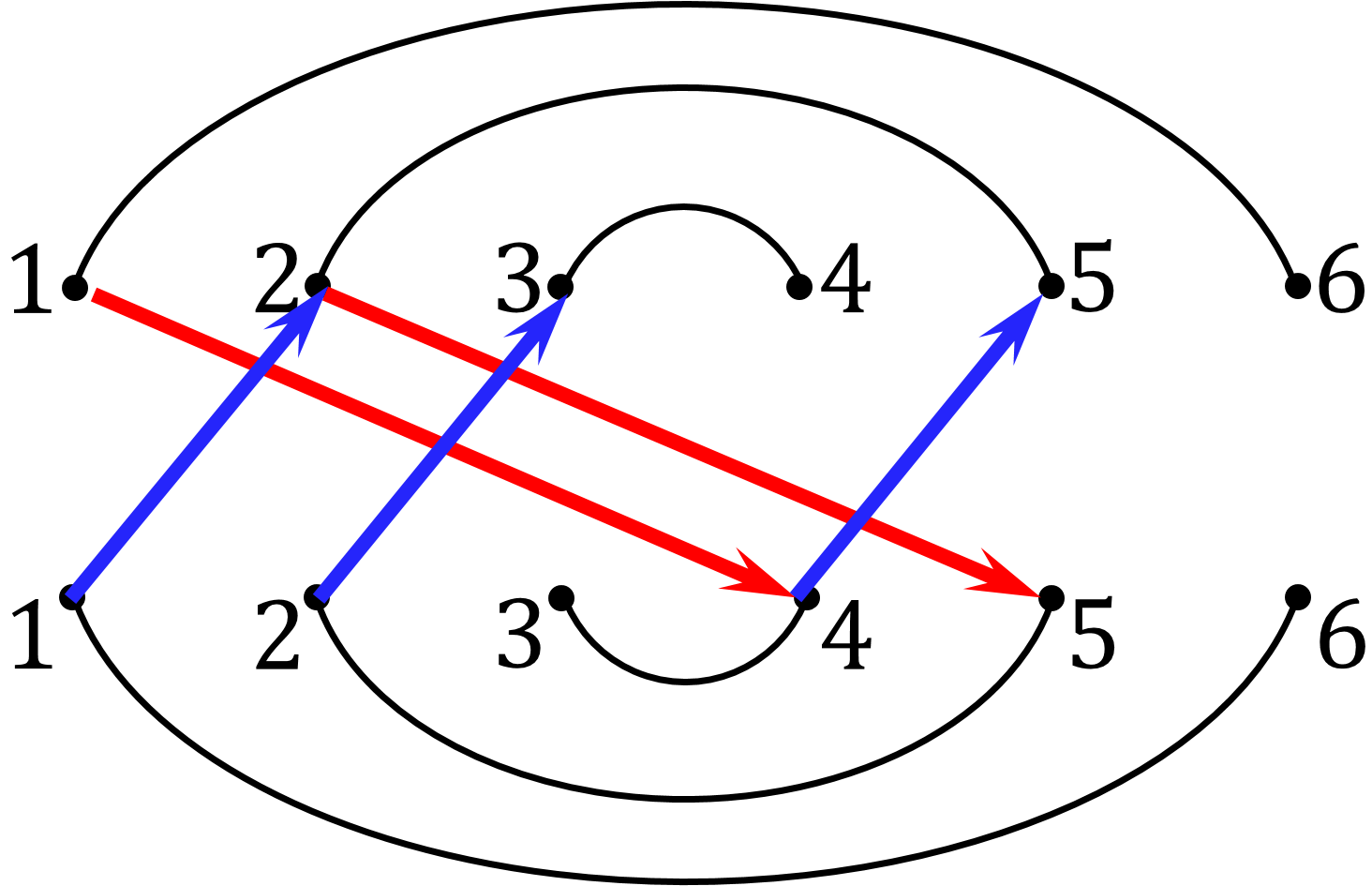}
\caption{$\Gamma_1^r = \{1,2\}$, $\Gamma_2^r = \{4,5\}$;\protect\newline\hphantom{\textbf{\emph{(a)}}} $\Gamma_1^c = \{2,3,5\}$, \ $\Gamma_2^c = \{1,2,4\}$.}
\label{f:bdgraphs_1}
\end{subfigure}\hspace{0.2in}
\begin{subfigure}[t]{2.4in}
\includegraphics[scale=0.45]{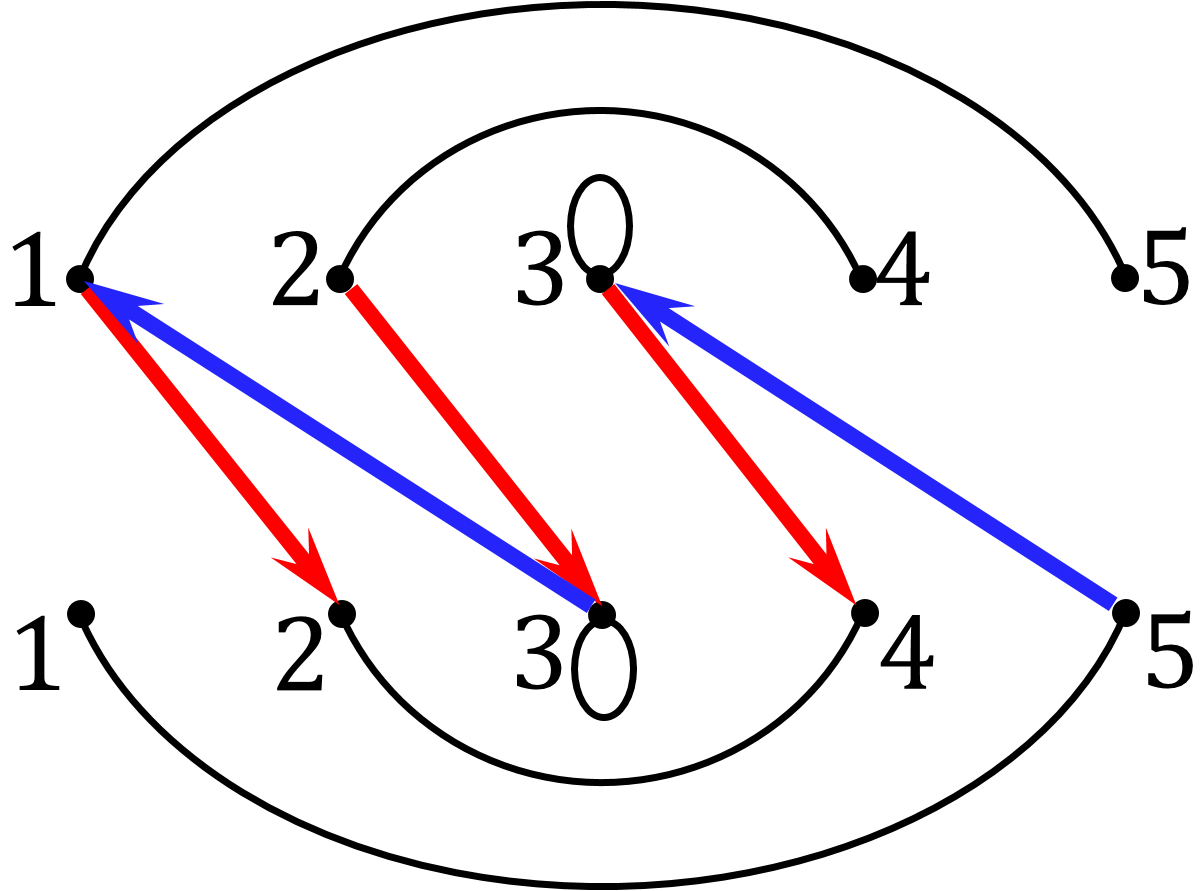}
\caption{$\Gamma_1^r = \{1,2,3\}$, $\Gamma_2^r = \{2,3,4\}$;\protect\newline\hphantom{\textbf{\emph{(b)}}} $\Gamma_1^c = \{1,3\}$, $\Gamma_2^c = \{3,5\}$.}
\label{f:bdgraphs_2}
\end{subfigure}
\caption{Examples of BD graphs. The vertical directed edges coming from $\bg^r$ and $\bg^c$ are painted in red and blue, respectively.}
\label{f:bdgraphs}
\end{center}
\end{figure}
\paragraph{Paths in $G_{(\bg^r,\bg^c)}$ and aperiodicity.} There is no orientation assigned to horizontal edges, hence we allow them to be traversed in both directions. An \emph{alternating path} in $G_{(\bg^r,\bg^c)}$ is a path in which horizontal and vertical edges alternate. A path is a \emph{cycle} if it starts where it ends. Now, we call the pair $(\bg^r, \bg^c)$ \emph{aperiodic} if $G_{(\bg^r, \bg^c)}$ has no alternating cycles (equivalently, the map $\gamma_c^{-1} w_0 \gamma_r w_0$ is nilpotent, where $w_0$ is the longest Weyl group element of $[1,n-1]$). In  examples in this paper, we denote alternating paths as $\cdots \rightarrow i \xrightarrow{X} i^\prime \xrightarrow{\gamma_r} j \xrightarrow{Y} j^\prime \xrightarrow{\gamma_c^*} i^\prime \rightarrow \cdots$, where $X$ and $Y$ indicate whether an edge is in the upper or the lower part of $G_{(\bg^r,\bg^c)}$, respectively, and $\gamma_r$ and $\gamma_c^*$ indicate a vertical edge directed downwards or upwards.
\paragraph{Oriented BD triples.} Let $\bg = (\Gamma_1,\Gamma_2,\gamma)$ be a BD triple. Since $\gamma$ is an isometry, if $\gamma(i) = j$ and $i+1 \in \Gamma_1$, then $\gamma(i+1) = j \pm 1$. We call the BD triple $\bg$ \emph{oriented} if $\gamma(i+1) = j+1$ for every $i \in \Gamma_1$ such that $i+1 \in \Gamma_1$. A pair of BD triples $(\bg^r, \bg^c)$ is called \emph{oriented} if both $\bg^r$ and $\bg^c$ are oriented. 
\paragraph{Runs.} Let $\bg = (\Gamma_1,\Gamma_2,\gamma)$ be a BD triple. For an arbitrary $i \in [1,n]$, set
\[
i_- := \max\{ j \in [0,n] \setminus \Gamma_1 \ | \ j < i\}, \ \ i_+ := \min \{j \in [1,n] \setminus \Gamma_1 \ | \ j \geq i \}.
\]
An \emph{$X$-run} of $i$ is the interval $\Delta(i) := [i_-+1, i_+]$. Replacing $\Gamma_1$ with $\Gamma_2$ in the above formulas, we obtain the definition of a \emph{$Y$-run} $\bar{\Delta}(i)$ of $i$. The $X$-runs partition the set $[1,n]$, and likewise the $Y$-runs. A run is called \emph{trivial} if it consists of a single element. Evidently, the map $\gamma$ can be viewed as a bijection between the set of nontrivial $X$-runs and the set of nontrivial $Y$-runs. For a pair of BD triples $(\bg^r,\bg^c)$, the runs that correspond to $\bg^r$ and $\bg^c$ are called \emph{row} and \emph{column} runs, respectively. We will indicate with an upper index $r$ or $c$ whether a run is from $\bg^r$ or $\bg^c$.

\begin{example}\label{ex:cgstr}
Consider the BD graphs on Figure \ref{f:bdgraphs}. Evidently, both are aperiodic and encode pairs of oriented BD triples. Here's a list of all runs that correspond to the graph on the right:
\begin{itemize}
\item Row runs: $\Delta_1^r = [1, 4],\ \Delta_2^r = [5],\ \Delta_3^r = [6]; \ \ \bar{\Delta}_1^r = [1], \ \bar{\Delta}_2^r = [2, 5], \ \bar{\Delta}_3^r = [6]$;
\item Column runs: $\Delta_1^c = [1,2], \ \Delta_2^c = [3,4],\  \Delta_3^c = [5], \ \Delta_3^c = [6]; \ \ \bar{\Delta}_1^c = [1],\ \bar{\Delta}_2^c = [2]$,  \mbox{$\bar{\Delta}_3^c = [3, 4]$}, $\bar{\Delta}_4^c = [5,6].$
\end{itemize}
\end{example}
\subsection{Construction of $\mathcal{L}$-matrices}\label{s:lmatr}

Let $X$ and $Y$ be two $n\times n$ matrices of indeterminates, which represent the standard coordinates on $\GL_n \times \GL_n$. For this section, fix an aperiodic oriented BD pair $\bg:=(\bg^r, \bg^c)$, and let $G_{\bg}$ be the BD graph associated with the pair, which is constructed in Section \ref{s:bdgraphs}. The construction described in this section follows Section 3.2 from \cite{plethora}.

We associate a matrix $\mathcal{L} = \mathcal{L}(X,Y)$ to every maximal alternating path in $G_{(\bg^r, \bg^c)}$ in the following way. If the path traverses a horizontal edge $i \rightarrow i^\prime$ in the upper part of the graph, we assign to the edge a submatrix of $X$ via\footnote{Note: if $j \in \Gamma_2^c$, then $\bar{\Delta}^c(j) = \bar{\Delta}^c(j+1)$, and similarly, if $i^\prime \in \Gamma_1^r$, then $\Delta^r(i^\prime) = \Delta^r(i^\prime + 1)$. Adding the ones in the formulas matters only for the beginning and the end of the path.}

\vspace{.1in}
\hspace{.85in}
\begin{minipage}{.8in}
{\large $X^{[1,\beta]}_{[\alpha,n]}$}
\end{minipage}
\begin{minipage}{3.5in}
$\beta:=$\ \  the right endpoint of $\Delta^c(i)$\ \  $= i_+(\Gamma_1^c)$;\\
$\alpha:=$\ \  the left endpoint of $\Delta^r(i^\prime+1)$\ \  $= (i^\prime+1)_-(\Gamma_1^r)+1$.
\end{minipage}
\\ \\
\noindent Similarly, we assign a submatrix of $Y$ to every horizontal edge $j^\prime \rightarrow j$ in the lower part of the graph that appears in the path via

\vspace{.1in}
\hspace{.85in}
\begin{minipage}{.8in}
{\large $Y^{[\bar{\beta},n]}_{[1,\bar{\alpha}]}$}
\end{minipage}
\begin{minipage}{3.5in}
$\bar{\beta}:=$\ \ the left endpoint of $\bar{\Delta}^c(j+1)$\ \ $= (j+1)_-(\Gamma_2^c)+1$;\\
$\bar{\alpha}:=$\ \ the right endpoint of $\bar{\Delta}^r(j^\prime)$\ \  $= j^\prime_+(\Gamma_2^r)$.
\end{minipage}
\\ \\
We call these submatrices $X$- and \emph{$Y$-blocks}, respectively. Now, the first block in the path becomes the bottom right corner of the matrix $\mathcal{L}$. As we move along the path, we collect the $X$- and $Y$-blocks and align them together according to the following patterns:

\noindent\hspace{0.5in}
\begin{minipage}{1.2in}
\[
i^\prime \xrightarrow{X} i \xrightarrow{\gamma_r} j \xrightarrow{Y} j^\prime
\]
\end{minipage}
\begin{minipage}{4in}
\vspace{5mm}
\[
\begin{array}{cc|}
\hline
\cdts  & y_{1n} \\
\vdts & \vdts \\
\cdts & y_{j_-+1,n} \\
\vdts & \vdts \\
\cdts & y_{jn} \\
\vdts & \vdts \\
\cdts & y_{j_+n} \\
\hline
 &\phantom{\vdots} \\
 & 
\end{array}
\begin{array}{cc}
 &\\
 &\phantom{\vdots} \\
\hline x_{i_-+1,1} \vphantom{y_{j_-+1},n} & \cdts \\
\vdts & \vdts \\
x_{i1} \vphantom{y_{jn}} & \cdts \\
\vdts & \vdts \\
x_{i_+,1} & \cdts \\
\vdts & \vdts \\
x_{n1} & \cdts \\
\hline 
\end{array}
\]
\end{minipage}

\vspace{5mm}
\noindent\hspace{0.5in}
\begin{minipage}{1.2in}
\[
j^\prime \xrightarrow{Y} j \xrightarrow{\gamma_c^*} i \xrightarrow{X} i^\prime
\]
\end{minipage}
\begin{minipage}{4in}
\begin{center}
\includegraphics[scale=0.5]{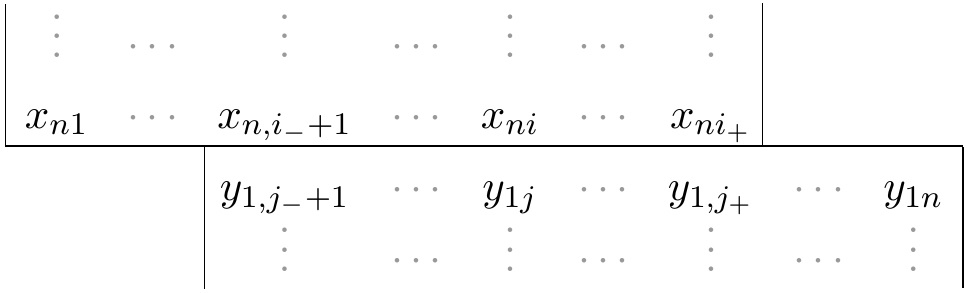}
\end{center}
\end{minipage}

\vspace{5mm}
\noindent The lower plus and minus in the first scheme correspond to $\bg^r$ (aligning along rows), and in the second scheme they correspond to $\bg^c$ (aligning along columns). In other words, once the first block is set in the bottom right part of $\mathcal{L}$, the algorithm of adding the blocks as one moves along the path can be described as follows:
\begin{enumerate}[1)]
\item If the $X$-block that corresponds to an edge $i^\prime \rightarrow i$ is placed in $\mathcal{L}$ and $\gamma_r(i) = j$, proceed to the edge $j\rightarrow j^\prime$ in the lower part of the graph and put the corresponding $Y$-block to the left of the $X$-block so that $y_{jn}$ and $x_{i1}$ are adjacent and belong to the same row;
\item If the $Y$-block that corresponds to an edge $j^\prime \rightarrow j$ is placed in $\mathcal{L}$ and $\gamma_c^*(j) = i$, proceed to the edge $i \rightarrow i^\prime$ in the upper part of the graph and put the corresponding $X$-block on top of the $Y$-block so that $x_{ni}$ and $y_{1j}$ are adjacent and belong to the same column;
\item Repeat until the path reaches its end.
\end{enumerate}

\begin{example}\label{ex:cg4}
Let $n=4$ and $\bg^r$ and $\bg^c$ be Cremmer-Gervais triples. In other words, $\gamma_r(i) = \gamma_c(i) = i+1$ for $i \in \{1, 2\}$ (see the BD graph below).

\vspace{3mm}
\begin{minipage}{1in}
\includegraphics[scale=0.45]{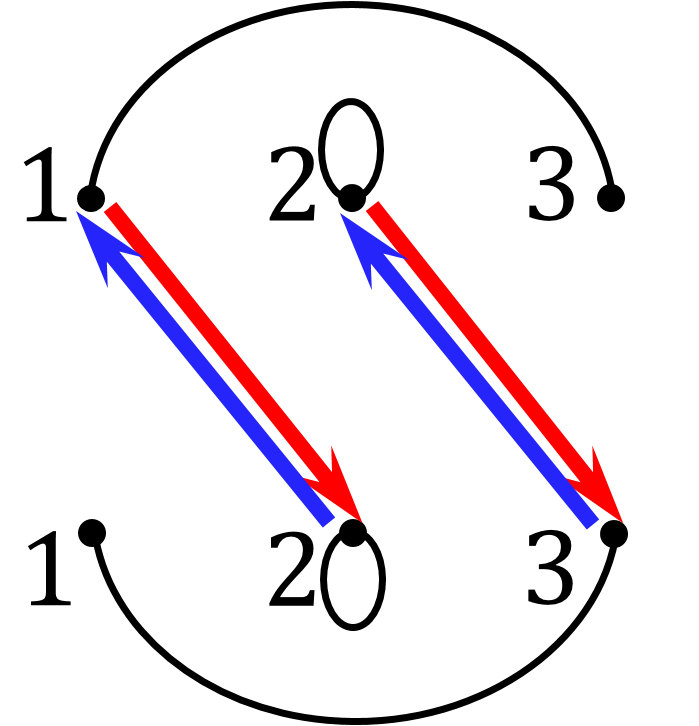}
\end{minipage}
\hspace{.4in}
\begin{minipage}{4in}
The runs in the upper part are $\Delta_1 = [1,3]$ and $\Delta_2 = [4]$; in the lower part, $\bar{\Delta}_1 = [1]$ and $\bar{\Delta}_2 = [2,4]$ (we don't leave an upper index, for $\bg^r = \bg^c$). There are two maximal alternating paths:
\[\begin{split}
&3 \xrightarrow{X} 1 \xrightarrow{\gamma_r} 2 \xrightarrow{Y} 2 \xrightarrow{\gamma_c^*}1 \xrightarrow{X} 3\\
&1 \xrightarrow{Y} 3 \xrightarrow{\gamma_c^*} 2 \xrightarrow{X} 2 \xrightarrow{\gamma_r} 3 \xrightarrow{Y} 1.
\end{split}\] 
\end{minipage}

\vspace{5mm}
\noindent Denote by $\mathcal{L}_1$ and $\mathcal{L}_2$ the matrices that correspond to these paths. For the first one, the blocks that correspond to the edges are $X_{[1,4]}^{[1,3]}$, $Y_{[1,4]}^{[2,4]}$, $X_{[1,1]}^{[1,3]}$, in the order in the path. Aligning these blocks according to the algorithm, we obtain $\mathcal{L}_1$ (see below). In a similar way one can obtain $\mathcal{L}_2$.
\[
\mathcal{L}_1(X,Y) = \begin{bmatrix}x_{41} & x_{42} & x_{43} & 0 & 0 & 0\\ y_{12} & y_{13} & y_{14} & 0 & 0 & 0\\ y_{22} & y_{23} & y_{24} & x_{11} & x_{12} & x_{13} \\ y_{32} & y_{33} & y_{34} & x_{21} & x_{22} & x_{23}\\ y_{42} & y_{43} & y_{44} & x_{31} & x_{32} & x_{33}\\  0 & 0 & 0 & x_{41} & x_{42} & x_{43}  \end{bmatrix}, \ \ \ 
\mathcal{L}_2(X,Y) = \begin{bmatrix}
y_{12} & y_{13} & y_{14} & 0 & 0 & 0\\
y_{22} & y_{23} & y_{24} & x_{11} & x_{12} & x_{13}\\
y_{32} & y_{33} & y_{34} & x_{21} & x_{22} & x_{23}\\
y_{42} & y_{43} & y_{44} & x_{31} & x_{32} & x_{33}\\
0 & 0 & 0 & x_{41} & x_{42} & x_{43}\\
0 & 0 & 0 & y_{12} & y_{13} & y_{14}
\end{bmatrix}
\]
\end{example}

\paragraph{Properties of $\mathcal{L}$-matrices.} Observe the following:
\begin{itemize}
\item The blocks are aligned in such a way that the indices in the blocks that correspond to the runs $\Delta^r$ and $\bar{\Delta}^r$ (or $\Delta^c$ and $\bar{\Delta}^c$) are in the same rows (columns);
\item For any variable $x_{ij}$ or $y_{ji}$ with $i> j$, there is a unique $\mathcal{L}(X,Y)$ that contains it on the diagonal;
\item If $i_1 \rightarrow i_2 \rightarrow \cdots \rightarrow i_{2n}$ is the maximal alternating path that gives rise to $\mathcal{L}$, then the size $N(\mathcal{L}) \times N(\mathcal{L})$ of $\mathcal{L}$ can be determined as
\[
N(\mathcal{L}) = \sum_{k=1}^{n} i_{2k-1}.
\]
\end{itemize}
\subsection{Initial cluster and $\gc(\bg)$}\label{s:iniclust}
In this section, we describe the initial extended cluster of the generalized cluster structure $\gc(\bg)$ on $\GL_n\times\GL_n$ induced by an aperiodic oriented BD pair $\bg = (\bg^r,\bg^c)$, as well as the choice of the ground ring.

\paragraph{Description of $\varphi$-, $f$- and $c$-functions.}
Set $U:= X^{-1} Y$ and define
\[
F_{kl}(X,Y) := | X^{[n-k+1,n]} \ Y^{[n-l+1,n]}|_{[n-k-l+1,n]}, \ \ k, l \geq 1, \ k + l \leq n-1;
\]
\[
\Phi_{kl}(X,Y) = [(U^0)^{[n-k+1,n]} \, U^{[n-l+1,n]} \, (U^2)^{[n]} \, \cdots \, (U^{n-k-l+1})^{[n]}], \ \ k,l \geq 1, \ k+l \leq n;
\]
set $\tilde{\varphi}_{kl}(X,Y) := \det \Phi_{kl}(X,Y)$ and
\begin{equation}\label{eq:ftphi}
f_{kl}(X,Y) := \det F_{kl}(X,Y), \ \ \varphi_{kl}(X,Y):=s_{kl} (\det X)^{n-k-l+1} \tilde{\varphi}_{kl}(X,Y),
\end{equation}
where
\[
s_{kl} = \begin{cases}
(-1)^{k(l+1)} &n \text{ is even,}\\
(-1)^{(n-1)/2 + k(k-1)/2 + l(l-1)/2} &n \text{ is odd.}
\end{cases}
\]
All $f$- and $\varphi$-functions are considered as cluster variables. The $c$-functions are defined via
\[
\det (X+\lambda Y) = \sum_{i=0}^{n} \lambda^{i} s_i c_i(X,Y),
\]
where $s_i = (-1)^{i(n-1)}$. Note that $c_{0} = \det X$ and $c_n = \det Y$. The functions $c_1,\ldots,c_{n-1}$ are considered as isolated stable variables, and the only nontrivial string, which is attached to $\varphi_{11}$, is given by the tuple $(1,c_1,\ldots,c_{n-1},1)$.

\paragraph{Description of $g$- and $h$-functions.} For $i > j$, let $\mathcal{L}$ be an $\mathcal{L}$-matrix such that $\mathcal{L}_{ss}(X,Y) = x_{ij}$ for some $s$, and let $N(\mathcal{L})$ be the size of $\mathcal{L}$. Set $g_{ij} := \det \mathcal{L}^{[s,N(\mathcal{L})]}_{[s,N(\mathcal{L})]}$. Similarly, if $\mathcal{L}$ is such that $\mathcal{L}_{ss}(X,Y) = y_{ji}$, we set $h_{ji} := \det \mathcal{L}^{[s,N(\mathcal{L})]}_{[s,N(\mathcal{L})]}$. In addition, we let $g_{ii} := \det X^{[i,n]}_{[i,n]}$ and  $h_{ii} := \det Y^{[i,n]}_{[i,n]}$, $1 \leq i \leq n$. The functions $h_{11}$ and $g_{11}$, as well as the determinants of the $\mathcal{L}$-matrices, are considered as stable variables.

\paragraph{Conventions.} The following identifications are frequently used in the text:
\begin{equation}\label{eq:fconv}
  \begin{aligned}
f_{n-l,l} &:= \varphi_{n-l,l},& &1 \leq l \leq n-1;\\
f_{0,l} &:= h_{n-l+1,n-l+1},& &1 \leq l \leq n-1;\\
f_{k,0} & := g_{n-k+1,n-k+1},& &1 \leq k \leq n-1.
\end{aligned}
\end{equation}
The above equalities are set in concordance with the defining formulas for the variables, for which one simply extends the range of the allowed indices. Furthermore, for $g$-functions we set
\begin{equation}\label{eq:gconv}
g_{n+1,i+1} := \begin{cases}
h_{1,j+1}  &\text{if} \ \gamma_c^*(j) = i,\\
1 &\text{otherwise;}
\end{cases} \ \ \ g_{i,0} := \begin{cases}
h_{jn} & \text{if} \ \gamma_r(i)=j,\\
1 & \text{otherwise;}
\end{cases} 
\end{equation}
and for $h$-functions we set
\begin{equation}\label{eq:hconv}
h_{j+1,n+1} := \begin{cases}
g_{i+1,1} &\text{if} \ \gamma_r(i) = j,\\
1 &\text{otherwise;}
\end{cases} \ \ \ h_{0,j} := \begin{cases}
g_{ni} &\text{if}\ \gamma_c^*(j) = i,\\
1 &\text{otherwise.}
\end{cases}
\end{equation}
The meaning of these identifications follows from the following observation: If $g_{ni} = \det \mathcal{L}^{[s,N(\mathcal{L})]}_{[s,N(\mathcal{L})]}$ and $\gamma_c^*(j) = i$, then $h_{1,j+1} = \det \mathcal{L}^{[s+1,N(\mathcal{L})]}_{[s+1,N(\mathcal{L})]}$; similarly, if $h_{jn} = \det \mathcal{L}^{[s,N(\mathcal{L})]}_{[s,N(\mathcal{L})]}$ and $\gamma_r(i) = j$, then $g_{i+1,1} = \det \mathcal{L}^{[s+1,N(\mathcal{L})]}_{[s+1,N(\mathcal{L})]}$.

\paragraph{Description of $\gc(\bg)$.} The description of the initial quiver is given later in Section~\ref{s:quiver}. The initial extended cluster is given by the union
\[\begin{split}
\{g_{ij}, \ h_{ji} \ | \ 1 \leq j \leq i \leq n\} \cup \{f_{kl} \ | \ k, l \geq 1&, \ k+l \leq n-1 \}\cup \\ &\cup \{ \varphi_{kl} \ | \ k, l \geq 1, \ k+l \leq n \} \cup \{ c_{i} \ | \ 1 \leq i \leq n-1\}.
\end{split}
\]
Let $\mathcal{L}_1(X,Y), \ldots, \mathcal{L}_m(X,Y)$ be the list of all $\mathcal{L}$-matrices in $\gc(\bg)$. The ground ring $\hat{\mathbb{A}} = \hat{\mathbb{A}}(\gc(\bg))$ is set to be
\[
\hat{\mathbb{A}} := \mathbb{C}[c_1, \ldots, c_{n-1}, h_{11}^{\pm 1}, g_{11}^{\pm 1}, \det \mathcal{L}_1,\ldots, \det \mathcal{L}_m].
\]
All mutation relations are ordinary except the mutation at $\varphi_{11}$. It is given by
\begin{equation}\label{eq:phiexch11}
\varphi_{11} \varphi_{11}^\prime = \sum_{r=0}^n c_r \varphi_{21}^r \varphi_{12}^{n-r}.
\end{equation}
A variable $\psi$ is frozen if and only if either $\psi = c_i$ for $0 \leq i \leq n$, or $\psi = g_{i+1,1}$ for $i \in \Pi\setminus \Gamma_1^r$, or $\psi = h_{1,j+1}$ for $j \in \Pi \setminus \Gamma_2^c$.

\subsection{Operators and the bracket}\label{s:opers}
In this section, we describe various operators and their properties used throughout the text, especially in sections on compatibility.

\paragraph{The operators $\gamma, \gamma^* : \gl_n(\mathbb C) \rightarrow \gl_n(\mathbb C)$.} Let $\bg := (\Gamma_1,\Gamma_2,\gamma)$ be an oriented BD triple. Let $\Delta_1,\ldots,\Delta_k$ be the list of all nontrivial $X$-runs, and set $\bar{\Delta}_1,\ldots,\bar{\Delta}_k$ to be the list of the corresponding $Y$-runs, where $\gamma(\Delta_i) = \bar{\Delta}_i$, $1 \leq i \leq k$. Set $\gl(\Delta_i)$ to be a subalgebra of $\gl_n(\mathbb{C})$ of the matrices that are zero outside of the block $\Delta_i \times \Delta_i$ (and similarly for $\gl(\Delta_i)$). Define $\gamma_i : \gl_n(\Delta_i) \rightarrow \gl_n(\bar{\Delta}_i)$ to be the map that shifts the $\Delta_i\times \Delta_i$ block to $\bar{\Delta}_i \times \bar{\Delta}_i$. Then the map $\gamma: \gl_n(\mathbb{C}) \rightarrow \gl_n(\mathbb{C})$ is defined as the direct sum $\gamma:=\bigoplus_{i=1}^k\gamma_i$ extended by zero to $\gl_n(\mathbb{C})$. Similarly, one sets $\gamma_i^* : \gl_n(\bar{\Delta}_i)\rightarrow \gl_n(\Delta_i)$ to be the map that shifts the $\bar{\Delta}_i \times \bar{\Delta}_i$ block to $\Delta_i \times \Delta_i$. The map $\gamma^* : \gl_n(\mathbb{C}) \rightarrow \gl_n(\mathbb{C})$ is obtained as the direct sum $\gamma^*:=\bigoplus_{i=1}^k\gamma_i^*$ extended by zero to $\gl_n(\mathbb{C})$.

\begin{remark}The resulting maps were denoted in~\cite{plethora} as $\mathring{\gamma}$ and $\mathring{\gamma}^*$, in order to distinguish them from their $\sll_{n}$-counterparts that were constructed in Section~\ref{s:plgrps} (note: $\gamma|_{\sll_n(\mathbb{C})}$ may be different from the map constructed in Section~\ref{s:plgrps} on the Cartan subalgebra of $\sll_n(\mathbb{C})$).
\end{remark}

\begin{example}
Let us consider a BD pair defined by its BD graph below (note: $\Gamma_1^c = \emptyset$):

\vspace{3mm}
\begin{minipage}{1in}
\noindent\includegraphics[scale=0.45]{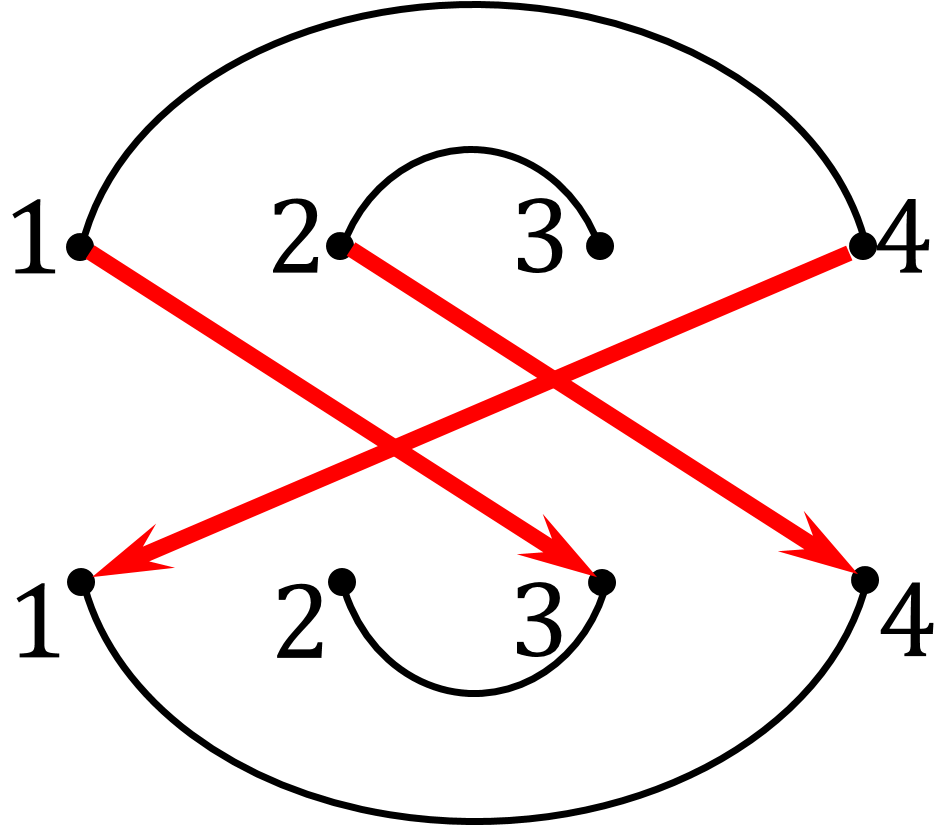}
\end{minipage}
\hspace{.7in}
\begin{minipage}{4.5in}
Let $\gamma:=\gamma_r$. Its action on $\gl_5(\mathbb C)$ is given by
\[
\gamma \begin{bmatrix} a_{11} & a_{12} & a_{13} & a_{14} & a_{15}\\a_{21} & a_{22} & a_{23} & a_{24} & a_{25}\\ a_{31} & a_{32} & a_{33} & a_{34} & a_{35}\\ a_{41} & a_{42} & a_{43} & a_{44} & a_{45}\\a_{51} & a_{52} & a_{53} & a_{54} & a_{55}\end{bmatrix} = \begin{bmatrix} a_{44} & a_{45} & 0 & 0 & 0 \\ a_{54} & a_{55} & 0 & 0 & 0 \\ 0 &0  &  a_{11} & a_{12} & a_{13}\\ 0& 0 & a_{21} & a_{22} & a_{23}\\0 & 0 & a_{31} & a_{32} & a_{33}   \end{bmatrix},
\]
\end{minipage}
\vspace{5mm}

\noindent Similarly, the action of $\gamma^*$ is given by
\[
\gamma^* \begin{bmatrix} a_{11} & a_{12} & a_{13} & a_{14} & a_{15}\\a_{21} & a_{22} & a_{23} & a_{24} & a_{25}\\ a_{31} & a_{32} & a_{33} & a_{34} & a_{35}\\ a_{41} & a_{42} & a_{43} & a_{44} & a_{45}\\a_{51} & a_{52} & a_{53} & a_{54} & a_{55}\end{bmatrix} = \begin{bmatrix}a_{33} & a_{34} & a_{35} & 0 & 0\\ a_{43} & a_{44} & a_{45} & 0 & 0 \\ a_{53} & a_{54} & a_{55} & 0 & 0 \\ 0 & 0 & 0 & a_{11} & a_{12} \\ 0 & 0 & 0 & a_{21} & a_{22}\end{bmatrix}.
\]
\end{example}

\paragraph{The group homomorphisms $\tilde{\gamma}$ and $\tilde{\gamma}^*$.} The maps $\gamma, \gamma^* : \gl_n(\mathbb C) \rightarrow \gl_n(\mathbb C)$ are not Lie algebra homomorphisms; however, their restrictions to the Borel subalgebras $\mathfrak{b}_+$ and $\mathfrak{b}_-$ are Lie algebra homomorphisms, hence we can define group homomorphisms $\tilde{\gamma}, \tilde{\gamma}^* : \mathfrak{B}_{\pm} \rightarrow \mathfrak{B}_{\pm}$, where $\mathfrak{B}_+$ and $\mathfrak{B}_-$ are the corresponding Borel subgroups. Notice that if the BD triple is oriented and $N_{\pm}$ is a unipotent (upper or lower) triangular matrix, then $\tilde{\gamma}(N_{\pm}) = \gamma(N_\pm - I) + I$, where $I$ is the identity matrix, and similarly for $\tilde{\gamma}^*$. Likewise, let $\GL(\Delta)\hookrightarrow\GL_n$ be the group of invertible $|\Delta|\times|\Delta|$ matrices viewed as a block in $\GL_n$ that occupies $\Delta\times \Delta$; since $\gamma : \gl(\Delta) \rightarrow \gl(\bar{\Delta})$ is an isomorphism of Lie algebras, it can be integrated to an isomorphism of groups $\gamma : \GL(\Delta) \rightarrow \GL(\bar{\Delta})$ (and similarly for $\gamma^*$).

\begin{remark}
The maps $\tilde{\gamma}$ and $\tilde{\gamma}^*$ were denoted in~\cite{plethora} as $\exp(\gamma)$ and $\exp(\gamma^*)$. We have changed the notation to avoid a possible confusion with the matrix exponential.
\end{remark}

\paragraph{Differential operators.} For a rational function $f \in \mathbb C(\GL_n \times \GL_n)$, set
\[
\nabla_X f := \left( \frac{\partial f}{\partial x_{ji}} \right) _{i,j=1}^{n}, \ \ \nabla_Y f := \left( \frac{\partial f}{\partial y_{ji}} \right) _{i,j=1}^{n}.
\]
Define
\begin{align*}
E_Lf &:= \nabla_X f \cdot X + \nabla_Y f \cdot Y, & E_R f &:= X \nabla_X f + Y \nabla_Y f, \\
\xi_L f &:= \gamma_c(\nabla_X f \cdot X) + \nabla_Y f \cdot Y, & \xi_R f &:= X\nabla_X f + \gamma_r^*(Y \nabla_Y f), \\
\eta_L f &:= \nabla_X f \cdot X + \gamma_c^*(\nabla_Y f \cdot Y), & \eta_R f&:= \gamma_r(X \nabla_X f) + Y \nabla_Y f.
\end{align*}
Let $\ell$ denote $r$ or $c$. Define subalgebras
\[
\mathfrak{g}_{\Gamma^\ell_1} := \bigoplus_{i=1}^{k} \gl(\Delta_i^\ell), \ \ \ \  \mathfrak{g}_{\Gamma^\ell_2} := \bigoplus_{i=1}^{k} \gl(\bar{\Delta}^\ell_i),
\]
where $\gl(\Delta_i^\ell)$ and $\gl(\bar{\Delta}_i^\ell)$ are constructed above. Let $\pi_{\Gamma_1^\ell}$ and $\pi_{\Gamma_2^\ell}$ be the projections onto $\mathfrak{g}_{\Gamma^\ell_1}$ and $\mathfrak{g}_{\Gamma^\ell_2}$, respectively; also, let $\pi_{\hat{\Gamma}_1^\ell}$ and $\pi_{\hat{\Gamma}_2^\ell}$ be the projections onto the orthogonal complements of $\mathfrak{g}_{\Gamma^\ell_1}$ and $\mathfrak{g}_{\Gamma^\ell_2}$ with respect to the trace form. There are numerous identities that relate the differential operators among each other and with the projections; they are easily derivable and extensively used in the paper. Let us mention some of them:
\begin{align*}
E_L &= \xi_L + (1-\gamma_c) (\nabla_X X), & E_R &= \xi_R + (1-\gamma_r^*)(Y\nabla_Y),\\
E_L &= \eta_L + (1-\gamma_c^*)(\nabla_Y Y), & E_R &= \eta_R + (1-\gamma_r)(X \nabla X),\\
\xi_L &= \gamma_c(\eta_L) + \pi_{\hat{\Gamma}_2^c}(\nabla_Y Y), & \xi_R &= \gamma_r^*(\eta_R) + \pi_{\hat{\Gamma}_1^r} (X \nabla_X),\\
\eta_L &= \gamma_c^*(\xi_L) + \pi_{\hat{\Gamma}_1^c} (\nabla_X X), & \eta_R &= \gamma_r(\xi_R) + \pi_{\hat{\Gamma}_2^r}(Y\nabla_Y).
\end{align*}

\paragraph{The bracket and $R_0$.} 
For any choice of $(R_0^r,R_0^c)$ on $\SL_n\times \SL_n$, the variables $c_0,c_1,\ldots,c_{n-1},c_n$ are Casimirs of the Poisson bracket. However, there is only one choice of $(R_0^r,R_0^c)$ for which these variables are Casimirs on $\GL_n\times \GL_n$:

\begin{enumerate}[a)]
\item The functions $c_0, c_1,\ldots, c_{n-1},c_n$ are Casimirs if and only if the identity matrix is an eigenvector of both $R_0^r$ and $R_0^c$ (in this case, $R_0^r(I) = R_0^c(I) = (1/2)I$ from $R_0 + R_0^* = \id_{\mathfrak h}$).
\end{enumerate}

\noindent However, there is an important alternative choice of $(R_0^r,R_0^c)$:

\begin{enumerate}[b)]
\item For a BD triple $(\Gamma_1,\Gamma_2,\gamma)$, a solution $R_0$ of the system~\eqref{eq:r0}-\eqref{eq:ralg} is such that
\begin{equation}\begin{aligned}\label{eq:roid}
R_0 (1-\gamma) &= \pi_{\Gamma_1} + R_0 \pi_{\hat{\Gamma}_1} \qquad&\qquad R_0(1-\gamma^*) &= -\gamma^* +R_0 \pi_{\hat{\Gamma}_2}\\
R_0^*(1-\gamma) &= -\gamma + R_0^*\pi_{\hat{\Gamma_1}}\qquad&\qquad R_0^*(1-\gamma^*) &= \pi_{\Gamma_2} + R_0^*\pi_{\hat{\Gamma}_2}
\end{aligned}\end{equation}
(the identities are viewed relative the Cartan subalgebra $\mathfrak{h}$ of $\gl_n$).
\end{enumerate}
Note that these conditions do not follow from the system~\eqref{eq:r0}-\eqref{eq:ralg}. For instance, if $I_{\Delta}:=\sum_{i\in\Delta}e_{ii}$, the first condition specifies the value of $R_0$ on $I_{\Delta}-I_{\bar{\Delta}}$ as
\[
R_0(I_{\Delta}-I_{\bar{\Delta}}) = I_{\Delta}.
\]
Choosing $R_0^r$ and $R_0^c$ that equation satisfy~\eqref{eq:roid} eases some of the computations with Poisson brackets, so this choice is employed in the proofs; however, in Section~\ref{s:depr} we show that the results of the paper hold regardless of the choice of $(R_0^r,R_0^c)$. Moreover, when $R_0:= R_0^r = R_0^c$ and $R_0$ satisfies equation~\eqref{eq:roid}, the connected Poisson dual $\GL_n^*$ of $\GL_n$ can be viewed as a subgroup of the direct product of certain parabolic subgroups modulo a relation (see Section~\ref{s:pgfroz} for details). Lastly, the Poisson bracket~\eqref{eq:brackgen} attains the following form on $\GL_n\times\GL_n$:
\[
\{f_1,f_2\} = \langle R_+^c (E_Lf_1), E_L f_2\rangle - \langle R_+^r(E_R f_1), E_R f_2\rangle + \langle X \nabla_X f_1, Y\nabla_Y f_2 \rangle - \langle \nabla_X f_1 \cdot X, \nabla_Y f_2 \cdot Y \rangle.
\]

\subsection{Invariance properties}
In this section, we describe the invariance properties of the functions from the initial extended cluster.
\paragraph{Invariance properties of $f$- and $\varphi$-functions.} Let $f$ be any $f$-function and $\tilde{\varphi}$ be any $\tilde{\varphi}$-function (recall that $\tilde{\varphi}$ differs from $\varphi$ by a factor of $\det X$; see equation~\eqref{eq:ftphi}). Pick any unipotent upper triangular matrix $N_+$, a pair of any unipotent lower triangular matrices $N_-$ and $N_-^\prime$, and let $A$ be any invertible matrix. Then
\begin{equation}\label{eq:invfphi}
f(X,Y) = f(N_+X N_-, N_+ Y N_-^\prime), \ \ \ \tilde{\varphi}(X,Y) = \tilde{\varphi}(AXN_-, AYN_-).
\end{equation}
Let $\mathfrak{b}_+$ and $\mathfrak{b}_-$ be the subspaces of upper and lower triangular matrices. The infinitesimal version of equation~\eqref{eq:invfphi} is 
\begin{equation}\begin{aligned}
\nabla_X f \cdot X, \ \nabla_Y f\cdot Y \in &\,\mathfrak{b}_-, \qquad&\qquad &E_Rf \in \mathfrak{b}_+;\\
E_L\tilde{\varphi} \in &\,\mathfrak{b}_-, \qquad&\qquad &E_R\tilde{\varphi} = 0.
\end{aligned}\end{equation}
Moreover, 
\begin{equation}\begin{aligned}\label{eq:invarfphi}
\pi_0 E_L \log f &= \text{const}, \qquad&\qquad \pi_0 E_R \log f &= \text{const}, \\
\pi_0 E_L \log \varphi &= \text{const}, \qquad&\qquad \pi_0 E_R \log \varphi &= \text{const},
\end{aligned}\end{equation}
where $\pi_0$ is the projection onto the space of diagonal matrices; by \emph{const} we mean that the left-hand sides of the formulas do not depend on $(X,Y)$. For the $c$-functions, 
\[
\pi_0 E_L \log c_i = \pi_0 E_R \log c_i = I, \ \ \ 0 \leq i \leq n,
\]
where $I$ is the identity matrix.

\paragraph{Invariance properties of $g$- and $h$-functions.} Let $\psi$ be any $g$- or $h$-function, and let $N_+$ and $N_-$ be any unipotent upper and lower triangular matrices. Then
\begin{equation}\label{eq:uniinvar}
\psi(N_+ X, \tilde{\gamma}_r(N_+)Y) = \psi(X\tilde{\gamma}_c^*(N_-),YN_-) = \psi(X,Y).
\end{equation}
Let $T$ be any diagonal matrix; then we also have
\begin{equation}\label{eq:diaginv}
\begin{aligned}
\psi(X\tilde{\gamma}_c^*(T),YT) &= \hat{\xi}_L(T)\psi(X,Y), &  \psi(TX,\tilde{\gamma}_r(T)Y) &= \hat{\xi}_R(T) \psi(X,Y),\\
\psi(XT,Y\tilde{\gamma}_c(T)) &= \hat{\eta}_L(T) \psi(X,Y), & \psi(\tilde{\gamma}_r^*(T)X,TY) &= \hat{\eta}_R(T) \psi(X,Y).
\end{aligned}\end{equation}
where $\hat{\xi}_R$, $\hat{\xi}_L$, $\hat{\eta}_R$ and $\hat{\eta}_L$ are constants that depend only on $T$ and $\psi$ (they can be viewed as characters on the group of invertible diagonal matrices). The infinitesimal version of equation $\eqref{eq:uniinvar}$ is
\begin{equation}\label{eq:infinvpsi}
\xi_L \psi \in \mathfrak{b}_-,\ \ \ \xi_R \psi \in \mathfrak{b}_+,
\end{equation}
and the infinitesimal version of equation \eqref{eq:diaginv} is
\begin{equation}
\begin{aligned}\label{eq:xiconst}
\pi_0 \xi_L \log \psi &= \text{const}, \qquad&\qquad \pi_0 \xi_R \log \psi &= \text{const},\\
\pi_0 \eta_L \log \psi &= \text{const}, \qquad&\qquad \pi_0 \eta_R \log \psi &= \text{const}.
\end{aligned}
\end{equation}
Finally, let us mention the results of Lemma 4.4 and Corollary~4.6 from~\cite{plethora}. If $\Delta^r$, $\Delta^c$, $\bar{\Delta}^r$ and $\bar{\Delta}^c$ are any $X$- and $Y$- row and column runs (trivial or not), then
\begin{equation}\begin{aligned}\label{eq:deltatraces}
\tr( (\nabla_X\log \psi\cdot X)_{\Delta^c}^{\Delta^c}) &= \text{const}, \qquad&\qquad \tr( (X \nabla_X\log \psi)_{\Delta^r}^{\Delta^r}) &= \text{const}, \\
\tr( (\nabla_Y \log \psi\cdot Y)_{\bar{\Delta}^c}^{\bar{\Delta}^c}) & = \text{const}, \qquad&\qquad \tr( (Y\nabla_Y \log \psi)_{\bar{\Delta}^r}^{\bar{\Delta}^r}) & = \text{const};
\end{aligned}\end{equation}
also,
\begin{equation}\label{eq:trconsti}\begin{aligned}
\tr( \nabla_X\log \psi\cdot X) &= \text{const}, \qquad&\qquad \tr(X \nabla_X\log \psi) &= \text{const}, \\
\tr( \nabla_Y \log \psi\cdot Y) & = \text{const}, \qquad&\qquad \tr(Y\nabla_Y \log \psi) & = \text{const}.
\end{aligned}
\end{equation}
\begin{remark}
Notice that there are four identities~\eqref{eq:diaginv} for diagonal elements and only two~\eqref{eq:uniinvar} for unipotent ones. The other two identities for unipotent matrices that one might think of do not hold.
\end{remark}
\subsection{Initial quiver}\label{s:quiver}
In this section, we describe the initial quiver for $\gc(\bg)$ defined by an aperiodic oriented BD pair $\bg = (\bg^r,\bg^c)$. We first describe the quiver for the trivial BD pair (based on~\cite{double}), and then we explain the necessary adjustments for a nontrivial BD pair. For particular examples of quivers, see Section~\ref{s:exs}. Throughout the section, we assume that $n \geq 3$ (the case $n=2$ is described in~\cite{double}).

\subsubsection{The quiver for the trivial BD pair}
Below one can find pictures of the neighborhoods of all variables in the initial quiver in the case of the trivial BD pair. A few of remarks beforehand:
\begin{itemize}
\item The circled vertices are mutable (in the sense of ordinary exchange relations~\eqref{eq:ordexchrel}), the square vertices are frozen, the rounded square vertices may or may not be mutable depending on the indices and the hexagon vertex is a mutable vertex with a generalized mutation relation (see equations~\eqref{eq:exchr} and~\eqref{eq:phiexch11});
\item Since $c_1,\ldots,c_{n-1}$ are isolated variables, they are not shown on the resulting quiver;
\item For $k=2$ and $n > 3$, the vertices $\varphi_{1k}$ and $\varphi_{k-1,2}$ coincide; hence, the pictures provided below suggest that there are two edges pointing from $\varphi_{21}$ to $\varphi_{12}$ (however, there is only one arrow in $n = 3$).
\end{itemize}

\begin{figure}[htb]
\begin{center}
\includegraphics[scale=0.75]{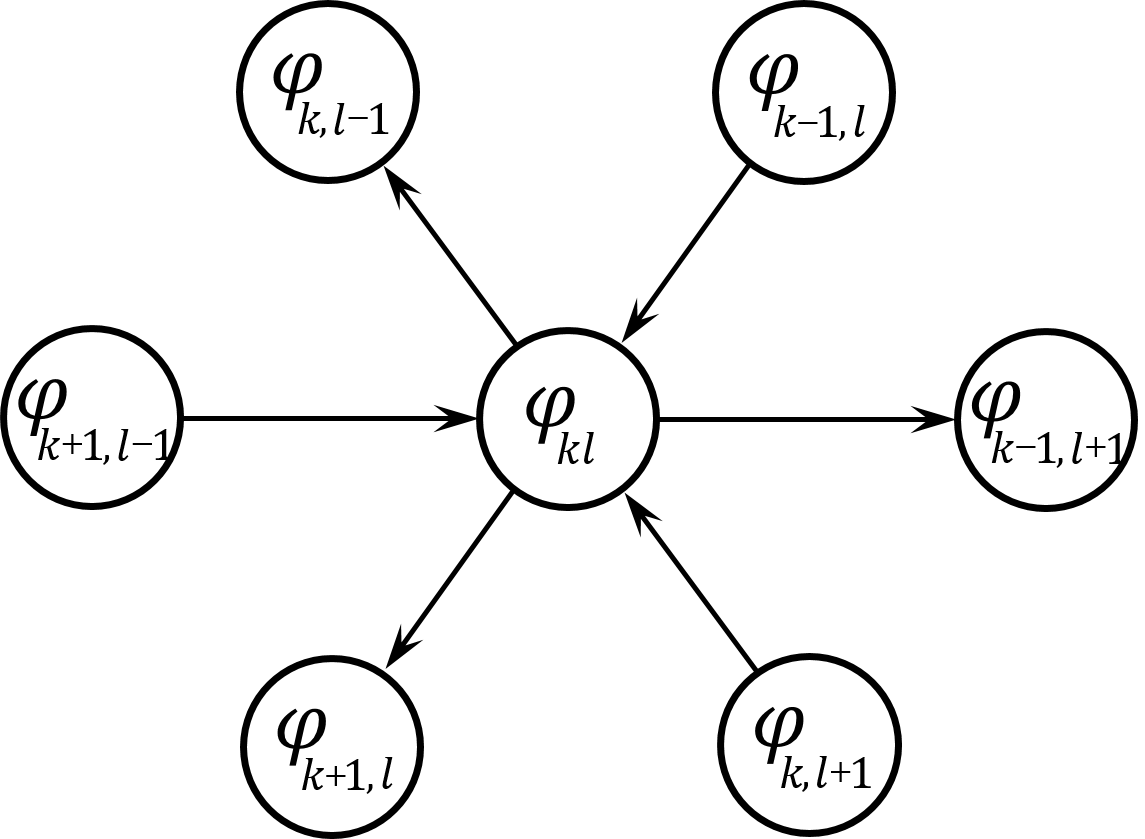}
\end{center}
\caption{The neighborhood of $\varphi_{kl}$ for $k,l \neq 1$, $k+l < n$.}
\label{f:nbd_phikl}
\end{figure}

\vspace{5mm}
\begin{figure}[htb]
\begin{subfigure}[t]{3in}
\begin{center}
\includegraphics[scale=0.75]{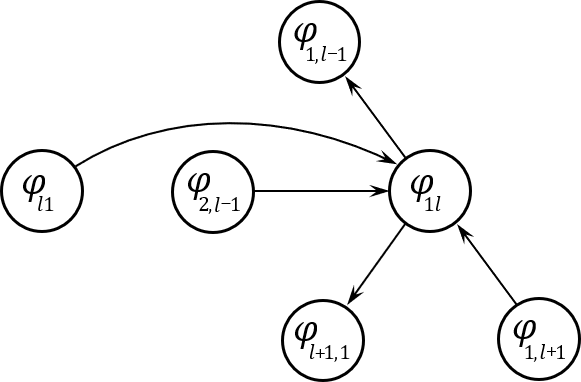}
\end{center}
\subcaption{Case $2 \leq l \leq n-2$.}
\label{f:nbd_phi1l}
\end{subfigure}
\begin{subfigure}[t]{2.6in}
\begin{center}
\includegraphics[scale=0.75]{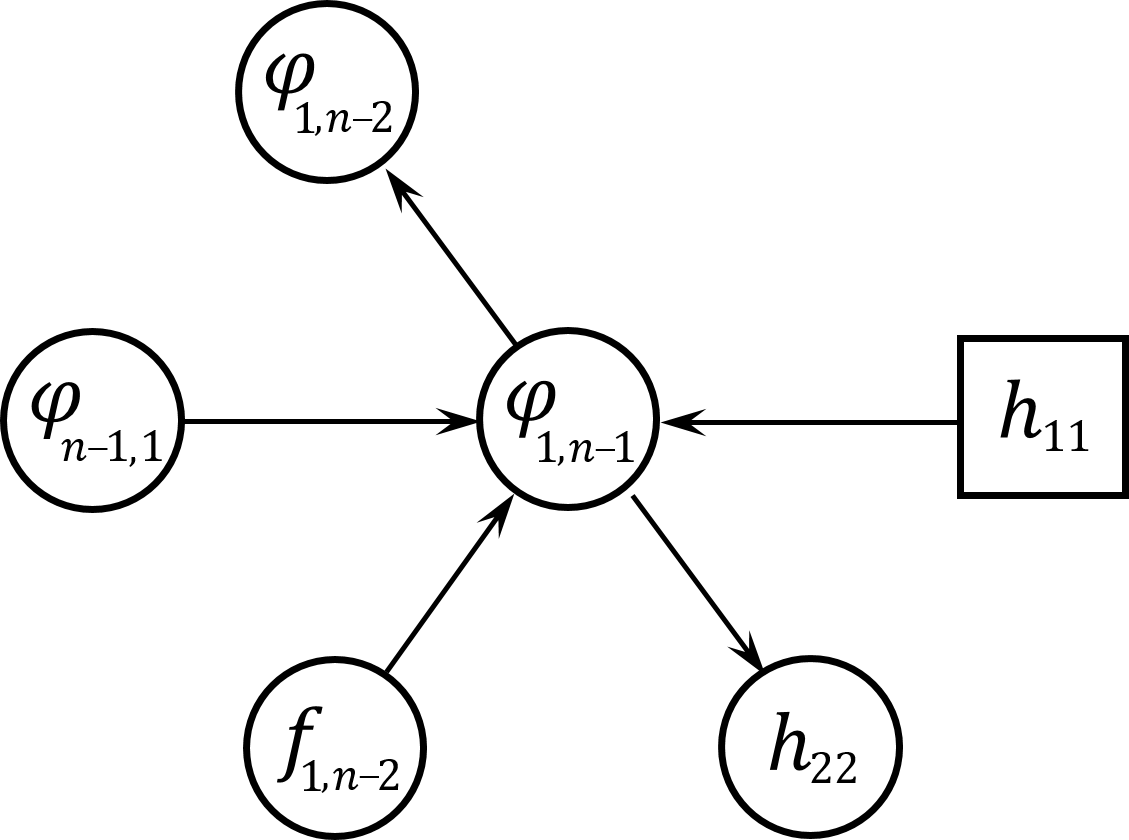}
\end{center}
\subcaption{Case $l = n-1$.}
\label{f:nbd_phi1n1}
\end{subfigure}
\caption{The neighborhood of $\varphi_{1l}$ for $2 \leq l \leq n-1$.}
\label{f:nbd_phi1}
\end{figure}

\vspace{5mm}
\begin{figure}[htb]
\begin{subfigure}[t]{2.9in}
\begin{center}
\includegraphics[scale=0.75]{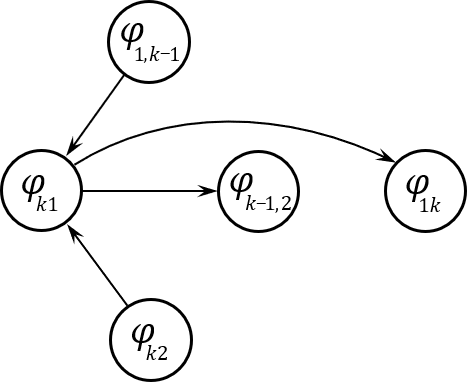}
\end{center}
\subcaption{Case $2 \leq k \leq n-2$.}
\label{f:nbd_phik1}
\end{subfigure}
\begin{subfigure}[t]{2.7in}
\begin{center}
\includegraphics[scale=0.75]{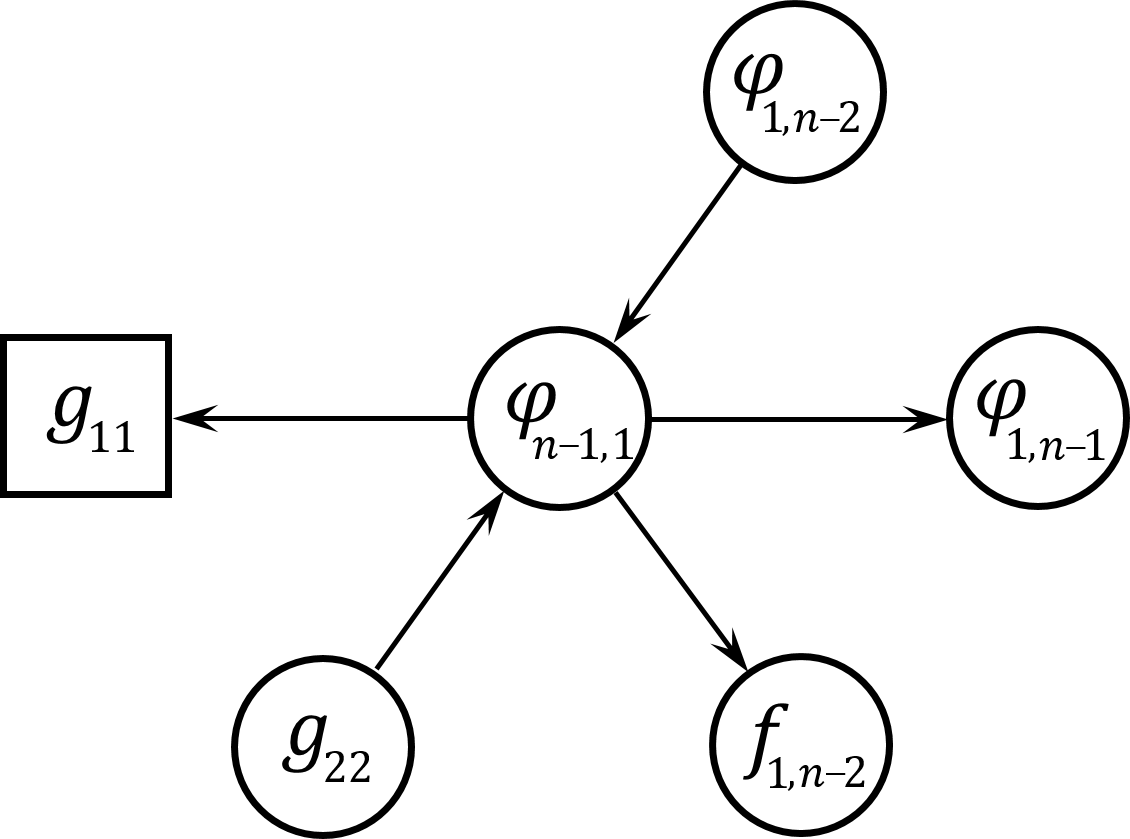}
\end{center}
\subcaption{Case $k = n-1$.}
\label{f:nbd_phin11}
\end{subfigure}
\caption{The neighborhood of $\varphi_{k1}$ for $2 \leq k \leq n-1$.}
\label{f:nbd_phik}
\end{figure}

\vspace{5mm}
\begin{figure}[htb]
\begin{subfigure}[t]{3in}
\begin{center}
\includegraphics[scale=0.75]{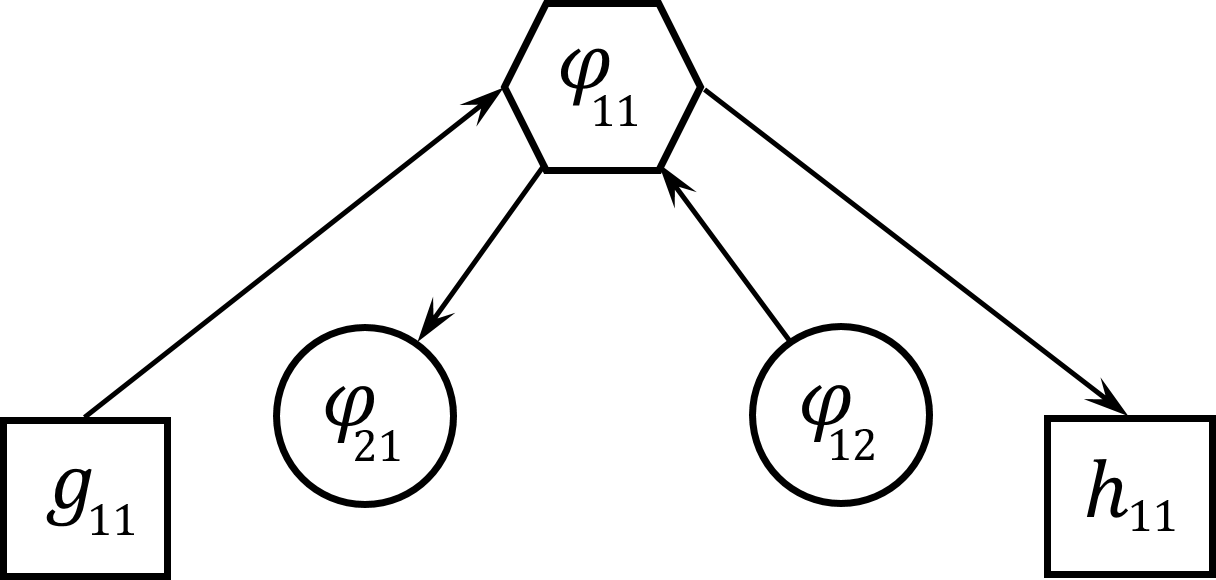}
\end{center}
\subcaption{Case $k=l=1$.}
\label{f:nbd_phi11}
\end{subfigure}
\begin{subfigure}[t]{2.5in}
\begin{center}
\includegraphics[scale=0.75]{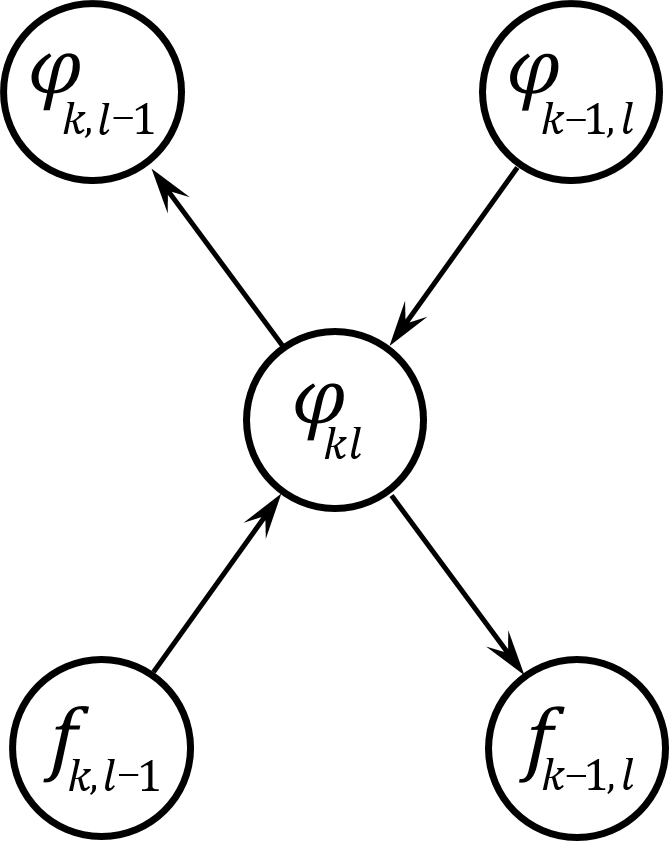}
\end{center}
\subcaption{Case $k + l = n$.}
\label{f:nbd_phikl_b}
\end{subfigure}
\caption{The neighborhood of $\varphi_{kl}$ for (a) $k=l=1$ and (b) $k+l=n$.}
\label{f:nbd_boundary}
\end{figure}

\vspace{5mm}
\begin{figure}[htb]
\begin{center}
\includegraphics[scale=0.75]{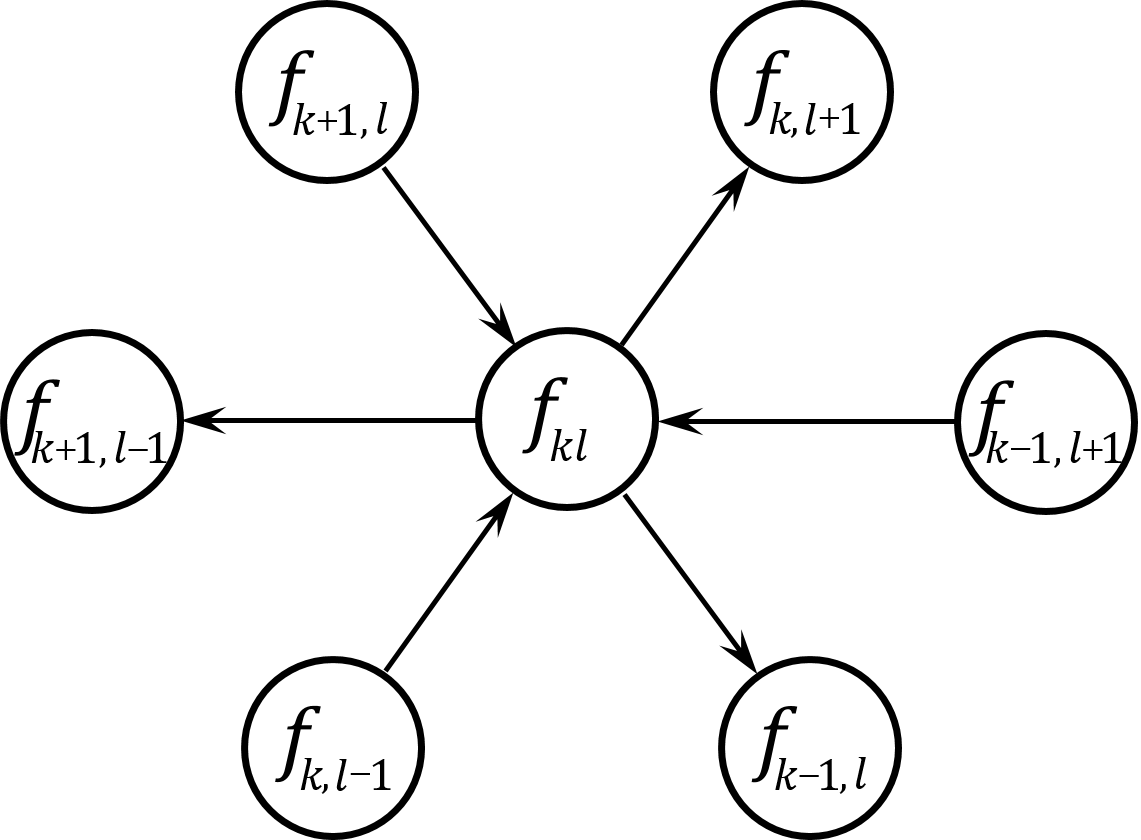}
\end{center}
\caption{The neighborhood of $f_{kl}$ for $k+l <n$ (convention~\eqref{eq:fconv} is in place).}
\label{f:nbd_fkl}
\end{figure}

\vspace{5mm}
\begin{figure}[htb]
\begin{subfigure}[t]{2.8in}
\begin{center}
\includegraphics[scale=0.75]{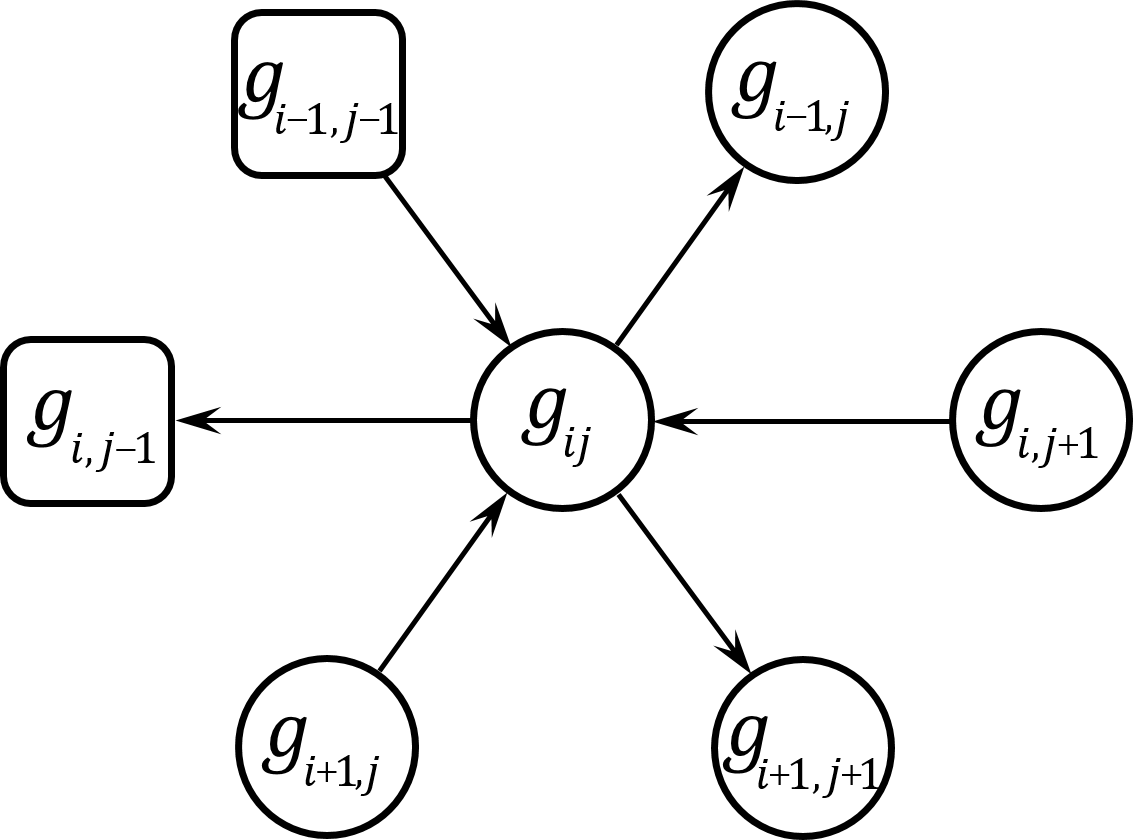}
\end{center}
\subcaption{Case $i < n$.}
\label{f:nbd_gij}
\end{subfigure}
\begin{subfigure}[t]{2.8in}
\begin{center}
\includegraphics[scale=0.75]{inbd_gnj}
\end{center}
\subcaption{Case $i = n$.}
\label{f:nbd_gnj}
\end{subfigure}
\caption{The neighborhood of $g_{ij}$ for $1 < j \leq i\leq n$ (convention~\eqref{eq:gconv} is in place).}
\label{f:nbd_g}
\end{figure}

\vspace{5mm}
\begin{figure}[htb]
\begin{subfigure}[t]{1.8in}
\begin{center}
\includegraphics[scale=0.75]{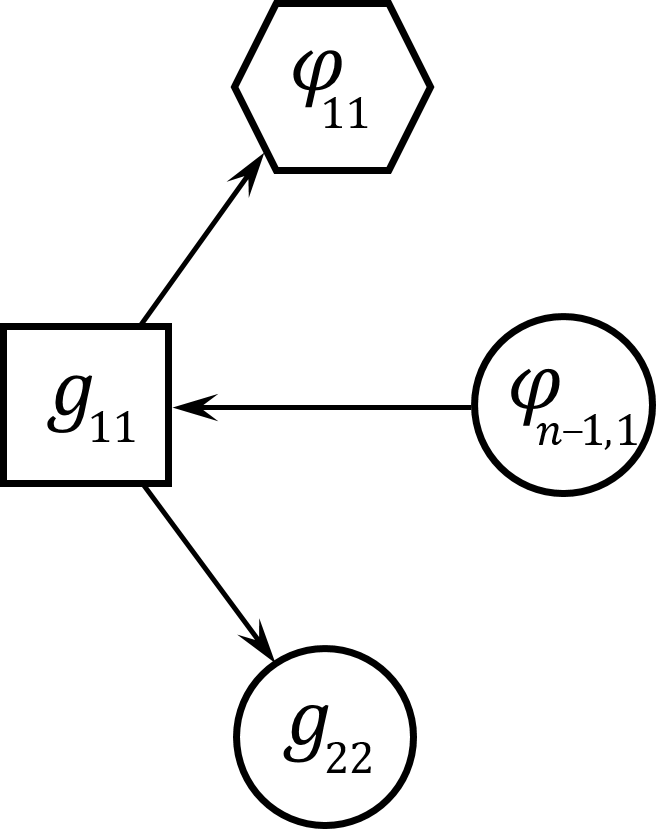}
\end{center}
\subcaption{Case $i=1$.}
\label{f:nbd_g11}
\end{subfigure}
\begin{subfigure}[t]{1.8in}
\begin{center}
\includegraphics[scale=0.75]{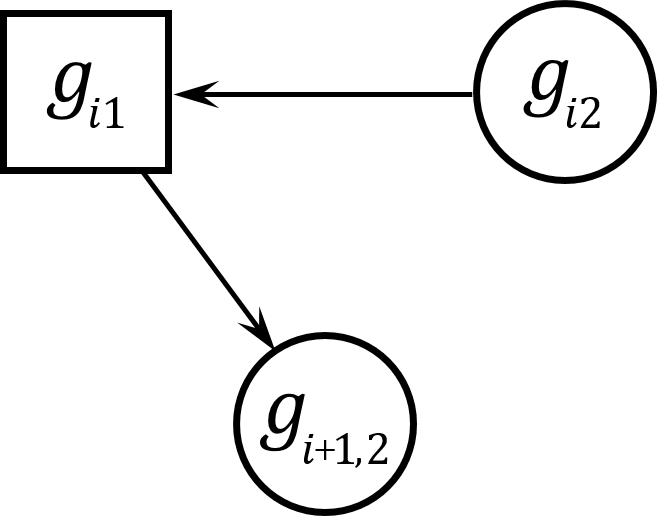}
\end{center}
\subcaption{Case $i < n$.}
\label{f:nbd_gi1}
\end{subfigure}
\begin{subfigure}[t]{1.8in}
\begin{center}
\includegraphics[scale=0.75]{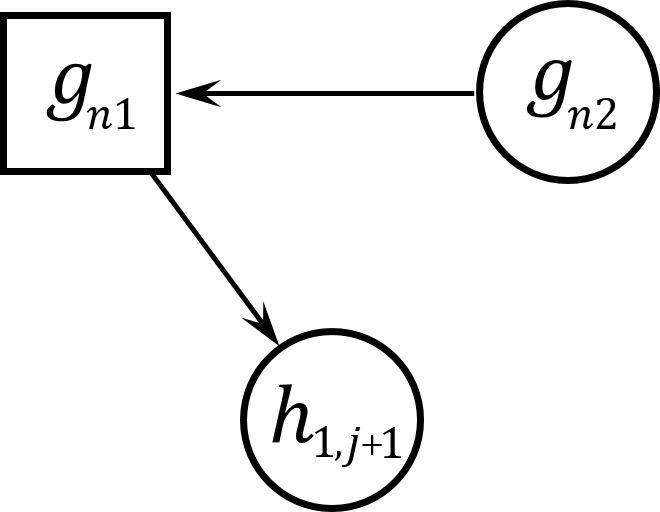}
\end{center}
\subcaption{Case $i=n$.}
\label{f:nbd_gn1}
\end{subfigure}
\caption{The neighborhood of $g_{i1}$ for $1 \leq i \leq n$.}
\label{f:nbd_g1}
\end{figure}

\vspace{5mm}
\begin{figure}[htb]
\begin{subfigure}[t]{2.8in}
\begin{center}
\includegraphics[scale=0.75]{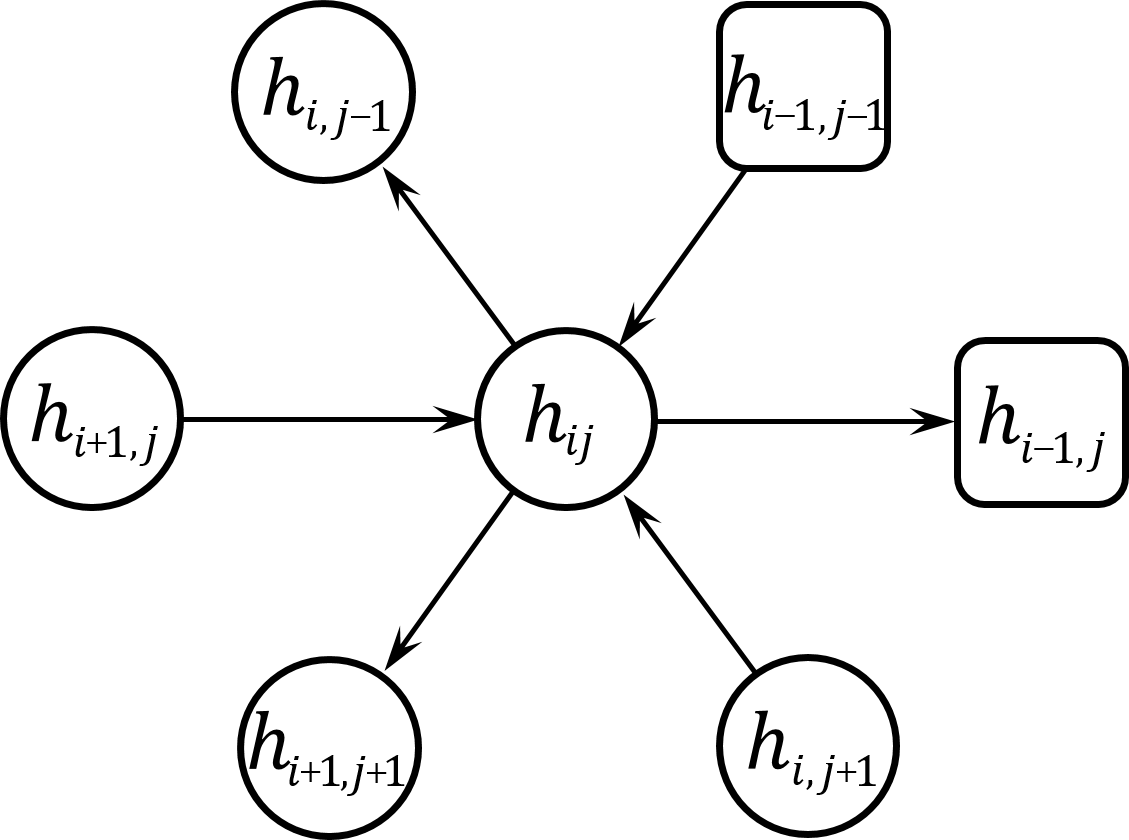}
\end{center}
\subcaption{Case $j < n$.}
\label{f:nbd_hij}
\end{subfigure}
\begin{subfigure}[t]{2.8in}
\begin{center}
\includegraphics[scale=0.75]{inbd_hin}
\end{center}
\subcaption{Case $j=n$.}
\label{f:nbd_hin}
\end{subfigure}
\caption{The neighborhood of $h_{ij}$ for $1 < i < j \leq n$.}
\label{f:nbd_h}
\end{figure}

\vspace{5mm}
\begin{figure}[htb]
\begin{subfigure}[t]{2.8in}
\begin{center}
\includegraphics[scale=0.75]{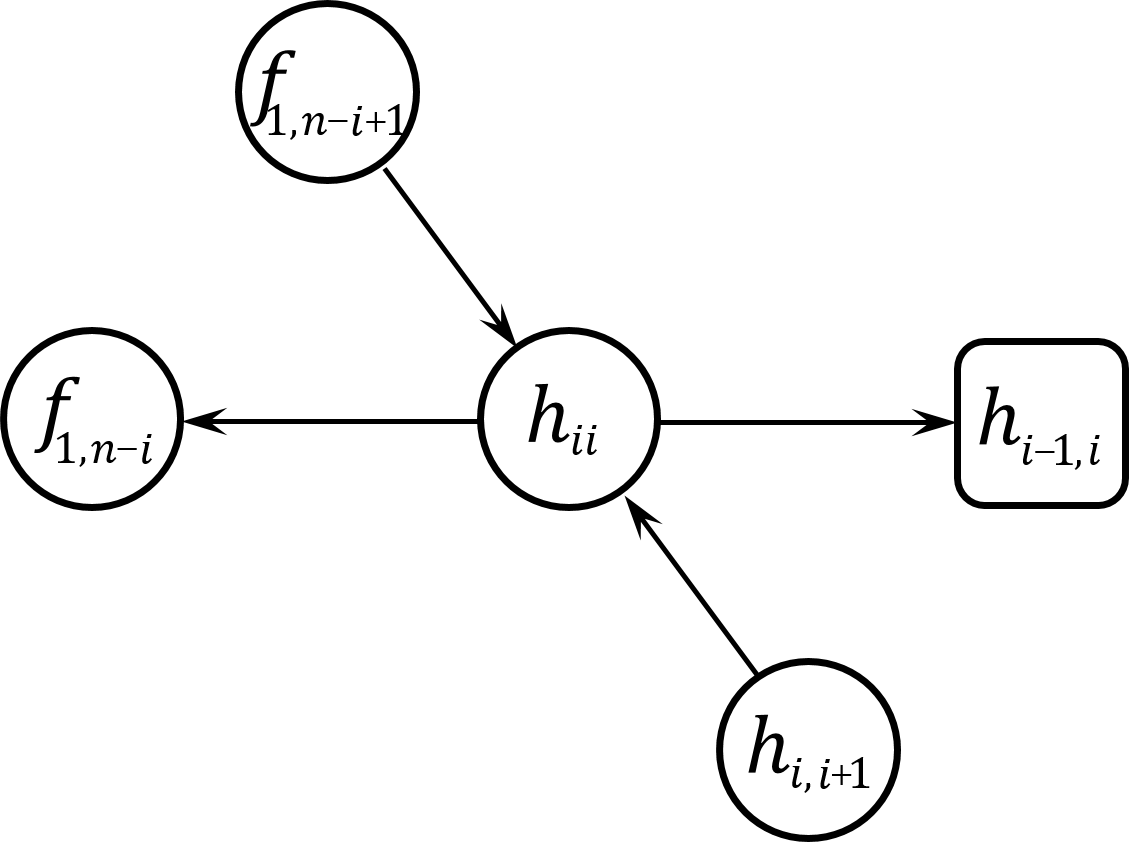}
\end{center}
\subcaption{Case $1 < i = j < n$.}
\label{f:nbd_hii}
\end{subfigure}
\begin{subfigure}[t]{2.8in}
\begin{center}
\includegraphics[scale=0.75]{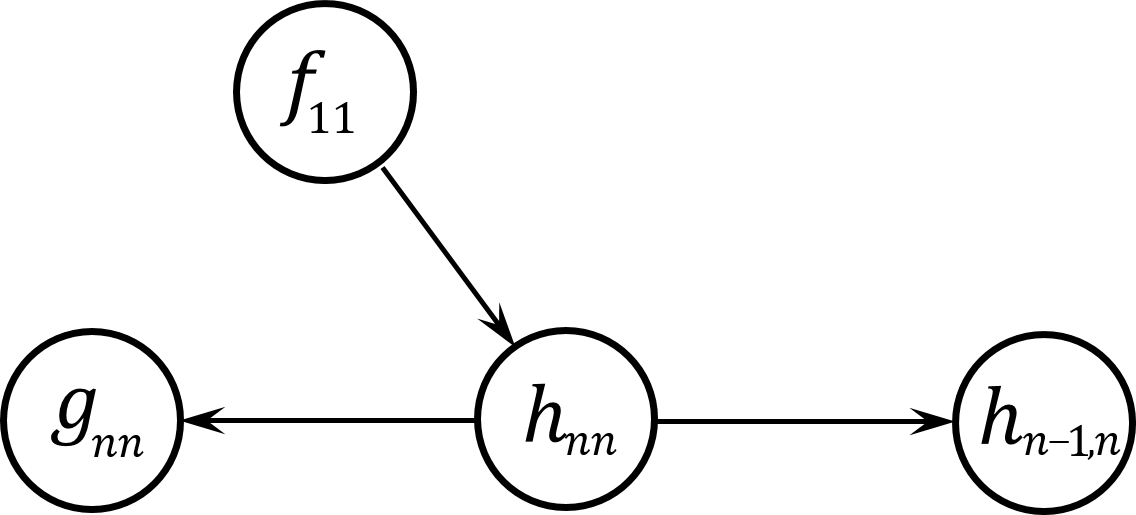}
\end{center}
\subcaption{Case $i = j = n$.}
\label{f:nbd_hnn}
\end{subfigure}
\caption{The neighborhood of $h_{ij}$ for $1 < i = j \leq n$.}
\label{f:nbd_hi}
\end{figure}
\vspace{5mm}
\begin{figure}[htb]
\begin{subfigure}[t]{1.8in}
\begin{center}
\includegraphics[scale=0.75]{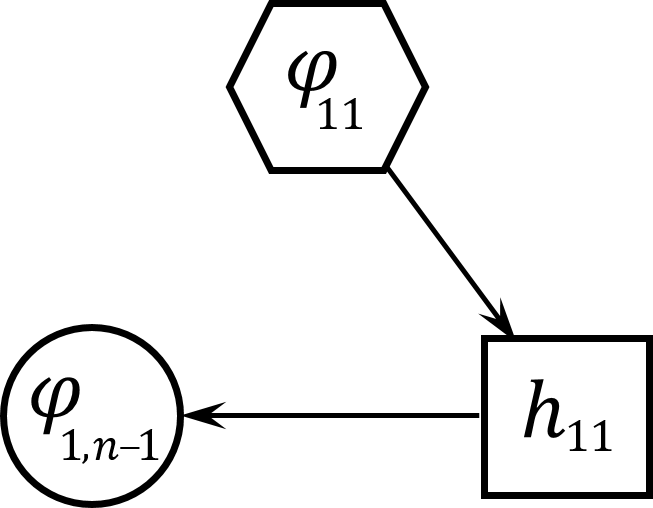}
\end{center}
\subcaption{Case $j=1$.}
\label{f:nbd_h11}
\end{subfigure}
\begin{subfigure}[t]{1.8in}
\begin{center}
\includegraphics[scale=0.75]{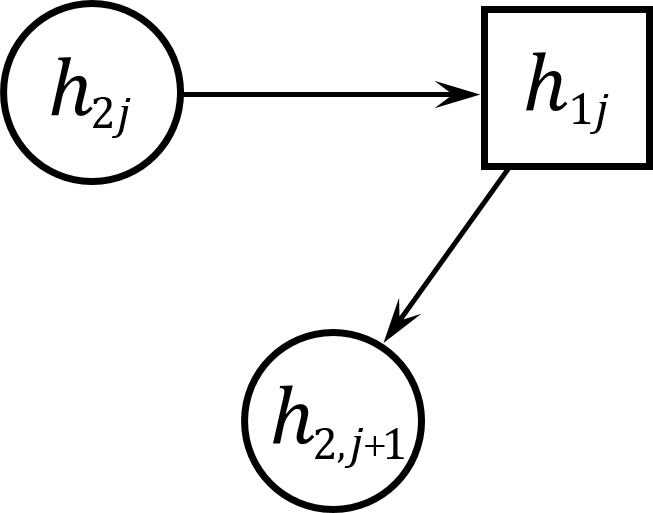}
\end{center}
\subcaption{Case $1 < j < n$.}
\label{f:nbd_h1j}
\end{subfigure}
\begin{subfigure}[t]{1.8in}
\begin{center}
\includegraphics[scale=0.75]{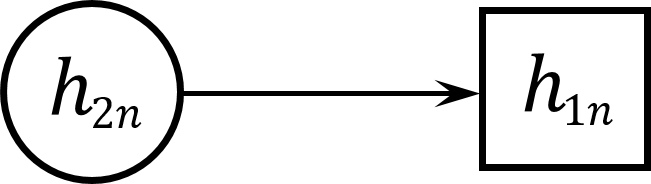}
\end{center}
\subcaption{Case $j=n$.}
\label{f:nbd_h1n}
\end{subfigure}
\caption{The neighborhood of $h_{1j}$.}
\label{f:nbd_h1}
\end{figure}

\clearpage

\subsubsection{The quiver for a nontrivial BD pair (algorithm)}
If $\bg = (\bg^r, \bg^c)$ is nontrivial, one proceeds as follows. First, draw the quiver for the case of the trivial BD pair, employing the neighborhoods as described above. Second, add new arrows as prescribed by the following algorithm:
\begin{enumerate}[1)]
\item If $i \in \Gamma_1^r$, unfreeze $g_{i+1,1}$ and add additional arrows, as indicated in Figure \ref{f:nbd_extra}\subref{f:nbd_gi11};
\item If $j \in \Gamma_2^c$, unfreeze $h_{1,j+1}$ and add additional arrows, as indicated in Figure \ref{f:nbd_extra}\subref{f:nbd_h1j1};
\item Repeat for all roots in $\Gamma_1^r$ and $\Gamma_2^c$.
\end{enumerate}
Note that the algorithm does not depend on the order of the roots of $\Gamma_1^r$ and $\Gamma_2^c$. Indeed, adding new arrows corresponds to adding a certain matrix (determined by the figure) to the current adjacency matrix of the quiver; since addition of matrices is commutative, the order of the roots is irrelevant. 

\vspace{5mm}
\begin{figure}[htb]
\begin{subfigure}[t]{2.8in}
\begin{center}
\includegraphics[scale=0.75]{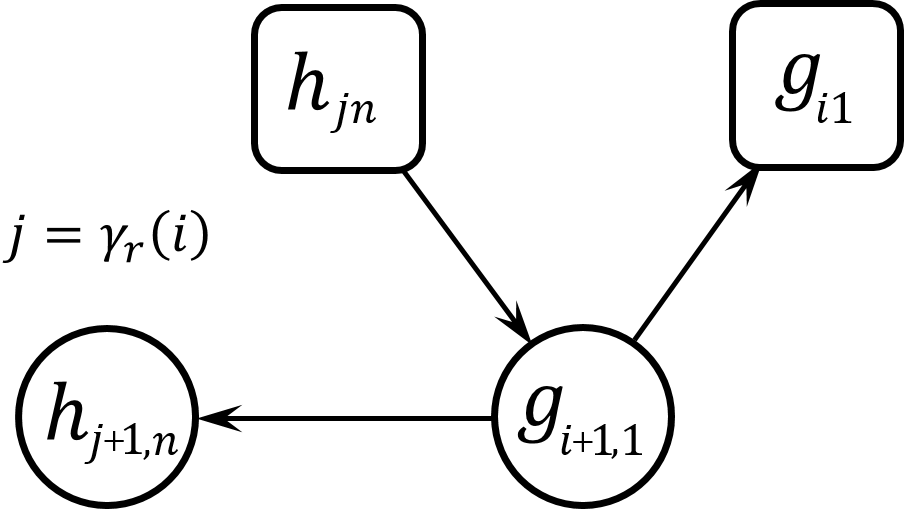} 
\end{center}
\subcaption{Case $i \in \Gamma_1^r$.}
\label{f:nbd_gi11}
\end{subfigure}
\begin{subfigure}[t]{2.8in}
\begin{center}
\includegraphics[scale=0.75]{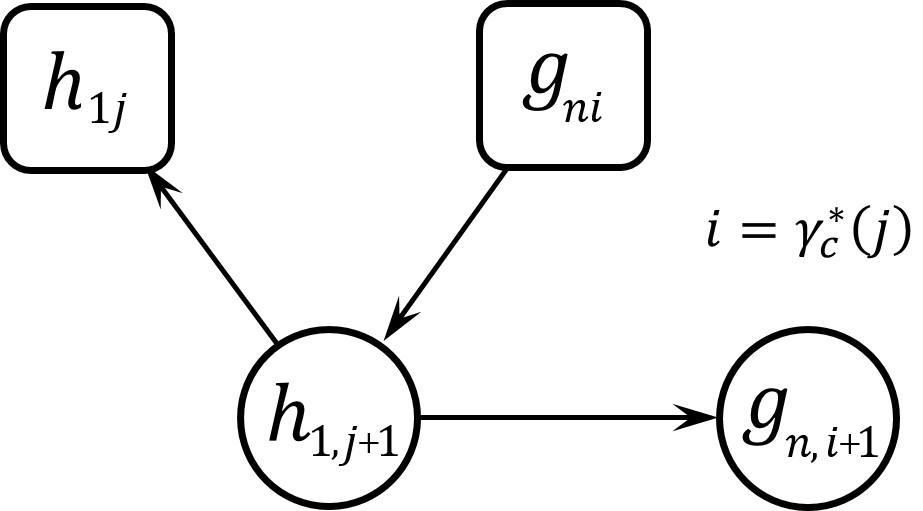}
\end{center}
\subcaption{Case $j \in \Gamma_2^c$.}
\label{f:nbd_h1j1}
\end{subfigure}
\caption{Additional arrows for $g_{i+1,1}$ and $h_{1,j+1}$.}
\label{f:nbd_extra}
\end{figure}
\subsubsection{The quiver for a nontrivial BD pair (explicit)}
As an alternative to the algorithm described in the previous paragraph, we provide explicit neighborhoods of the variables $g_{i1}$, $h_{1i}$, $g_{ni}$, $h_{in}$, $1 \leq i \leq n$ in the case of a nontrivial BD pair. All the other neighborhoods are the same as in the case of the trivial BD pair.

\vspace{5mm}
\begin{figure}[htb]
\begin{center}
\begin{subfigure}[t]{2.6in}
\begin{center}
\includegraphics[scale=0.75]{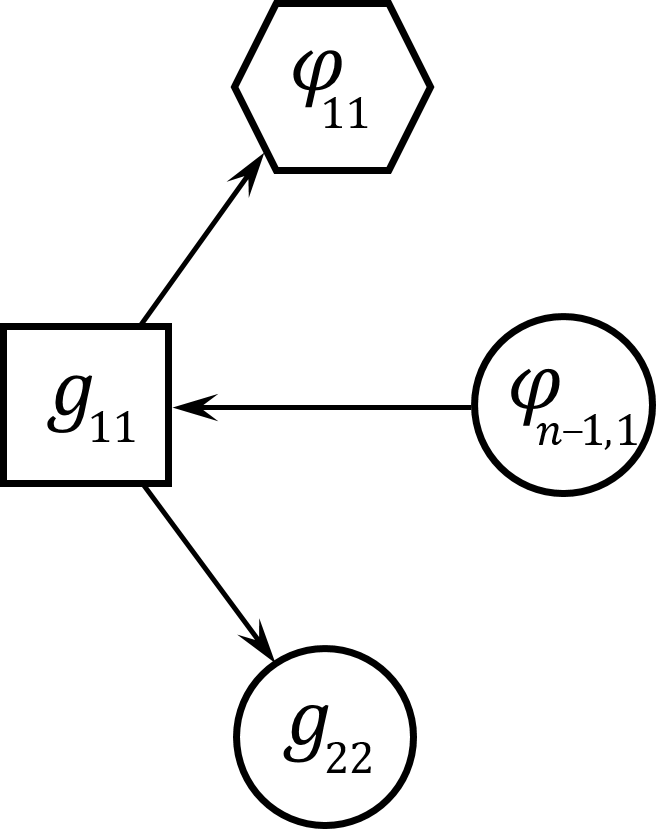} 
\end{center}
\subcaption{Case $1 \notin \Gamma_1^r$.}
\label{f:inbd_g11_0}
\end{subfigure}
\begin{subfigure}[t]{2.6in}
\begin{center}
\includegraphics[scale=0.75]{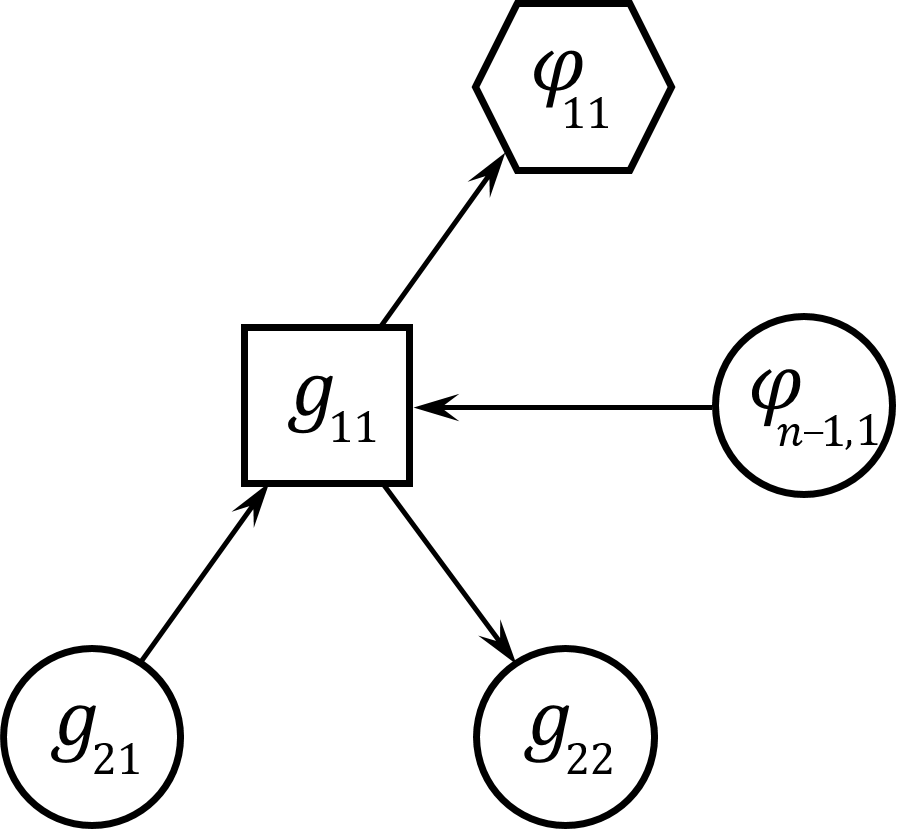}
\end{center}
\subcaption{Case $1 \in \Gamma_1^r$.}
\label{f:inbd_g11_1}
\end{subfigure}
\caption{The neighborhood of $g_{11}$.}
\label{f:inbd_g11}
\end{center}
\end{figure}

\vspace{5mm}
\begin{figure}[htb]
\begin{center}
\begin{subfigure}[t]{2.6in}
\begin{center}
\includegraphics[scale=0.75]{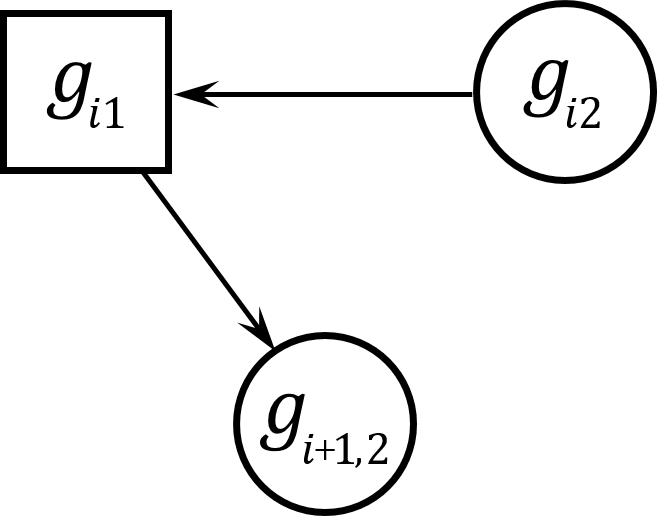} 
\end{center}
\subcaption{Case $i-1,i \notin \Gamma_1^r$.}
\label{f:inbd_gi1_0}
\end{subfigure}
\begin{subfigure}[t]{2.6in}
\begin{center}
\includegraphics[scale=0.75]{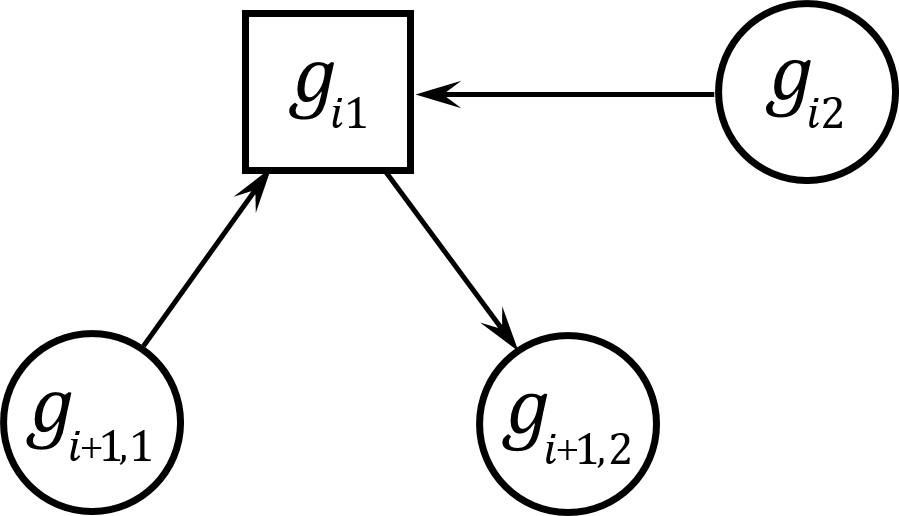}
\end{center}
\subcaption{Case $i-1 \notin \Gamma_1^r$, $i \in \Gamma_1^r$, $j:=\gamma_r(i)$.}
\label{f:inbd_gi1_1}
\end{subfigure}
\begin{subfigure}[t]{2.6in}
\vspace{4mm}
\begin{center}
\includegraphics[scale=0.75]{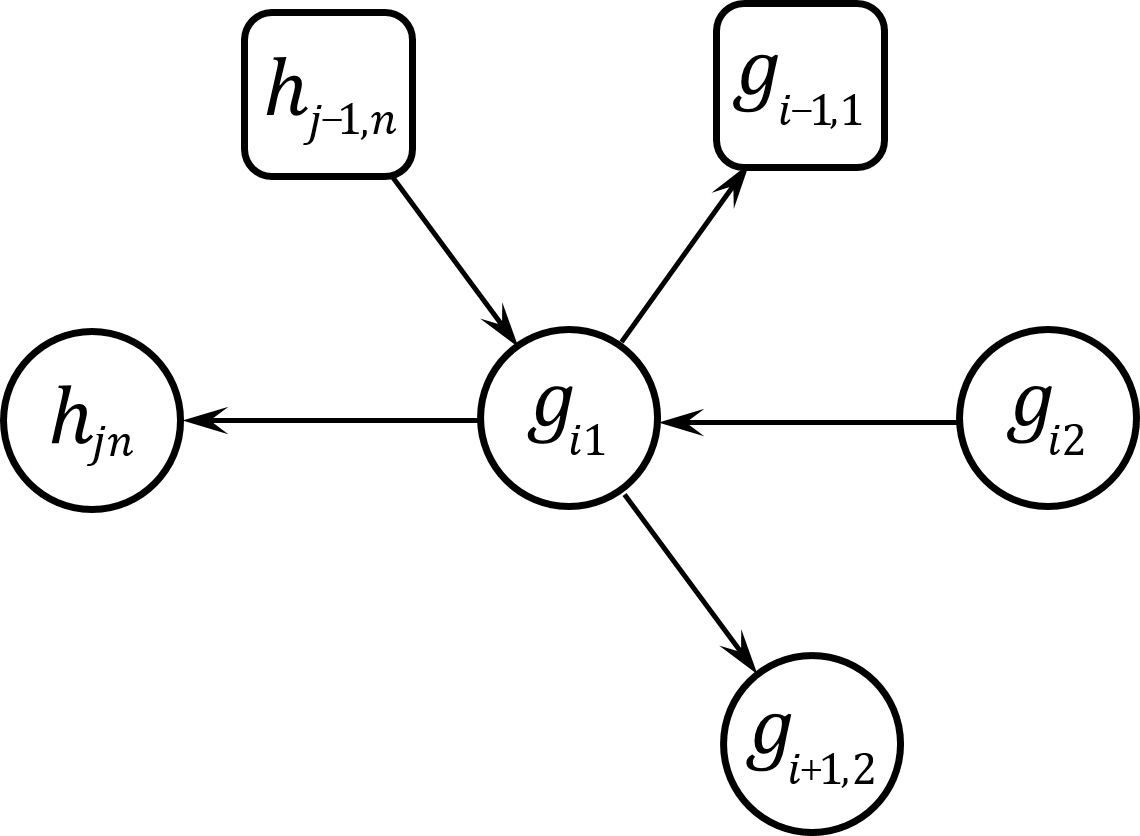}
\end{center}
\subcaption{Case $i-1 \in \Gamma_1^r$, $i \notin \Gamma_1^r$,\newline\hphantom{Cas} $j-1:=\gamma_r(i-1)$.}
\label{f:inbd_gi1_2}
\end{subfigure}
\hspace{7mm}
\begin{subfigure}[t]{2.6in}
\vspace{4mm}
\begin{center}
\includegraphics[scale=0.75]{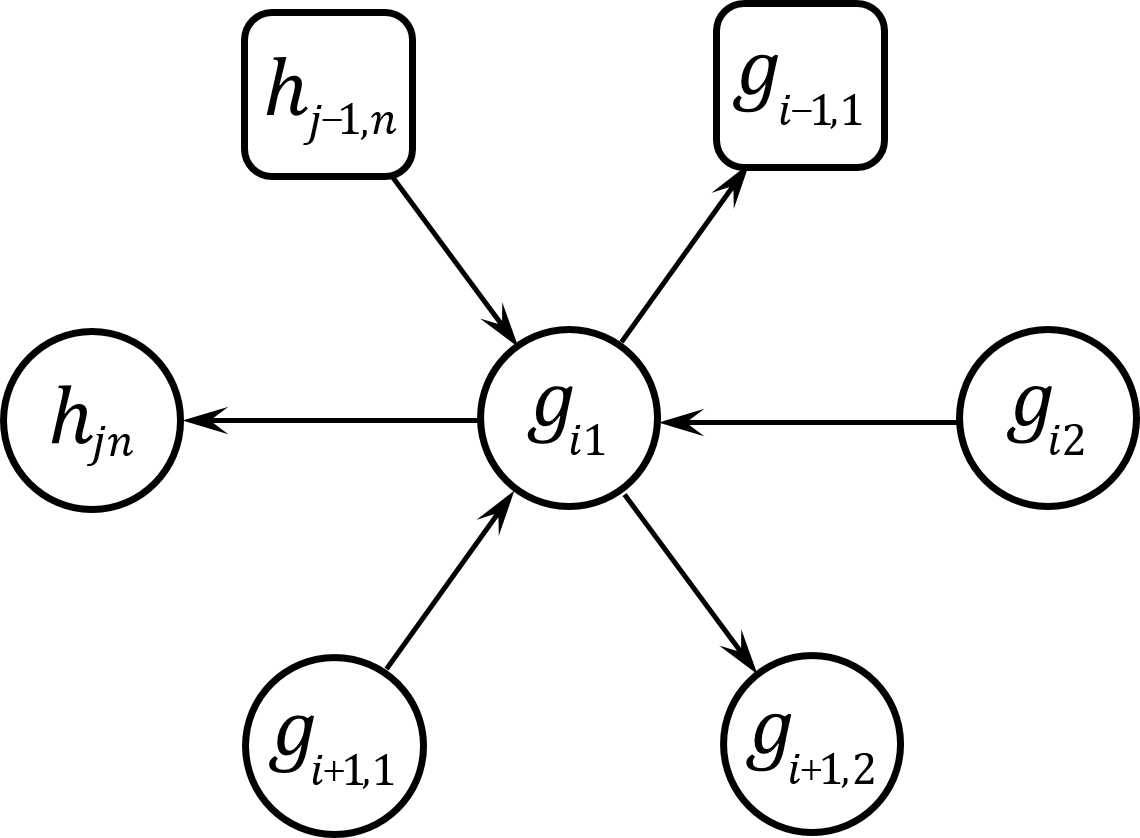}
\end{center}
\subcaption{Case $i-1,i\in \Gamma_1^r$, $j:=\gamma_r(i)$.}
\label{f:inbd_gi1_3}
\end{subfigure}
\caption{The neighborhood of $g_{i1}$ for $1 < i < n$.}
\label{f:inbd_gi1}
\end{center}
\end{figure}

\vspace{5mm}
\begin{figure}[htb]
\begin{center}
\begin{subfigure}[t]{2.6in}
\begin{center}
\includegraphics[scale=0.75]{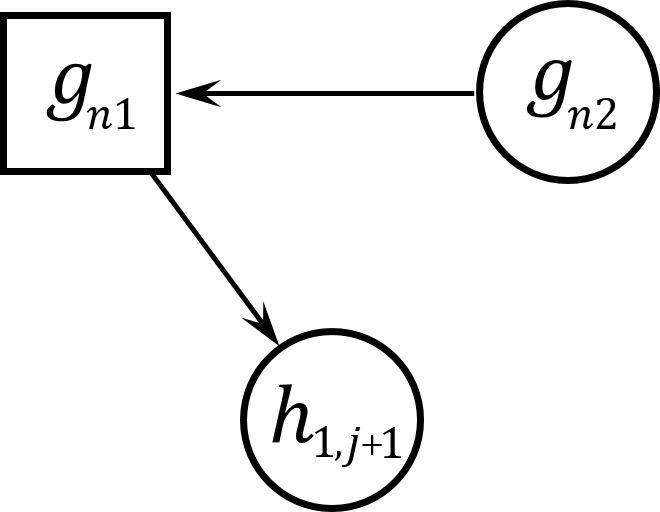} 
\end{center}
\subcaption{Case $1 \notin \Gamma_1^c$, $n-1 \notin \Gamma_1^r$.}
\label{f:inbd_gn1_0}
\end{subfigure}
\begin{subfigure}[t]{2.6in}
\begin{center}
\includegraphics[scale=0.75]{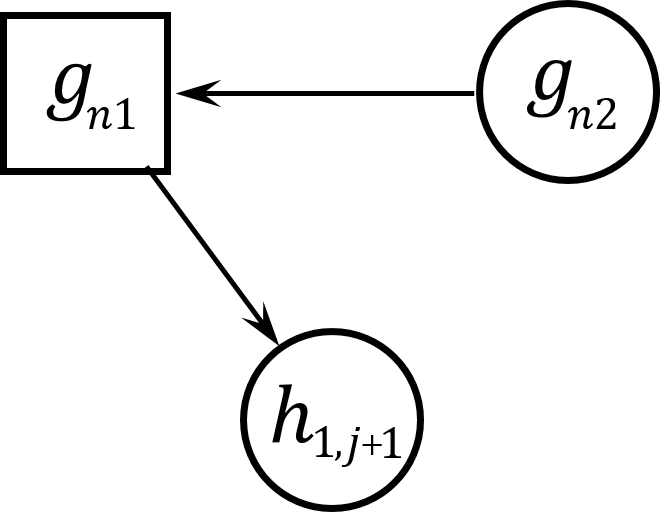}
\end{center}
\subcaption{Case $1\in\Gamma_1^c$, $n-1 \notin \Gamma_1^r$, $j:=\gamma_c(1)$}
\label{f:inbd_gn1_1}
\end{subfigure}
\vspace{5mm}
\begin{subfigure}[b]{2.6in}
\vspace{4mm}
\begin{center}
\includegraphics[scale=0.75]{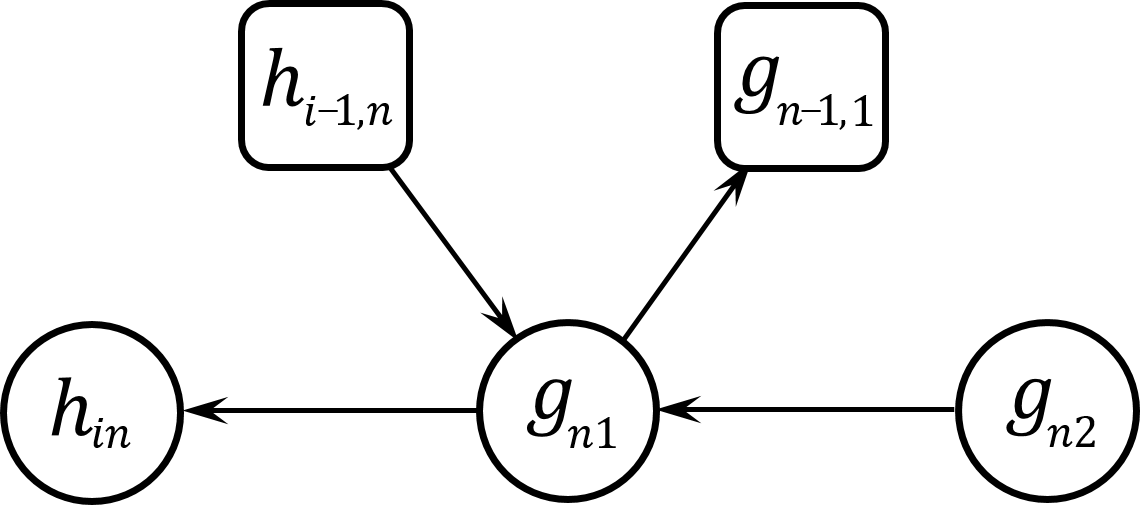}
\end{center}
\subcaption{Case $1\notin\Gamma_1^c$, $n-1 \in \Gamma_1^r$, \newline\hphantom{Cas} $i-1:=\gamma_r(n-1)$.}
\label{f:inbd_gn1_2}
\end{subfigure}
\hspace{7mm}
\begin{subfigure}[b]{2.6in}
\vspace{4mm}
\begin{center}
\includegraphics[scale=0.75]{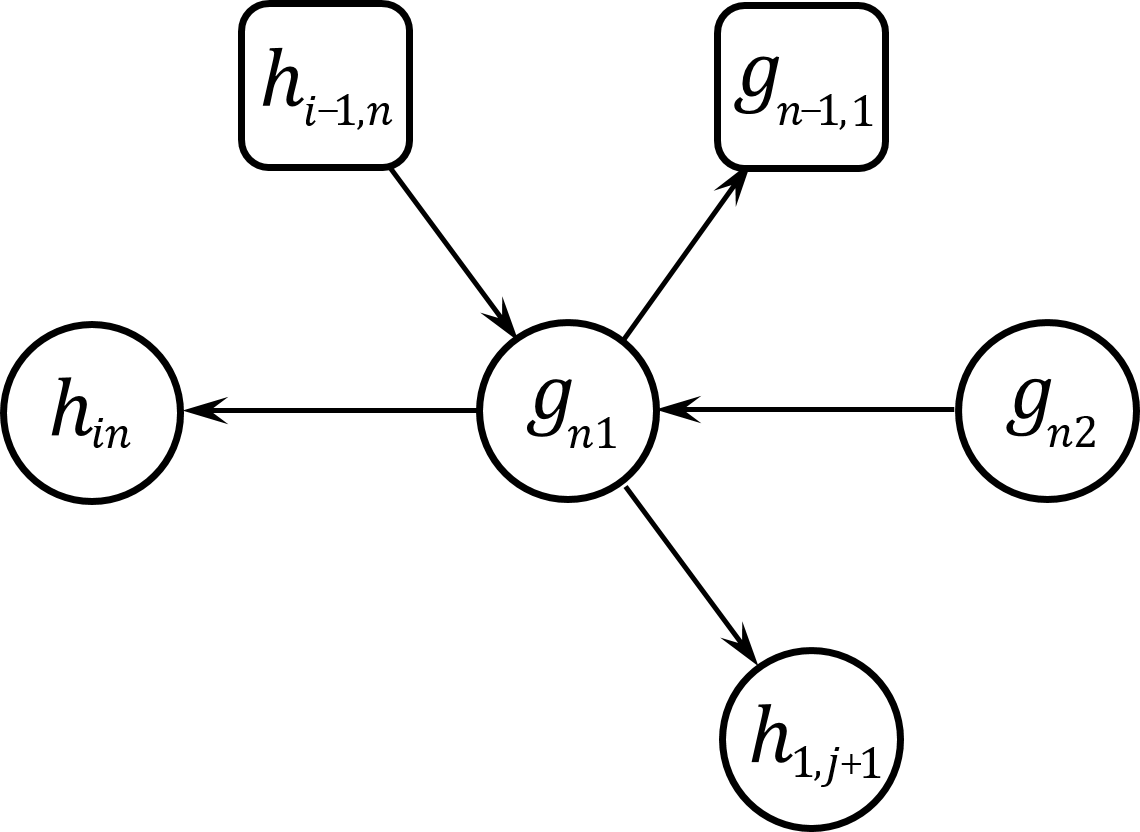}
\end{center}
\subcaption{Case $1\in\Gamma_1^c$, $n-1 \in \Gamma_1^r$,\newline\hphantom{Cas} $j:=\gamma_c(1)$, \mbox{$i-1:=\gamma_r(n-1)$}.}
\label{f:inbd_gn1_3}
\end{subfigure}
\caption{The neighborhood of $g_{n1}$.} 
\label{f:inbd_gn1}
\end{center}
\end{figure}

\vspace{5mm}
\begin{figure}[htb]
\begin{center}
\begin{subfigure}[t]{2.6in}
\begin{center}
\includegraphics[scale=0.75]{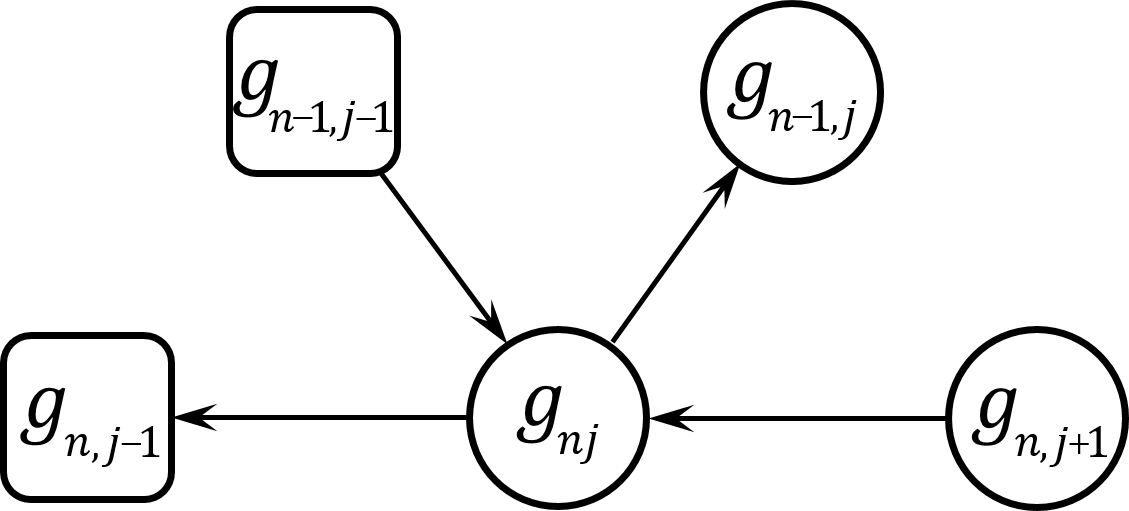} 
\end{center}
\subcaption{Case $j-1,j\notin\Gamma_1^c$.}
\label{f:inbd_gnj_0}
\end{subfigure}
\begin{subfigure}[t]{2.6in}
\begin{center}
\includegraphics[scale=0.75]{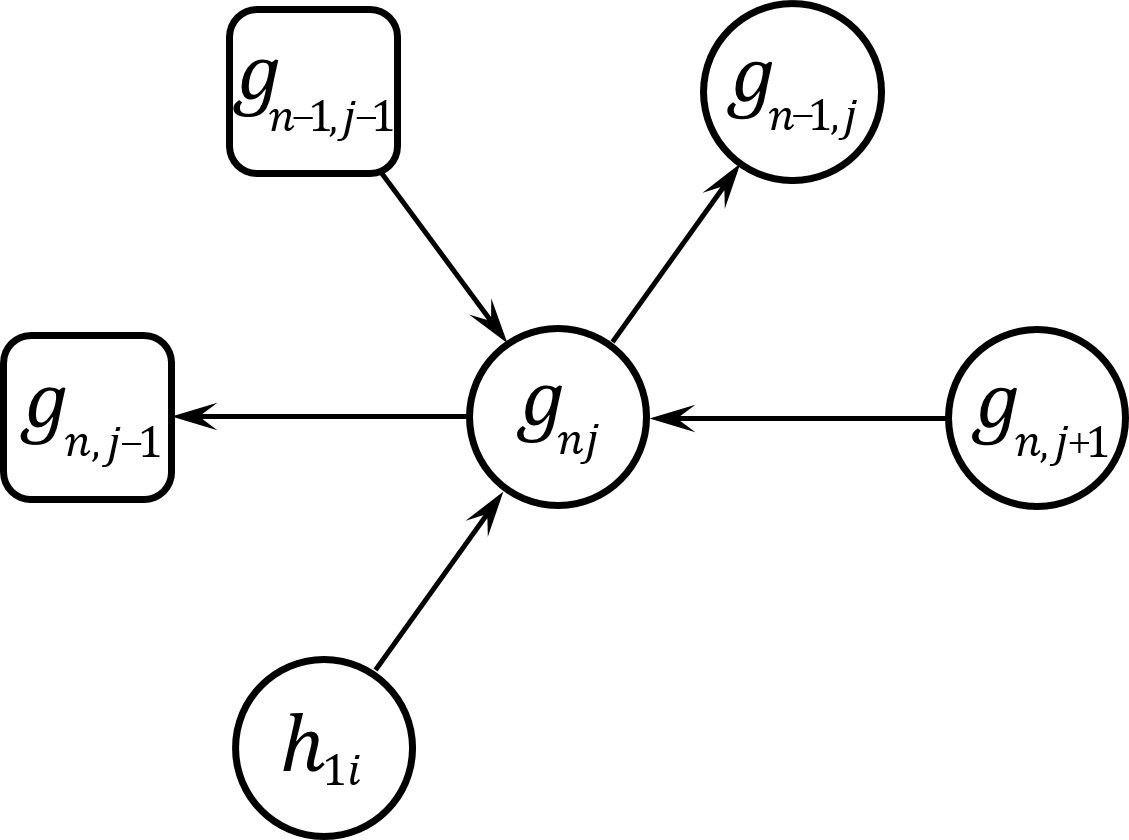}
\end{center}
\subcaption{Case $j-1\in\Gamma_1^c$, $j \notin \Gamma_1^c$,\newline\hphantom{Cas} $i-1:=\gamma_c(j-1)$.}
\label{f:inbd_gnj_1}
\end{subfigure}
\begin{subfigure}[b]{2.6in}
\vspace{4mm}
\begin{center}
\includegraphics[scale=0.75]{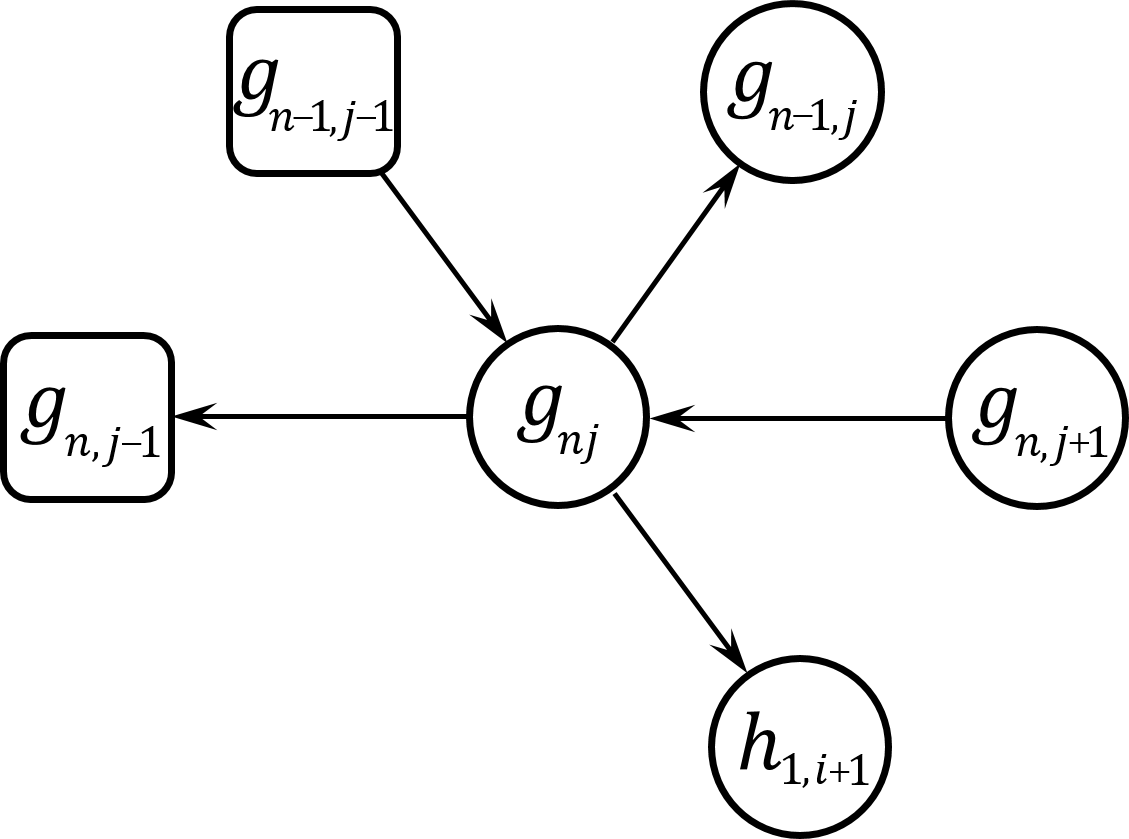}
\end{center}
\subcaption{Case $j-1\notin\Gamma_1^c$, $j \in \Gamma_1^c$, $i:=\gamma_c(j)$.}
\label{f:inbd_gnj_2}
\end{subfigure}
\hspace{7mm}
\begin{subfigure}[b]{2.6in}
\vspace{4mm}
\begin{center}
\includegraphics[scale=0.75]{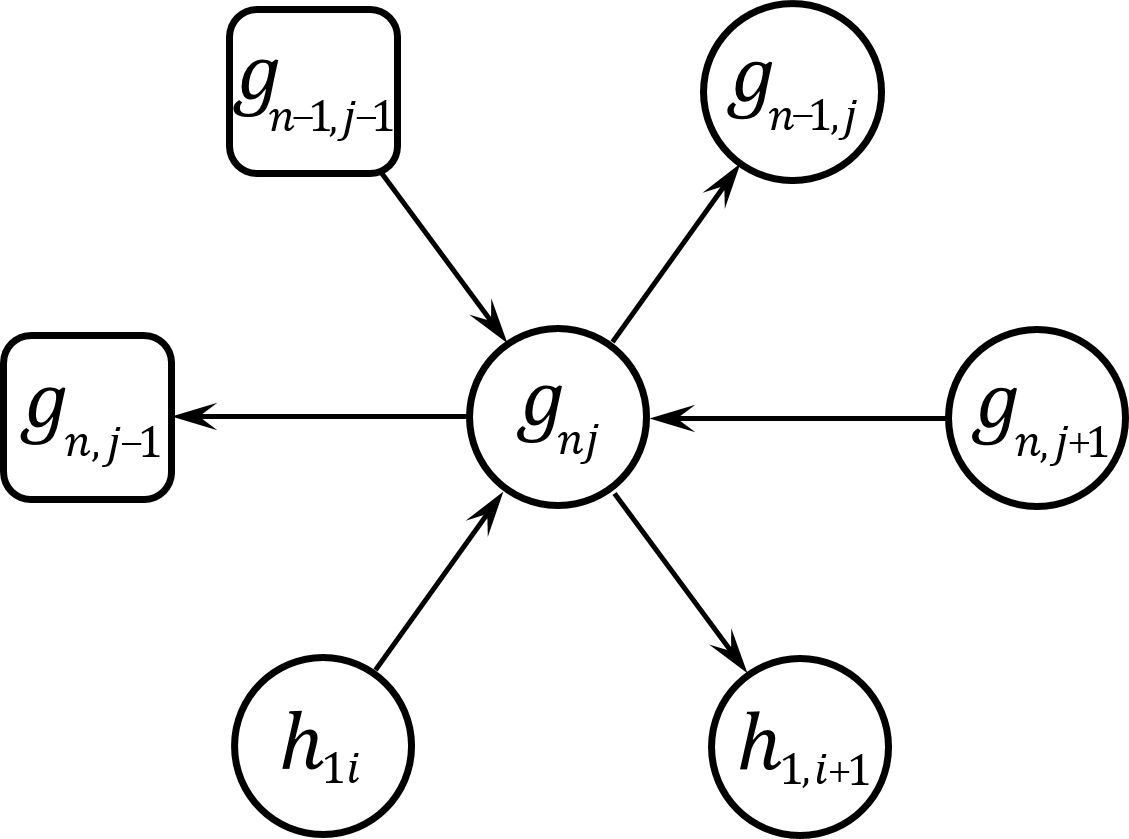}
\end{center}
\subcaption{Case $j-1,j\in\Gamma_1^c$, $i:=\gamma_c(j)$.}
\label{f:inbd_gnj_3}
\end{subfigure}
\caption{The neighborhood of $g_{nj}$ for $2 \leq j \leq n$.}
\label{f:inbd_gnj}
\end{center}
\end{figure}

\vspace{5mm}
\begin{figure}[htb]
\begin{subfigure}[t]{2.6in}
\begin{center}
\includegraphics[scale=0.75]{inbd_hnn}
\end{center}
\subcaption{Case $n-1 \notin \Gamma_2^r$.}
\label{f:inbd_hnn_0}
\end{subfigure}
\begin{subfigure}[t]{2.6in}
\begin{center}
\includegraphics[scale=0.75]{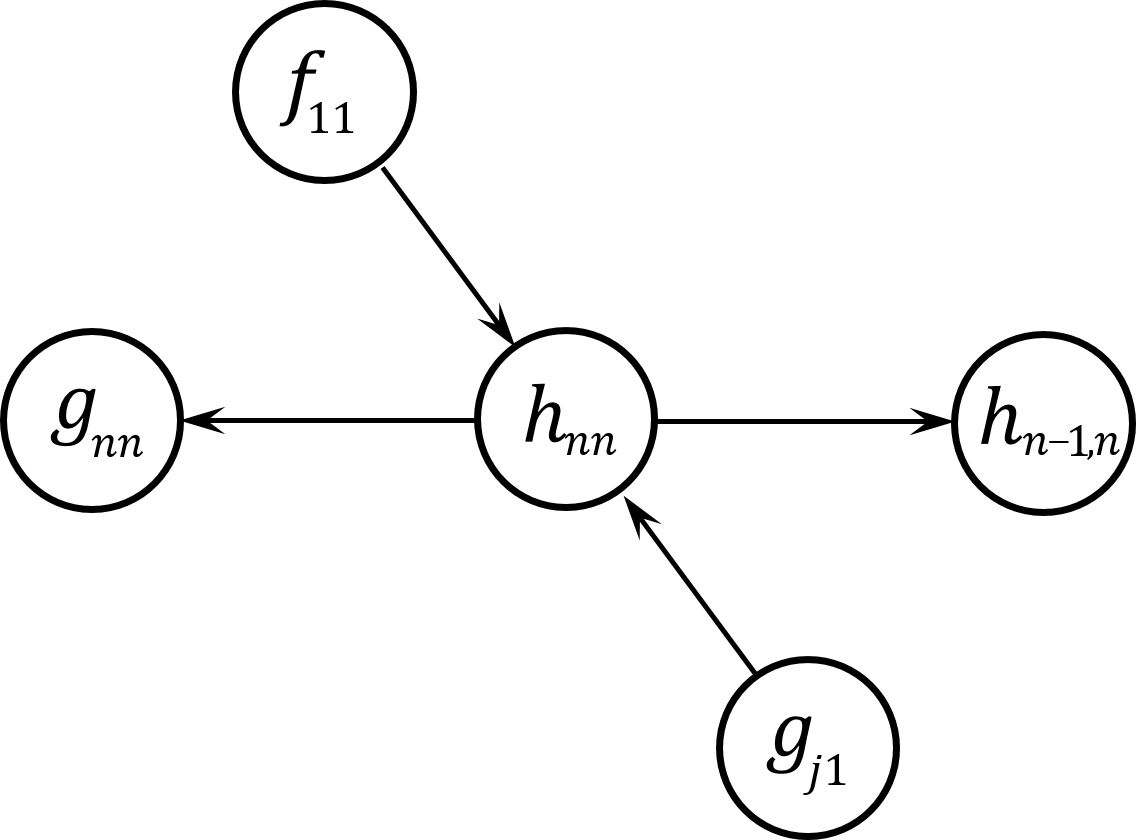}
\end{center}
\subcaption{Case $n-1 \in \Gamma_2^r$, $j-1:=\gamma_r^*(n-1)$.}
\label{f:nbd_hnn_1}
\end{subfigure}
\caption{The neighborhood of $h_{nn}$.}
\label{f:inbd_hnn}
\end{figure}

\vspace{5mm}
\begin{figure}[htb]
\begin{center}
\begin{subfigure}[t]{2.6in}
\begin{center}
\includegraphics[scale=0.75]{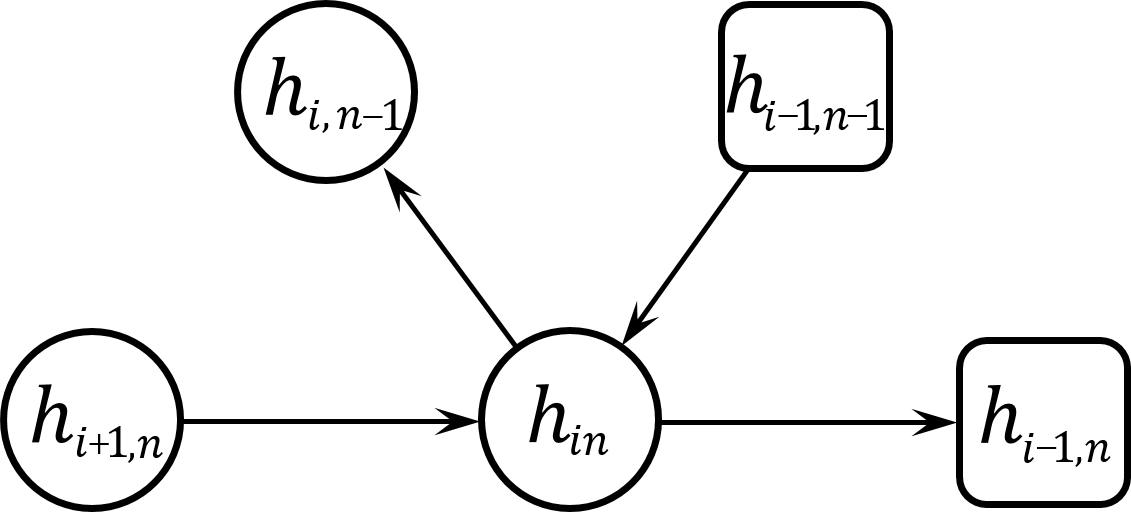} 
\end{center}
\subcaption{Case $i-1,i\notin\Gamma_2^r$.}
\label{f:inbd_hin_0}
\end{subfigure}
\begin{subfigure}[t]{2.6in}
\begin{center}
\includegraphics[scale=0.75]{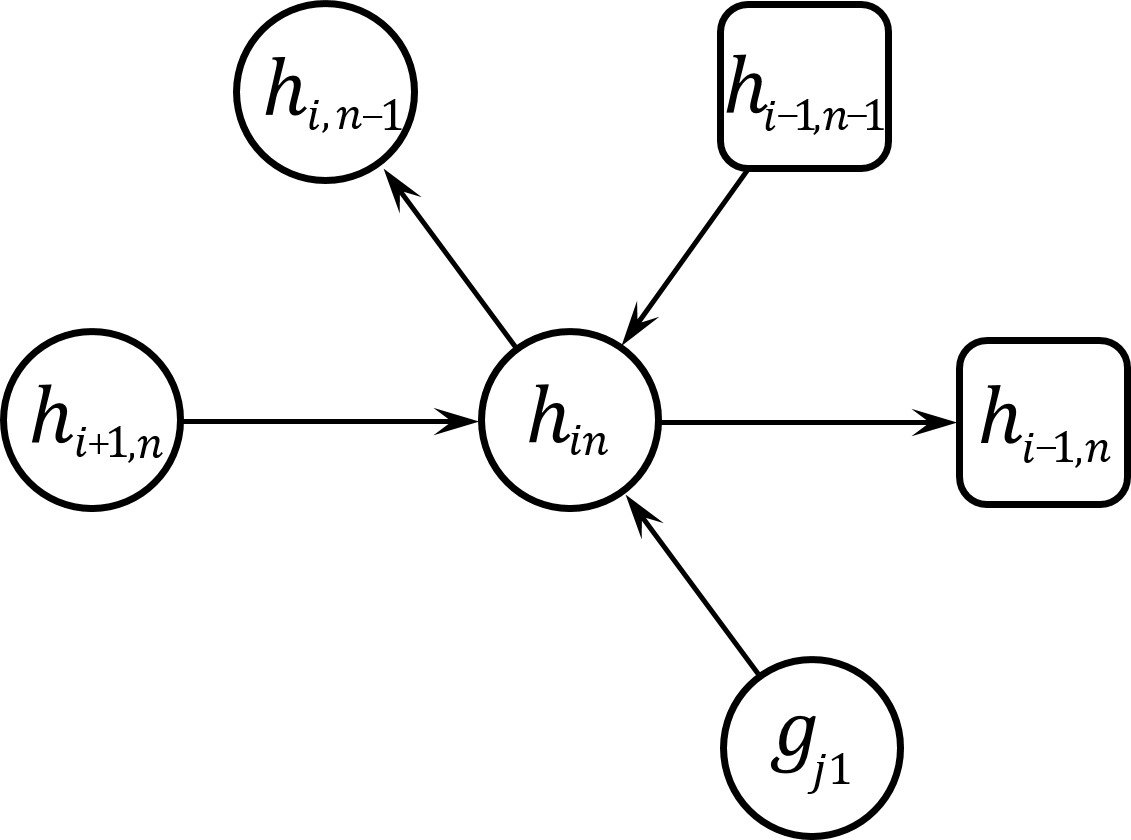}
\end{center}
\subcaption{Case $i-1\in\Gamma_2^r$, $i \notin \Gamma_2^r$, \newline\hphantom{Cas}$j-1:=\gamma_r^*(i-1)$.}
\label{f:inbd_hin_1}
\end{subfigure}
\hspace{7mm}
\begin{subfigure}[t]{2.6in}
\vspace{4mm}
\begin{center}
\includegraphics[scale=0.75]{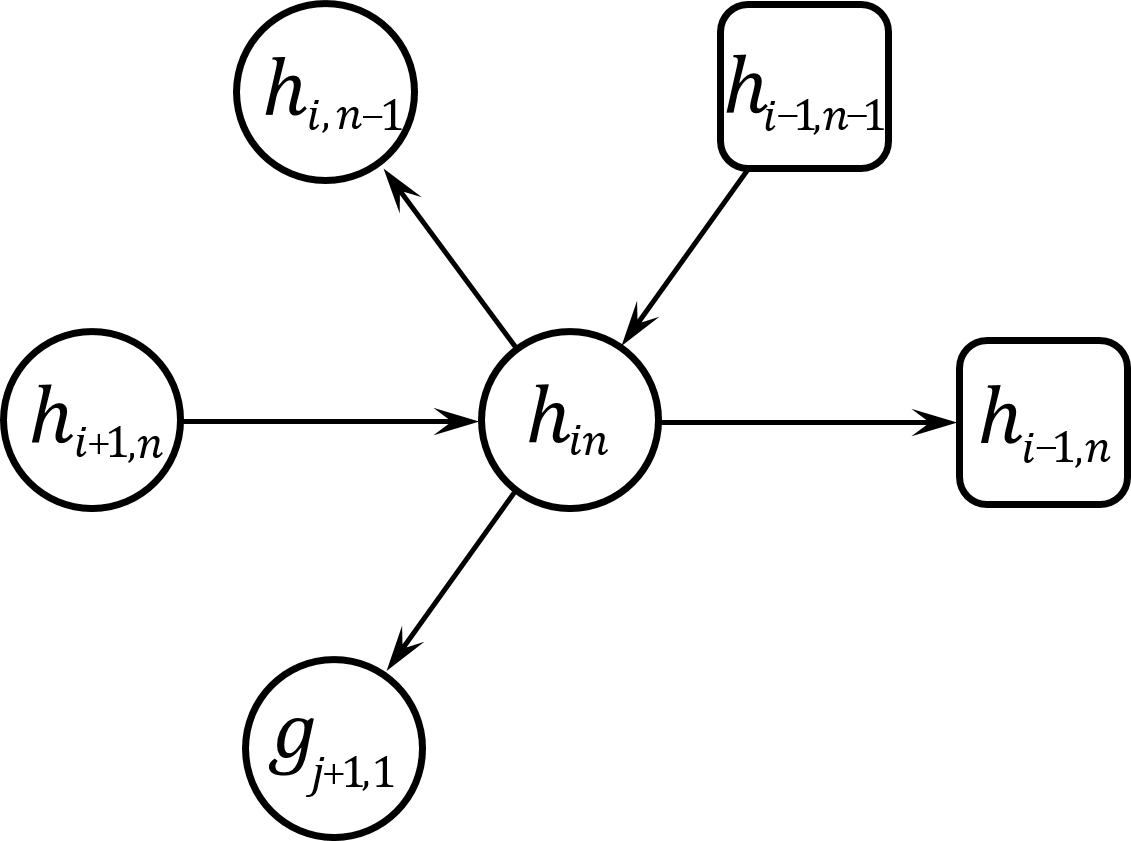}
\end{center}
\subcaption{Case $i-1\notin\Gamma_2^r$, $i \in \Gamma_2^r$, $i:=\gamma_r^*(j)$.}
\label{f:inbd_hin_2}
\end{subfigure}
\hspace{7mm}
\begin{subfigure}[t]{2.6in}
\vspace{4mm}
\begin{center}
\includegraphics[scale=0.75]{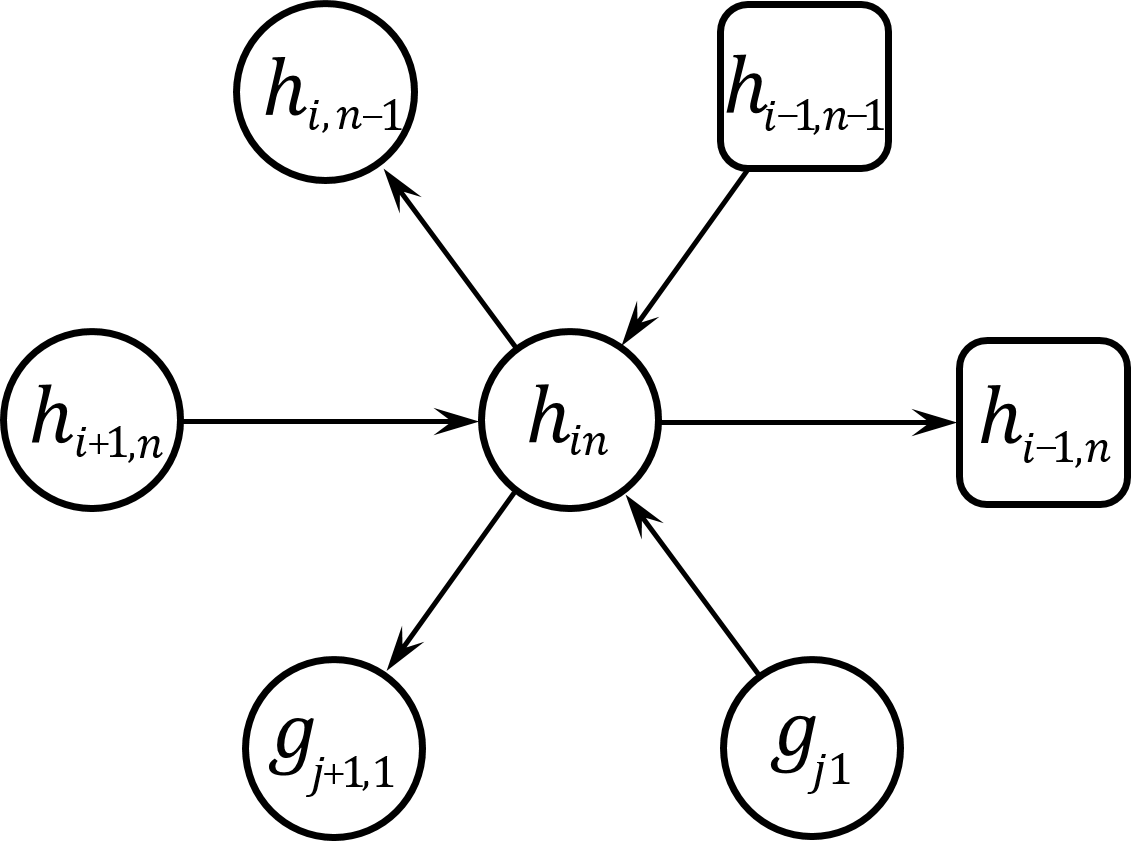}
\end{center}
\subcaption{Case $i-1,i\in\Gamma_2^r$, $i:=\gamma_r^*(j)$.}
\label{f:inbd_hin_3}
\end{subfigure}
\caption{The neighborhood of $h_{in}$ for $2 \leq j \leq n-1$.} 
\label{f:inbd_hin}
\end{center}
\end{figure}

\vspace{5mm}
\begin{figure}[htb]
\begin{center}
\begin{subfigure}[t]{2.6in}
\begin{center}
\includegraphics[scale=0.75]{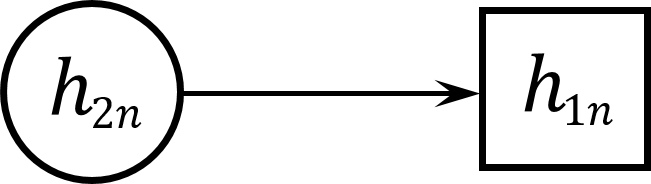} 
\end{center}
\subcaption{Case $1\notin \Gamma_2^r$, $n-1\notin \Gamma_2^c$.}
\label{f:inbd_h1n_0}
\end{subfigure}
\begin{subfigure}[t]{2.6in}
\begin{center}
\includegraphics[scale=0.75]{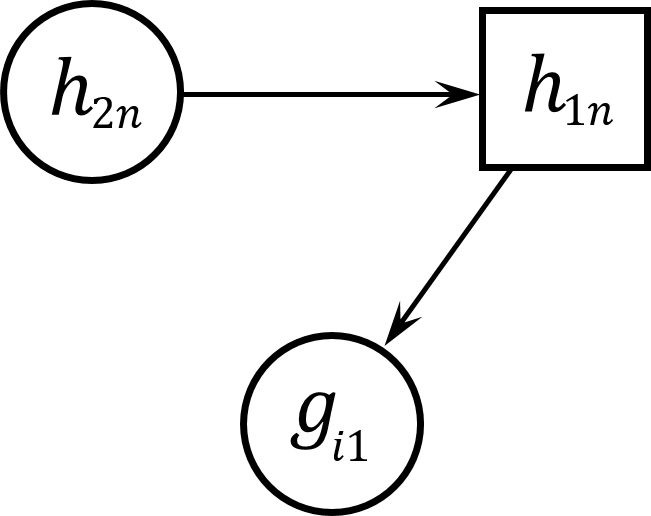}
\end{center}
\subcaption{Case $1\in \Gamma_2^r$, $n-1\notin\Gamma_2^c$, $i-1:=\gamma_r^*(1)$.}
\label{f:inbd_h1n_1}
\end{subfigure}
\begin{subfigure}[b]{2.6in}
\vspace{4mm}
\begin{center}
\includegraphics[scale=0.75]{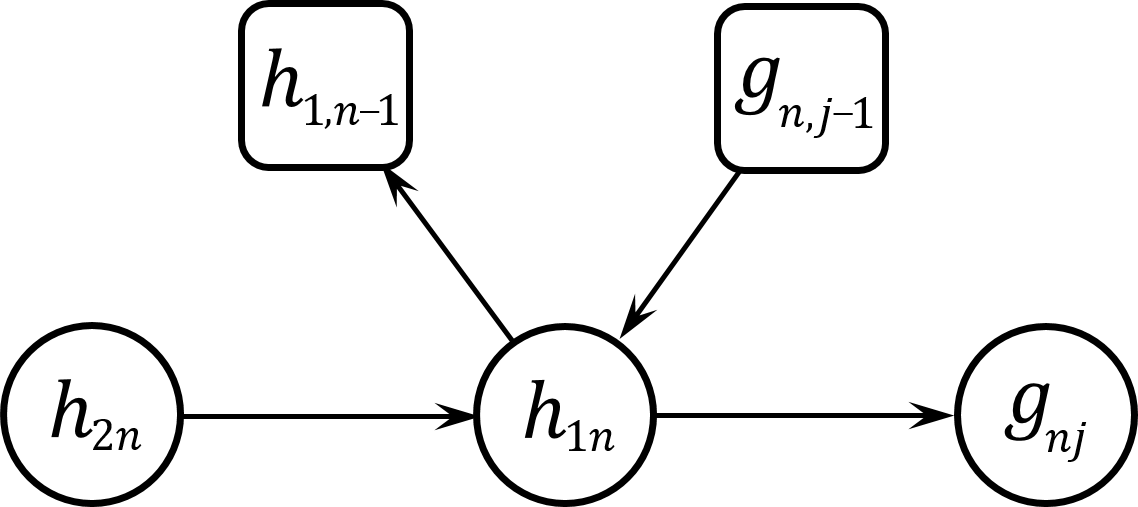}
\end{center}
\subcaption{Case $1\notin \Gamma_2^r$, $n-1\in \Gamma_2^c$,\newline\hphantom{Cas} $j-1:=\gamma_c^*(n-1)$.}
\label{f:inbd_h1n_2}
\end{subfigure}
\hspace{7mm}
\begin{subfigure}[b]{2.6in}
\vspace{4mm}
\begin{center}
\includegraphics[scale=0.75]{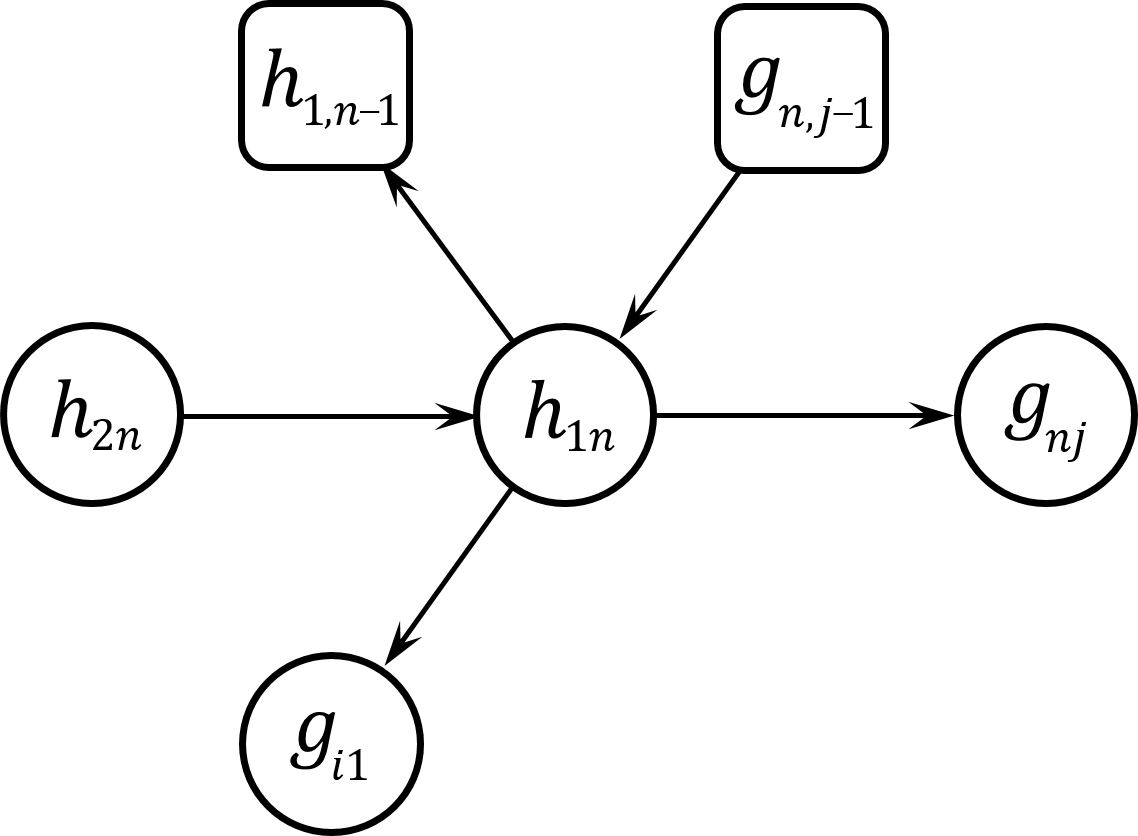}
\end{center}
\subcaption{Case $1\in\Gamma_2^r$, $n-1 \in \Gamma_2^c$,\newline\hphantom{Cas} $i-1:=\gamma_r^*(1)$, $j-1:=\gamma_c^*(n-1)$.}
\label{f:inbd_h1n_3}
\end{subfigure}
\caption{The neighborhood of $h_{1n}$.} 
\label{f:inbd_h1n}
\end{center}
\end{figure}

\vspace{5mm}
\begin{figure}[htb]
\begin{center}
\begin{subfigure}[t]{2.6in}
\begin{center}
\includegraphics[scale=0.75]{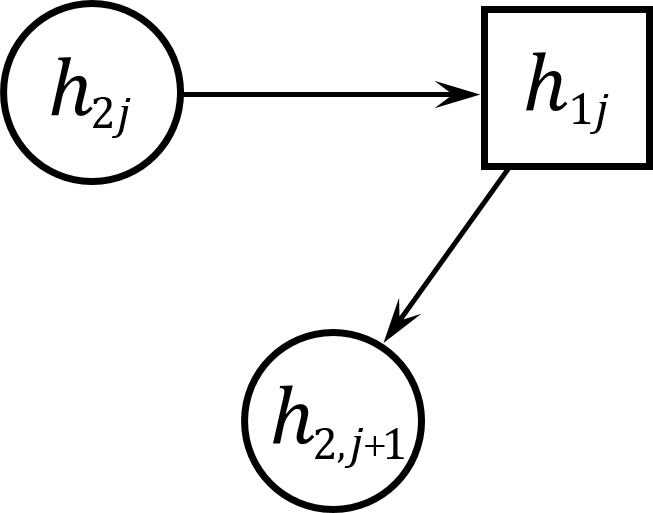} 
\end{center}
\subcaption{Case $j-1,j\notin\Gamma_2^c$.}
\label{f:inbd_h1j_0}
\end{subfigure}
\begin{subfigure}[t]{2.6in}
\begin{center}
\includegraphics[scale=0.75]{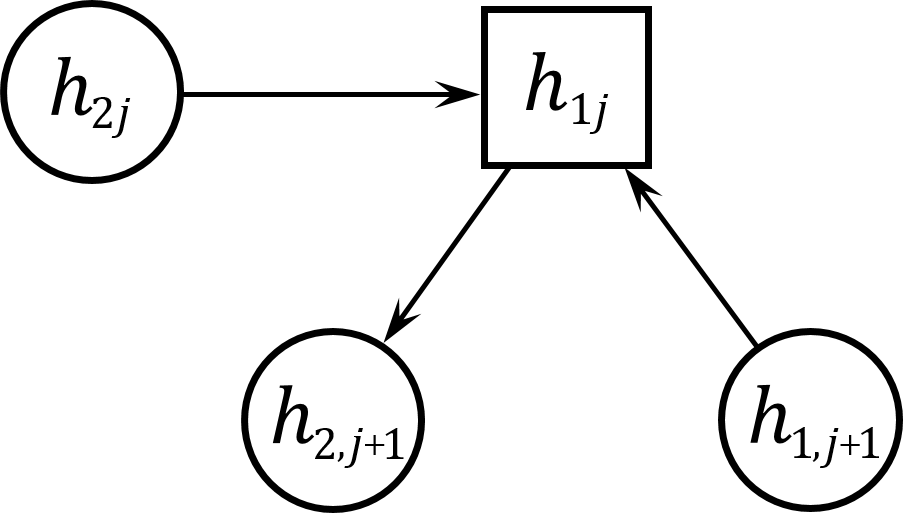}
\end{center}
\subcaption{Case $j-1\notin\Gamma_2^c$, $j \in \Gamma_2^c$.}
\label{f:inbd_h1j_1}
\end{subfigure}
\begin{subfigure}[t]{2.6in}
\vspace{4mm}
\begin{center}
\includegraphics[scale=0.75]{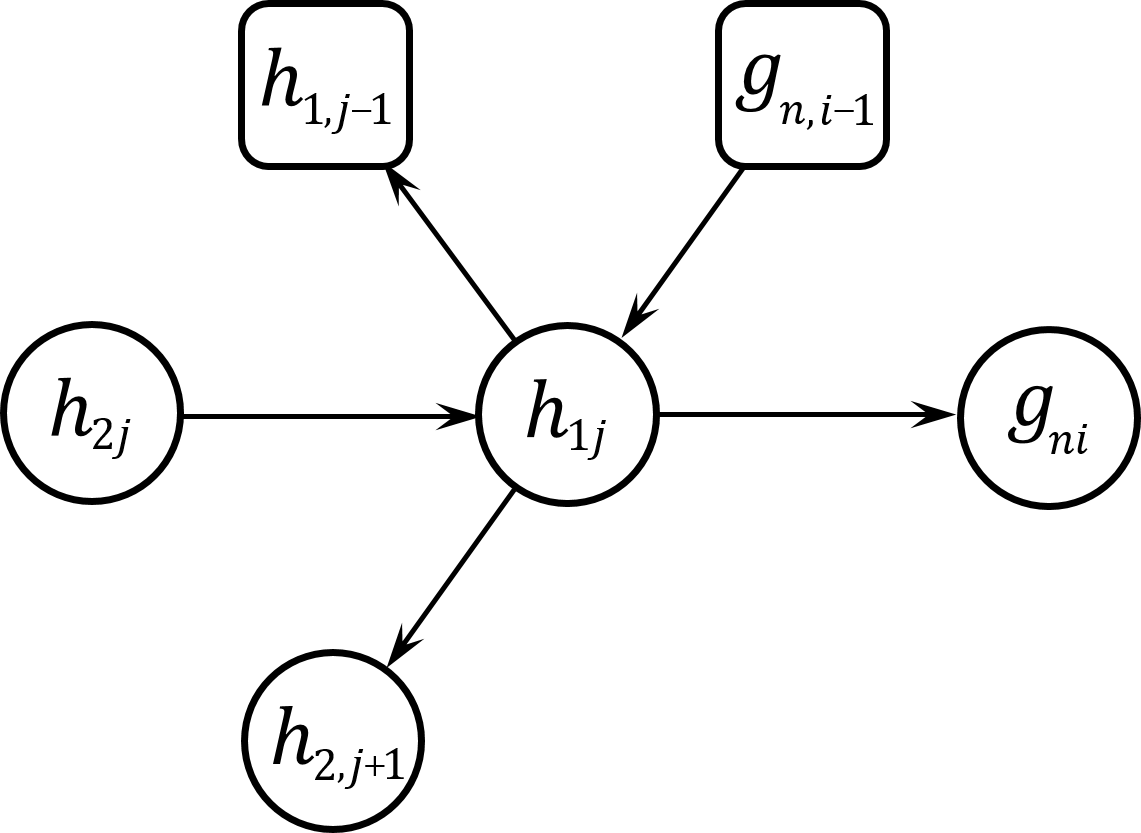}
\end{center}
\subcaption{Case $j-1\in\Gamma_2^c$, $j \notin \Gamma_2^c$,\newline\hphantom{Cas} $i-1:=\gamma_c^*(j-1)$.}
\label{f:inbd_h1j_2}
\end{subfigure}
\hspace{7mm}
\begin{subfigure}[t]{2.6in}
\vspace{4mm}
\begin{center}
\includegraphics[scale=0.75]{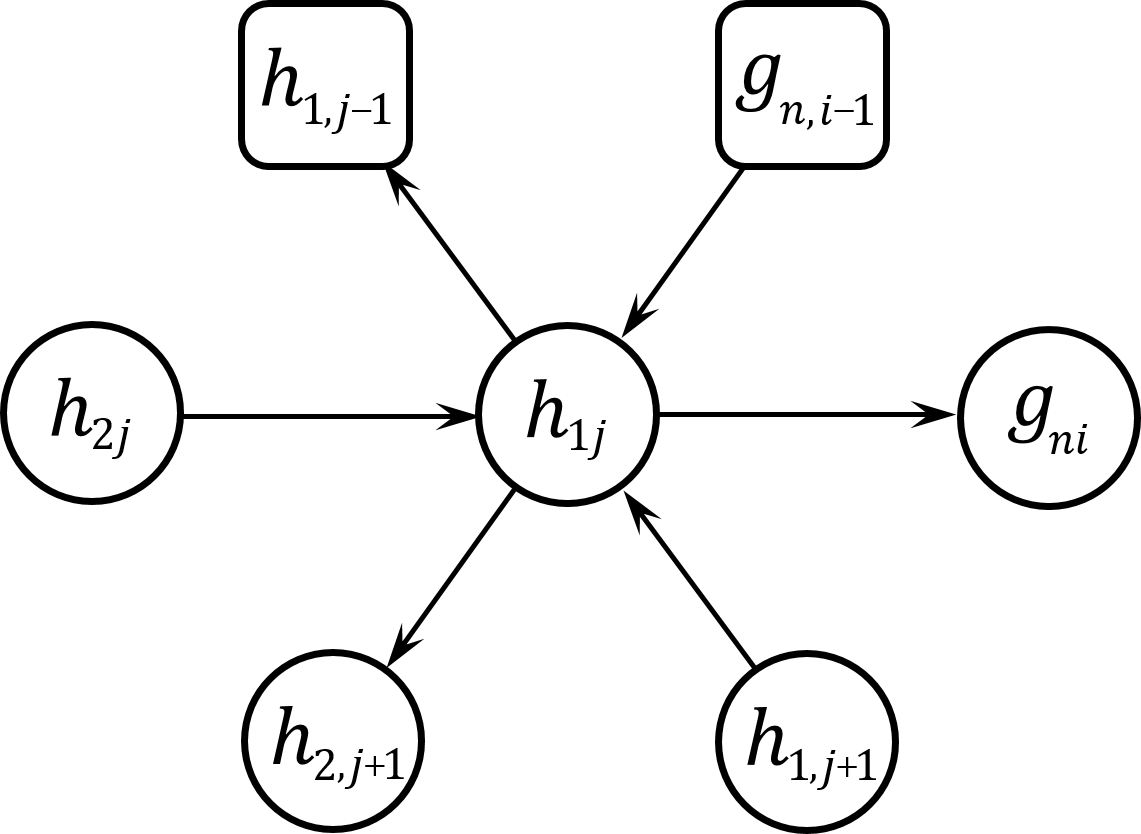}
\end{center}
\subcaption{Case $j-1,j\in\Gamma_2^c$, $i:=\gamma_c^*(j)$.}
\label{f:inbd_h1j_3}
\end{subfigure}
\caption{The neighborhood of $h_{1j}$ for $1<j<n$.} 
\label{f:inbd_h1j}
\end{center}
\end{figure}
\clearpage
\vspace{5mm}
\begin{figure}[htb]
\begin{center}
\begin{subfigure}[t]{2.6in}
\begin{center}
\includegraphics[scale=0.75]{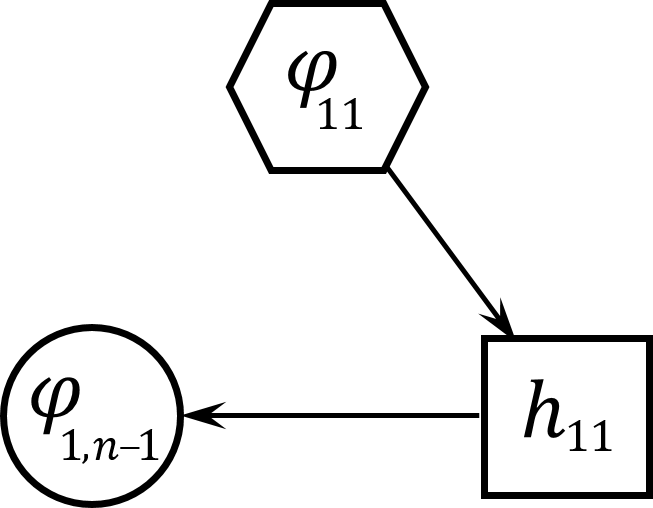} 
\end{center}
\subcaption{Case $1 \notin \Gamma_2^c$.}
\label{f:inbd_h11_0}
\end{subfigure}
\begin{subfigure}[t]{2.6in}
\begin{center}
\includegraphics[scale=0.75]{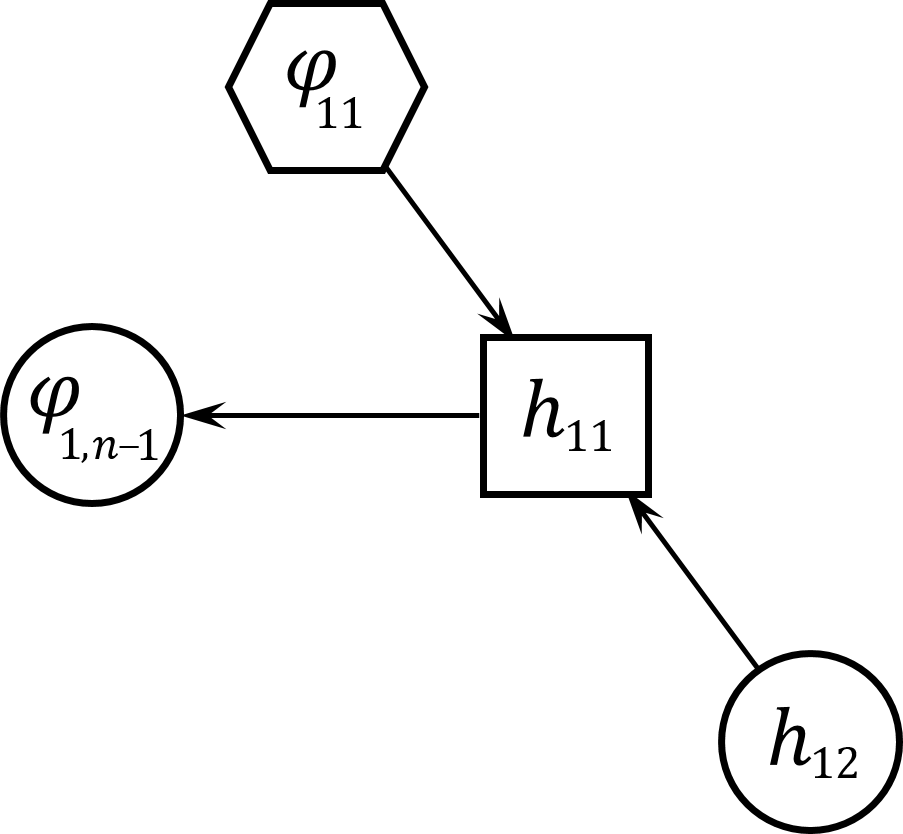}
\end{center}
\subcaption{Case $1 \in \Gamma_2^c$.}
\label{f:inbd_h11_1}
\end{subfigure}
\caption{The neighborhood of $h_{11}$.}
\label{f:inbd_h11}
\end{center}
\end{figure}
\begin{remark}
Once the initial quiver is constructed for a BD pair $\bg = (\bg^r,\bg^c)$, one can obtain the initial quiver for the cluster structure $\mathcal{C}(\bg)$ on $\GL_n$ described in~\cite{plethora}, in the following way: 1) remove all $f$- and $\varphi$-vertices; 2) for each $1 \leq i \leq n$, merge the vertex $h_{ii}$ with the vertex $g_{ii}$ (but retain the edges); 3) in the resulting quiver, remove the loop at the vertex $g_{nn}$.
\end{remark}
\subsection{Toric action}\label{s:tor}
Let $\bg = (\bg^r, \bg^c)$ be an aperiodic oriented BD pair that defines the generalized cluster structure $\gc(\bg)$, and let $\mathfrak{h}^{\sll_n}$ be the Cartan subalgebra of $\sll_n$. For each $\ell\in \{r,c\}$, define a subalgebra
\[
\mathfrak{h}_{\bg^{\ell}} := \{h \in \mathfrak{h}^{\sll_n} \ | \ \alpha(h) = \beta(h) \ \text{if} \ \gamma^j_{\ell}(\alpha) = \beta \ \text{for some} \ j\}.
\]
Notice that its dimension is
\[
\dim\mathfrak{h}_{\bg^\ell} = k_{\bg^\ell} = |\Pi \setminus \Gamma^\ell|,
\]
where $\Pi$ is the set of simple roots. Let $\mathcal{H}_{\bg^r}$ and $\mathcal{H}_{\bg^c}$ be the connected subgroups of $\SL_n$ that correspond to $\mathfrak{h}_{\bg^r}$ and $\mathfrak{h}_{\bg^c}$, respectively. We let $\mathcal{H}_{\bg^r}$ act upon $D(\GL_n)$ on the left and $\mathcal{H}_{\bg^c}$ to act upon $D(\GL_n)$ on the right; that is,
\[
\begin{split}
H.(X,Y) = (HX,HY), \ \ &H \in \mathcal{H}_{\bg^r};\\
(X,Y).H = (XH,YH),  \ \  &H \in \mathcal{H}_{\bg^c}.
\end{split}
\]
We also let scalar matrices act upon $D(\GL_n)$ via
\[
(aI,bI).(X,Y) = (aX,bY), \ \ a, b \in \mathbb{C}^*.
\]
As we shall see in Section~\ref{s:torpr}, the left-right action of $\mathcal{H}_{\bg^r} \times \mathcal{H}_{\bg^c}$ together with the action by scalar matrices induces a global toric action on $\gc(\bg)$ of rank $k_{\bg^r} + k_{\bg^c}+2$.

\subsection{Poisson-geometric properties of frozen variables}\label{s:pgfroz} As we explained above, if $(R_0^r,R_0^c)$ is chosen in such a way that $R_0^r(I) = R_0^c(I) = 1/2$, then the frozen variables $c_0,c_1,\ldots,c_{n-1},c_n$ are Casimirs of the Poisson bracket (for the case of $D(\SL_n)$, the statement is true for any $(R_0^r,R_0^c)$). In particular, the symplectic leaves of the Poisson bracket are contained in the level sets of these Casimirs. The other frozen variables are given by the determinants of $\mathcal{L}$-matrices. Given such a frozen variable $\psi(X,Y) := \det\mathcal{L}(X,Y)$, the proposition below, which was proved in~\cite{plethora}, implies that the nonsingular part of the zero locus of $\psi$ is a Poisson submanifold; hence, it foliates into a union of its own symplectic leaves. However, we do not know\footnote{After publishing this paper, we became aware of  results by Yakimov et. al., from which it follows that indeed the zero loci foliate into a union of symplectic leaves of $D(\GL_n)$. This also requires a verification of the fact that the frozen variables are irreducible as elements of $\mathcal{O}(\GL_n)$, which is done in Section~\ref{s:coprim}.} if those symplectic leaves are also symplectic leaves of $D(\GL_n)$.

For a Belavin-Drinfeld triple $\bg = (\Gamma_1,\Gamma_2,\gamma)$, let $\mathcal{P}_+(\Gamma_1)$ and $\mathcal{P}_-(\Gamma_2)$ be the upper and lower parabolic subgroups of $\GL_n$ determined by the root data $\Gamma_1$ and $\Gamma_2$, respectively. Define a subgroup $\mathcal{D} \subseteq \mathcal{P}_+(\Gamma_1)\times \mathcal{P}_-(\Gamma_2)$ via
\[
\mathcal{D}:=\{(g_1,g_2)\in \mathcal{P}_+(\Gamma_1)\times \mathcal{P}_-(\Gamma_2) \ | \ \tilde{\gamma}(\Pi_{\Gamma_1}(g_1)) = \Pi_{\Gamma_2}(g_2)\}
\]
where $\Pi_{\Gamma_1} :\mathcal{P}_+(\Gamma_1) \rightarrow \Pi_{\Delta} \GL_n(\Delta)$ and $\Pi_{\Gamma_2} : \mathcal{P}_-(\Gamma_2) \rightarrow \prod_{\bar{\Delta}} \GL_n(\bar{\Delta})$ are group projections ($\Delta$ and $\bar{\Delta}$ are nontrivial $X$- and $Y$-runs, respectively), and $\GL_n(\Delta)$ are invertible $|\Delta|\times|\Delta|$ matrices embedded into $\GL_n$ as a $\Delta\times\Delta$ block (and likewise $\GL_n(\bar{\Delta})$). Given a Belavin-Drinfeld pair $(\bg^r,\bg^c)$, denote the respective groups $\mathcal{D}$ as $\mathcal{D}^r$ and $\mathcal{D}^c$.

\begin{proposition}
For any $\mathcal{L}$-matrix in $\gc(\bg)$, the following statements hold:
\begin{enumerate}[(i)]
\item \mbox{For any $(g_1,g_2) \in \mathcal{D}^r$, $\det \mathcal{L}(g_1X,g_2Y) = \chi^r(g_1,g_2) \det \mathcal{L}(X,Y)$, where $\chi^r$ is a character on $\mathcal{D}^r$;}
\item \mbox{For any $(g_1,g_2) \in \mathcal{D}^c$, $\det \mathcal{L}(Xg_1,Yg_2) = \chi^c(g_1,g_2) \det \mathcal{L}(X,Y)$, where $\chi^c$ is a character on $\mathcal{D}^c$;}
\item $\det \mathcal{L}(X,Y)$ is log-canonical with any $x_{ij}$ or $y_{ij}$.
\end{enumerate}
\end{proposition}

\begin{remark}
Assume that $\bg^r= \bg^c$ and $R_0:=R_0^r=R_0^c$ is chosen so that formulas~\eqref{eq:roid} are satisfied. Then the connected dual Poisson group $\GL_n^*$, viewed as a subgroup of $D(\GL_n)=\GL_n\times\GL_n$, is a subgroup of $\mathcal{D}$ as well. In the case of $D(\SL_n)$, such an issue with the choice of $R_0$ does not arise, so $\SL_n^* \subseteq \mathcal{D}$ (hence the determinants of the $\mathcal{L}$-matrices are semi-invariant with respect to the action of $\SL_n^*$ on the right and on the left).
\end{remark}
\section{Regularity}\label{s:regular}
Let $\bg = (\bg^r, \bg^c)$ be a BD pair that defines a generalized cluster structure $\gc(\bg)$ on $D(\GL_n)$ with the initial seed described in Section~\ref{s:descr}. In this section, we show\footnote{As of October 2023, these results can be obtained via the birational quasi-isomorphisms $\mathcal{U}$ with no additional computations. See our new paper on arXiv. } that the mutation of any cluster variable from the initial seed in $\gc(\bg)$ produces a regular function. We will prove in Section~\ref{s:coprim} that $\gc(\bg)$ satisfies coprimality conditions~\ref{p:sfcopri} and~\ref{p:sfregi} of Proposition~\ref{p:starfish}, which implies that $\gc(\bg)$ is a regular generalized cluster structure on $D(\GL_n)$.

\begin{proposition}\label{p:regular}
The mutation of the initial cluster of $\gc(\bg)$ in any direction yields a regular function.
\end{proposition}
\begin{proof}
The regularity at $g_{ij}$ and $h_{ji}$ for $i > j$ follows from Theorem 6.1 in \cite{plethora}; for $\varphi$- and $f$-functions, the regularity follows from Section 6.4 in \cite{double}. Therefore, all we need to prove is that the mutation at any $g_{ii}$ or $h_{ii}$ in the case of an aperiodic oriented BD pair yields a regular function.
\\

\noindent \emph{Mutation at $h_{ii}$.} First of all, note that if $n-1 \notin \Gamma^r_2$, then, according to the construction in Section~\ref{s:lmatr}, the functions $h_{i-1,i}$ for $2 \leq i \leq n$ coincide with the ones in the case of the standard BD pair. This situation was already studied in \cite{double}, so let us assume that $n-1 \in \Gamma^r_2$. For $i < n$, the mutation at $h_{ii}$ can be written as
\begin{equation}\label{eq:hiimut}
h_{ii} h^\prime_{ii} = h_{i,i+1} f_{1,n-i+1} + f_{1,n-i} h_{i-1,i}.
\end{equation}
Let $\mathcal{L}$ be the $\mathcal{L}$-matrix that defines the functions $h_{i-1,i}$, $2 \leq i \leq n$, and let $H_{i-1,i}$ be a submatrix of $\mathcal{L}$ such that $h_{i-1,i} = \det H_{i-1,i}$. Then $H_{i-1,i}$ can be written as a block-diagonal matrix
\[
H_{i-1,i} = \begin{bmatrix} Y^{[i,n]}_{[i-1,n]} & * \\ 0 & C \end{bmatrix},
\]
where $C$ is some $(m-1) \times m$ matrix and the asterisk denotes the part of $H_{i-1,i}$ that's not relevant to the proof. Recall that $F_{1,n-i+1} = | X^{[n,n]}\, Y^{[i,n]} |_{[i-1,n]}$. Define a block-diagonal matrix $A$ as
\[
A:= \begin{bmatrix} F_{1,n-i+1} & * \\ 0 & C\end{bmatrix}
\]
and let $N$ be the index of the last column of $A$. According to the Desnanot-Jacobi identity from Proposition~\ref{p:dj13}, we see that
\begin{equation}\label{eq:dnnhii}
\det A^{\hat{1}} \det A^{\hat{2} \hat{N}}_{\hat{1}} + \det A^{\hat{N}} \det A^{\hat{1}\hat{2}}_{\hat{1}} = \det A^{\hat{1}\hat{N}}_{\hat{1}} \det A^{\hat{2}}.
\end{equation}
Now notice that
\[
\det A^{\hat{1}} = h_{i-1,i}, \ \det A^{\hat{2}\hat{N}}_{\hat{1}} = f_{1,n-i} \det C^{\hat{m}}, \ \det A^{\hat{N}} = f_{1,n-i+1} \det C^{\hat{m}},
\]
\[
\hat A^{\hat{1}\hat{2}}_{\hat{1}} = h_{i,i+1}, \ \det A^{\hat{1}\hat{N}}_{\hat{1}} = h_{ii} \det C^{\hat{m}},
\]
hence equation~\eqref{eq:dnnhii} becomes
\[
h_{i-1,i} f_{1,n-i} \det C^{\hat{m}} + f_{1,n-i+1} \det C^{\hat{m}} = h_{ii} \det C^{\hat{m}} \det A^{\hat{2}}.
\]
Dividing both sides by $\det C^{\hat{m}}$ and comparing the resulting expression with equation \eqref{eq:hiimut}, we see that $h^\prime_{ii} = \det A^{\hat{2}}$. Hence $h^\prime_{ii}$ is a regular function.

Now let's study the mutation at $h_{nn}$. Since we assume $n-1 \in \Gamma_2^r$, let $\gamma_r(i) = n-1$. Then the mutation reads
\[
h_{nn} h^\prime_{nn} = f_{11} g_{i+1,1} + g_{nn} h_{n-1,n}.
\]
Set $H:= H_{n-1,n}$. Then $h_{n-1,n} = y_{n-1,n} g_{i+1,1} - y_{nn} \det H^{\hat{1}}_{\hat{2}}$ and
\[\begin{split}
h_{nn} h^\prime_{nn} &= (y_{nn} x_{n,n-1} - y_{n-1,n} x_{nn}) g_{i+1,1} + x_{nn} (y_{n-1,n} g_{i+1,1} - y_{nn} \det H^{\hat{1}}_{\hat{2}}) = \\ &= h_{nn} (x_{n,n-1} g_{i+1,1} - x_{nn} \det H^{\hat{1}}_{\hat{2}}).
\end{split}
\]
Therefore, $h^\prime_{nn} = x_{n,n-1} g_{i+1,1} - x_{nn} \det H^{\hat{1}}_{\hat{2}}$ is a regular function.
\\

\noindent \emph{Mutation at $g_{ii}$.} As in the previous case, if $n-1 \notin \Gamma_1^c$, then the functions $g_{i+1,i}$ coincide with the ones in case of the standard BD pair, which was already treated in \cite{double}. Therefore, assume $n-1 \in \Gamma_1^c$, which implies there is a $Y$-block attached to the bottom of the leading $X$-block of the functions $g_{i+1,i}$. For $i < n$, the mutation at $g_{ii}$ is given by
\[
g_{ii} g^\prime_{ii} = f_{n-i,1}g_{i-1,i-1}g_{i+1,i} + f_{n-i+1,1}g_{i+1,i+1}g_{i,i-1}.
\]
Define $\tilde{F}_{n-i,1} := [Y^{[n, n]} \, X^{[i, n]}]_{[i, n]}$. Note that $\det (\tilde{F}_{n-i,\ 1})^{\hat 2} = (-1)^{n-i}f_{n-i, 1}$. Let $G_{i,i-1}$ be a submatrix of the $\mathcal{L}$-matrix such that $\det G_{i,i-1} = g_{i,i-1}$; it can be written as
\[
G_{i,i-1} = \begin{bmatrix}
X^{[i-1,n]}_{[i, n]} & 0\\
\ast & C
\end{bmatrix},
\]
where $C$ is some $m \times (m-1)$ matrix. Define
\[
A := A(i-1) := \begin{bmatrix}
\tilde{F}_{n-i+1,1} & 0 \\
\ast & C
\end{bmatrix}
\]
Let $N$ be the index of the last row of $A$. The Desnanot-Jacobi identity from Proposition~\ref{p:dj22} tells us that
\[
\det A \cdot \det A^{\hat 1 \hat 2}_{\hat 1 \hat N} = \det A^{\hat 1}_{\hat 1} \det A^{\hat 2}_{\hat N} - \det A^{\hat 1}_{\hat N} \det A^{\hat 2}_{\hat 1}.
\]
Deciphering the last identity yields
\[
\det A \cdot g_{ii} \det C_{\hat m} = g_{i, i-1} (-1)^{n-i+1}f_{n-i+1, 1} \det C_{\hat m} - g_{i-1,i-1} \det C_{\hat m} \det A^{\hat 2}_{\hat 1}
\]
or
\begin{equation}\label{eq:mut1}
\det A \cdot g_{ii} = g_{i, i-1} (-1)^{n-i+1}f_{n-i+1, 1}  - g_{i-1,i-1} \det A^{\hat 2}_{\hat 1}.
\end{equation}
Let $B:= A(i) = A_{\hat 1}^{\hat 2}$. The Desnanot-Jacobi identity from Proposition~\ref{p:dj22} for $B$ yields
\begin{equation}\label{eq:mut2}
\det B \cdot g_{i+1, i+1} = g_{i+1, i} (-1)^{n-i} f_{n-i, 1} - g_{i i} \det B_{\hat 1}^{\hat 2}.
\end{equation}
Now, multiply equations \eqref{eq:mut1} by $g_{i+1,i+1}$ and \eqref{eq:mut2} by $g_{i-1, i-1}$, substitute $\det A_{\hat{1}}^{\hat{2}} \cdot g_{i+1,i+1} \cdot g_{i-1, i-1}$ in equation \eqref{eq:mut1} with the right-hand side (RHS) of equation \eqref{eq:mut2} and combine the terms. These algebraic manipulations result in
\[
g_{ii} (-1)^{n-i+1}(g_{i+1,i+1}\det A - g_{i-1, i-1} \det B_{\hat 1}^{\hat 2}) = g_{i,i-1}f_{n-i+1,1} g_{i+1,i+1} + g_{i+1,i} f_{n-i,1} g_{i-1,i-1}.
\]
Thus the mutation at $g_{ii}$ for $1 < i < n$ yields a regular function.

Now consider the mutation at $g_{nn}$. Since we assume $n-1 \in \Gamma_1^c$, let $\gamma_c(n-1)=j$. Then the mutation at $g_{nn}$ reads
\[
g_{nn} g^\prime_{nn} = g_{n-1, n-1}h_{nn}h_{1, j+1} + f_{11} g_{n, n-1}.
\]
Since $g_{nn} = x_{nn}$, all we need to check is that the RHS is divisible by $x_{nn}$. Let $G := G_{n, n-1}$. Expanding $g_{n,n-1}$ along the first row, we obtain $g_{n, n-1} = x_{n, n-1} h_{1,j+1} - x_{nn} G^{\hat{2}}_{\hat{1}}$. Writing out $g_{n-1,n-1}$ and $f_{11}$, we see that
\[\begin{split}
g_{nn} g^\prime_{nn} = (x_{n-1,n-1}x_{nn} - x_{n-1,n}x_{n-1,n})y_{nn} h_{1,j+1}+ \\+ (x_{n-1,n}y_{nn} - y_{n-1,n}x_{nn})(x_{n, n-1} h_{1,j+1} - x_{nn} \det G^{\hat{2}}_{\hat{1}})
\end{split}
\]
After expanding the brackets, it's easy to see that there are two terms $x_{n-1,n}x_{n,n-1}y_{nn}h_{1,j+1}$ with opposite signs, hence they cancel each other out; all the other terms are divisible by $x_{nn}$. Thus the proposition is proved.
\end{proof}
\section{Completeness}\label{s:complet}
In this section, we prove part~\ref{c:natiso} of Proposition~\ref{p:starfish}, which asserts that any regular function belongs to the upper cluster algebra. Together with the results on regularity from Section~\ref{s:regular}, we will conclude that the ring of regular functions on $D(\GL_n)$ can be identified with the upper cluster algebra.

\subsection{Birational quasi-isomorphisms $\mathcal{U}$}\label{s:birat}
For this section, let us fix an aperiodic oriented BD pair $\bg := (\bg^r, \bg^c)$, let $D(\GL_n)_{\bg}$ be the corresponding Drinfeld double, and let $\gc(\bg)$ be the generalized cluster structure on $D(\GL_n)_{\bg}$. We consider another BD pair $\tilde{\bg}$ obtained from $\bg$ by removing a root from $\Gamma_1^r$ (or from $\Gamma_1^c$) and its image in $\Gamma_1^r$ (or in $\Gamma_2^c$; see the cases below), and define another Drinfeld double\footnote{We loosely refer to $D(\GL_n)_{\bg}$ as the Drinfeld double of $\GL_n$ even when $\bg^r \neq \bg^c$; strictly speaking, it is a Drinfeld double if and only if $\bg^r = \bg^c$.} $D(\GL_n)_{\tilde{\bg}}$ endowed with the generalized cluster structure $\gc(\tilde{\bg})$. The objective of this section is to construct a certain rational map \[\mathcal{U}:D(\GL_n)_{\tilde{\bg}}\dashrightarrow D(\GL_n)_{\bg},\] which we later recognize as a quasi-isomorphism in the sense of Proposition~\ref{p:compar} and as a birational automorphism of $\GL_n\times \GL_n$. In view of these two properties, we refer to the maps $\mathcal{U}$ as \emph{birational quasi-isomorphisms}\footnote{We do not provide a general definition of birational quasi-isomorphisms, but we use this term for any map $\mathcal{U}$ constructed in this section. We will give a comprehensive general treatment of these objects in our future publications.  \emph{Update, October 2023:} we have posted a paper on arXiv with a general formalism of birational quasi-isomorphisms.}.

\paragraph{Notation.} We denote by $(X,Y)$ the standard coordinates on $D(\GL_n)$ (regardless of the associated BD pair). If $\psi$ is a cluster or stable variable in $\gc(\bg)$, then by $\tilde{\psi}$ we denote the corresponding variable in $\gc(\tilde{\bg})$; that is, $\psi$ and $\tilde{\psi}$ are either the variables attached to the same vertices in the initial quivers or in the quivers that are obtained via the same sequences of mutations. All $g$-, $h$-, $f$-, $\varphi$- and $c$- functions in the initial extended cluster of $\gc(\tilde{\bg})$ are marked with a tilde as well.

\paragraph{Removing the rightmost root from a row run.} Let $\Delta^r = [p+1,p+k]$ be a nontrivial row $X$-run in $\bg$ and let $\bar{\Delta}^r = [q+1,q+k] := \gamma_r(\Delta^r)$ be the corresponding row $Y$-run. Define $\tilde{\bg} = (\tilde{\bg}^r, \bg^c)$ with $\tilde{\bg}^r = (\tilde{\Gamma}_1^r, \tilde{\Gamma}_2^r, \gamma_r|_{\tilde{\Gamma}_1^r})$ given by $\tilde{\Gamma}_1^r = \Gamma_1^r \setminus \{p+k-1\}$ and $\tilde{\Gamma}_2^r = \Gamma_2^r \setminus \{q+k-1\}$. Let us examine the difference between the $\mathcal{L}$-matrices in $\gc(\bg)$ and $\gc(\tilde{\bg})$. For any $\mathcal{L}$-matrix $\mathcal{L}(X,Y)$ in $\gc(\bg)$, let $\tilde{\mathcal{L}}(X,Y)$ be a matrix obtained from $\mathcal{L}(X,Y)$ via removing the last row of each $Y$-block of the form $Y^{J}_{[1,q+k]}$. If $\mathcal{L}(X,Y)$ arises from a maximal alternating path in $G_{\bg}$ that does not pass through the edge $(p+k-1)\xrightarrow{\gamma_r}(q+k-1)$, then $\tilde{\mathcal{L}}(X,Y)$ is an $\mathcal{L}$-matrix in $\gc(\tilde{\bg})$ that arises from the same path in $G_{\tilde{\bg}}$. However, if $\mathcal{L}^*(X,Y):=\mathcal{L}(X,Y)$ is constructed from a path that does pass through $(p+k-1)\xrightarrow{\gamma_r}(q+k-1)$, then $\tilde{\mathcal{L}}^*(X,Y)$ is a reducible matrix with blocks that correspond to the remaining two $\mathcal{L}$-matrices in $\gc(\tilde{\bg})$. Let us set $s_0$ to be the number such that $\mathcal{L}^*_{s_0s_0}(X,Y) = x_{p+k,1}$. Define a rational map $\mathcal{U} : D(\GL_n)_{\tilde{\bg}} \dashrightarrow D(\GL_n)_{\bg}$ via the following data:
\begin{equation}\label{eq:alp_coef}
\alpha_i(X,Y) := \frac{1}{\tilde{g}_{p+k,1}(X,Y)} {\det \tilde{\mathcal{L}}^*}^{[s_0, N(\tilde{\mathcal{L}}^*)]}_{\{s_0-k+i\}\cup [s_0+1,N(\tilde{\mathcal{L}}^*)]}(X,Y), \ \ i=1,\ldots,k-1;
\end{equation}
\begin{equation}\label{eq:u0}
U_0(X,Y) = I + \sum_{i = 1}^{k-1} \alpha_i(X,Y) e_{q+i,q+k};
\end{equation}
\begin{equation}\label{eq:defu}
U_+(X,Y) := \left(\prod_{k\geq 1}^{\leftarrow} \tilde{\gamma}_r^k(U_0)\right)U_0;
\end{equation}
\begin{equation}\label{eq:udef_rr}
\mathcal{U}(X,Y) := \left(U_+(X,Y)X,U_+(X,Y)Y\right).
\end{equation}

\begin{proposition}\label{p:rightmost}
Let $\mathcal{U}:D(\GL_n)_{\tilde{\bg}} \dashrightarrow D(\GL_n)_{\bg}$ be the rational map given by equation~\eqref{eq:udef_rr}. Then the map $\mathcal{U}$ acts on the cluster and stable variables via the following formulas:
\begin{equation}\label{eq:rightmost}
\mathcal{U}^*(g_{ij}(X,Y)) = \begin{cases}
\tilde{g}_{ij}(X,Y)\tilde{g}_{p+k,1}(X,Y) \ &\text{if}\ \mathcal{L}^*_{ss}(X,Y) = x_{ij} \ \text{for} \ s < s_0;\\
\tilde{g}_{ij}(X,Y) \ &\text{otherwise};
\end{cases}
\end{equation}
\begin{equation}\label{eq:rightmost2}
\mathcal{U}^*(h_{ij}(X,Y)) = \begin{cases}
\tilde{h}_{ij}(X,Y)\tilde{g}_{p+k,1}(X,Y) \ &\text{if}\ \mathcal{L}^*_{ss}(X,Y) = y_{ij} \ \text{for} \ s < s_0;\\
\tilde{h}_{ij}(X,Y) \ &\text{otherwise};
\end{cases}
\end{equation}
if $\psi$ is any $\varphi$-, $f$- or $c$- function in the initial extended cluster, then
\begin{equation}\label{eq:rightmost3}
\mathcal{U}^*(\psi(X,Y)) = \tilde{\psi}(X,Y).
\end{equation}
\end{proposition}
Note that the first lines in equations \eqref{eq:rightmost} and \eqref{eq:rightmost2} reflect the fact that $\tilde{\mathcal{L}}^*(X,Y)$ is a reducible matrix with blocks equal to a pair of $\mathcal{L}$-matrices from $\gc(\tilde{\bg})$. The proof of the above proposition is exactly the same as in~\cite{plethora}.

\paragraph{Motivation of formulas~\eqref{eq:alp_coef}-\eqref{eq:udef_rr}.} Though the formulas are complicated, they are designed in concordance with the invariance properties of the variables. The $\varphi$-, $f$- and $c$-variables are the same in the initial extended clusters of $\gc(\bg)$ and $\gc(\tilde{\bg})$, and they are all invariant with respect to the left action $N_+.(X,Y) = (N_+X,N_+Y)$ by unipotent upper triangular matrices. Since $U_+$ is such, we see that formula~\eqref{eq:rightmost3} holds. Now, if $\psi$ is a $g$- or $h$-function, recall that one of its invariance properties reads
\[
\psi(N_+X,\tilde{\gamma}_r(N_+)Y) = \psi(X,Y).
\]
Notice that $\tilde{\gamma}_r(U_+)\cdot U_0 = U_+$; therefore,
\[
\mathcal{U}^*(\psi(X,Y)) = \psi(U_+X,U_+Y) = \psi(U_+X,\tilde{\gamma}_r(U_+)U_0Y) = \psi(X,U_0Y);
\]
hence, the main part of the proof of Proposition~\ref{p:rightmost} is to show that 
\[
\psi(X,U_0Y) = \tilde{\psi}(X,Y)\tilde{g}_{p+k,1}^\varepsilon(X,Y)
\] 
for some $\varepsilon \geq 0$. A similar reasoning explains formulas for $\mathcal{U}$ for other choices of roots below.

\paragraph{The inverse of $\mathcal{U}$.} Though we do not need formulas for the inverse of $\mathcal{U}$ in the proofs (except in some simple cases), let us state them for completeness. Let $\theta_r:=\gamma_r|_{\tilde{\Gamma}_1^r}$ be the BD map for the triple $\tilde{\bg}^r$. The verification of the formulas is similar to the proof of Proposition~\ref{p:rightmost} and is based on an application of a series of long Pl\"ucker relations.
\begin{equation}\label{eq:alp_coef_inv}
\beta_i(X,Y) := -\frac{1}{g_{p+k,1}(X,Y)} {\det (\mathcal{L}}^*)^{[s_0, N(\tilde{\mathcal{L}}^*)]}_{\{s_0-k+i\}\cup [s_0+1,N(\tilde{\mathcal{L}}^*)]}(X,Y), \ \ i=1,\ldots,k-1;
\end{equation}
\begin{equation}\label{eq:uinv0}
\tilde{U}_0(X,Y) = I + \sum_{i = 1}^{k-1} \beta_i(X,Y) e_{q+i,q+k};
\end{equation}
\begin{equation}\label{eq:defuinv}
\tilde{U}_+(X,Y) := \left(\prod_{k\geq 1}^{\leftarrow} \tilde{\theta}_r^k(\tilde{U}_0)\right)\tilde{U}_0;
\end{equation}
\begin{equation}\label{eq:udef_rrinv}
\mathcal{U}^{-1}(X,Y) := \left(\tilde{U}_+(X,Y)X,\tilde{U}_+(X,Y)Y\right).
\end{equation}
The formulas for the inverse of $\mathcal{U}$ in the other cases below are obtained via the same scheme: 1) add an extra negative sign in front of the coefficients; 2) substitute $\tilde{\mathcal{L}}^*$ with $\mathcal{L}^*$ and the frozen variable in the denominator with the corresponding cluster variable from $\gc(\bg)$; 3) substitute $\tilde{\gamma}$ with $\tilde{\theta}$.

\paragraph{Removing the leftmost root from a row run.} As before, let $\Delta^r = [p+1,p+k]$ be a nontrivial row $X$-run in $\bg$ and let $\bar{\Delta}^r = [q+1,q+k] := \gamma_r(\Delta^r)$ be the corresponding row $Y$-run. Define $\tilde{\bg} = (\tilde{\bg}^r, \bg^c)$ with $\tilde{\bg}^r = (\tilde{\Gamma}_1^r, \tilde{\Gamma}_2^r, \gamma_r|_{\tilde{\Gamma}_1^r})$ given by $\tilde{\Gamma}_1^r = \Gamma_1^r \setminus \{p+1\}$ and $\tilde{\Gamma}_2^r = \Gamma_2^r \setminus \{q+1\}$. For an $\mathcal{L}$-matrix $\mathcal{L}(X,Y)$ in $\gc(\bg)$, let $\tilde{\mathcal{L}}(X,Y)$ be a matrix that is obtained from $\mathcal{L}(X,Y)$ by removing the first row of each $X$-block of the form $X^{J}_{[p+1,n]}$. If $\mathcal{L}(X,Y)$ arises from a path that does not traverse the edge $(p+1) \xrightarrow{\gamma_r} (q+1)$, then $\tilde{\mathcal{L}}(X,Y)$ is an $\mathcal{L}$-matrix in $\gc(\tilde{\bg})$; if it does traverse the latter edge, $\tilde{\mathcal{L}}(X,Y)$ is a reducible matrix with blocks that are $\mathcal{L}$-matrices in $\gc(\tilde{\bg})$. Let us denote the $\mathcal{L}$-matrix that corresponds to the latter path as $\mathcal{L}^*(X,Y)$, and let us denote by $s_0$ the number for which $\mathcal{L}^*_{s_0s_0}(X,Y) = x_{p+2,2}$.  
We define the rational map $\mathcal{U}:D(\GL_n)_{\tilde{\bg}} \dashrightarrow D(\GL_n)_{\bg}$ via the following data:
\begin{equation}
\alpha_i(X,Y) := (-1)^{i-1}\frac{1}{\tilde{g}_{p+2,1}(X,Y)} {\det \tilde{\mathcal{L}}^*}^{[s_0, N(\tilde{\mathcal{L}}^*)]}_{[s_0-1,N(\tilde{\mathcal{L}}^*)]\setminus\{s_0+i-1\}}(X,Y), \ \ i = 1,\ldots,k-1;
\end{equation}
\begin{equation}
U_0 := I + \sum_{i=1}^{k-1}\alpha_{i}(X,Y)e_{q+1,q+i+1};
\end{equation}
\begin{equation}
U_+:=\left(\prod_{k\geq 1}^{\leftarrow} \tilde{\gamma}_r^k(U_0)\right)U_0;
\end{equation}
\begin{equation}\label{eq:leftmostu}
\mathcal{U}(X,Y) := \left(U_+(X,Y)X,U_+(X,Y)Y\right).
\end{equation}
The next proposition corresponds to Theorem 7.3 in \cite{plethora} and can be proved in exactly the same way:
\begin{proposition}\label{p:leftmost}
Let $\mathcal{U}:D(\GL_n)_{\tilde{\bg}} \dashrightarrow D(\GL_n)_{\bg}$ be the rational map given by equation~\eqref{eq:leftmostu}. Then the map $\mathcal{U}$ acts on the cluster and stable variables via the following formulas:
\begin{equation}\label{eq:leftmost}
\mathcal{U}^*(g_{ij}(X,Y)) = \begin{cases}
\tilde{g}_{ij}(X,Y)\tilde{g}_{p+2,1}(X,Y) \ &\text{if}\ \mathcal{L}^*_{ss}(X,Y) = x_{ij} \ \text{for} \ s < s_0;\\
\tilde{g}_{ij}(X,Y) \ &\text{otherwise};
\end{cases}
\end{equation}
\begin{equation}\label{eq:leftmost2}
\mathcal{U}^*(h_{ij}(X,Y)) = \begin{cases}
\tilde{h}_{ij}(X,Y)\tilde{g}_{p+2,1}(X,Y) \ &\text{if}\ \mathcal{L}^*_{ss}(X,Y) = y_{ij} \ \text{for} \ s < s_0;\\
\tilde{h}_{ij}(X,Y) \ &\text{otherwise};
\end{cases}
\end{equation}\label{eq:leftmost3}
if $\psi$ is any $\varphi$-, $f$- or $c$- function in the initial extended cluster, then
\begin{equation}
\mathcal{U}^*(\psi(X,Y)) = \tilde{\psi}(X,Y).
\end{equation}
\end{proposition}

\paragraph{Removing roots from column runs.} For a BD triple $\bg = (\Gamma_1,\Gamma_2,\gamma)$, let us define the \emph{opposite} BD triple $\bg^{\text{op}}$ as $\bg^{\text{op}}:= (\Gamma_2,\Gamma_1,\gamma^*)$; likewise, if $\bg = (\bg^r,\bg^c)$ is a BD pair, we call $\bg^{\text{op}} := ((\bg^c)^{\text{op}}, (\bg^r)^{\text{op}})$ the \emph{opposite} BD pair. As explained in \cite{plethora}, the $\mathcal{L}$-matrices in $\gc(\bg)$ and $\gc(\bg^{\text{op}})$ are related via the involution
\[
\mathcal{L}(X,Y) \mapsto \mathcal{L}(Y^t, X^t)^t.
\]
In particular, the involution $(X,Y)\mapsto (Y^t, X^t)$ maps $g$- and $h$-functions from $\gc(\bg)$ to $h$- and $g$-functions from $\gc(\bg^{\text{op}})$. This allows one to translate the construction of the rational maps from the case of removing a pair of roots from row runs to the case of removing a pair of roots from column runs. In the latter case, for some unipotent lower triangular matrix $U_0:=U_0(X,Y)$, we set
\[
U_- := U_0 \prod_{k \geq 1}^{\rightarrow} (\tilde{\gamma}_c^*)^k(U_0)
\]
and define the rational map $\mathcal{U} : D(\GL_n)_{\tilde{\bg}} \dashrightarrow D(\GL_n)_{\bg}$ via
\[
\mathcal{U}(X,Y) := (XU_-(X,Y),YU_-(X,Y)).
\]
As one can observe in the previous cases, the entries of the matrix $U_0$ belong to the localization $\mathcal{O}(D(\GL_n))[\tilde{\psi}^{\pm}_{\square}]$ where the variable $\tilde{\psi_{\square}}$ is a stable variable in $\gc(\tilde{\bg})$ such that $\psi_{\square}$ is a cluster variable in $\gc(\bg)$ (see the paragraph on notation above). Let $\Delta^c = [p+1,p+k]$ be a nontrivial column $X$-run and $\gamma(\Delta^c) = [q+1,q+k]$ be the corresponding column $Y$-run. Then, if $p+1$ and $q+1$ are removed, the variable $\psi_{\square}$ is $\psi_{\square}=h_{1,q+2}$; if $p+k-1$ and $q+k-1$ are removed, then $\psi_{\square} = h_{1,q+k}$.

\paragraph{The action of $\mathcal{U}$ upon other clusters.} 
Let $\tilde{\bg}$ be obtained from $\bg$ in one of the four ways described above (i.e., by removing a pair of rightmost or leftmost roots from row or column runs), and let $\psi_{\square}$ be the variable that is cluster in $\gc(\bg)$ and such that the corresponding variable $\tilde{\psi}_{\square}$ is stable in $\gc(\tilde{\bg})$. The following proposition corresponds to Proposition~7.4 in~\cite{plethora} and describes the action of the maps $\mathcal{U}$ defined above upon an extended cluster other than the initial one.
\begin{proposition}\label{p:comparu}
If $\psi$ and $\tilde{\psi}$ are cluster variables in $\gc(\bg)$ and $\gc(\tilde{\bg})$ that are obtained via the same sequences of mutations, then
\[
\mathcal{U}^*(\psi(X,Y)) = \tilde{\psi}(X,Y) \tilde{\psi}_{\square}(X,Y)^\lambda, \ \ \ \lambda:={\frac{\deg(\psi)-\deg(\tilde{\psi})}{\deg{\tilde{\psi}_{\square}}}}
\]
where $\deg$ denotes the polynomial degree.
\end{proposition}
\begin{proof}
The proposition is a straight consequence of Proposition~\ref{p:compar}. The required global toric actions in $\gc(\bg)$ and $\gc(\tilde{\bg})$ have their weight vectors formed by the degrees of the cluster and stable variables considered as polynomials, and the map $\theta$ coincides with the map $\mathcal{U}$. Indeed, if $\psi$ is any variable from the initial extended cluster such that $\deg(\psi) = \deg(\tilde{\psi})$, then $\theta(\psi) = \tilde{\psi} = \mathcal{U}(\psi)$. However, if $\psi$ and $\tilde{\psi}$ have different degrees, then the formulas for $\mathcal{U}$ (see Proposition~\ref{p:leftmost} or Proposition~\ref{p:rightmost}) suggest that $\deg\psi - \deg\tilde{\psi} = \deg \psi_{\square} = \deg \tilde{\psi}_{\square}$. Therefore, 
\[
\theta(\psi) = \tilde{\psi} \tilde{\psi}_{\square}^{\frac{\deg\psi-\deg\tilde{\psi}}{\det \tilde{\psi}_{\square}}} = \tilde{\psi} \tilde{\psi}_{\square} = \mathcal{U}(\psi).
\]
Thus the maps $\theta$ and $\mathcal{U}$ are the same (when viewed as maps between the rings generated by the initial extended clusters), and the conclusion of Proposition~\ref{p:compar} for the map $\theta$ is exactly the statement of Proposition~\ref{p:comparu}.
\end{proof}
 
For the next corollaries, if $\Sigma$ is any seed in $\gc(\bg)$, we set $\mathcal{L}_{\mathbb{C}}(\Sigma):=\mathcal{L}(\Sigma)\otimes \mathbb{C}$ to be the complexification of the ring of Laurent polynomials associated with the seed $\Sigma$ (see equation~\eqref{eq:laurring} for the definition). Likewise, $\tilde{\mathcal{L}}_{\mathbb{C}}(\tilde{\Sigma})$ denotes the ring of Laurent polynomials associated to a seed $\tilde{\Sigma} \in \gc(\tilde{\bg})$. The below corollaries appeared in~\cite{plethora} in a disguised form in the proof of Theorem~3.12.

\begin{corp}\label{c:comparu}
Let $\Sigma$ and $\tilde{\Sigma}$ be seeds in $\gc(\bg)$ and $\gc(\tilde{\bg})$ obtained via the same sequences of mutations from the initial seeds, and let $\mathcal{L}_{\mathbb{C}}(\Sigma)$ and $\tilde{\mathcal{L}}_{\mathbb{C}}(\tilde{\Sigma})$ be the corresponding rings of Laurent polynomials. If $\mathcal{O}(\GL_n) \subseteq \tilde{\mathcal{L}}_{\mathbb{C}}(\tilde{\Sigma})$, then $\mathcal{O}(\GL_n) \subseteq \mathcal{L}_{\mathbb{C}}(\Sigma)$.
\end{corp}
\begin{proof}

It's a consequence of Proposition~\ref{p:comparu} that $\mathcal{U}^*$ can be viewed as an isomorphism $\mathcal{L}_{\mathbb{C}}(\Sigma) \xrightarrow{\sim} \tilde{\mathcal{L}}_{\mathbb{C}}(\tilde{\Sigma})[\tilde{\psi}_{\square}^{\pm 1}]$. Since $\mathcal{U}^*(\mathcal{O}(\GL_n)) \subseteq \mathcal{O}(\GL_n)[\tilde{\psi}_{\square}^{\pm 1}] \subseteq \tilde{\mathcal{L}}_{\mathbb{C}}(\tilde{\Sigma})[\tilde{\psi}_{\square}^{\pm 1}]$, we see that $\mathcal{O}(\GL_n) \subseteq \mathcal{L}_{\mathbb{C}}(\Sigma)$.\qedhere
\end{proof}

\begin{corp}\label{c:comparu2}
Let $\tilde{\mathcal{N}}$ be a nerve in $\gc(\tilde{\bg})$ and $\mathcal{N}^\prime$ be the corresponding set of seeds in $\gc(\bg)$. Set $\mathcal{N} := \mathcal{N}^\prime \cup \Sigma_{\psi_{\square}}$ to be a nerve in $\gc(\bg)$, where $\Sigma_{\psi_{\square}}$ is a seed adjacent to any seed of $\mathcal{N}^\prime$ in the direction of $\psi_{\square}$. If $\mathcal{O}(D(\GL_n))\subseteq \bar{\mathcal{A}}_{\mathbb{C}}(\gc(\tilde{\bg}))$ and $\mathcal{O}(D(\GL_n)) \subseteq \mathcal{L}_{\mathbb{C}}(\Sigma_{\psi_{\square}})$, then $\mathcal{O}(D(\GL_n))\subseteq\bar{\mathcal{A}}_{\mathbb{C}}(\gc(\bg))$.
\end{corp}
\begin{proof}
Since $\bar{\mathcal{A}}_{\mathbb{C}}(\gc(\tilde{\bg}) = \bigcap_{\tilde{\Sigma} \in \tilde{\mathcal{N}}} \tilde{\mathcal{L}}_{\mathbb{C}}(\tilde{\Sigma})$, it follows from Corollary~\ref{c:comparu} that
\[
\mathcal{O}(\GL_n) \subseteq \bigcap_{\Sigma \in \mathcal{N}^\prime}\mathcal{L}_{\mathbb{C}}(\Sigma);
\]
since in addition $\mathcal{O}(D(\GL_n)) \subseteq \mathcal{L}_{\mathbb{C}}(\Sigma_{\psi_{\square}})$, we conclude that
\[
\mathcal{O}(\GL_n) \subseteq\bigcap_{\Sigma \in \mathcal{N}}\mathcal{L}_{\mathbb{C}}(\Sigma) = \bar{\mathcal{A}}_{\mathbb{C}}(\mathcal{N}) = \bar{\mathcal{A}}_{\mathbb{C}}(\gc(\bg)).\qedhere
\]
\end{proof}
The conclusion of Corollary~\ref{c:comparu2} corresponds to Condition~\ref{c:natiso} of Proposition~\ref{p:starfish}. Hence if the other two conditions of the proposition are satisfied, then $\mathcal{O}(D(\GL_n))$ is naturally isomorphic to $\bar{\mathcal{A}}_{\mathbb{C}}(\gc(\bg))$.
\subsection{Auxiliary mutation sequences}
As in~\cite{plethora}, we will use the same inductive argument on the size $|\Gamma_1^r|+|\Gamma_1^c|$ in order to prove that $\bar{\mathcal{\mathcal{A}}}$ is naturally isomorphic to $\mathcal{O}(D(\GL_n))$. The step of the induction is simple and relies upon Corollary~\ref{c:comparu2} and the existence of at least two different birational quasi-isomorphisms (i.e., arising from a removal of different roots). However, for the base of the induction, which is $|\Gamma_1^r|+|\Gamma_1^c| = 1$, we will need to express manually the standard coordinates $x_{ij}$ and $y_{ij}$ as elements of $\mathcal{L}_{\mathbb{C}}(\Sigma_{\psi_{\square}})$, where the seed $\Sigma_{\psi_{\square}}$ is adjacent to the initial one in the direction of $\psi_{\square}$ (see the previous section). This, in turn, will be substantially based on the Laurent phenomenon: If we know that a certain polynomial $p(X,Y)$ is a cluster variable, then $p(X,Y) \in \mathcal{L}_{\mathbb{C}}(\Sigma_{\psi_{\square}})$, and therefore $p(X,Y)$ can be used in the production of Laurent expressions\footnote{Note that the only invertible elements of $\mathcal{L}_{\mathbb{C}}(\Sigma_{\psi_{\square}})$ are monomials in the invertible frozen variables and cluster variables of $\Sigma_{\psi_{\square}}$, so if $p(X,Y)$ does not belong to $\Sigma_{\psi_{\square}}$, we cannot divide by $p(X,Y)$, but we can add it and multiply by it in the process.} for $x_{ij}$ or $y_{ij}$ even if we do not know a precise Laurent expansion of $p(X,Y)$ in terms of the variables of $\Sigma_{\psi_{\square}}$. Thus, the objective of this section is to enrich our database of cluster variables, which will be used in manufacturing Laurent expressions of the standard coordinates $x_{ij}$ and $y_{ij}$. 

\subsubsection{Sequence $B_s$ in the standard $\gc$}
For this section, let us consider only one generalized cluster structure on $D(\GL_n)$ induced by the standard BD pair. For $2 \leq s \leq n$, define a sequence of mutations $B_s$ as
\[\begin{split}
h_{sn} \rightarrow h_{s,n-1} \rightarrow \cdots \rightarrow h_{s,s+1} \rightarrow \\ \rightarrow h_{ss} \rightarrow f_{1,n-s} \rightarrow f_{2,n-s-1} \rightarrow \cdots \rightarrow f_{n-s,1} \rightarrow \\ \rightarrow g_{ss} \rightarrow g_{s,s-1} \rightarrow \cdots \rightarrow g_{s2}.
\end{split}
\]
The pathway of the sequence is illustrated in Figure~\ref{f:bs_highlight}. 
\begin{figure}
\begin{center}
\includegraphics[scale=0.42]{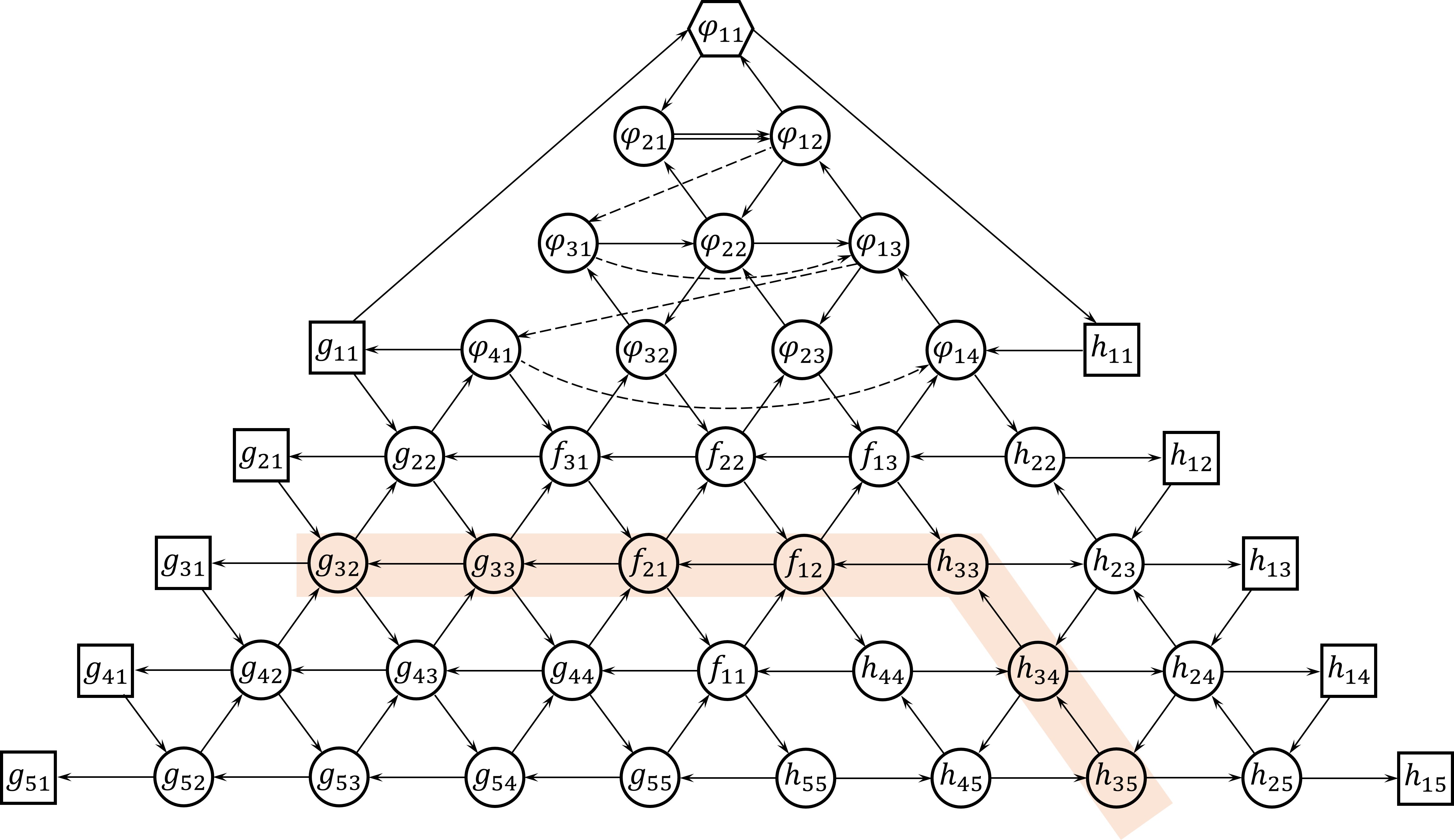}
\end{center}
\caption{The initial quiver of the standard $\gc$ in $n=5$. The vertices of the sequence $B_s$ for $s=3$ are highlighted.}
\label{f:bs_highlight}
\end{figure}
\begin{lemma}\label{l:base_ind}
Apply the mutation sequence $B_s$ to the initial seed. Then the resulting seed contains the following cluster variables:
\begin{equation}\label{eq:bshsn}
h^\prime_{s,n-i+1} = \det Y^{[n-i,n]}_{\{s-1\} \cup [s+1,s+i]}, \ \ i \in [1,n-s];
\end{equation}
\begin{equation}\label{eq:bsf}
f_{i,n-s-i+1}^\prime = \det [ X^{[n-i,n]} \ Y^{[s+i+1,n]}]_{\{s-1\}\cup [s+1,n]} , \  \ i \in [0,n-s];
\end{equation}
\begin{equation}\label{eq:bsgss}
g_{s,s-i+1}^\prime = \det X^{[s-i,n-i]}_{\{s-1\}\cup [s+1,n]}, \ \ i \in [1,s-1].
\end{equation}
\end{lemma}
\begin{proof}
\noindent The mutation at $h_{sn}$ reads
\[
h_{sn} h^\prime_{sn} = h_{s+1,n} h_{s-1,n-1} + h_{s,n-1} h_{s-1,n},
\]
which is simply
\[
y_{sn} h^\prime_{sn} = y_{s+1,n} \det \begin{bmatrix} y_{s-1,n-1} & y_{s-1,n}\\ y_{s,n-1} & y_{sn} \end{bmatrix} + y_{s-1,n} \det \begin{bmatrix} y_{s,n-1} & y_{sn} \\ y_{s+1,n-1} & y_{s+1,n} \end{bmatrix}
\]
hence $h^\prime_{sn} = \det Y^{[n,n-1]}_{\{s-1,s+1\}}$. Once we've mutated along the sequence $h_{sn} \rightarrow \cdots \rightarrow h_{s,n-i+1}$, the mutation at $h_{s,n-i}$ reads
\[
h_{s,n-i} h_{s,n-i}^\prime = h_{s,n-i-1} h_{s,n-i+1}^\prime + h_{s-1,n-i-1} h_{s+1,n-i}.
\]
This is a Desnanot-Jacobi identity from Proposition~\ref{p:dj13} applied to the matrix
\[
\begin{matrix}
 & \downarrow &  &  &  & \\
\rightarrow & y_{s-1,n-i-1} & y_{s-1,n-i} & y_{s-1,n-i+1} & \ldots & y_{s-1,n}\\
\rightarrow & y_{s,n-i-1} & y_{s,n-i} & y_{s,n-i+1} & \ldots & y_{s,n}\\
&\vdots & \vdots & \vdots & \ldots & \vdots\\
\rightarrow & y_{s+1+i,n-i-1} & y_{s+1+i,n-i} & y_{s+1+i,n-i+1} & \ldots & y_{s+1+i,n}
\end{matrix}
\]
with rows and columns chosen as indicated by arrows (the first two rows, the last row and the first column). We obtain $h^\prime_{s,n-i+1} = \det Y^{[n-i,n]}_{\{s-1\}\cup [s+1,s+i]}$. Next, the mutation at $h_{ss}$ is given by
\[
h_{ss} h^\prime_{ss} = f_{1,n-s}h_{s,s+1}^\prime + f_{1,n-s+1} h_{s+1,s+1}.
\]
This is a Desnanot-Jacobi identity from Proposition~\ref{p:dj22} applied to the matrix
\[
\begin{matrix}
& \downarrow & \downarrow & & & \\
\rightarrow & x_{s-1,n} & y_{s-1,s} & y_{s-1,s+1} & \ldots & y_{s-1,n}\\
\rightarrow & x_{sn} & y_{ss} & y_{s,s+1} & \ldots & y_{sn}\\
 & \vdots & \vdots &\vdots & \ldots & \vdots \\
 & x_{nn} & y_{ns} & y_{n,s+1} & \ldots & y_{nn}
\end{matrix}
\]
hence $h_{ss}^\prime = \det [X^{[n]} \ Y^{[s+1,n]}]_{\{s-1\}\cup [s+1,n]}$. Next, assuming the conventions $f_{0,n-j} = h_{j+1,j+1}$ and $f_{n-j,0} = g_{j+1,j+1}$ (see equation~\eqref{eq:fconv}), the subsequent mutations along the path $f_{1,n-s} \rightarrow \cdots \rightarrow f_{n-s,1}$ yield
\[
f_{i,n-s-i+1} f_{i,n-s-i+1}^\prime = f_{i+1,n-s-i+1} f_{i,n-s-i} + f_{i+1,n-s-i} f^\prime_{i-1,n-s-i+2} , \ i \in [1,n-s]
\]
Assuming by induction that $f_{i-1,n-s-i+2}^\prime = \det [X^{[n-i+1,n]} \ Y^{[s+i,n]}]_{\{s-1\}\cup [s+1,n]}$, the latter relation becomes a Desnanot-Jacobi identity from Proposition~\ref{p:dj22} for the matrix $[X^{[n-i,n]} \ Y^{[s+i-2,n]}]_{\{s-1\}\cup [s+1,n]}$ applied as indicated:
\[
\begin{matrix}
& \downarrow & & & \downarrow & & & &\\
\rightarrow & x_{s-1,n-i} & x_{s-1,n-i+1} & \ldots & x_{s-1,n} & y_{s-1,i+s} & y_{s-1,i+s+1} & \ldots & y_{s-1,n}\\
\rightarrow & x_{s,n-i} & x_{s,n-i+1} & \ldots & x_{s,n} & y_{s,i+s} & y_{s,i+s+1} & \ldots & y_{s,n}\\
 &\vdots & \vdots & \ldots & \vdots & \vdots & \vdots & \ldots & \vdots\\
& x_{n,n-i} & x_{n,n-i+1} & \ldots & x_{nn} & y_{n,i+s} & y_{n,i+s+1} & \ldots & y_{n,n}
\end{matrix}
\]
Therefore, $f_{i,n-s-i+1}^\prime = \det [X^{[n-i,n]} \ Y^{[s+i+1,n]}]_{\{s-1\}\cup[s+1,n]}$ (note that $f_{n-s,1}$ consists entirely of variables from $X$). Lastly, the mutations along the path $g_{ss}  \rightarrow \cdots \rightarrow g_{s2}$ read
\[
g_{s,s-i+1} g_{s,s-i+1}^\prime = g_{s,s-i} g_{s,s-i+2}^\prime + g_{s-1,s-i} g_{s+1,s-i+1},  i \in [1,s-1].
\] 
Assuming by induction $g^\prime_{s,s-i+2} = \det X^{[s-i+1,n-i+1]}_{\{s-1\}\cup[s+1,n]}$, apply the Desnanot-Jacobi identity to the matrix
\[
\begin{matrix}
&\downarrow & & & \downarrow \\
\rightarrow & x_{s-1,s-i} & x_{s-1,s-i+1} & \ldots & x_{s-1,n-i+1} \\
\rightarrow & x_{s,s-i} & x_{s,s-i+1} & \ldots & x_{s,n-i+1}\\
&\vdots& \vdots & \ldots & \vdots \\
&x_{n,s-i} & x_{n,s-i+1} & \ldots & x_{n,n-i+1}\\
\end{matrix}
\]
At last, we obtain that $g_{s,s-i+1}^\prime = \det X^{[s-i,n-i]}_{\{s-1\} \cup [s+1,n]}$.
\end{proof}

\subsubsection{Sequence $B_{s-k} \rightarrow \ldots \rightarrow B_{s}$ in the standard $\gc$}

\begin{figure}
\begin{center}
\includegraphics[scale=0.42]{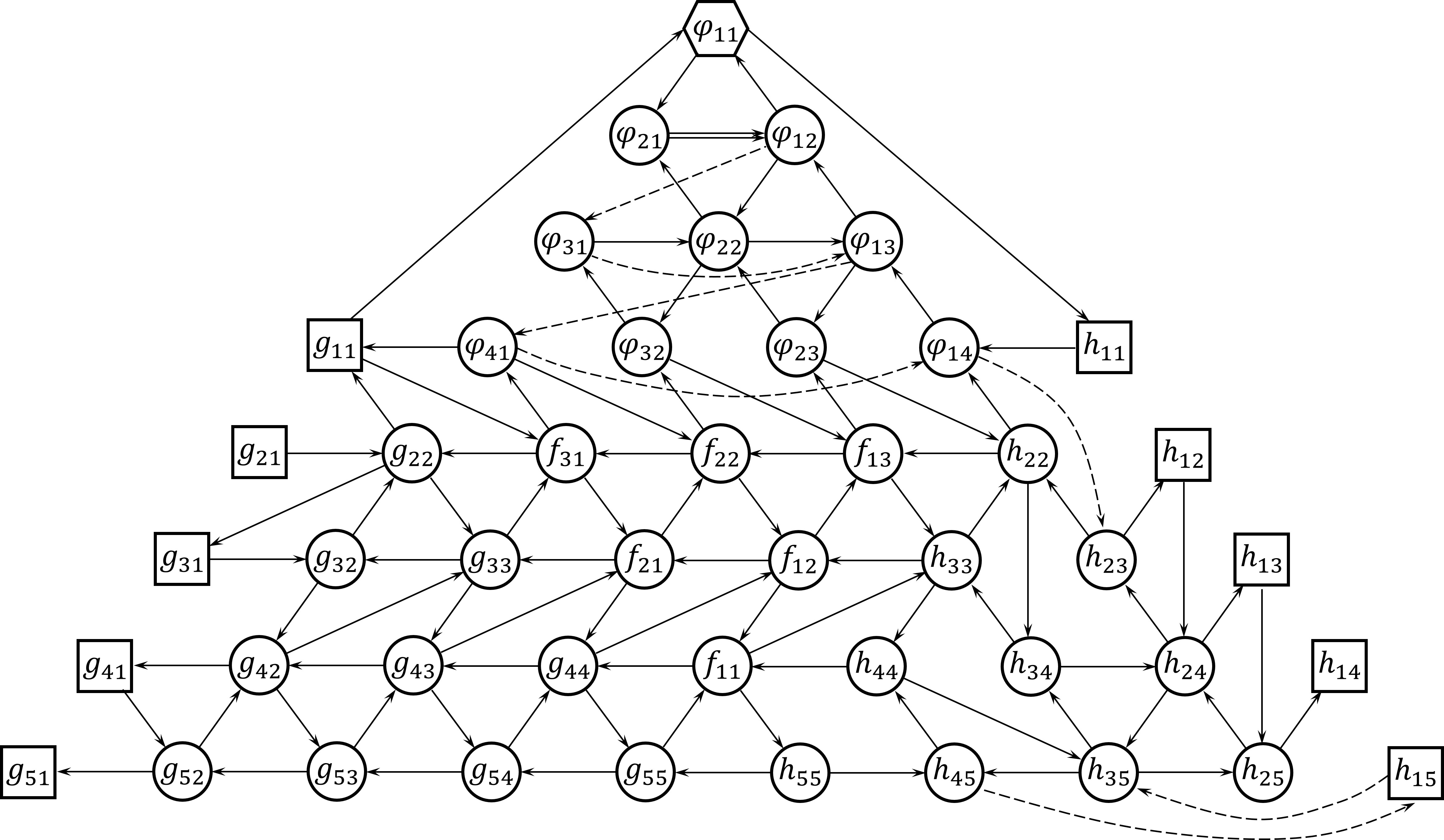}
\end{center}
\caption{The result of mutating the initial quiver along the sequence $B_2 \rightarrow B_3$ ($n=5$, the standard $\gc$).}
\label{f:b2b3}
\end{figure}

\begin{lemma}\label{l:seqs_c}
Let us apply the mutation sequence $B_{s-k}\rightarrow \cdots \rightarrow B_s$ to the initial seed. Then the resulting seed contains the following cluster variables:
\begin{equation}\label{eq:bskh}
h_{s,n-i+1}^\prime = \det Y_{\{s-k-1\}\cup[s+1,s+i]}^{[n-i,n]}, \ i \in [1,n-s];
\end{equation}
\begin{equation}\label{eq:bskf}
f^\prime_{i,n-s-i+1} = \det [ X^{[n-i,n]} Y^{[s+i+1,n]}]_{\{s-k-1\}\cup [s+1,n]}, \ \ i \in [0,n-s];
\end{equation}
\begin{equation}\label{eq:bskg}
 g^\prime_{s,s-i+1} = \det X^{[s-i,n-i]}_{\{s-k-1\}\cup[s+1,n]}, \ \ i \in [1,s-1].
\end{equation}
\end{lemma}
\begin{proof}
We prove by induction on $k$. For $k=0$, the formulas coincide with formulas~\eqref{eq:bshsn}-\eqref{eq:bsgss}. Let us apply the sequence $B_{s-k} \rightarrow \cdots \rightarrow B_{s-1}$ to the initial seed and assume that the formulas hold. We will show that the same formulas hold after a further mutation along the sequence $B_s$. The mutation at $h_{sn}$ reads
\[
h_{sn} h_{sn}^\prime = h^\prime_{s-1,n}h_{s+1,n} + h_{s-k-1,n} h_{s,n-1}.
\]
This is a Desnanot-Jacobi identity for the matrix
\[
\begin{matrix}
&\downarrow & \\
\rightarrow & y_{s-k-1,n-1} & y_{s-k-1,n}\\
\rightarrow & y_{s,n-1} & y_{sn}\\
\rightarrow & y_{s+1,n-1} & y_{s+1,n}
\end{matrix}
\]
hence $h_{sn}^\prime = \det Y^{[n-1,n]}_{\{s-k-1,s+1\}}$. The subsequent mutations are
\[
h_{s,n-i} h_{s,n-i}^\prime = h_{s+1,n-i} h^\prime_{s-1,n-i} + h_{s,n-i-1} h_{s,n-i+1}^\prime
\]
These are Desnanot-Jacobi identities applied to the matrix
\[
\begin{matrix}
 & \downarrow &  &  & \\
\rightarrow & y_{s-k-1,n-i-1} & y_{s-k-1,n-i} &\ldots&y_{s-k-1,n}\\
\rightarrow & y_{s,n-i-1} & y_{s,n-i} & \ldots & y_{sn}\\
& y_{s+1,n-i-1} & y_{s,n-i} & \ldots & y_{s+1,n}\\
& \vdots & \vdots & \ldots & \vdots \\
\rightarrow & y_{s+i+1,n-i-1} & y_{s+i+1,n-i} & \ldots & y_{s+i+1,n}
\end{matrix}
\]
Therefore, $h^\prime_{s,n-i} = \det Y^{[n-i-1,n]}_{\{s-k-1\}\cup [s+1,s+i+1]}$. Next, the mutation at $h_{ss}$ for $s < n$ reads
\[
h_{ss} h^\prime_{ss} = h_{s-1,s-1}^\prime h_{s+1,s+1} + h^\prime_{s,s+1} f_{1,n-s}.
\]
This is a Desnanot-Jacobi identity applied to the matrix
\[
\begin{matrix}
&\downarrow & \downarrow & & & \\
\rightarrow & x_{s-k-1,n} & y_{s-k-1,s} & y_{s-k-1,s+1} & \ldots & y_{s-k-1,n}\\
\rightarrow & x_{sn} & y_{ss} & y_{s,s+1} & \ldots & y_{sn}\\
& x_{s+1,n} & y_{s+1,s} & y_{s+1,s+1} & \ldots & y_{s+1,n}\\
& \vdots & \vdots & \vdots & \ldots & \vdots \\
& x_{nn} & y_{ns} & y_{n,s+1} & \ldots & y_{nn}
\end{matrix}
\]
hence $h^\prime_{ss} = \det [X^{[n,n]} \ Y^{[s+1,n]}]_{\{s-k-1\}\cup [s+1,n]}$. If $s = n$, then the mutation is
\[
h_{nn} h_{nn}^\prime = h^\prime_{n-1,n-1} + g_{nn} h_{n-k-1,n},
\]
which expands as
\[
y_{nn} h_{nn}^\prime = \det \begin{bmatrix} x_{n-k-1,n} & y_{n-k-1,n}\\ x_{nn} & y_{nn} \end{bmatrix} + x_{nn} y_{n-k-1,n} = x_{n-k-1,n}y_{nn},
\]
hence $h_{nn}^\prime = x_{n-k-1,n}$. The subsequent mutations along $f_{1,n-s} \rightarrow \cdots \rightarrow f_{n-s,1}$ read
\[
f_{i,n-s-i+1} f^\prime_{i,n-s-i+1} = f^\prime_{i,n-s-i+2} f_{i,n-s-i} + f_{i+1,n-s-i} f^\prime_{i-1,n-s-i+2}, \ i \in [1,n-s].
\]
These are Desnanot-Jacobi identities applied to the matrices of the form
\[
\begin{matrix}
& \downarrow & & & & & \downarrow & & & \\
\rightarrow & x_{s-k-1,n-i} & x_{s-k-1,n-i+1} & \cdots&x_{s-k-1,n} & y_{s-k-1,i+s} & y_{s-k-1,i+s+1} & \cdots & y_{s-k-1,n}\\
\rightarrow & x_{s,n-i} & x_{s,n-i+1} & \cdots&x_{sn} & y_{s,i+s} & y_{s,i+s+1} & \cdots & y_{sn}\\
& \vdots & \vdots& \cdots & \vdots & \vdots & \vdots &\cdots & \vdots \\
& x_{n,n-i} & x_{n,n-i+1} & \cdots & x_{nn}& y_{n,i+s} & y_{n,i+s+1} & \cdots & y_{nn}\\
\end{matrix}
\]
hence $f_{i,n-s-i+1}^\prime = \det [X^{[n-i,n]} \ Y^{[s+i+1,n]}]_{\{s-k-1\} \cup [s+1,n]}$. Lastly, consider the consecutive mutations along the path $g_{ss} \rightarrow \cdots \rightarrow g_{s2}$. The mutation at $g_{s,s-i+1}$ yields
\[
g_{s,s-i+1}g^\prime_{s,s-i+1} = g_{s-k-1,s-i+1} g_{s+1,s-i+1}^\prime + g_{s,s-i} g_{s,s-i+2}^\prime, \ \ i \in [1,s-1].
\]
This is a Desnanot-Jacobi identity for the matrix
\[
\begin{matrix}
& \downarrow & & & & \downarrow\\
\rightarrow & x_{s-k-1,s-i} & x_{s-k-1,s-i+1} & \cdots & x_{s-k-1,n-i} & x_{s-k-1,n-i+1}\\
\rightarrow & x_{s,s-i} & x_{s,s-i+1} & \cdots & x_{s,n-i} & x_{s,n-i+1}\\
& x_{s+1,s-i} & x_{s+1,s-i+1} & \cdots & x_{s+1,n-i} & x_{s+1,n-i+1}\\
& \vdots & \vdots & \cdots & \vdots & \vdots\\
& x_{n,s-i} & x_{n,s-i+1} & \cdots & x_{n,n-i} & x_{n,n-i+1}
\end{matrix}
\]
Thus the lemma is proved.
\end{proof}

\begin{remark}\label{r:xvar} Applying the sequence $B_{n-k} \rightarrow \cdots \rightarrow B_{n}$ to the initial seed, we obtain 
\[
h^\prime_{nn} = x_{s-k-1,n}, \ \ g^\prime_{n,n-i+1} = x_{s-k-1,n-i}, \ \ i \in [1,n-1].
\]
Therefore, this mutation sequence provides an alternative way of showing that $x_{ij}$'s are cluster variables (another sequence is shown in \cite{double}, but it doesn't translate well to a nontrivial BD pair). Figure~\ref{f:b2b3} illustrates the quiver in $n=5$ obtained after applying $B_2\rightarrow B_3$.
\end{remark}

\subsubsection{Sequence $B_{s-k} \rightarrow \ldots \rightarrow B_s$ in the case $|\Gamma_1^r|+|\Gamma_1^c| = 1$}
\begin{lemma}\label{l:seq_cc}
Let $\bg:=(\bg^r,\bg^c)$ be a BD pair such that $\Gamma_1^r = \{p\}$, $\Gamma_2^r = \{q\}$, and $\Gamma_1^c = \emptyset$, and let $\gc(\bg)$ be the corresponding generalized cluster structure on $D(\GL_n)$. Let $s$ and $k$ be nonnegative numbers that satisfy $2 \leq s-k \leq n$, $2 \leq s \leq n$, $s-k \neq q+1$. Apply the mutation sequence $B_{s-k} \rightarrow \cdots \rightarrow B_{s}$ to the initial seed of $\gc(\bg)$. Then the resulting seed contains the following cluster variables:
\begin{equation}\label{eq:fhsC}
h^\prime_{s,n-i+1} = \det Y^{[n-i,n]}_{\{s-k-1\} \cup [s+1,s+i]}, \ \ i \in [1,n-s]\setminus\{q-s\};
\end{equation}
\begin{equation}\label{eq:fminC}
f_{i,n-s-i+1}^\prime = \det [ X^{[n-i,n]} \ Y^{[s+i+1,n]}]_{\{s-k-1\}\cup [s+1,n]} , \  \ i \in [0,n-s];
\end{equation}
\begin{equation}\label{eq:gminC}
g_{s,s-i+1}^\prime = \det X^{[s-i,n-i]}_{\{s-k-1\}\cup [s+1,n]}, \ \ i \in [1,s-1].
\end{equation}
\end{lemma}
\begin{proof}
Let $\tilde{\bg}$ be the standard BD pair. Let $\mathcal{U} : D(\GL_n)_{\tilde{\bg}} \dashrightarrow D(\GL_n)_{\bg}$ be the birational quasi-isomorphism from Section~\ref{s:birat}. In this case, it is given by
\[
\mathcal{U}(X,Y):= (U_0X,U_0Y), \ \ U_0(X,Y) := I + \alpha(X,Y) e_{q,q+1}, \ \ \alpha(X,Y):=\frac{\det X_{\{p\}\cup[p+2,n]}^{[1,n-p]}}{\det X_{[p+1,n]}^{[1,n-p]}}.
\]
Now, notice that if $I \subseteq [1,n]$ and $J \subseteq [1,2n]$ are two sets of indices of the same size, and if either $\{q,q+1\}\subseteq I$ or $I \cap \{q\}  = \emptyset$, then
\[
\mathcal{U}^*(\det\begin{bmatrix}X & Y\end{bmatrix}_{I}^J) = \det\begin{bmatrix}X & Y\end{bmatrix}_{I}^{J}.
\]
Therefore, if $p(X,Y)$ is any polynomial from equations~\eqref{eq:fhsC}-\eqref{eq:gminC}, $\mathcal{U}^*(p(X,Y)) = p(X,Y)$; but since $\mathcal{U}$ is invertible and $p(X,Y)$ is a cluster variable in $\gc(\tilde{\bg})$ (see Lemma~\ref{l:seqs_c}), it follows from Proposition~\ref{p:comparu} that $p(X,Y)$ is a cluster variable in $\gc(\bg)$ as well.
\end{proof}
\begin{lemma}\label{l:bscc}
Let $\bg :=(\bg^r,\bg^c)$ be a BD pair such that $\Gamma_1^r = \emptyset$, $\Gamma_1^c = \{p\}$, $\Gamma_2^c = \{q\}$. Let $s$ and $k$ be nonnegative numbers that satisfy $2 \leq s-k \leq n$, $2 \leq s \leq n$. Apply the mutation sequence $B_{s-k} \rightarrow \cdots \rightarrow B_{s}$ to the initial seed of $\gc(\bg)$. Then the resulting seed contains the following cluster variables:
\begin{equation}\label{eq:fhsCC}
h^\prime_{s,n-i+1} = \det Y^{[n-i,n]}_{\{s-k-1\} \cup [s+1,s+i]}, \ \ i \in [1,n-s];
\end{equation}
\begin{equation}\label{eq:fminCC}
f_{i,n-s-i+1}^\prime = \det [ X^{[n-i,n]} \ Y^{[s+i+1,n]}]_{\{s-k-1\}\cup [s+1,n]} , \  \ i \in [0,n-s];
\end{equation}
\begin{equation}\label{eq:gminCC}
g_{s,s-i+1}^\prime = \det X^{[s-i,n-i]}_{\{s-k-1\}\cup [s+1,n]}, \ \ i \in [1,s-1] \setminus \{n-p\}
\end{equation}
\end{lemma}
\begin{proof}
The proof proceeds along the same lines as the proof of Lemma~\ref{l:seq_cc}. In this case, the birational quasi-isomorphism is given by
\[
\mathcal{U}(X,Y):=(XU_0,YU_0), \ \ U_0(X,Y) := I + \alpha(X,Y) e_{p+1,p}, \ \ \alpha(X,Y) := \frac{\det Y^{\{q\}\cup[q+2,n]}_{[1,n-q]}}{\det Y^{[q+1,n]}_{[1,n-q]}} \qedhere
\]
\end{proof}
\subsubsection{Sequence $W$ in the standard $\gc$}
 \begin{figure}
 \includegraphics[scale=0.4]{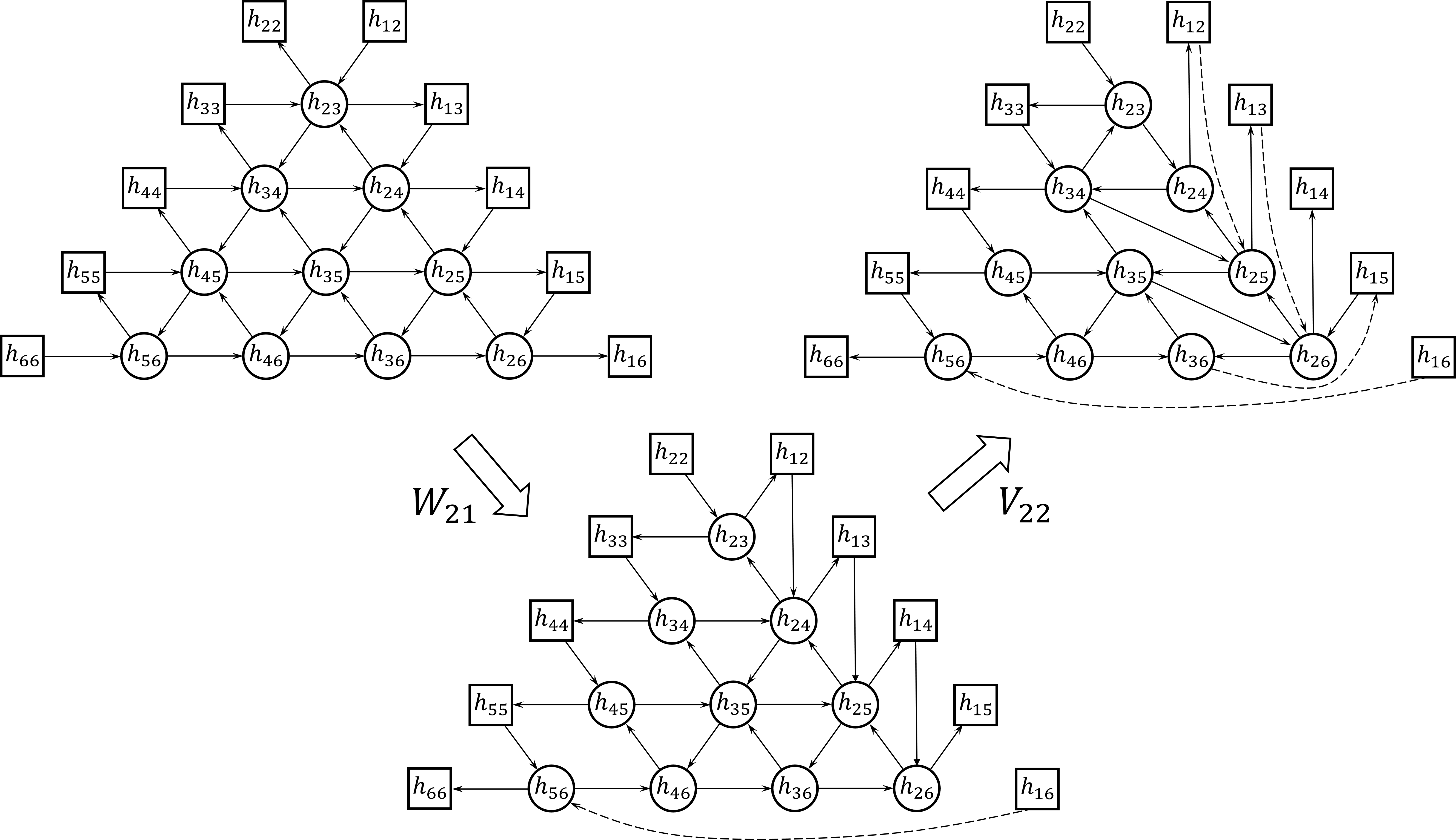}
 \caption{An illustration of the sequence $W_{21}$ and $W_{21}\rightarrow V_{22}$ in $n=6$. Vertices $h_{ii}$ are frozen for convenience, and the vertices that do not participate in mutations are removed.}
 \label{f:wseq}
 \end{figure}
Let $\bg$ be the trivial BD pair and $\gc(\bg)$ be the corresponding generalized cluster structure on $D(\GL_n)$. For $2 \leq s \leq n-1$ and $1 \leq t \leq n-s$, define a sequence of mutations $V_{s,t}$ by 
\[
h_{sn} \rightarrow h_{s,n-1}\rightarrow \cdots \rightarrow h_{s,s+t},
\]
and define a sequence $W_{s,t}$ as
\[
V_{s,t} \rightarrow V_{s+1,t} \rightarrow \cdots \rightarrow V_{n-t,t}.
\]
An illustration of the sequences is shown in Figure~\ref{f:wseq}.
\begin{lemma}\label{l:kilcol}
Apply the mutation sequence $W_{s,1}\rightarrow W_{s,2} \rightarrow \cdots \rightarrow W_{s,t-1} \rightarrow V_{s,t}$ to the initial seed of $\gc(\bg)$. Then the resulting seed contains the following cluster variables:
\[
h_{s,n-i}^{(t)} = \det Y^{[n-(t+i), n]}_{[s-1,s+t-2]\cup [s+t,s+t+i]}, \ \ i \in [0,n-s-t],
\]
where the upper index indicates the number of times the corresponding vertex of the quiver was mutated along the sequence.
\end{lemma}
\begin{proof}
Notice that $V_{s,1}$ is a part of $B_{s}$ sequence. Moreover, the cluster variables obtained along the sequence $W_{s,1}$ can be as well collected from the sequence $B_{s} \rightarrow \cdots \rightarrow B_{n-1}$. More generally, if we've already mutated along the sequence 
\[\begin{split}
W_{s,1} \rightarrow W_{s,2} \rightarrow \cdots \rightarrow W_{s,t-2} \rightarrow V_{s,t-1} \rightarrow V_{s+1,t-1} &\rightarrow \cdots \rightarrow V_{s+k-1,t-1}\rightarrow  \\ &\rightarrow h_{s+k,n}^{(t-2)} \rightarrow \cdots \rightarrow h_{s+k,n-i+1}^{(t-2)},
\end{split}
\]
the mutation at $h_{s+k,n-i}^{(t-2)}$ yields a cluster variable
\begin{equation}\label{eq:hskni}
h_{s+k,n-i}^{(t-1)} = \det Y^{[n-(t-1+i),n]}_{[s-1,s+(t-1)-2]\cup [s+(t-1)+k,s+(t-1)+i+k]}.
\end{equation}
Proceeding with the proof, the mutation of $h_{sn}^{(t-1)}$ can be written as
\[
h_{sn}^{(t-1)} h_{sn}^{(t)} = h_{s,n-1}^{(t-1)} h_{s-1,n-t+1} + h_{s+1,n}^{(t-1)}h_{s-1,n-t},
\]
which is a Desnanot-Jacobi identity applied to the matrix
\[
\begin{matrix}
  & \downarrow &   &   &  \\
  & y_{s-1,n-t} & y_{s-1,n-t+1} & \cdots & y_{s-1,n}\\
  & y_{s,n-t} & y_{s,n-t+1} & \cdots & y_{sn} \\
  & \vdots & \vdots & \cdots & \vdots \\
  & y_{s+t-3,n-t} & y_{s+t-3,n-t+1} & \cdots & y_{s+t-3,n}\\
\rightarrow & y_{s+t-2,n-t} & y_{s+t-2,n-t+1} & \cdots & y_{s+t-2,n}\\
\rightarrow & y_{s+t-1,n-t} & y_{s+t-1,n-t+1} & \cdots & y_{s+t-1,n}\\
\rightarrow & y_{s+t,n-t} & y_{s+t-1,n-t+1} & \cdots & y_{s+t-1,n}
\end{matrix}
\]
Proceeding along $h^{(t-1)}_{s,n} \rightarrow \cdots \rightarrow h^{(t-1)}_{s,n-i+1}$, the subsequent mutation at $h^{(t-1)}_{s,n-i}$ reads
\[
h^{(t-1)}_{s,n-i} h^{(t)}_{s,n-i} = h_{s,n-i+1}^{(t)} h^{(t-1)}_{s,n-i-1} + h_{s+1,n-i}^{(t-1)} h_{s-1,n-t-i}.
\]
This is again a Desnanot-Jacobi identity applied to the matrix
\[
\begin{matrix}
  & \downarrow &   &   &  \\
  & y_{s-1,n-(t+i)} & y_{s-1,n-(t+i)+1} & \cdots & y_{s-1,n}\\
  & \vdots & \vdots & \cdots & \vdots \\
\rightarrow & y_{s+t-2,n-(t+i)} & y_{s+t-2,n-(t+i)+1} & \cdots & y_{s+t-2,n}\\
\rightarrow & y_{s+t-1,n-(t+i)} & y_{s+t-1,n-(t+i)+1} & \cdots & y_{s+t-1,n}\\
  & \vdots & \vdots & \cdots & \vdots \\
\rightarrow & y_{s+t+i,n-(t+i)} & y_{s+t+i,n-(t+i)+1} & \cdots & y_{s+t+i, n}
\end{matrix}
\]
As for the variable $h_{s+k,n}^{(t-1)}$ in~\eqref{eq:hskni} for $k > 0$, the mutation relation is
\[
h_{s+k,n}^{(t)} h_{s+k,n}^{(t-1)} = h_{s+k,n-1}^{(t-1)} h_{s-1,n-t} + h_{s+k+1,n}^{(t-1)} h_{s+k-1,n}^{(t)}.
\]
This is a Desnanot-Jacobi identity applied to the matrix
\[
\begin{matrix}
  & \downarrow &   &   &  \\
  & y_{s-1,n-t} & y_{s-1,n-t+1} & \ldots & y_{s-1,n}\\
  & y_{s,n-t} & y_{s,n-t+1} & \ldots & y_{sn}\\
  & \vdots & \vdots & \ldots & \vdots\\
  & y_{s+(t-1)-2,n-t} & y_{s+(t-1)-2,n-t+1} & \ldots & y_{s+(t-1)-2,n}\\
\rightarrow & y_{s+t-2,n-t} & y_{s+t-2,n-t+1} & \ldots & y_{s+t-2,n}\\
\rightarrow & y_{s+(t-1)+k,n-t} & y_{s+(t-1)+k,n-t+1} & \ldots & y_{s+(t-1)+k,n}\\
\rightarrow & y_{s+t+k,n-t} & y_{s+t+k,n-t+1} & \ldots & y_{s+t+k,n}
\end{matrix}
\]
Lastly, for $i > 0$ and $k > 0$, the mutation at $h_{s+k,n-i}^{(t-1)}$ is
\[
h_{s+k,n-i}^{(t)} h_{s+k,n-i}^{(t-1)} = h^{(t)}_{s+k-1,n-i} h^{(t-1)}_{s+k+1,n-i} + h^{(t)}_{s+k,n-i+1} h^{(t-1)}_{s+k,n-i-1}.
\]
This is a Desnanot-Jacobi identity for
\[
\begin{matrix}
  & \downarrow &   &   &  \\
  & y_{s-1,n-(t+i)} & y_{s-1,n-(t+i)+1} & \cdots & y_{s-1,n}\\
  & \vdots & \vdots & \cdots & \vdots \\
\rightarrow & y_{s+t-2,n-(t+i)} & y_{s+t-2,n-(t+i)+1} & \cdots & y_{s+t-2,n}\\
\rightarrow & y_{s+(t-1)+k,n-(t+i)} & y_{s+(t-1)+k,n-(t+i)+1} & \cdots & y_{s+(t-1)+k,n}\\
& y_{s+t+k,n-(t+i)} & y_{s+t+k,n-(t+i)+1} & \cdots & y_{s+t+k,n}\\
  & \vdots & \vdots & \cdots & \vdots \\
	& y_{s+(t-1)+i+k,n-(t+i)} & y_{s+(t-1)+i+k,n-(t+i)+1} & \cdots & y_{s+(t-1)+i+k, n}\\
\rightarrow & y_{s+t+i+k,n-(t+i)} & y_{s+t+i+k,n-(t+i)+1} & \cdots & y_{s+t+i+k, n}\\
\end{matrix}
\]
Thus the lemma is proved.
\end{proof}
\subsubsection{Sequence $\mathcal{S}$ in the standard $\gc$}\label{s:seqs}
Let us briefly recall a special sequence of mutations from \cite{double} denoted as $\mathcal{S}$. The sequence was used in order to show that the entries of the matrix $U = X^{-1}Y$ belong to the upper cluster algebra, as well as to produce a generalized cluster structure on the variety $\GL_n^\dagger$ (see Section~\ref{s:plgrps} for the definition).
\begin{figure}
\begin{center}
\includegraphics[scale=0.4]{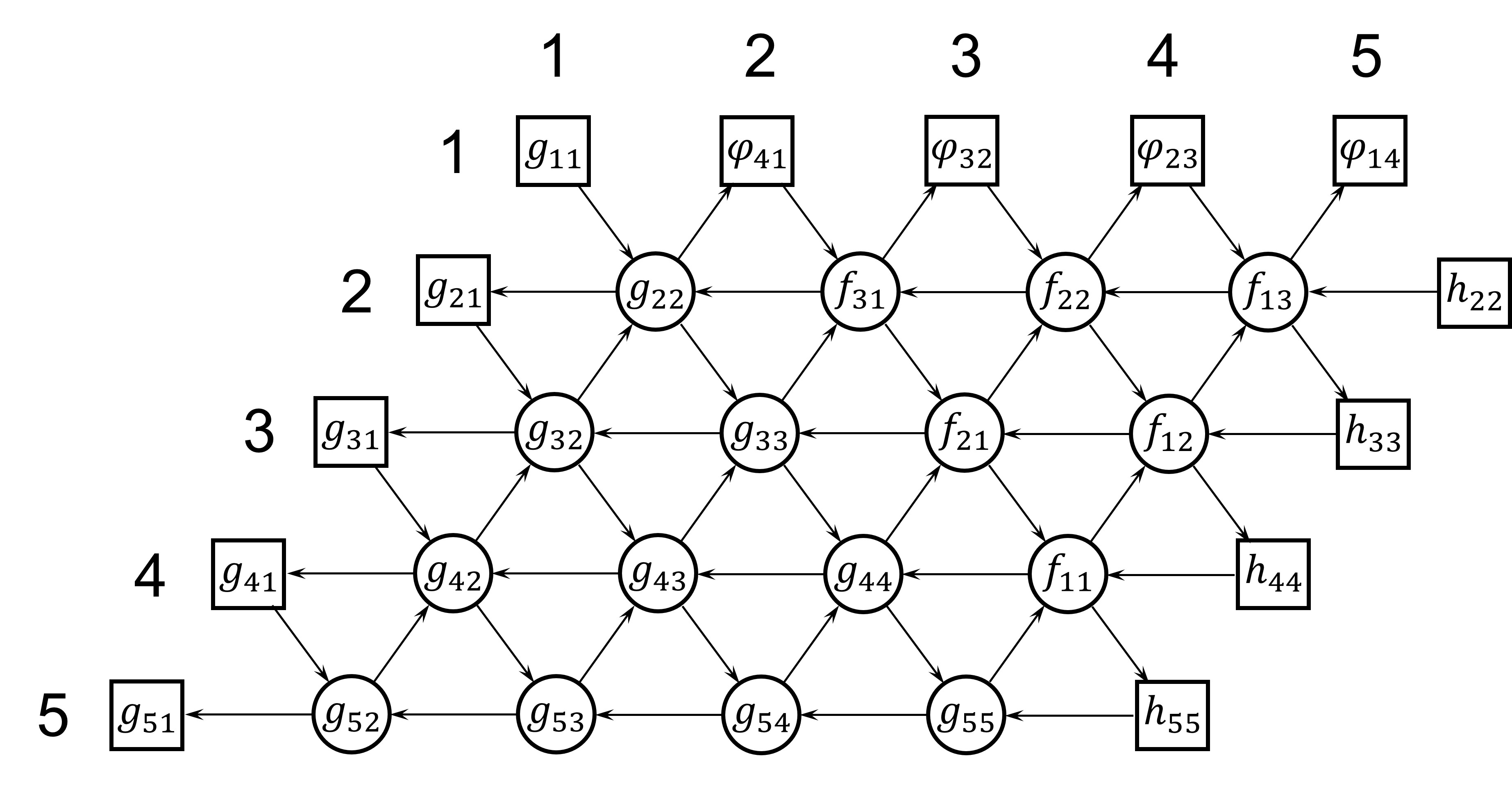}
\end{center}
\caption{Quiver $Q_0$ for $n=5$.}
\label{f:qzero}
\end{figure}
\paragraph{Quiver $Q_0$.} Let $Q$ be the initial quiver of the standard generalized cluster algebra $\gc$. Let us define a quiver $Q_0$ that consists of the vertices that contain all $g$- and $f$-functions, as well as all $h_{ii}$ for $2 \leq i \leq n$ and all $\varphi_{i,n-i}$ for $1 \leq i \leq n-1$; for convenience, let us freeze the vertices $\varphi_{i,n-i}$ and $h_{ii}$. Furthermore, we assign double indices $(i,j)$ to the vertices of $Q_0$, with $i$ enumerating the rows and running from top to bottom, and $j$ being responsible for the columns, running from left to right. Figure~\ref{f:qzero} represents $Q_0$ for $n=5$. The quiver $Q_0$ together with the functions attached to the vertices defines an ordinary cluster algebra of geometric type.

\paragraph{Sequence $\mathcal{S}_1$.} Let us mutate the quiver $Q_0$ along the diagonals starting from the bottom left and proceeding to the top right corner. More precisely: first, mutate at $(n,2)$; second, mutate along $(n-1,2) \rightarrow (n,3)$; third, mutate along $(n-2,2) \rightarrow (n-1,3) \rightarrow (n,4)$ and so on. The last mutation in the sequence is at the vertex $(2,n)$. Let us denote the resulting quiver as $Q_1$ and the resulting cluster variables as $\chi_{ij}^1$, $2 \leq i,\ j \leq n$. They are given by
\begin{equation}\label{eq:chij1}
\chi_{ij}^1 = \begin{cases}
\det X^{[1]\cup [j+1,n+j-i+1]}_{[i-1,n]} \ &\text{if} \ i > j\\
\det [X^{[1]\cup [j+1,n]} \ Y^{[n+i-j,n]}]_{[i-1,n]} \ & \text{if} \ i \leq j.
\end{cases}
\end{equation}

\paragraph{Sequence $\mathcal{S}_k$.} Once we've mutated along the sequence $\mathcal{S}_{k-1}$, the sequence $\mathcal{S}_k$ is defined as follows. First, freeze all the vertices in the $k$th row and in the $(n-k+2)$th column of the quiver $Q_{k-1}$. Then $\mathcal{S}_k$ is defined as a sequence of mutations along the diagonals: First, mutate at $(n,2)$; then mutate along $(n-1,2)\rightarrow (n,3)$, and so on. The resulting cluster variables are denoted as $\chi_{ij}^k$ and are given by
\[
\chi_{ij}^k = \begin{cases} \det X^{[1,k]\cup [j+k,n+j-i+k]}_{[i-k,n]} \ &\text{if} \ i-k+1>j\\
\det [X^{[1,k]\cup [j+k,n]} \ Y^{[n+i-j+1-k,n]}]_{[i-k,n]} \ &\text{if} \ i-k+1 \leq j.
\end{cases}
\]

\paragraph{Sequence $\mathcal{S}$.} The sequence $\mathcal{S}$ is defined as the composition $\mathcal{S}_{n-1} \circ \mathcal{S}_{n-2} \circ \cdots \circ \mathcal{S}_1$. The result of its application to the initial quiver is illustrated in Figure~\ref{f:sseqn4} for $n=4$. Notice that
\[
\chi_{k+1,j}^k = \det X \cdot (-1)^{(n-j-k+1)(n-k-1)} h_{k+1,n-j+2}(U), \ \ \ 2 \leq j \leq n-k+1.
\]
It was shown in \cite{double} that the entries of $U$ in the standard $\gc$ can be written as Laurent polynomials in terms of the following variables: $c$-functions, $\varphi$-functions and the functions $\chi_{k+1,j}^k$ obtained from the sequence $\mathcal{S}$.
\begin{figure}
\begin{center}
\includegraphics[scale=0.5]{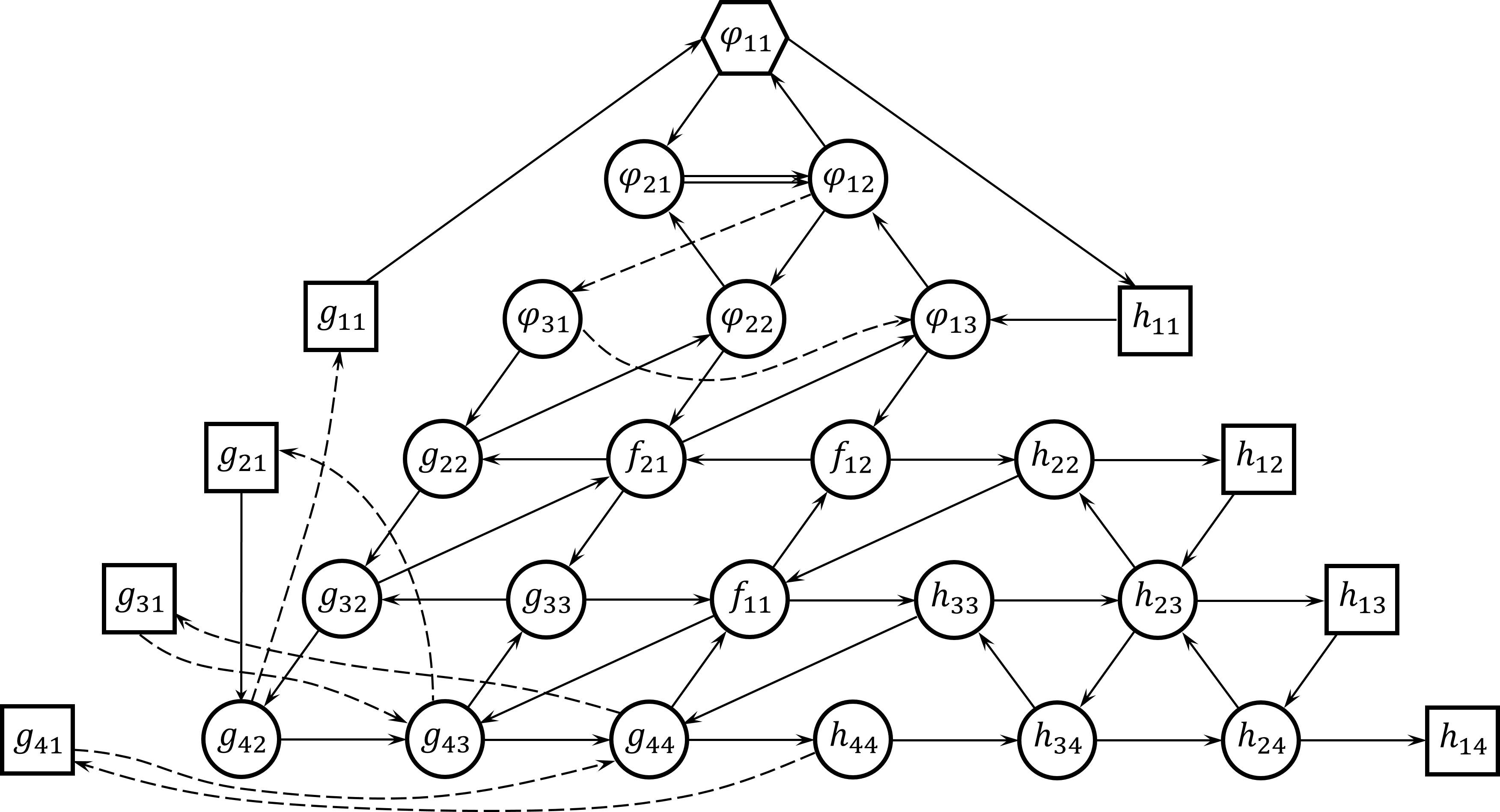}
\end{center}
\caption{An application of the sequence $\mathcal{S}$ to the initial quiver of the standard $\gc$, $n=4$.}
\label{f:sseqn4}
\end{figure}

\subsubsection{Sequence $\mathcal{S}$ in the case $|\Gamma_1^r|+|\Gamma_1^c| = 1$}
\begin{lemma}\label{l:sseqr}
Let $\bg:=(\bg^r,\bg^c)$ be a BD pair such that $\Gamma_1^r = \{p\}$, $\Gamma_2^r = \{q\}$ and $\Gamma_1^c = \emptyset$, and let $\gc(\bg)$ be the corresponding generalized cluster structure on $D(\GL_n)$. Apply the sequence $\mathcal{S}$ to the initial seed of $\gc(\bg)$. Then for $1 \leq k \leq n-1$, $2 \leq j \leq n-k+1$, the resulting seed contains the cluster variables
\begin{equation}\label{eq:chir}
\chi_{k+1,j}^k = \det X \cdot (-1)^{(n-j-k+1)(n-k-1)} h_{k+1,n-j+2}(U)
\end{equation}
\end{lemma}
\begin{proof}
The proof is similar to Lemma~\ref{l:seq_cc} and Lemma~\ref{l:bscc}.
\end{proof}

\begin{remark}
Though not needed in this paper, a similar lemma can be proved for the case $\Gamma_1^c = \emptyset$ and $\Gamma_1^r = \{p\}$, $\Gamma_1^r = \{q\}$. Then the resulting seed also contains the cluster variables~\eqref{eq:chir} except for $k = p$. 
\end{remark}
The case of a nontrivial $\bg^c$ will require a different result:
\begin{lemma}\label{l:cclust}
Let $\bg:=(\bg^r,\bg^c)$ be a BD pair such that $\Gamma_1^r = \emptyset$, $\Gamma_1^c = \{p\}$ and $\Gamma_2^c = \{q\}$, and let $\gc(\bg)$ be the corresponding generalized cluster structure on $D(\GL_n)$. There exist extended clusters $\Psi:=(\psi_1,\ldots,\psi_{2n})$ and $\tilde{\Psi}:=(\tilde{\psi}_1,\ldots,\tilde{\psi}_{2n})$ in $\gc(\bg)$ and $\gc(\tilde{\bg})$, respectively, such that $\psi_i(X,Y) = \tilde{\psi}_i(X,Y)$ if and only if $\psi_i \neq g_{n-p+1,1}$.
\end{lemma}
\begin{proof}
Indeed, if $p = 1$, then the initial extended clusters of $\gc(\bg)$ and $\gc(\tilde{\bg})$ satisfy the requirement; if $p > 1$, then let  $\Psi_{\mathcal{S}_1}$ and $\tilde{\Psi}_{\mathcal{S}_1}$ be the extended clusters in $\gc(\bg)$ and $\gc(\tilde{\bg})$, respectively, that are obtained from the initial extended clusters via an application of $\mathcal{S}_1$. Let $\mathcal{U}:D(\GL_n)_{\tilde{\bg}} \dashrightarrow D(\GL_n)_{\bg}$ be the birational quasi-isomorphism described in Section~\ref{s:birat}. It is given by
\[
\mathcal{U}(X,Y):=(XU_0,YU_0), \ \ U_0(X,Y) := I + \alpha(X,Y) e_{p+1,p}, \ \ \alpha(X,Y) := \frac{\det Y^{\{q\}\cup[q+2,n]}_{[1,n-q]}}{\det Y^{[q+1,n]}_{[1,n-q]}}.
\]
It follows that $\mathcal{U}(\chi_{ij}^1(X,Y)) = \chi_{ij}^1(X,Y)$, where $\chi_{ij}^1$ are defined in equation~\eqref{eq:chij1}; therefore, Proposition~\ref{p:comparu} implies that $\chi_{ij}^1$ are cluster variables of $\Psi_{\mathcal{S}_1}$. Since all the other variables (except $g_{n-p+1,1}(X,Y)$ and $\tilde{g}_{n-p+1,1}(X,Y)$) are equal as elements of $\mathcal{O}(D(\GL_n))$, we conclude that $\Psi_{\mathcal{S}_1}$ and $\tilde{\Psi}_{\mathcal{S}_1}$ are the required extended clusters.
\end{proof}

\subsection{Completeness for $|\Gamma_1^r| = 1$ and $|\Gamma_1^c| = 0$}
Let $\gc(\bg)$ be a generalized cluster structure on $D(\GL_n)$ defined by a BD pair with $\Gamma_1^r = \{p\}$, $\Gamma_2^r = \{q\}$ and $\Gamma_1^c = \emptyset$, and let $\gc(\tilde{\bg})$ be the standard generalized cluster structure.
\begin{lemma}\label{l:entru}
The entries of $U = X^{-1}Y$ in $\gc(\bg)$ belong to the upper cluster algebra.
\end{lemma}
\begin{proof}
It was shown in~\cite{double} that the entries of $U$ can be expressed as Laurent polynomials in terms of the $\varphi$-variables, $c$-variables, and the variables $\chi^k_{k+1,j}$ (see Section~\ref{s:seqs}), as well as in terms of any mutations of these variables. By Lemma~\ref{l:sseqr}, all of these variables are present in $\gc(\bg)$; thus, the entries of $U$ belong to $\bar{\mathcal{A}}(\gc(\bg))$.\qedhere
\end{proof}

\begin{proposition}\label{p:recr}
Under the setup of the current section, the entries of $X$ and $Y$ belong to the upper cluster algebra.
\end{proposition}
\begin{proof}
Due to Corollary~\ref{c:comparu2}, it suffices to show that the entries of $X$ and $Y$ can be expressed as Laurent polynomials in the cluster $\Psi$ adjacent to the initial one in the direction of $g_{p+1,1}$. It follows from Lemma~\ref{l:seq_cc} and Remark~\ref{r:xvar} that all the entries of $X$ except the $q$th row belong to the upper cluster algebra, for they are themselves cluster variables. Due to Lemma~\ref{l:entru} and the relation $XU = Y$, all the entries of $Y$ except the $q$th row also belong to the upper cluster algebra. Therefore, we only need to find Laurent expressions for the $q$th rows of $X$ and $Y$.

The mutation at $g_{p+1,1}$ yields
\begin{equation}\label{eq:psisqr}
g^\prime_{p+1,1}(X,Y) = \det \begin{bmatrix} 
y_{qn} & x_{p,2} & x_{p,3} & \ldots & x_{p,n-p+1}\\
y_{q+1,n} & x_{p+1,2} & x_{p+1,3} & \ldots & x_{p+1,n-p+1}\\
0 & x_{p+2,2} & x_{p+2,3} & \ldots & x_{p+2,n-p+1}\\
\vdots & \vdots & \vdots & \ldots & \vdots \\
0 & x_{n,2} & x_{n,3} & \ldots & x_{n,n-p+1} 
\end{bmatrix},
\end{equation}
which can be seen via an appropriate application of a Pl\"ucker relation. Expanding $g^\prime_{p+1,1}(X,Y)$ along the first column yields
\[
g_{p+1,1}^\prime(X,Y) = y_{qn} g_{p+1,2}(X,Y)- y_{q+1, n} \det X_{\{p\} \cup [p+2,n]}^{[2,n-p+1]}.
\]
Since $p \neq q$, it follows from Lemma~\ref{l:seq_cc} that $\det X_{\{p\} \cup [p+2,n]}^{[2,n-p+1]}$ is a Laurent polynomial in terms of the variables of $\Psi$. Together with the above relation, we see that $y_{qn}$ is a Laurent polynomial in terms of the variables of $\Psi$ as well.

Let us assume by induction that for $i > q$ the variables $y_{qj}$ are already recovered, where $j \geq i$. Expanding the function $h_{q,i-1}(Y) = \det Y^{[i-1,n]}_{[q,q+n-i+1]}$ along the first row, we see that
\[
h_{q,i-1}(Y) = y_{q,i-1} h_{q+1,i}(Y) + P_1(Y)
\]
where $P_1(Y)$ is a polynomial in all entries of $Y^{[i-1,n]}_{[q,q+n-i+1]}$ except $y_{q,i-1}$, and hence $P_1(Y)$ is a Laurent polynomial in the variables of $\Psi$. Therefore, we've recovered all $y_{qi}$ for $i \geq q$. To proceed further, we make use of $f$-functions. The variable $x_{qn}$ can be recovered via expanding $f_{1,n-q}(X,Y)$ along the first row:
\[
f_{1,n-q}(X,Y) = x_{qn} h_{q+1,q+1}(Y) + P_2(X,Y)
\]
where $P_2$ is now a polynomial in all entries of $[X^{[n,n]}\, Y^{[q+1,n]}]_{[q,n]}$ except $x_{qn}$, and therefore $P_2(X,Y)$ is a Laurent polynomial in $\Psi$. If for some $i > q+1$ the variables $x_{qj}$ are already recovered, where $j \geq i$, then $x_{q,i-1}$ can be recovered via expanding $f_{n-i+1,i-q}(X,Y)$ along the first row:
\[
f_{n-i+1,i-q}(X,Y) = x_{q,i-1} f_{n-i+1,i-q-1}(X,Y) + P_3(X,Y)
\]
where $P_3(X,Y)$ is again a polynomial in entries that are already known to be Laurent polynomials in terms of $\Psi$. We conclude at this moment that the variables $x_{q,q+1},\ldots,x_{qn}$ are Laurent polynomials in $\Psi$. Using the same idea, we recover the variables $x_{q1},\ldots,x_{qq}$ consecutively starting from $x_{qq}$ and using the $g$-functions: Each $x_{qi}$ is recovered via the expansion along the first row of the function $g_{qi}(X,Y)$. 

Lastly, since $x_{q1}, \ldots, x_{qn}$ are recovered as Laurent polynomials in terms of $\Psi$, the remaining variables $y_{q1},\ldots,y_{q,q-1}$ are recovered via $XU = Y$. Thus all the entries of $X$ and $Y$ are Laurent polynomials in the variables of $\Psi$.
\end{proof}

\subsection{Completeness for $|\Gamma_1^r| = 0$ and $|\Gamma_1^c| = 1$}
Similarly to the previous section, let $\gc(\bg)$ be a generalized cluster structure on $D(\GL_n)$ defined by a BD pair with $\Gamma_1^c = \{p\}$, $\Gamma_2^c = \{q\}$ and $\Gamma_1^r = \emptyset$, and let $\gc(\tilde{\bg})$ be the standard generalized cluster structure. We need the following abstract result:
\begin{lemma}\label{l:field}
Let $\mathcal{F}$ be a field of characteristic zero, and let $\alpha$ and $\beta$ be distinct transcendental elements over $\mathcal{F}$ and such that $\mathcal{F}(\alpha) = \mathcal{F}(\beta)$. If there is a relation
\begin{equation}\label{eq:abk}
\sum_{k=1}^{m} (\alpha^k - \beta^k) p_k = 0
\end{equation}
for $p_k \in \mathcal{F}$, then all $p_k = 0$.
\end{lemma}
\begin{proof}
Let us set $x:=\alpha$ for convenience. Since $\mathcal{F}(\alpha) = \mathcal{F}(\beta)$, we can express $\beta$ as $\beta = \frac{a x + b}{cx + d}$ with $ad-bc \neq 0$. Now, if $c = 0$, each $p_k$ in equation~\eqref{eq:abk} must be zero due to the linear independence of the polynomials $x^k - ((a/d)x+(b/d))^k$. Otherwise, if $c \neq 0$, we can look at the order of the pole $x = -d/c$ and show that $p_m = 0$, and then, via a descending induction starting at $m$, that all $p_k = 0$. Thus the statement holds.
\end{proof}

Contrary to Lemma~\ref{l:entru}, in the case of a nontrivial column BD triple we first treat the entries of $YX^{-1}$:
\begin{lemma}\label{l:entru_c}
The entries of $YX^{-1}$ belong to the upper cluster algebra of $\gc(\bg)$.
\end{lemma}
\begin{proof}
By Lemma~\ref{l:cclust}, there exist extended clusters $\Psi:=(\psi_1,\ldots,\psi_{2n})$ and $\tilde{\Psi}:=(\tilde{\psi}_1,\ldots,\tilde{\psi}_{2n})$ in $\gc(\bg)$ and $\gc(\tilde{\bg})$, respectively, that differ only in the variable $g_{n-p+1,1}$. Let $\mathcal{U}:D(\GL_n)_{\tilde{\bg}}\dashrightarrow D(\GL_n)_{\bg}$ be the birational quasi-isomorphism defined in Section~\ref{s:birat}.  
Let $\Psi^\prime$ be the extended cluster adjacent to $\Psi$ in the direction of $h_{1,q+1}$. By Corollary~\ref{c:comparu2}, it suffices to show that the entries of $YX^{-1}$ belong to the ring of Laurent polynomials $\mathcal{L}_{\mathbb{C}}(\Psi^\prime)$. Let us fix an entry $v:=(YX^{-1})_{ij}$; since $v \in \tilde{\mathcal{L}}_{\mathbb{C}}(\tilde{\Psi})$, we can write $v$ as
\begin{equation}\label{eq:vv1}
v = p_0 + \sum_{k\geq 1}(\tilde{g}_{n-p+1,1})^kp_k
\end{equation}
where $p_{i}$ are elements of $\tilde{\mathcal{L}}_{\mathbb{C}}(\tilde{\Psi})$ that do not contain $\tilde{g}_{n-p+1,1}$ (in other words, we view $\tilde{\mathcal{L}}_{\mathbb{C}}(\tilde{\Psi})$ as a polynomial ring in one variable $\tilde{g}_{n-p+1,1}$). Since $\mathcal{U}(YX^{-1}) = YX^{-1}$ and due to the choice of $\Psi$ and $\tilde{\Psi}$, we see that an application of $(\mathcal{U}^*)^{-1}$ yields
\begin{equation}\label{eq:vv2}
v = p_0 + \sum_{k \geq 1}\left( \frac{{g}_{n-p+1,1}}{{h}_{1,q+1}}\right)^k p_k,
\end{equation}
hence, subtracting equation~\eqref{eq:vv2} from equation~\eqref{eq:vv1}, we arrive at
\[
\sum_{k \geq 1}\left(\tilde{g}_{n-p+1,1}^k -  \left( \frac{g_{n-p+1,1}}{h_{1,q+1}}\right)^k \right) p_k = 0.
\]
By Lemma~\ref{l:field}, $p_k = 0$ for all $k \geq 1$. Therefore, there exists a Laurent expression for $v$ in terms of $\mathcal{L}_{\mathbb{C}}(\Psi)$ that does not involve a division by $h_{1,q+1}$ (for $\tilde{h}_{1,q+1}$ is not invertible in $\tilde{\mathcal{L}}_{\mathbb{C}}(\tilde{\Psi})$); therefore, if $h_{1,q+1} h_{1,q+1}^\prime = M$ is an exchange relation for $h_{1,q+1}$, substituting $h_{1,q+1}$ with $M/h_{1,q+1}^\prime$ in the Laurent expression for $v$ yields an expression in the ring $\mathcal{L}_{\mathbb{C}}(\Psi^\prime)$. Thus the lemma is proved.\qedhere
\end{proof}
\begin{proposition}\label{p:recc}
In the setup of the current section, all entries of $X$ and $Y$ belong to the upper cluster algebra of $\gc(\bg)$.
\end{proposition}
\begin{proof}
It follows from Lemma~\ref{l:bscc} that all entries of $X$ except the $p$th column are cluster variables in $\gc(\bg)$. Since the entries of $YX^{-1}$ belong to the upper cluster algebra due to Lemma~\ref{l:entru_c} and since $Y = (YX^{-1})X$, we see that all entries of $Y$ except the $p$th column also belong to the upper cluster algebra. The mutation at $h_{1,q+1}(X,Y)$ yields
\begin{equation}\label{eq:psisqc}
h^\prime_{1,q+1}(X,Y) = \begin{bmatrix}
x_{np} & x_{n,p+1} & 0 & \cdots & 0\\
y_{2q} & y_{2,q+1} & y_{2,q+2} & \cdots & y_{2n}\\
\vdots & \vdots & \vdots & \cdots & \vdots\\
y_{n-q+1,q} & y_{n-q+1,q+1} & y_{n-q+1,q+2} &\cdots & y_{n-q+1,n}
\end{bmatrix},
\end{equation}
and the expansion along the first row yields
\begin{equation}\label{eq:xnp}
h^\prime_{1,q+1}(X,Y) = x_{np} h_{2,q+1}(X,Y) - x_{n,p+1}\det Y^{\{q\}\cup[q+2,n]}_{[2,n-q+1]}.
\end{equation}
A further expansion of the minor $\det Y^{\{q\}\cup[q+2,n]}_{[2,n-q+1]}$ along its first column yields
\[
\det Y^{\{q\}\cup[q+2,n]}_{[2,n-q+1]} = \sum_{k=2}^{n-q+1} y_{kq} \det Y^{[q+2,n]}_{[2,k-1]\cup[k+1,n-q+1]}.
\]
In turn, the minors $\det Y^{[q+2,n]}_{[2,k-1]\cup[k+1,n-q+1]}$ are known to be cluster variables:
\[
\det Y^{[q+2,n]}_{[2,k-1]\cup[k+1,n-q+1]} = \begin{cases}
h_{3,q+2} \ &k=2,\\
h^{(k-2)}_{3,q+k} \ &2 < k < n-q+1,\\
h_{2,q+2} \ &k = n-q+1,
\end{cases}
\]
where the variables $h^{(k-2)}_{3,q+k}$ come from the $W$-sequences studied in Lemma~\ref{l:kilcol} (they are applicable in $|\Gamma_1^c| = 1$ as well, for the $h$-functions are the same as in the standard structure). It follows from equation~$\eqref{eq:xnp}$ that $x_{np}$ belongs to the upper cluster structure. Now, the rest of the proof is similar to Proposition~\ref{p:recr}: To recover the variables $x_{n-i,p}$ for $1 \leq i \leq p$, one uses the functions $g_{n-i,p}$; due to the relation $Y = (YX^{-1})X$ and Lemma~\ref{l:entru_c}, these variables together with $x_{np}$ yield $y_{np},\ldots, y_{pp}$. To proceed further, one recovers consecutively $y_{p,p+i}$ from $h_{p,p+i}$, and then one can obtain $x_{p,p+i}$ back from the relation $X = (YX^{-1})^{-1} Y$. Thus the proposition is proved. \qedhere
\end{proof}
\subsection{Coprimality}\label{s:coprim}
Let $\gc(\bg)$ be the generalized cluster structure on $D(\GL_n)$ induced by an aperiodic oriented \mbox{BD pair $\bg$}. In this section, we prove that all cluster and frozen variables from the initial extended cluster are irreducible as elements of $\mathcal{O}(D(\GL_n))$, as well as (for cluster variables) coprime\footnote{As of October 2023, these results are already outdated. We have published a work on birational quasi-isomorphisms and proved similar statements in a greater generality, thus covering the whole program on the GSV conjecture. The new proofs are clearer and simpler.} with their mutations in $\mathcal{O}(D(\GL_n))$. Together with Proposition~\ref{p:regular}, we will conclude that $\gc(\bg)$ is a regular generalized cluster structure. 
\begin{lemma}\label{l:irred}
Assume that $\bg$ is nontrivial and $\tilde{\bg}$ is obtained from $\bg$ by the removal of a pair of roots. Let $\psi_{\square}$ be the cluster variable from the initial cluster of $\gc(\bg)$ such that $\tilde{\psi}_{\square}$ is frozen in $\gc(\tilde{\bg})$ (see Section~\ref{s:birat}). Let $\tilde{\psi}\neq\tilde{\psi}_{\square}$ be a cluster or frozen variable in $\gc(\tilde{\bg})$ and $\psi$ be the corresponding variable in $\gc(\bg)$. Suppose that $\tilde{\psi}$ and $\tilde{\psi}_{\square}$ are irreducible as elements of $\mathcal{O}(D(\GL_n))$. Then there exist $f \in \mathcal{O}(D(\GL_n))$ and $\lambda \geq 0$ such that $f$ is coprime with $\psi_{\square}$ and $\psi = f \psi_{\square}^{\lambda}$; moreover, $\psi$ is irreducible in $\mathcal{O}(D(\GL_n))$ if and only if $\psi$ is not divisible by $\psi_{\square}$.
\end{lemma}
\begin{proof}
Indeed, let $\mathcal{U}:D(\GL_n)_{\tilde{\bg}} \dashrightarrow D(\GL_n)_{\bg}$ be the birational quasi-isomorphism constructed in Section~\ref{s:birat}. By Proposition~\ref{p:comparu}, $\mathcal{U}^*(\psi) = \tilde{\psi} \tilde{\psi}_{\square}^{\varepsilon}$ for some $\varepsilon \geq 0$. Assume that $\psi = f_1\cdot f_2$ for some regular coprime functions $f_1$ and $f_2$. Set
\[
\tilde{f}_i \tilde{\psi}_{\square}^{\lambda_i}:=\mathcal{U}^*(f_i), \ \ i \in \{1,2\}, \ \lambda_i \in \mathbb{Z}
\]
where $\tilde{f}_i \in \mathcal{O}(D(\GL_n))$ and $\tilde{\psi}_{\square}$ are coprime (by the assumption, $\tilde{\psi}_{\square}$ is irreducible, so we can find such $\tilde{f}_i$). Applying $\mathcal{U}^*$ to $\psi$, we arrive at
\[
\tilde{\psi}\tilde{\psi}_{\square}^{\varepsilon} = \tilde{f}_1\tilde{f}_2 \tilde{\psi}_{\square}^{\lambda_1+\lambda_2}.
\]
Since $\tilde{\psi}_{\square}$ is coprime with $\tilde{\psi}$, $\tilde{f}_1$ and $\tilde{f}_2$, we see that $\varepsilon = \lambda_1 + \lambda_2$; since $\tilde{\psi}$ is irreducible by the assumption, without loss of generality $\tilde{f}_2$ is a unit in $\mathcal{O}(\GL_n)$; that is, $\tilde{f}_2 = a \det X^{k}\det Y^{l}$ for some $k,l\in \mathbb{Z}$ and $a \in \mathbb{C}$. Since $(\mathcal{U}^*)^{-1}(\tilde{f}_2) = \tilde{f}_2$, we see that $\psi = \tilde{f}_2 f_1\psi_{\square}^{\lambda_2}$. Setting $f:=\tilde{f}_2f_1$ and $\lambda:=\lambda_2$ proves the first claim. Moreover, if $\psi$ is not divisible by $\psi_{\square}$, then $\lambda_2 = 0$, hence $f_2$ is a unit and thus $\psi$ is irreducible.
\end{proof}

\begin{proposition}\label{p:irred}
All cluster and frozen variables in the initial extended cluster of $\gc(\bg)$ are irreducible polynomials.
\end{proposition}
\begin{proof}
For the standard BD pair, it's the statement of Theorem~3.10 in~\cite{double}. For other BD pairs, let us use an induction on the size $|\Gamma_1^r| + |\Gamma_1^c|$. If $|\Gamma_1^r| + |\Gamma_1^c| = 1$, then the only variables from the initial extended cluster that differ from the case of the standard BD pair are $g$- and $h$-functions; these are irreducible by Frobenius theorem~\cite[p.~15]{schneider}. From now on, assume that $|\Gamma_1^r|+|\Gamma_1^c| \geq 2$.

Let $\tilde{\bg}$ be obtained from $\bg$ by removing a pair of leftmost or rightmost roots, and let $\psi_1$ be the variable that is cluster in $\gc(\bg)$ but such that $\tilde{\psi}_1$ is frozen in $\gc(\tilde{\bg})$. Let $\mathcal{U}_1 : D(\GL_n)_{\tilde{\bg}} \dashrightarrow D(\GL_n)_{\bg}$ be the associated birational quasi-isomorphism. Since $|\Gamma_1^r|+|\Gamma_1^c| \geq 2$, we can find yet another pair of roots from $\bg$ to remove; let us denote by $\psi_2$ the corresponding variable.

\emph{The variables $\psi_1$ and $\psi_2$ are irreducible.} Indeed, let us write $\psi_1 = f\psi_2^{\lambda_2}$ and $\psi_2 = g \psi_1^{\lambda_{1}}$ for some $\lambda_1,\lambda_2 \geq 0$ and $f,g\in\mathcal{O}(D(\GL_n))$ coprime with $\psi_2$ and $\psi_1$, respectively. Applying $\mathcal{U}_1$ to $\psi_1 = f \psi^{\lambda_2}_2$, we see that there exists $\tilde{f}$ coprime with $\tilde{\psi}_1$ and numbers $\theta \in \mathbb{Z}$, $\eta \geq 0$ such that 
\[
\tilde{\psi}_1 = \tilde{f}\tilde{\psi}_1^{\theta} \tilde{\psi}_2^{\lambda_2}\tilde{\psi}_1^{\eta\lambda_2}.
\]
It follows from the assumption of the induction that $1 = \theta + \eta\lambda_2$ and that $1 = \tilde{f}\tilde{\psi}_2^{\lambda_2}$. Since $\tilde{\psi}_2$ is not invertible in $\mathcal{O}(D(\GL_n))$, we conclude that $\lambda_2 = 0$. A similar argument shows that $\lambda_1 = 0$ (here one applies the other birational quasi-isomorphism). By Lemma~\ref{l:irred}, both $\psi_1$ and $\psi_2$ are irreducible.

\emph{Any cluster or frozen variable from the initial extended cluster is irreducible.} Indeed, let $\psi$ be such. If $\psi$ is not divisible by $\psi_1$ or $\psi_2$, then it follows from Lemma~\ref{l:irred} that $\psi$ is irreducible; otherwise, since $\psi_1$ and $\psi_2$ are irreducible, we can find $f \in \mathcal{O}(D(\GL_n))$ coprime with both $\psi_1$ and $\psi_2$, and numbers $\theta_1,\theta_2 \geq 1$ such that 
\[
\psi = f\psi_1^{\theta_1}\psi_2^{\theta_2}.
\]
Applying $\mathcal{U}_1^*$ to the above identity, we arrive at
\[
\tilde{\psi} \tilde{\psi}_1^{\varepsilon_1} = \tilde{f}\tilde{\psi}_1^{\eta_1} \tilde{\psi}_1^{\theta_1} \tilde{\psi}_2^{\theta_2}\tilde{\psi}_1^{\theta_2\zeta}
\]
where $\tilde{f}$ is coprime with $\tilde{\psi}_1$, $\eta_1, \in \mathbb{Z}$ and $\zeta \geq 0$. We see that $\varepsilon_1 = \eta_1 + \theta_1 + \theta_2\zeta$, hence $\tilde{\psi} = \tilde{f}\tilde{\psi}_2^{\theta_2}$. But since both $\tilde{\psi}$ and $\tilde{\psi}_2$ are coprime irreducible elements of $\mathcal{O}(D(\GL_n))$, we conclude that $\theta_2 = 0$. By Lemma~\ref{l:irred}, $\psi$ is irreducible. 
\end{proof}

\begin{proposition}\label{p:coprim}
Any cluster variable $\psi$ from the initial cluster of $\gc(\bg)$ is coprime with $\psi^\prime$.
\end{proposition}
\begin{proof}
As in the previous proposition, let us run an induction on the size $|\Gamma_1^r| + |\Gamma_1^c|$. For the standard BD pair, the statement was proved in~\cite{double}. For $|\Gamma_1^r|+|\Gamma_1^c| \geq 1$, let $\tilde{\bg}$ be obtained from $\bg$ by the removal of a pair of leftmost or rightmost roots, and let $\psi_{\square}$ be the cluster variable such that $\tilde{\psi}_{\square}$ is frozen in $\gc(\tilde{\bg})$. For any variable $\psi \neq \psi_{\square}$, since $\psi$ is irreducible (see Proposition~\ref{p:irred}), we can write $\psi^\prime = p \cdot \psi^{\lambda}$ for some $\lambda \geq 0$ and some $p$ coprime with $\psi$. Applying the corresponding birational quasi-isomorphism, we find $\varepsilon, \eta \geq 0$, $\theta \in \mathbb{Z}$ and an element $\tilde{p}$ coprime with $\tilde{\psi}_{\square}$ such that
\[
\tilde{\psi}^\prime \tilde{\psi}_{\square}^{\varepsilon} = \tilde{p} \tilde{\psi}_{\square}^{\theta} \tilde{\psi}^{\lambda} \tilde{\psi}_{\square}^{\eta \lambda}.
\]
Since $\tilde{\psi}^\prime$ and $\tilde{\psi}$ are coprime by the assumption of the induction, we see that $\lambda = 0$. Therefore, $\psi$ is coprime with $\psi^\prime$.

Now, let us address the case of $\psi = \psi_{\square}$. If $|\Gamma_1^r|+|\Gamma_1^c| \geq 2$, the coprimality of $\psi_{\square}$ with $\psi_{\square}^\prime$ follows from the existence of another birational quasi-isomorphism, which is associated with a different pair of roots. If $|\Gamma_1^r|+|\Gamma_1^c| = 1$, we observe from formula~\eqref{eq:psisqr}, formula~\eqref{eq:psisqc} and Frobenius theorem~\cite[p. 15]{schneider} that $\psi_{\square}^\prime$ is irreducible and coprime with $\psi_{\square}$. Thus the proposition is proved.
\end{proof}
 
Combining Proposition~\ref{p:regular}, Proposition~\ref{p:irred} and Proposition~\ref{p:coprim}, we see that $\gc(\bg)$ satisfies the first two conditions of Proposition~\ref{p:starfish}; thus, $\gc(\bg)$ is a regular generalized cluster structure.

\subsection{The final proof}
\begin{proposition}
Let $\gc(\bg)$ be a generalized cluster structure $\gc(\bg)$ on $D(\GL_n)$ that arises from an aperiodic oriented BD pair $\bg$. Then the ring of regular functions on $D(\GL_n)$ is naturally isomorphic to the the upper cluster algebra of $\gc(\bg)$.
\end{proposition}
\begin{proof}
The fact that $\gc(\bg)$ is a regular generalized cluster structure is the content of Section~\ref{s:regular} and Section~\ref{s:coprim}, hence we only need to verify the third condition of Proposition~\ref{p:starfish}. The proof is based on an inductive argument on the size $|\Gamma_1^r| + |\Gamma_1^c|$. The base of induction is ${|\Gamma_1^c| + |\Gamma_1^r| = 1}$, which is the content of Proposition~\ref{p:recr} and Proposition~\ref{p:recc}, and the inductive step is based on Corollary~\ref{c:comparu2} and the existence of at least two distinct birational quasi-isomorphisms. The proof can be executed verbatim as in~\cite{plethora}.
\end{proof}
\section{Toric action}\label{s:torpr}
Let $\bg = (\bg^r, \bg^c)$ be an aperiodic oriented BD pair that induces the generalized cluster structure $\gc(\bg)$ on $D(\GL_n)$. Let $\mathfrak{h}^{\sll_n}$ be the Cartan subalgebra of $\sll_n$. In Section~\ref{s:tor}, we defined subalgebras
\[
\mathfrak{h}_{\bg^{\ell}} := \{h \in \mathfrak{h}^{\sll_n} \ | \ \alpha(h) = \beta(h) \ \text{if} \ \gamma^j(\alpha) = \beta \ \text{for some} \ j\}
\]
and we let $\mathcal{H}_{\bg^r}$ and $\mathcal{H}_{\bg^c}$ be the connected subgroups of $\SL_n$ that correspond to $\mathfrak{h}_{\bg^r}$ and $\mathfrak{h}_{\bg^c}$, respectively. Then we let the groups $\mathcal{H}_{\bg^r}$ and $\mathcal{H}_{\bg^c}$ act upon $D(\GL_n)$ on the left and on the right, respectively, and we also defined an action by scalar matrices on each component of $D(\GL_n) = \GL_n \times \GL_n$. Note that $\dim \mathcal{H}_{\bg^r} = k_{\bg^r}:=|\Pi\setminus\Gamma_1^r|$ and $\dim \mathcal{H}_{\bg^c} = k_{\bg^c}:=|\Pi\setminus\Gamma_1^c|$, where $\Pi = [1,n-1]$ is the set of simple roots of type $A_n$. In this section, we show that the cumulative action of the three groups induces a global toric action on $\gc(\bg)$ of rank $k_{\bg^r}+k_{\bg^c}+2$.
\begin{lemma}\label{l:torsmv}
All cluster and frozen variables from the initial extended cluster are semi-invariant with respect to the left action by $\mathcal{H}_{\bg^r}$, the right action by $\mathcal{H}_{\bg^c}$, and the action by scalar matrices.
\end{lemma}
\begin{proof}
The $\varphi$-, $f$- and $c$-functions are semi-invariant with respect to the left action $T.(X,Y) = (TX,TY)$ and the right action $(X,Y).T = (XT,YT)$, where $T$ is any invertible diagonal matrix (see Theorem~6.1 in~\cite{double}). The $g$- and $f$-functions are semi-invariant with respect to the actions by $\mathcal{H}_{\bg^r}$ and $\mathcal{H}_{\bg^c}$ by Lemma~6.2 from~\cite{plethora}. Their semi-invariance relative the action $(a,b).(X,Y) = (aX,bY)$, $a,b \in \mathbb{C}^*$, follows from its infinitesimal counterpart~\eqref{eq:trconsti}.
\end{proof}

\paragraph{How the toric action is induced.} If $H \in \mathcal{H}_{\bg^r}$ and $\psi$ is any cluster or stable variable, then $\psi(HX,HY) = \chi(H) \psi(X,Y)$ for some character $\chi$ on $\mathcal{H}_{\bg^r}$ that depends on $\psi$. The character is a monomial in $k_{\bg^r}$ independent parameters that describe the group $\mathcal{H}_{\bg^r}$, and the exponents of the parameters become the weight vector assigned to $\psi$. Thus one induces a local toric action from the left-right action of $\mathcal{H}_{\bg^r}\times \mathcal{H}_{\bg^c}$ and the action by scalar matrices.

\begin{proposition}
The toric action induced by the left action of $\mathcal{H}_{\bg^r}$, the right action of $\mathcal{H}_{\bg^c}$ and the action by scalar matrices is $\gc$-extendable.
\end{proposition}
\begin{proof}
Let $\tilde{B}$ be the initial extended exchange matrix and $W$ be the weight matrix of the resulting toric action. 
The fact that $W$ has the full rank can be proved in exactly the same way as in~\cite{plethora} using the fact that the upper cluster algebra can be identified with the ring of regular functions on $D(\GL_n)$ (which is proved in Section~\ref{s:complet}): If we assume that $\rank W < k_{\bg^r} + k_{\bg^c}+2$, then one can construct a toric action of rank $1$ from the given action that leaves all cluster and stable variables invariant; but any $x_{ij}$ and $y_{ij}$ is a Laurent polynomial in the initial cluster and stable variables, and the constructed action does not fix them, which leads to a contradiction. Thus $\rank W = k_{\bg^r} + k_{\bg^c}+2$.

Now, let us show that $\tilde{B}W = 0$. This reduces to showing that if $\psi(X,Y) \psi^\prime(X,Y) = M(X,Y)$ is an exchange relation in the initial cluster, then $M(X,Y)$ is a semi-invariant of the three actions. For the exchange relations for $g$- and $h$- functions (except $h_{ii}$ and $g_{ii}$, $2 \leq i \leq n$), the latter was already shown in~\cite{plethora}. For $\varphi$- and $f$-functions, the statement was verified in~\cite{double}. Therefore, we need to check that $\tilde{B}W = 0$ holds for $h_{ii}$ and $g_{ii}$ when $2 \leq i \leq n$ (i.e., for the rows of $\tilde{B}$ that correspond to these functions). 

The mutation at $h_{ii}$ reads
\[
h_{ii} h_{ii}^\prime = h_{i-1,i}f_{1,n-i} + f_{1,n-i+1}h_{i,i+1}.
\]
Set $H := \diag(t_1,\ldots,t_n) \in \mathcal{H}_{\bg^r}$ and $M(X,Y)$ the right-hand side of the above mutation relation, and let us act by $H$ on $M(X,Y)$. If we set $h_{i,i+1}(HX,HY) = \alpha h_{i,i+1}(X,Y)$, where $\alpha = \alpha(t_1,\ldots,t_n)$, then $h_{i-1,i}(HX,HY) = t_{i-1}\alpha h_{i-1,i}(X,Y)$; similarly, if we write $f_{1,n-i}(HX,HY) = \beta f_{1,n-i}(X,Y)$, then $f_{1,n-i+1}(HX,HY) = t_{i-1} \beta f_{1,n-i+1}(X,Y)$. Overall, 
\[
M(HX,HY) = t_{i-1}\alpha \beta M(X,Y),
\] 
that is, it's a semi-invariant of $\mathcal{H}_{\bg^r}$.

Next, the mutation at $g_{ii}$ is
\[
g_{ii} g_{ii}^\prime = g_{i+1,i+1} f_{n-i+1,1} g_{i,i-1} + g_{i-1,i-1} f_{n-i,1} g_{i+1,i}.
\]
Let $M(X,Y)$ be the right-hand side of the latter mutation relation, and let us act by the same $H$ on $M(X,Y)$. If we write $g_{i+1,i+1}(HX,HY) = \alpha g_{i+1,i+1}(X,Y)$, where now we have some different $\alpha = \alpha(t_1,\ldots,t_n)$, then $g_{i-1,i-1}(HX,HY) = t_{i-1} t_i \alpha g_{i-1,i-1}(X,Y)$; next, if we write $f_{n-i,1}(HX,HY) = \beta f_{n-i,1}(X,Y)$, then $f_{n-i+1,1}(HX,HY) = t_{i-1} \beta f_{n-i+1,1}(X,Y)$; and lastly, if $g_{i+1,i}(HX,HY) = \gamma g_{i+1,i}(X,Y)$, then $g_{i,i-1}(HX,HY) = t_{i} \gamma g_{i,i-1}(X,Y)$. Overall, the action by $H$ on $M(X,Y)$ yields 
\[
M(HX,HY) = t_{i-1} t_i \alpha \beta \gamma M(X,Y).
\]
Reasoning along the same lines, one can prove that the RHS of the above mutation relations are also semi-invariant with respect to the right action by $\mathcal{H}_{\bg^c}$ and the action by scalar matrices. Lastly, the Casimirs $\hat{p}_{1r}$ from the statement of Proposition~\ref{p:toric} are given by $\hat{p}_{1r} = c_{r}^n g_{11}^{r-n} h_{11}^{-r}$, $1\leq r \leq n-1$. Their invariance was shown in~\cite{double}. Thus the toric action is $\gc$-extendable.
\end{proof}
\section{Log-canonicity in the initial cluster}\label{s:logcan}
The objective of this section is to prove that the brackets between all functions in the initial extended cluster are log-canonical. It was proved in \cite{plethora} that the brackets between $g$- and $h$-functions are such, so the rest is to show that $f$- and $\varphi$-functions are log-canonical between themselves and each other, as well as log-canonical with $g$- and $h$-functions. The former is straightforward:

\begin{proposition}
The $f$- and $\varphi$-functions are log-canonical between themselves and each other.
\end{proposition}
\begin{proof}
Notice that if a function $\phi$ satisfies $\pi_0 E_R \phi \in \mathfrak{b}_+$ and $\pi_0 E_L \phi \in \mathfrak{b}_-$, and if $\pi_0 E_R \log \phi = \text{const}$ and $\pi_0 E_L \log \phi = \text{const}$, then the first two terms of the bracket of two such functions are constant:
\[\begin{split}
&\langle R_+^c(E_L \log \phi_1), E_L \log \phi_2\rangle = \langle R_0^c\pi_0(E_L \log \phi_1), \pi_0 E_L\log \phi_2 \rangle = \text{const};\\
&-\langle R_+^r(E_R \log \phi_1), E_R \log \phi_2 \rangle = - \langle R_0^r\pi_0 (E_R \log \phi_1), E_R \log \phi_2 \rangle = \text{const}.
\end{split}
\]
This means that the difference between $\{\log \phi_1, \log \phi_2 \}$ and $\{\log \phi_1, \log \phi_2 \}_{\text{std}}$ (the standard bracket studied in \cite{double}) is constant. Since $f$- and $\varphi$-functions enjoy such properties, they are log-canonical between themselves and each other.
\end{proof}

Before we proceed to proving the log-canonicity for the remaining pairs, let us derive a preliminary formula, which is also needed in Section~\ref{s:compb}:
\begin{lemma}\label{l:phipsi}
Let $\phi$ be any $f$- or $\varphi$-function, and let $\psi$ be any $g$- or $h$-function. Then the following formula holds:
\begin{equation}\label{eq:phipsi}
\{\phi, \psi\} = -\langle \pi_0 E_L \phi, \nabla_Y \psi Y\rangle + \langle \pi_0 E_R \phi, Y\nabla_Y \psi \rangle + \langle R_0^c \pi_0 E_L \phi, E_L\psi\rangle - \langle R_0^r \pi_0 E_R\phi, E_R \psi\rangle.
\end{equation}
\end{lemma}
\begin{proof}
If $\phi$ is either a $\varphi$- or $f$-function, then $\pi_0 E_R \phi \in \mathfrak{b}_+$ and $\pi_0 E_L \phi \in \mathfrak{b}_-$. Let's use the following form of the bracket:
\[
\{\phi, \psi\} = \langle R_+^c(E_L \phi), E_L \psi\rangle - \langle R_+^r(E_R \phi), E_R \psi \rangle + \langle E_R \phi, Y\nabla_Y \psi\rangle - \langle E_L \phi, \nabla_Y \psi Y \rangle.
\]
Recall that $E_L \psi = \xi_L \psi + (1-\gamma_c) (\nabla_X \psi X)$, where $\pi_0\xi_L \psi \in \mathfrak{b}_-$; with that in mind, rewrite the first term as
\[\begin{split}
\langle R_+^c(E_L \phi)&, E_L \psi\rangle = -\langle \frac{\gamma_c^*}{1-\gamma_c^*} \pi_{<} E_L\phi, E_L\psi\rangle + \langle R_0^c \pi_0 E_L\phi, E_L\psi \rangle = \\ &= -\langle \pi_{<} E_L\phi, \gamma_c(\nabla_X\psi X) \rangle +\langle R_0^c \pi_0 E_L\phi, E_L\psi \rangle = \langle \pi_{<} E_L\phi, \nabla_Y \psi Y\rangle + \langle R_0^c \pi_0 E_L\phi, E_L\psi \rangle.
\end{split}
\]
Similarly, applying $E_R \psi = \xi_R \psi + (1-\gamma_r^*) (Y\nabla_Y \psi)$, we can rewrite the second term of the bracket as
\[
-\langle R_+^r(E_R\phi), E_R \psi \rangle = -\langle \pi_{>} E_R\phi, Y\nabla_Y \psi\rangle - \langle R_0^r \pi_0 E_R\phi, E_R \psi \rangle.
\]
Combining all together, the result follows.
\end{proof}

\begin{proposition}
All $f$- and $\varphi$-functions are log-canonical with all $g$- and $h$-functions.
\end{proposition}
\begin{proof}
Let $\phi$ be any $f$- or $\varphi$-function, and let $\psi$ be any $g$- or $h$-function. Only for this proof, call two rational functions log-equivalent ($\loge$) if their difference is a multiple of $\phi \psi$. Therefore, we aim at proving that $\{\phi,\psi\} \loge 0$. Let us pick a pair of solutions $(R_0^r,R_0^c)$ of the system~\eqref{eq:r0}-\eqref{eq:ralg} with the properties~\eqref{eq:roid} (as a reminder, in Section~\ref{s:depr} we show that log-canonicity doesn't depend on the choice of $R_0$).

Recall from \cite{plethora} or from Section~\ref{s:iniclust} that all the following quantities
\[
\pi_0\xi_R \psi, \ \ \pi_0 \eta_R \psi, \ \ \pi_0 \xi_L \psi, \ \ \pi_0 \eta_L \psi,\ \ \pi_0 \pi_{\hat{\Gamma}_1^r}(X \nabla_X \psi), \ \ \pi_0 \pi_{\hat{\Gamma}_2^r} (Y\nabla_Y \psi),  \ \ \pi_0 \pi_{\hat{\Gamma}_1^c}(\nabla_X \psi X),\ \ \pi_0 \pi_{\hat{\Gamma}_2^c} (\nabla_Y\psi Y), \ \ 
\]
are multiples of $\psi$. Also, recall from \cite{double} or from Section~\ref{s:iniclust} that $\pi_0 E_R \phi$ and $\pi_0 E_L \phi$ are multiples of $\phi$. Therefore, rewriting $E_L \psi$ as $E_L \psi = \eta_L \psi + (1-\gamma_c^*)(\nabla_Y \psi Y)$ and using $(R_0^c)^*(1-\gamma_c^*) = \pi_{\Gamma_2^c} + (R_0^c)^* \pi_{\hat{\Gamma}_2^c}$, we see that
\[\begin{split}
\langle R_0^c \pi_0 E_L \phi, E_L\psi \rangle &= \overbrace{\langle R_0^c \pi_0 E_L \phi, \eta_L \psi \rangle}^{\loge 0} + \langle \pi_0 E_L\phi, \pi_0(R_0^c)^*(1-\gamma_c^*)(\nabla_Y\psi Y)\rangle \loge\\&\loge \langle \pi_0 E_L\phi, \pi_{\Gamma_2^c} \nabla_Y\psi Y \rangle + \overbrace{\langle R_0^c \pi_0 E_L \phi, \pi_{\hat{\Gamma}_2^c} \nabla_Y \psi Y \rangle}^{\loge 0} \loge \langle \pi_0 E_L \phi, \pi_{\Gamma_2^c} \nabla_Y \psi Y \rangle.
\end{split}
\]
Similarly, rewriting $E_R\psi = \xi_R + (1-\gamma_r^*) Y\nabla_Y \psi$ and using $(R_0^r)^*(1-\gamma_r^*) = \pi_{\Gamma_2^r} + (R_0^r)^* \pi_{\hat{\Gamma}_2^r}$, we arrive at
\[
-\langle R_0^r \pi_0 E_R\phi, E_R \psi \rangle \loge - \langle \pi_0 E_R \phi, \pi_{\Gamma_2^r} Y\nabla_Y \psi \rangle.
\]
Now, combining these together with formula \eqref{eq:phipsi}, we see that 
\[\begin{split}
\{\phi, \psi\} &\loge -\langle \pi_0 E_L \phi, \nabla_Y \psi Y\rangle + \langle \pi_0 E_R \phi, Y\nabla_Y \psi \rangle + \langle \pi_0 E_L \phi, \pi_{\Gamma_2^c} \nabla_Y \psi Y \rangle - \langle \pi_0 E_R \phi, \pi_{\Gamma_2^r} Y\nabla_Y \psi \rangle \loge \\ & \loge -\langle \pi_0 E_L\phi, \pi_{\hat{\Gamma}_2^c} \nabla_Y \psi Y\rangle - \langle \pi_0 E_R\phi, \pi_{\hat{\Gamma}_2^r} Y\nabla_Y \psi \rangle \loge 0.
\end{split}
\]
Thus the result follows.
\end{proof}
\section{Compatibility}\label{s:compb}
The objective of this section is to prove Condition \ref{pr:c2} of Proposition \ref{p:compb}. The matrix $\Delta$ from the proposition is the identity matrix; therefore, we show that $\{\log y_{i}, \log x_j\} = \delta_{ij}$, where $y_i$ is the $y$-coordinate of a cluster variable $x_i$. Together with the results from Section \ref{s:logcan}, we will conclude, in particular, that any extended cluster in $\gc(\bg^r,\bg^c)$ is log-canonical with respect to the Poisson bracket. If $\psi_1$ is any cluster $g$- or $h$-function that is not equal to $g_{ii}$ or $h_{ii}$, $1 \leq i \leq n$, and $\psi_2$ is any $g$- or $h$-function, then it was shown in \cite{plethora} that
\[
\{\log y(\psi_1), \log \psi_2\} = \begin{cases}
1, \ &\psi_1 = \psi_2\\
0, \ &\text{otherwise.}
\end{cases}
\]
In this section, we treat all the other pairs of functions from the initial cluster.

\subsection{Diagonal derivatives}
\label{s:diagders}
In this subsection, we state technical formulas that compare the diagonal derivatives of the variables that are adjacent in the quiver. We remind the reader that when the indices are seemingly out of range, the conventions~\eqref{eq:fconv}-\eqref{eq:hconv} are in place.

\paragraph{Case of $f$- and $\varphi$-functions.} Set $\Delta(i,j) := \sum_{k=i}^{j} e_{kk}.$ The following formulas are drawn from the text of~\cite[p.~25]{double}: for $k,l\geq 0$, $1 \leq k+l\leq n$,
\begin{equation}\label{eq:elfexp}
\pi_0 E_L \log f_{kl} = \Delta(n-k+1,n) + \Delta(n-l+1,n), \ \  \pi_0E_R \log f_{kl} = \Delta(n-k-l+1,n);
\end{equation}
for $k,l \geq 1, \ k+l \leq n$,
\begin{equation}\begin{split}
\pi_0 E_L \log \varphi_{kl} &= (n-k-l)(I+\Delta(n,n))+ \Delta(n-k+1,n) + \Delta(n-l+1,n), \\ \pi_0 E_R \log\varphi_{kl} &= (n-k-l+1)I.
\end{split}
\end{equation}
Let $y(f_{kl})$ and $y(\varphi_{kl})$ be $y$-coordinates. Examining the neighborhoods of $\varphi$- and $f$-functions and applying the above formulas yield
\begin{equation}\label{eq:deryfphi}\begin{split}
&\pi_0 E_L y(f_{kl}) = \pi_0 E_R y(f_{kl}) = 0, \ \ k,l\geq 1, \ \ k+l \leq n-1;\\
&\pi_0 E_L y(\varphi_{kl}) = \pi_0 E_R y(\varphi_{kl}) = 0, \ \ k,l \geq 1, \ \ k+l \leq n.
\end{split}
\end{equation}

\paragraph{Case of $g$- and $h$-functions.} For $1 \leq i \leq j \leq n$, let us denote
\begin{equation}\begin{split}
g&:= \log g_{ij} - \log g_{i+1,j+1},\\
h&:= \log h_{ji} - \log h_{j+1,i+1}.
\end{split}
\end{equation}
Then $g$ satisfies the following list of formulas:
\begin{equation}
\begin{aligned}
\pi_0 \xi_L g &= \gamma_c(e_{jj}), & \pi_0 \xi_R g &= e_{ii}, \\
\pi_0 \eta_L g &= e_{jj}, & \pi_0 \eta_R g &= \gamma_r(e_{ii}); 
\end{aligned}
\end{equation}
\begin{equation}
\begin{aligned}
\pi_0\pi_{\hat{\Gamma}_2^c} (\nabla_Y g \cdot Y) &= 0, & \pi_0 \pi_{\hat{\Gamma}_1^r} (X\nabla_X g) &= \pi_{\hat{\Gamma}_1^r} e_{ii},\\
\pi_0 \pi_{\hat{\Gamma}_1^c}(\nabla_X g \cdot X) &= \pi_{\hat{\Gamma}_1^c} (e_{jj}), & \pi_0 \pi_{\hat{\Gamma}_2^r}(Y\nabla_Y g) & = 0;
\end{aligned}
\end{equation}
and for any runs $\Delta^r$, $\Delta^c$, $\bar{\Delta}^r$ and $\bar{\Delta}^c$,
\begin{equation}\label{eq:ylg}
\begin{aligned}
\tr(\nabla_X g X)_{\Delta^c}^{\Delta^c} &= 1_{\Delta^c}(j), & \tr(X\nabla_X g)_{\Delta^r}^{\Delta^r} &= 1_{\Delta^r}(i),\\
\tr (\nabla_Y g Y)_{\bar{\Delta}^c}^{\bar{\Delta}^c} &= 0, & \tr (Y\nabla_Y g)_{\bar{\Delta}^r}^{\bar{\Delta}^r}  & = 0
\end{aligned}
\end{equation}
where $1_{\Delta^c}$ and $1_{\Delta^r}$ are indicators. Similarly, $h$ satisfies the following list:
\begin{equation}\label{eq:diagxiet}
\begin{aligned}
\pi_0 \xi_L h &= e_{ii}, & \pi_0 \xi_R h &= \gamma_r^*(e_{jj}), \\
\pi_0 \eta_L h &= \gamma_c^*(e_{ii}), & \pi_0 \eta_R h &= e_{jj}; 
\end{aligned}
\end{equation}
\begin{equation}
\begin{aligned}
\pi_0\pi_{\hat{\Gamma}_2^c} (\nabla_Y h \cdot Y) &= \pi_{\hat{\Gamma}_2^c}e_{ii}, & \pi_0 \pi_{\hat{\Gamma}_1^r} (X\nabla_X h) &= 0,\\
\pi_0 \pi_{\hat{\Gamma}_1^c}(\nabla_X h \cdot X) &=0, & \pi_0 \pi_{\hat{\Gamma}_2^r}(Y\nabla_Y h) & = \pi_{\hat{\Gamma}_2^r} e_{jj};
\end{aligned}
\end{equation}
\begin{equation}\label{eq:xrh}
\begin{aligned}
\tr(\nabla_X h X)_{\Delta^c}^{\Delta^c} &= 0, & \tr(X\nabla_X h)_{\Delta^r}^{\Delta^r} &= 0,\\
\tr (\nabla_Y h Y)_{\bar{\Delta}^c}^{\bar{\Delta}^c} &= 1_{\bar{\Delta}^c}(i), & \tr (Y\nabla_Y h)_{\bar{\Delta}^r}^{\bar{\Delta}^r}  & = 1_{\bar{\Delta}^r}(j).
\end{aligned}
\end{equation}
The above formulas easily follow from a close inspection of the invariance properties from equation \eqref{eq:diaginv}; the formulas for traces follow from the proof of Lemma 4.4 in~\cite{plethora}. As a corollary, if $D \in \{\xi_L,\xi_R, \eta_L, \eta_R\}$, then
\begin{equation}\label{eq:ygh}
\pi_0 D y(g_{ij}) = \pi_0 D y(h_{ji}) = 0, \ \ 1 \leq j < i \leq n.
\end{equation}
For $i=j$, formula \eqref{eq:ygh} is true only for $D = \eta_L$ and $D=\xi_L$. It follows from equation \eqref{eq:ylg} that for any $1 \leq j \leq i \leq n$,
\begin{equation}\label{eq:trylg}
\tr( \nabla_X  y(g_{ij}) X)_{{\Delta}^c}^{{\Delta}^c} = \tr( X\nabla_X  y(g_{ij})_{{\Delta}^r}^{{\Delta}^r} =
\tr( \nabla_Y  y(g_{ij}) Y)_{\bar{\Delta}^c}^{\bar{\Delta}^c} =\tr( Y\nabla_Y  y(g_{ij}))_{\bar{\Delta}^r}^{\bar{\Delta}^r} = 0;
\end{equation}
and for any $1 \leq j < i \leq n$, it's a consequence of equation \eqref{eq:xrh} that
\begin{equation}\label{eq:trxrh}
\tr( \nabla_X  y(h_{ji}) X)_{{\Delta}^c}^{{\Delta}^c} = \tr( X\nabla_X  y(h_{ji}))_{{\Delta}^r}^{{\Delta}^r} =
\tr( \nabla_Y  y(h_{ji}) Y)_{\bar{\Delta}^c}^{\bar{\Delta}^c} =\tr( Y\nabla_Y  y(h_{ji}))_{\bar{\Delta}^r}^{\bar{\Delta}^r} = 0.
\end{equation}
\subsection{Dependence on the choice of $R_0$}\label{s:depr}
In this subsection we show that the compatibility of the Poisson bracket with the generalized cluster structure $\gc(\bg^r,\bg^c)$ does not depend on the choice of the solutions of the system~\eqref{eq:r0}-\eqref{eq:ralg}. Specifically, let $(R_0^r,R_0^c)$ and $(\tilde{R}_0^r,\tilde{R}_0^c)$ be solutions that correspond to $(\bg^r,\bg^c)$, and let us consider two Poisson brackets on $D(\GL_n)$ that depend on these choices: $\{\cdot,\cdot\}_{(R_0^r,R_0^c)}$ and $\{\cdot,\cdot\}_{(\tilde{R}_0^r,\tilde{R}_0^c)}$.
\begin{proposition}\label{p:r0deplog}
If the initial extended cluster of $\gc(\bg^r,\bg^c)$ is log-canonical with respect to $\{\cdot, \cdot\}_{(R_0^r,R_0^c)}$, then it's also log-canonical with respect to $\{\cdot,\cdot \}_{(\tilde{R}_0^r,\tilde{R}_0^c)}$.
\end{proposition}
\begin{proof}
Indeed, let $\psi_1$ and $\psi_2$ be any two variables from the initial extended cluster and let $\mathfrak{h}$ be the Cartan subalgebra of $\gl_n(\mathbb{C})$. Then the difference of the brackets can be written as
\[
\{\psi_1,\psi_2\}_{(R_0^r,R_0^c)} - \{\psi_1,\psi_2\}_{(\tilde{R}_0^r,\tilde{R}_0^c)} = \langle s_0^c \pi_0 E_L\psi_1, \pi_0 E_L\psi_2 \rangle - \langle s_0^r \pi_0 E_R \psi_1,\pi_0 E_R \psi_2\rangle
\]
where $s_0^\ell : \mathfrak{h}\rightarrow \mathfrak{h}$ is a skew-symmetric linear transformation such that $s_0^\ell(\alpha- \gamma_\ell(\alpha)) = 0$ for $\alpha \in \Gamma_1^\ell$, $\ell \in \{r,c\}$. Now it suffices to prove that\footnote{Let $A_i:= \pi_0 E_L\log\psi_i$. If we show that $s_0^c A_1$ and $s_0^cA_2$ are constant, then we can write $s_0^cA_1 = s_0^c\tilde{A}_1$ for some constant $\tilde{A}_1$; hence, $\langle s_0^cA_1, A_2\rangle = -\langle \tilde{A}_1,s_0^cA_2\rangle = \text{const}$.} $s_0^c\pi_0 E_L\log \psi = \text{const}$ and $s_0^r \pi_0E_R \log\psi = \text{const}$, where $\psi$ is any function from the initial extended cluster. Let us only deal with the case of $s_0^c$, the other case is similar. If $\psi$ is a $\varphi$- or $f$-function, then it follows from equation~\eqref{eq:invarfphi} that $\pi_0 E_L \log \psi = \text{const}$; if $\psi$ is a $g$- or $h$-function, then we write $\pi_0 E_L \psi = \pi_0\xi_L \psi + \pi_0(1-\gamma^c) X \nabla_X \psi$. Recall from equation~\eqref{eq:xiconst} that $\pi_0 \xi_L \log \psi = \text{const}$, hence it's left to study $s_0^c\pi_0 (1-\gamma^c)(X\nabla_X \psi)$. Let us enumerate all nontrivial column $X$-runs as $\Delta^c_1,\ldots,\Delta^c_k$, and let us decompose the space of all diagonal matrices $\mathfrak{h}$ as
\begin{equation}\label{eq:hdecomp}
\mathfrak{h} = \left(\bigoplus_{i=1}^k \mathfrak{h}_i\right) \oplus \left( \bigoplus_{i=1}^k \langle I_i \rangle \right) \oplus (\h_{\Gamma_1^c})^\perp
\end{equation}
where $\mathfrak{h}_i$ is a subspace generated by the roots $\Delta_i^c \cap \Gamma_1^c$, $I_i := \sum_{j \in \Delta_i^c} e_{jj}$, $\langle I_i \rangle$ is the span of $I_i$ and $\mathfrak{h}_{\Gamma_1^c}$ is the span of $\{e_{jj} \ | \ \exists i \in [1,k], \  j \in \Delta_i^c\}$. Now, $\pi_{\hat{\Gamma}_1^c}\pi_0 X\nabla_X\log \psi$ is constant by \eqref{eq:xiconst}, and the application of $s_0^c(1-\gamma_c)$ to $X\nabla_X\log \psi$ is zero on the first component of equation~\eqref{eq:hdecomp}. The projection of $X\nabla_X\log \psi$ onto the second component is equal to 
\begin{equation}\label{eq:diagdecomp}
\sum_{i=1}^k \frac{1}{|\Delta_i^c|} \tr(X\nabla_X\log \psi)_{\Delta_i^c}^{\Delta_i^c},
\end{equation}
which is constant by \eqref{eq:deltatraces} (or by Lemma 4.4 from \cite{plethora}). Thus the statement holds.\qedhere
\end{proof}

\begin{proposition}
If $\gc(\bg^r,\bg^c)$ is compatible with the Poisson bracket $\{\cdot,\cdot\}_{(R_0^r,R_0^c)}$, then it's also compatible with $\{\cdot,\cdot\}_{(\tilde{R}_0^r,\tilde{R}_0^c)}$.
\end{proposition}
\begin{proof}
Let $\psi_1$ and $\psi_2$ by any two variables from the initial extended cluster with $\psi_1$ being non-frozen. As the proof of Proposition~\ref{p:r0deplog} shows, we need to prove that
\[
\langle s_0^c E_L y(\psi_1), \psi_2\rangle = 0 \ \ \text{and} \ \ \langle s_0^r E_R y(\psi_1),\psi_2\rangle = 0.
\]
If $\psi_1$ is any $\varphi$- or $f$-function, then the above identities follow from formulas~\eqref{eq:deryfphi}. Assume that $\psi_1 = g_{ij}$ for $1 \leq j \leq i \leq n$ or $\psi_1 = h_{ji}$ for $1 \leq j < i \leq n$ (and $\psi_1$ is not frozen). Then we can write $E_L = \xi_L + (1-\gamma_c)(\nabla_X X)$ and recall that $\pi_0\xi_L y(\psi_1) = 0$ by equation~\eqref{eq:ygh}, 
$\tr(\nabla_X y(\psi_1))_{\Delta_i^c}^{\Delta_i^c} = 0$ by equation~\eqref{eq:trylg} and equation~\eqref{eq:trxrh}; and finally, $\pi_0 \pi_{\hat{\Gamma}_1^c}(\nabla_X y(\psi_1) X) = 0$ by~\eqref{eq:ygh}; therefore, $\langle s_0^c E_L y(\psi_1), \psi_2\rangle = 0$. In a similar way one can prove $\langle s_0^r E_R y(\psi_1),\psi_2\rangle = 0$. The only exception is $\psi_1 = h_{ii}$ for $2 \leq i \leq n$. In this case, we set $h = \log h_{i-1,i} - \log h_{i,i+1}$ and $f = \log f_{1,n-i} - \log f_{1,n-i+1}$ so that $\log y(h_{ii}) = h + f$, and let $\Delta_1^r,\ldots,\Delta_{m}^r$ be the list of all nontrivial row $X$-runs; then, by equations \eqref{eq:elfexp}, \eqref{eq:diagxiet} and~\eqref{eq:xrh},
\[\begin{split}
\langle s_0^r E_R& \log y(h_{ii}),\psi_2\rangle = \langle s_0^r \eta_R(h), E_R \psi \rangle + \langle s_0^r(1-\gamma_r)X\nabla_X h, E_R \psi\rangle + \langle s_0^r E_R f, E_R \psi \rangle = \\ &= \langle s_0^r e_{i-1,i-1}, E_R\psi\rangle + \sum_{k=1}^m \frac{1}{|\Delta_k^r|} \tr(X \nabla_X h)_{\Delta^r}^{\Delta^r}\langle s_0^r I_k, E_R \psi\rangle + \langle s_0(-e_{i-1,i-1}),E_R\psi\rangle = 0.\qedhere
\end{split}
\]
\end{proof}
\subsection{Computation of $\{y(\phi),\psi\}$ and $\{y(\psi),\phi\}$}
Let $\phi$ be any $f$- or $\varphi$-function, and let $\psi$ be any $g$- or $h$-function. The objective of this subsection is to show that $\{y(\phi),\psi\} = \{y(\psi), \phi\} = 0$ (for $y(\psi)$, we assume that $\psi$ is a cluster variable).

\begin{proposition}
$\{y(\phi),\psi\} = 0$.
\end{proposition}
\begin{proof}
Let us apply formula \eqref{eq:phipsi}:
\[
\{y(\phi), \psi\} = -\langle \pi_0 E_L y(\phi), \nabla_Y \psi Y\rangle + \langle \pi_0 E_R y(\phi), Y\nabla_Y \psi \rangle + \langle R_0^c \pi_0 E_L y(\phi), E_L\psi\rangle - \langle R_0^r \pi_0 E_Ry(\phi), E_R \psi\rangle,
\]
and now recall from equation \eqref{eq:deryfphi} that $\pi_0E_Ly(\phi) = \pi_0E_Ry(\phi) = 0$. Thus $\{y(\phi),\psi\} = 0$.
\end{proof}

\begin{proposition}
$\{y(\psi),\phi\} = 0$.
\end{proposition}
\begin{proof}
Let us pick a pair $(R_0^r,R_0^c)$ of solutions of the system~\eqref{eq:r0}-\eqref{eq:ralg} such that both $R_0^r$ and $R_0^c$ satisfy the identities~\eqref{eq:roid}.\\
\noindent \emph{Case 1, $i\neq j$.} Assume first that $\psi$ is any cluster $g_{ij}$ or $h_{ji}$ for $1 \leq j < i \leq n$. Similarly to equation \eqref{eq:phipsi}, we can write
\begin{equation}\label{eq:ypsiphi}
\{y(\psi),\phi\} = \langle \pi_0 X \nabla_X y(\psi), E_R\phi\rangle - \langle \pi_0 \nabla_X y(\psi)\cdot X, E_L\phi\rangle + \langle R_0^c \pi_0 E_L y(\psi), E_L\phi\rangle -\langle R_0^r \pi_0 E_R y(\psi), E_R\phi\rangle.
\end{equation}
Using equation \eqref{eq:ygh} and the formula $E_L = \xi_L + (1-\gamma_c)(\nabla_X X)$, we can write $E_L y(\psi) = (1-\gamma_c) \nabla_X y(\psi) X$ and $\pi_0 \pi_{\hat{\Gamma}_1^c} \nabla_X y(\psi) X = 0$; hence, the second and the third terms combine into
\[\begin{split}
- \langle \pi_0 \nabla_X y(\psi)\cdot X&, E_L\phi\rangle + \langle R_0^c \pi_0 E_L y(\psi), E_L\phi\rangle =\\&= - \langle \pi_0\pi_{\Gamma_1^c} \nabla_X y(\psi)\cdot X, E_L\phi\rangle + \langle R_0^c \pi_0 (1-\gamma_c)\pi_{\Gamma_1^c} (\nabla_X y(\psi)\cdot X), E_L\phi\rangle= 0.
\end{split}
\]
Similarly, the first term cancels out with the fourth one if we write $E_R y(\psi) = (1-\gamma_r)(X\nabla_X y(\psi))$ and apply $R_0^r(1-\gamma_r)\pi_0 = \pi_0 \pi_{\Gamma_1^r} + R_0^r\pi_0 \pi_{\hat{\Gamma}_1^r}$.

\noindent \emph{Case 2. $\psi = h_{ii}$, $2 \leq i \leq n$.} Let us denote $\hat{h}:= \log h_{i-1,i} - \log h_{i,i+1}$, $\hat{f} := \log f_{1,n-i} - \log f_{1,n-i+1}$, $\hat{\phi} = \log \phi$. Then $\log y(h_{ii}) = \hat{h} + \hat{f}$. The bracket $\{\hat{h}, \hat{\phi}\}$ can be expressed as in equation \eqref{eq:ypsiphi}:
\[
\{\hat{h},\hat{\phi}\} = \langle \pi_0 X \nabla_X \hat{h}, E_R\hat{\phi}\rangle - \langle \pi_0 \nabla_X \hat{h}\cdot X, E_L\hat{\phi}\rangle + \langle R_0^c \pi_0 E_L \hat{h}, E_L\hat{\phi}\rangle -\langle R_0^r \pi_0 E_R \hat{h}, E_R\hat{\phi}\rangle.
\]
Using the diagonal derivatives formulas from Section \ref{s:diagders} and $E_L = \xi_L + (1-\gamma_c)(\nabla_X X)$, we can expand the second and the third terms as
\[\begin{split}
- \langle \pi_0 \nabla_X \hat{h}\cdot X&, E_L\hat{\phi}\rangle + \langle R_0^c \pi_0 E_L \hat{h}, E_L\hat{\phi}\rangle = - \langle \pi_0 \pi_{\Gamma_1^c}\nabla_X \hat{h}\cdot X, E_L\hat{\phi}\rangle + \langle R_0^c e_{ii}, E_L\hat{\phi}\rangle + \\ &+ \langle R_0^c(1-\gamma_c)(\nabla_X \hat{h} \cdot X), E_L\hat{\phi}\rangle = \langle R_0^c e_{ii}, E_L\hat{\phi}\rangle;
\end{split}
\]
similarly, using $E_R = \eta_R + (1-\gamma_r)(X\nabla_X)$, we write
\[
\langle \pi_0 X \nabla_X \hat{h}, E_R\hat{\phi}\rangle-\langle R_0^r \pi_0 E_R \hat{h}, E_R\hat{\phi}\rangle = -\langle R_0^r e_{i-1,i-1}, E_R \hat{\phi} \rangle,
\]
hence $\{\hat{h}, \hat{\phi}\} = \langle R_0^c e_{ii}, E_L\hat{\phi}\rangle - \langle R_0^r e_{i-1,i-1}, E_R \hat{\phi} \rangle$. Using the invariance properties of $f$-functions together with the diagonal derivatives formulas for $\hat{f}$, we can write $\{\hat{f},\hat{\phi}\}$ as
\[
\{\hat{f}, \hat{\phi}\} = -\langle R_0^c e_{ii}, E_L\hat{\phi}\rangle + \langle R_0^r e_{i-1,i-1}, E_R \hat{\phi}\rangle + \langle X \nabla_X \hat{f}, Y\nabla_Y \hat{\phi}\rangle - \langle \nabla_X \hat{f} X, \nabla_Y \hat{\phi} Y\rangle.
\]
Altogether, we see that
\[
\{\log y(\psi), \log \phi\} = \langle X \nabla_X \hat{f}, Y\nabla_Y \hat{\phi}\rangle - \langle \nabla_X \hat{f} X, \nabla_Y \hat{\phi} Y\rangle.
\]
The latter expression depends only on $f$- and $\varphi$-functions, which stay the same for all oriented aperiodic BD pairs. Since it was proved in \cite{double} that for the standard pair we have $\{y(\psi), \phi\} = 0$, we see that the same is true for any other BD pair.

\noindent \emph{Case 2. $\psi = g_{ii}$, $2 \leq i \leq n$.} Let us denote $\hat{g}:=\log g_{i,i-1} - \log g_{i+1,i}$, $\hat{g}^\prime := \log g_{i+1,i+1} - \log g_{i-1,i-1}$, $\hat{f} := \log f_{n-i+1,1} - \log f_{n-i,1}$. Then $\log y(g_{ii}) = \hat{g} + \hat{g}^\prime + \hat{f}$. The bracket between these three pieces and $\phi$ can be computed as in the previous case:
\[
\{\hat{g}, \hat{\phi}\} = \langle R_0^c e_{i-1,i-1}, E_L \hat{\phi} \rangle - \langle R_0^r e_{ii}, E_R \hat{\phi} \rangle + \langle e_{ii}, E_R \hat{\phi}\rangle - \langle e_{i-1,i-1}, E_L\hat{\phi} \rangle;
\]
\[
\{\hat{f},\hat{\phi}\} = \langle R_0^c e_{ii}, E_L\hat{\phi}\rangle - \langle R_0^r e_{i-1,i-1}, E_R \hat{\phi} \rangle + \langle X \nabla_X \hat{f}, Y\nabla_Y \hat{\phi}\rangle - \langle \nabla_X \hat{f} X, \nabla_Y \hat{\phi} Y \rangle;
\]
\[\begin{split}
\{ \hat{g}^\prime, \hat{\phi} \} &= -\langle R_0^c(e_{i-1,i-1} + e_{ii}), E_L \hat{\phi}\rangle + \langle R_0^r(e_{i-1,i-1} + e_{ii}), E_R \hat{\phi}\rangle - \\ &-\langle e_{i-1,i-1} + e_{ii}, E_R \hat{\phi} \rangle + \langle e_{ii} + e_{i-1,i-1}, E_L\hat{\phi} \rangle.
\end{split}
\]
Now, summing up the above three equations, we see that all terms with $R_0^c$ or $R_0^r$ cancel out. The remaining terms do not depend on the choice of BD pair, hence they coincide with the expression in the standard case, which is zero.
\end{proof}
\subsection{Bracket for $g$- and $h$-functions}\label{s:bragh}
The main objective of this subsection is to derive a formula for the Poisson bracket between $g$- and $h$-functions that's subsequently used below. 


\paragraph{Shorthand notation.} Whenever we fix two functions $\psi_1$ and $\psi_2$, let us denote the gradients of their logarithms (and operators associated with them) via augmenting the operators with upper indices $1$ or $2$. For instance, $\nabla^1_X \cdot X := \nabla_X \log \psi_1 \cdot X$ or $\eta_R^2 := \eta_R \log \psi_2$. For conciseness, any other data associated with either of the two functions (e.g., blocks or $\mathcal{L}$-matrices) is also augmented with upper indices $1$ or $2$.

\begin{lemma}
Let $\psi_1$ and $\psi_2$ be any $g$- or $h$-functions. Then the bracket between them can be expressed as
\begin{equation}\label{eq:brack}
\begin{split}
\{\log \psi_1,\log \psi_2\} = &= -\langle \pi_{<} \eta_L^1, \pi_{>} \eta_L^2 \rangle - \langle \pi_{>} \eta_R^1, \pi_{<} \eta_R^2 \rangle + \\ &+ \langle \gamma_r \xi_R^1, \gamma_r X \nabla_X^2 \rangle + \langle \gamma_c^* \xi_L^1, \gamma_c^* \nabla_Y^2 Y \rangle + D\\
\end{split}
\end{equation}
where $D$ is given by
\begin{equation}\label{eq:diagpart}\begin{split}
D &= -\langle \pi_0 \gamma_c^* \xi_L^1, \gamma_c^*(\nabla^2_Y Y)\rangle - \langle \pi_0 \gamma_r \xi_R^1, \gamma_r(X\nabla^2_X)\rangle + \langle R_0^c \pi_0 E_L^1, E_L^2 \rangle - \langle R_0^r \pi_0 E_R^1, E_R^2 \rangle - \\ &-\langle \pi_0 \nabla^1_X X, E_L^2\rangle + \langle \pi_0 X\nabla^1_X, E_R^2 \rangle.
\end{split}
\end{equation}
\end{lemma}
We refer to $D$ as \emph{the diagonal part} of the bracket. 
\begin{proof}
Recall that the bracket is defined as
\begin{equation}\label{eq:prbrack}
\{ \log \psi_1, \log \psi_2\} = \langle R_+^c(E_L^1), E_L^2 \rangle - \langle R_+^r(E_R^1), E_R^2\rangle + \langle X \nabla^1_X, Y \nabla^1_Y \rangle - \langle \nabla^1_X X, \nabla^2_Y Y \rangle.
\end{equation}
Recall that $E_L^i = \xi_L^i + (1-\gamma_c) (\nabla_X^i X)$ with $\xi_L^i \in \mathfrak{b}_-$, $i \in \{1,2\}$; with that, the first term becomes
\begin{equation}\label{eq:prfrstt}
\begin{split}
\langle R_+^c(E_L^1), E_L^2 \rangle &= \langle \frac{1}{1-\gamma_c} \pi_{>} E_L^1, E_L^2 \rangle - \langle \frac{\gamma_c^*}{1-\gamma_c^*} \pi_{<} E_L^1, E_L^2 \rangle + \langle R_0^c \pi_0 E_L^1, E_L^2 \rangle = \\ &= \langle \pi_{>} \nabla_X^1 X, E_L^2 \rangle - \langle \pi_{<} E_L^1, \gamma_c(\nabla_X^2 X)\rangle + \langle R_0^c \pi_0 E_L^1, E_L^2 \rangle.
\end{split}
\end{equation}
Similarly, recall that $E_R^i = \xi_R^i + (1-\gamma_r^*) (Y \nabla_Y^i)$ and $\xi_R^i \in \mathfrak{b}_+$; using these formulas, we rewrite the second term of the bracket as
\begin{equation}\label{eq:prscndt}
\begin{split}
-\langle R_+^r (E_R^1), E_R^1\rangle &= -\langle \frac{1}{1-\gamma_r}\pi_{>}E_R^1, E_R^2 \rangle + \langle \frac{\gamma_r^*}{1-\gamma_r^*} \pi_{<}E_R^1, E_R^2 \rangle - \langle R_0^r \pi_0 E_R^1, E_R^2 \rangle = \\ &= -\langle \pi_{>} E_R^1, Y\nabla^2_Y\rangle + \langle \pi_{<}\gamma_r^* (Y\nabla^2_Y), E_R^2 \rangle - \langle R_0^r \pi_0 E_R^1, E_R^2 \rangle.
\end{split}
\end{equation}
With equations \eqref{eq:prfrstt} and \eqref{eq:prscndt}, we can rewrite equation \eqref{eq:prbrack} as
\begin{equation}\label{eq:prrewr1}
\begin{split}
\{\log \psi_1,\log \psi_2\} &= \langle \pi_{>} \nabla^1_X X, E_L^2\rangle - \langle \pi_{<} E_L^1, \gamma_c (\nabla^2_X X)\rangle - \langle \nabla^1_X X, \nabla^2_Y Y\rangle - \\ & -\langle \pi_{>} E_R^1, Y\nabla^2_Y \rangle + \langle \pi_{<} \gamma_r^*(Y\nabla^1_Y), E_R^2\rangle + \langle X\nabla^1_X, Y\nabla^2_Y \rangle + \\ &+ \langle R_0^c \pi_0 E_L^1, E_L^2\rangle - \langle R_0^r \pi_0 E_R^1, E_R^2 \rangle.
\end{split}
\end{equation}
Let's deal with the first three terms of equation \eqref{eq:prrewr1}. Rewrite $E_L^1 = \xi_L^1 + (1-\gamma_c)(\nabla_X^1 X)$ and $\pi_{>}\gamma_c(\nabla^2_X X)=-\pi_{>}\nabla^2_Y Y$, and combine the first and the third terms:
\begin{equation}\label{eq:prl}
\begin{split}
&\langle \pi_{>} \nabla^1_X X, E_L^2\rangle - \langle \pi_{<} E_L^1, \gamma_c (\nabla^2_X X)\rangle - \langle \nabla^1_X X, \nabla^2_Y Y\rangle = \\ = &-\langle \pi_{\leq} \nabla^1_X X, \nabla^2_Y Y\rangle + \langle \pi_{>} \nabla^1_X X, \nabla^2_X X\rangle + \langle \pi_{<} \xi_L^1, \nabla_Y^2 Y\rangle + \langle \pi_{<} (1-\gamma_c)(\nabla_X^1 X), \nabla_Y^2 Y\rangle;
\end{split}
\end{equation}
the first and the fourth terms in the latter expression combine into
\[
\begin{split}
-\langle \pi_{\leq} \nabla^1_X X&, \nabla^2_Y Y\rangle + \langle \pi_{<} (1-\gamma_c)(\nabla_X^1 X), \nabla_Y^2 Y\rangle = - \langle \pi_0 \nabla^1_X X, \nabla^2_Y Y\rangle - \langle \pi_{<} \gamma_c \nabla^1_X X, \nabla^2_Y Y\rangle =\\ &= - \langle \pi_0 \nabla^1_X X, \nabla^2_Y Y\rangle - \langle \pi_{<} \nabla^1_X X, \eta_L^2 \rangle + \langle \pi_{<} \nabla^1_X X, \nabla^2_X X\rangle = \\ &= - \langle \pi_0 \nabla^1_X X, \nabla^2_Y Y\rangle - \langle \pi_{<} \eta_L^1, \eta_L^2 \rangle + \langle \pi_{<} \gamma_c^*(\nabla^1_Y Y), \eta_L^2 \rangle + \langle \pi_{<} \nabla^1_X X, \nabla^2_X X\rangle.
\end{split}
\]
Since $\xi_L^2 = \gamma_c(\eta_L^2) + \pi_{\hat{\Gamma}_2^c} (\nabla_Y^2 Y) \in \mathfrak{b}_-$, we see that $\pi_>(\gamma_c(\eta_L^2)) = -\pi_{>}\pi_{\hat{\Gamma}_2^c}(\nabla_Y^2 Y)$, hence the term $\langle \pi_{<} \gamma_c^* (\nabla^1_Y Y), \eta_L^2 \rangle$ can be combined with $\langle \pi_{<} \xi_L^1, \nabla^2_Y Y\rangle$ from equation~\eqref{eq:prl} as
\[\begin{split}
\langle \pi_{<} \xi_L^1, \nabla^2_Y Y\rangle + \langle \pi_{<} \gamma_c^* (\nabla^1_Y Y), \eta_L^2 \rangle &= \langle \pi_{<} \xi_L^1, \nabla^2_Y Y\rangle - \langle \pi_{<} \pi_{\hat{\Gamma}_2^c}(\nabla^1_Y Y), \nabla^2_Y Y \rangle = \langle \pi_{<} \pi_{\Gamma_2^c} \xi_L^1, \nabla_Y^2 Y\rangle = \\ &= \langle \pi_{<} \gamma_c^*(\xi_L^1), \gamma_c^*(\nabla_Y^2 Y)\rangle,
\end{split}
\]
for $\pi_{\hat{\Gamma}_2^c}(\xi_L^1) = \pi_{\hat{\Gamma}_2^c}(\nabla_Y^1 Y)$. Overall, equation \eqref{eq:prl} (which is the first three terms of equation \eqref{eq:prrewr1}) updates to
\[
-\langle \pi_{<} \eta_L^1, \eta_L^2\rangle + \langle \pi_{<}\gamma_c^*(\xi_L^1), \gamma_c^*(\nabla^2_Y Y)\rangle + \langle \nabla^1_X X, \nabla^2_X X\rangle - \langle\pi_0\nabla^1_X X, \nabla^2_X X\rangle - \langle \pi_0 \nabla^1_X X, \nabla^2_Y Y\rangle.
\]
Next, let's study the contribution of the fourth, fifth and the sixth terms in equation \eqref{eq:prrewr1} together with $\langle \nabla^1_X X, \nabla^2_X X\rangle = \langle X \nabla^1_X, X \nabla^2_X \rangle$. First, rewrite $-\langle \pi_{>} E_R^1, Y\nabla^2_Y\rangle$ as
\begin{equation}\label{eq:eryny}
\begin{split}
&-\langle \pi_{>} E_R^1, Y\nabla^2_Y\rangle = -\langle \pi_{>} \eta_R^1, Y\nabla^2_Y \rangle - \langle \pi_{>} (1-\gamma_r) (X\nabla^1_X), Y\nabla^2_Y\rangle = \\ &= -\langle \pi_{>} \eta_R^1, \eta_R^2\rangle + \langle \pi_{>} \eta_R^1, \gamma_r (X\nabla^2_X) \rangle - \langle \pi_{>} X\nabla^1_X, Y \nabla^2_Y\rangle + \langle \pi_{>} \gamma_r (X \nabla^1_X), Y\nabla^2_Y\rangle;
\end{split}
\end{equation}
Since $\gamma_r^*(\eta_R^1) = \pi_{\Gamma_1^r} \xi_R^1$, we see that $\langle \pi_{>} \eta_R^1, \gamma_r(X \nabla^2_X)\rangle = \langle \pi_{>} \pi_{\Gamma_1^r} \xi_R^1, X\nabla^2_X\rangle$. The last two terms in equation \eqref{eq:eryny} together with $\langle \nabla^1_X X, \nabla^2_X X\rangle$ and the fifth and the sixth terms in equation \eqref{eq:prrewr1} contribute
\[\begin{split}
&- \langle \pi_{>} X\nabla^1_X, Y \nabla^2_Y\rangle + \langle \pi_{>} \gamma_r (X \nabla^1_X), Y\nabla^2_Y\rangle + \langle X \nabla^1_X, X \nabla^2_X \rangle + \langle \pi_{<} \gamma_r^*(Y\nabla^1_Y), E_R^2\rangle + \langle X\nabla^1_X, Y\nabla^2_Y \rangle = \\ &= \langle \pi_{\leq} X\nabla^1_X, Y\nabla^2_Y\rangle - \langle \pi_{>} X\nabla^1_X, X\nabla^2_X \rangle + \langle X\nabla^1_X,X\nabla^2_X \rangle - \langle \pi_{<}X\nabla^1_X, E_R^2\rangle = \\ &= \langle \pi_0 X\nabla^1_X,Y\nabla^2_Y\rangle +\langle \pi_0 X\nabla^1_X,X\nabla^2_X \rangle.
\end{split}
\]
Combining everything together, we obtain the formula.
\end{proof}

\begin{lemma}
If $(R_0^r,R_0^c)$ are chosen so that the identities~\eqref{eq:roid} hold, then the diagonal part $D$ from equation \eqref{eq:diagpart} can be further expanded as
\begin{equation}\label{eq:diagpart2}\begin{split}
D &= -\langle \pi_0 \gamma_c^* \xi_L^1, \gamma_c^*(\nabla^2_Y Y)\rangle - \langle \pi_0 \gamma_r \xi_R^1, \gamma_r(X\nabla^2_X)\rangle + \langle R_0^c \pi_0 \xi_L^1, E_L^2 \rangle - \langle \pi_{\hat{\Gamma}_1^c} \pi_0 \nabla^1_X X, E_L^2 \rangle + \\ &+ \langle R_0^c \pi_0 \pi_{\hat{\Gamma}_1^c} \nabla^1_X X, E_L^2 \rangle - \langle R_0^r \pi_0 \eta_R^1, E_R^2 \rangle +  \langle \pi_0 \pi_{\hat{\Gamma}_1^r} X \nabla^1_X, E_R^2 \rangle - \langle R_0^r \pi_0 \pi_{\hat{\Gamma}_1^r} X \nabla^1_X, E_R^2 \rangle.
\end{split}
\end{equation}
\end{lemma}
\begin{proof}
Observe that 
\[
R_0^c \pi_0 E_L^1 = R_0^c \pi_0 \xi_L^1 + R_0^c \pi_0 (1-\gamma_c) (\nabla^1_X X) = R_0^c \pi_0 \xi_L^1 + \pi_0 \pi_{\Gamma_1^c} \nabla^1_X X + R_0^c \pi_0 \pi_{\hat{\Gamma}_1^c} \nabla^1_X X.
\]
Therefore, the corresponding terms together with $-\langle \pi_0 \nabla^1_X X, E_L^2\rangle$ contribute
\[
\langle R_0^c \pi_0 E_L^1, E_L^2 \rangle - \langle \pi_0 \nabla^1_X X, E_L^2 \rangle = \langle R_0^c \pi_0 \xi_L^1, E_L^2 \rangle - \langle \pi_{\hat{\Gamma}_1^c} \pi_0 \nabla^1_X X, E_L^2 \rangle + \langle R_0^c \pi_0 \pi_{\hat{\Gamma}_1^c} \nabla^1_X X, E_L^2 \rangle.
\]
Similarly,
\[
R_0^r \pi_0 E_R^1 = R_0^r \pi_0 \eta_R^1 + \pi_0 \pi_{\Gamma_1^r} X \nabla_X^1 + R_0^r \pi_0 \pi_{\hat{\Gamma}_1^r} X\nabla_X^1,
\]
hence
\[
-\langle R_0^r \pi_0 E_R^1, E_R^2\rangle + \langle \pi_0 X\nabla^1_X, E_R^2 \rangle = -\langle R_0^r \pi_0 \eta_R^1, E_R^2 \rangle + \langle \pi_0 \pi_{\hat{\Gamma}_1^r} X \nabla^1_X, E_R^2 \rangle - \langle R_0^r \pi_0 \pi_{\hat{\Gamma}_1^r} X \nabla^1_X, E_R^2 \rangle.
\]
Now the result is obtained via combining the two formulas.
\end{proof}

\subsection{Block formulas}\label{s:blockf}
In this subsection, we state a further expansion from \cite{plethora} of the first four terms in \mbox{formula \eqref{eq:brack}.}

\paragraph{Block intervals.} Let $\mathcal{L}$ be an $\mathcal{L}$-matrix. We enumerate the blocks of $\mathcal{L}$ in such a way that blocks $X_t$ and $Y_t$ are aligned along their rows (i.e., using $\gamma_r$) and $Y_t$ and $X_{t+1}$ are aligned along their columns, $t \geq 1$. Let us denote by $K_t$ and $L_t$, respectively, the row and column indices in $\mathcal{L}$ that are occupied by $X_t$; similarly, $\bar{K}_t$ and $\bar{L}_t$ are row and column indices occupied by $Y_t$ in $\mathcal{L}$. Furthermore, we set $\Phi_t := K_t \cap \bar{K}_t$ and $\Psi_t := L_t \cap \bar{L}_{t-1}$. Figure \ref{f:lstruct} depicts an $\mathcal{L}$-matrix with the intervals.

\begin{figure}
\begin{center}
\includegraphics[scale=0.4]{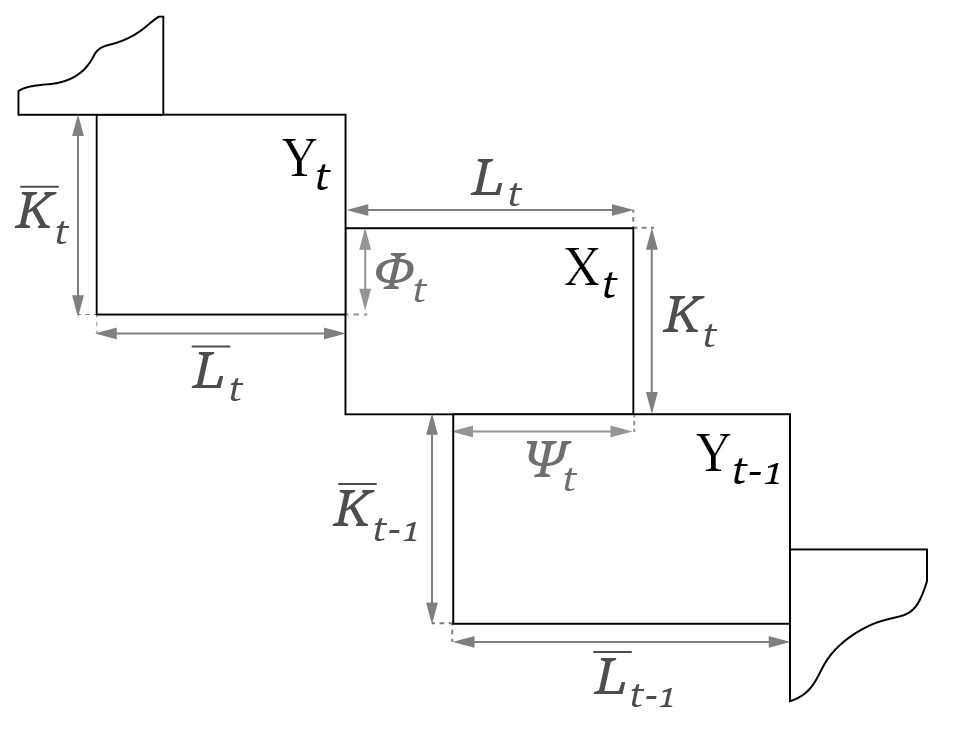}
\end{center}
\caption{An illustration of the intervals $\bar{K}_t$, $\bar{L}_t$, $L_t$, $K_t$, $\Phi_t$, $\Psi_t$.}
\label{f:lstruct}
\end{figure}
\begin{remark}
The authors of \cite{plethora} additionally define empty blocks in the beginning and in the end of the sequence so that it always starts with an $X$-block and ends with a $Y$-block; furthermore, they attach row or column intervals to the empty blocks depending on a set of conditions. The convention with empty blocks is rather complicated and only helps to avoid one term in two formulas. Ergo, we write all formulas without assuming that there are extra empty blocks.
\end{remark}

\paragraph{$\mathcal{L}$-gradients.} Let $\psi$ be any $g$- or $h$-function, and let $\mathcal{L}$ be an $\mathcal{L}$-matrix such that $\psi = \det \mathcal{L}^{[s,N(\mathcal{L})]}_{[s,N(\mathcal{L})]}$ (in the case of $\psi = g_{ii}$ or $\psi = h_{ii}$, set $\mathcal{L}(X,Y):=X$ or $\mathcal{L}(X,Y): = Y$ accordingly). Notice that $(j,i)$ entry of $\nabla_X\psi$ is the sum of the cofactors computed at all occurrences of $x_{ij}$ in the matrix $\mathcal{L}^{[s,N(\mathcal{L})]}_{[s,N(\mathcal{L})]}$ (and similarly for $(j,i)$ entry of $\nabla_Y\psi$ and $y_{ij}$). We define an $N(\mathcal{L}) \times N(\mathcal{L})$ matrix $\nabla_{\mathcal{L}}\psi$ that has as its $(j,i)$ entry the $(i-s+1,j-s+1)$-cofactor of $\mathcal{L}^{[s,N(\mathcal{L})]}_{[s,N(\mathcal{L})]}$, where $i, j \geq s$, and zero everywhere else. Consequently, if $m$ is the last index in the block sequence in $\mathcal{L}$, we have
\[
\nabla_{X} \psi =  \sum_{t=1}^{m} (\nabla_{\mathcal{L}} \psi)_{L_t \rightarrow J_t}^{K_t \rightarrow I_t}, \ \ \ \nabla_Y \psi = \sum_{t=1}^{m} (\nabla_{\mathcal{L}} \psi)_{\bar{L}_t \rightarrow \bar{J}_t}^{\bar{K}_t \rightarrow \bar{I}_t},
\]
where $Y_t = Y^{\bar{J}_t}_{\bar{I}_t}$ and $X_t = X^{J_t}_{I_t}$. Evidently, by $\nabla_{\mathcal{L}}\log \psi$ we mean $(1/\psi) \nabla_{\mathcal{L}}\psi$. Let us mention the following simple formulas:
\[
\mathcal{L}\nabla_{\mathcal{L}}\log\psi = \begin{bmatrix}
0 & * \\
0 & I
\end{bmatrix}, \ \ \ \nabla_{\mathcal{L}}\log\psi \cdot \mathcal{L} = \begin{bmatrix}
0 & 0 \\
* & I
\end{bmatrix},
\]
where $I$ is the identity matrix that occupies $[s,N(\mathcal{L})]\times [s,N(\mathcal{L})]$ and $*$ indicates terms whose particular expressions are of no importance in the proofs.

\paragraph{Numbers $p$ and $q$.} Let $\psi(X,Y):= \det \mathcal{L}^{[s,N(\mathcal{L})]}_{[s,N(\mathcal{L})]}(X,Y)$ for some $s$ and $\mathcal{L}$. Let us call the \emph{leading block} of $\psi$ the $X$- or $Y$-block of $\mathcal{L}(X,Y)$ that contains the entry $(s,s)$. The number $q$ is defined as the index of the leading block. Furthermore, if the block is of type $X$, we set $p:=q$; if it's of type $Y$, we set $p:=q+1$.

\paragraph{Embeddings $\rho$ and $\sigma$.} Let us pick a pair of $g$- or $h$-functions $\psi_1$ and $\psi_2$ and let us mark all the data associated with either of the functions with an upper index $1$ or $2$. Pick the number $p$ associated with $\psi_1$ as in the previous paragraph and let $X_t^2 = X_{I_t^2}^{J_t^2}$ be an $X$-block of $\mathcal{L}^2$. If $I_t^2 \subseteq I_p^1$, define $\rho(K_t^2)$ to be a subset of $K_p^1$ that corresponds to $I_t^2$ viewed as a subset of $I_p^1$; similarly, if $J_t^2 \subseteq J_p^1$, we define $\rho(L_t^2)$ to be a subset of $L_p^1$ that occupies the column indices $I_t^2$ in $X_p^1$. Likewise, fix $Y_u^1 = Y_{\bar{I}_u^1}^{\bar{J}_u^1}$ in $\mathcal{L}^1$ and define an embedding $\sigma_u$ for $Y$-blocks as follows. If $\bar{J}_t^2 \subseteq \bar{J}_u^1$, then $\sigma_u(\bar{L}_t^2)$ is a subset of $\bar{L}_u^1$ that corresponds to $\bar{J}_t^2$; similarly, if $\bar{I}_t^2 \subseteq \bar{I}_u^1$, then $\sigma_u(\bar{K}_t^2)$ is a subset of $\bar{K}_u^1$ that occupies the indices $\bar{I}_t^2$ in $Y_u^1$. Note: The map $\rho$ always embeds into rows or columns of $X_p^1$ that is viewed as a submatrix of $\mathcal{L}^1$, whereas the targeting block for $\sigma$ might vary depending on its subscript $u$.

\paragraph{More on subblocks.} Recall that $X$- and $Y$-blocks have the form $X^{[1,\beta]}_{[\alpha,n]}$ and $Y^{[\bar{\beta},n]}_{[1,\bar{\alpha}]}$, where $\alpha,\beta,\bar{\alpha},\bar{\beta}$ are defined in Section \ref{s:lmatr}. For two matrices $A_1$ and $A_2$, let us write $A_1 \subseteq A_2$ if $A_1$ is a submatrix of $A_2$. Let us recall Proposition 4.3 from \cite{plethora}:
\begin{proposition}
Let $X_1$, $X_2$, $Y_1$ and $Y_2$ be arbitrary $X$- and $Y$-blocks, with $\alpha$'s and $\beta$'s indexed accordingly. Then the following holds:
\begin{enumerate}[(i)]
\item If $\beta_2 < \beta_1$ or $\alpha_2 > \alpha_1$, then $X_2 \subseteq X_1$;
\item If $\bar{\beta}_2 > \bar{\beta}_1$ or $\bar{\alpha}_2 < \bar{\alpha}_1$, then $Y_2 \subseteq Y_1$.
\end{enumerate}
\end{proposition}
\noindent Notice that the proposition in particular states that if $\beta_2 < \beta_1$, then necessarily $\alpha_2 \geq \alpha_1$, and likewise in all other instances. We implicitly refer to this fact in the proofs that follow.

\paragraph{The formulas.} Let us pick any $g$- or $h$-functions $\psi_1$ and $\psi_2$, and let $p$ and $q$ be the numbers defined above for $\psi_1$, and let $q^\prime$ be the index of the leading block of $\psi_2$. Following \cite{plethora}, define \emph{$B$-terms:}
\begin{align*}
B_t^{\ronum{1}} &:= -\langle \lgrl{1}_{\rho(\Phi_t^2)}^{\rho(\Phi_t^2)} \lm{2}_{\Phi_t^2}^{\bar{L}_t^2} \grl{2}_{\bar{L}_t^2}^{\Phi_t^2}\rangle, & B_t^{\ronum{2}} &:= \langle \grll{1}_{\rho(\Psi_t^2)}^{\rho(\Psi_t^2)}\grl{2}_{\Psi_t^2}^{\bar{K}_{t-1}^2}\lm{2}_{\bar{K}_{t-1}^2}^{\Psi_t^2}\rangle,\\
B_t^{\ronum{3}} &:= \langle \grll{1}_{\Psi_p^1}^{L_p^1\setminus \Psi_p^1} \grl{2}_{L_t^2\setminus \Psi_t^2}^{K_t^2} \lm{2}_{K_t^2}^{\Psi_t^2}\rangle, & B_t^{\ronum{4}} &:= \langle \lgrl{1}_{\rho(\Phi_t^2)}^{\rho(\Phi_t^2)} \lm{2}_{\Phi_t^2}^{L_t^2}\grl{2}_{L_t^2}^{\Phi_t^2}\rangle\\
\bar{B}_t^{\ronum{1}}(u) &:= -\langle \grll{1}_{\sigma_u(\Psi_{t+1}^2)}^{\sigma_u(\Psi_{t+1})} \grl{2}_{\Psi_{t+1}^2}^{K_{t+1}^2} \lm{2}_{K_{t+1}^2}^{\Psi_{t+1}^2}\rangle, & \bar{B}_t^{\ronum{2}}(u) &:= \langle \lgrl{1}_{\sigma_u(\Phi_t^2)}^{\sigma_u(\Phi_t^2)} \lm{2}_{\Phi_t^2}^{L_t^2} \grl{2}_{L_t^2}^{\Phi_t^2}\rangle\\
\bar{B}_t^{\ronum{3}} &:= \langle \lgrl{1}_{\bar{K}_q^1\setminus \Phi_q^1}^{\Phi_q^1} \lm{2}_{\Phi_t^2}^{\bar{L}_t^2} \grl{2}_{\bar{L}_t^2}^{\bar{K}_t^2\setminus \Phi_t^2}\rangle, & \bar{B}_t^{\ronum{4}}(u) &:= \langle \grll{1}_{\sigma_u(\Psi_{t+1}^2)}^{\sigma_u(\Psi_{t+1}^2)} \grl{2}_{\Psi_{t+1}^2}^{\bar{K}_t^2} \grl{2}_{\bar{K}_t^2}^{\Psi_{t+1}^2}\rangle
\end{align*}
Now, the formulas for the first four terms of equation \eqref{eq:brack} are:
\begin{equation}\label{eq:fetal}\begin{split}
\langle \pi_{<}\eta_L^1&,\pi_{>}\eta_L^2\rangle = \sum_{\beta_t^2 < \beta_p^1} (B_t^{\ronum{1}}+B_t^{\ronum{2}}) + \sum_{\beta_t^2=\beta_p^1}B_t^{\ronum{3}} +\\ &+\sum_{\beta_t^2 < \beta_p^1} \left( \left\langle \lgrl{1}_{\rho(K_t^2)}^{\rho(K_t^2)}\lgrl{2}_{K_t^2}^{K_t^2}\right\rangle-\left\langle \grll{1}_{\rho(L_t^2)}^{\rho(L_t^2)}\grll{2}_{L_t^2}^{L_t^2}\right\rangle\right);
\end{split}
\end{equation}
\begin{equation}\label{eq:fetar}\begin{split}
\langle \pi_{>}\eta_R^1&,\pi_{<}\eta_R^2 \rangle = \sum_{\bar{\alpha}_t^2 < \bar{\alpha}_q^1} (\bar{B}_t^{\ronum{1}}(q) + \bar{B}_t^{\ronum{2}}(q)) + \sum_{\bar{\alpha}_t^2 = \bar{\alpha}^1_q} \bar{B}_t^{\ronum{3}} + \\ &+  \sum_{\bar{\alpha}_t^2 < \bar{\alpha}^1_q} \left( \left\langle \grll{1}_{\sigma_q(\bar{L}_t^2}^{\sigma_q(\bar{L}_t^2)}\grll{2}_{\bar{L}_t^2}^{\bar{L}_t^2}\right\rangle - \left\langle \lgrl{1}_{\sigma_q(\bar{K}_t^2)}^{\sigma_q(\bar{K}_t^2)} \lgrl{2}_{\bar{K}_t^2}^{\bar{K}_t^2}\right\rangle \right);
\end{split}
\end{equation}
\begin{equation}\label{eq:fxil}
\begin{split}
\langle \gamma_c^* (\xi_L^1)&, \gamma_c^*(\nabla_Y^2Y)\rangle = \sum_{\beta_t^2 \leq \beta_p^1} B_t^{\ronum{2}} + \sum_{\bar{\beta}_t^2 > \bar{\beta}_{p-1}^1} \bar{B}_t^{\ronum{4}}(p-1) + (\Psi_p^1 = \emptyset) \sum_{\bar{\beta}_{p-1} = \bar{\beta}_t} \bar{B}_t^{\ronum{4}}(p-1) + \\
& +\sum_{u=1}^{p} \sum_{t=1}^{q^\prime} \langle (\nabla^1_{\mathcal{L}} \mathcal{L}^1)^{L_u^1 \rightarrow J_u^1}_{L_u^1 \rightarrow J_u^1}, \gamma_c^* (\nabla^2_{\mathcal{L}}\mathcal{L}^2)^{\bar{L}^2_t\setminus \Psi_{t+1} \rightarrow \bar{J}^2_t \setminus \bar{\Delta}(\bar{\beta}_t)}_{\bar{L}^2_t\setminus \Psi_{t+1} \rightarrow \bar{J}^2_t \setminus \bar{\Delta}(\bar{\beta}_t)} \rangle + \\
& + \sum_{u=1}^{p-1} \sum_{t=1}^{q^\prime} \langle (\nabla^1_{\mathcal{L}} \mathcal{L}^1)^{\bar{L}^1_u\setminus \Psi_{u+1} \rightarrow \bar{J}^1_u \setminus \bar{\Delta}(\bar{\beta}_u)}_{\bar{L}^1_u\setminus \Psi_{u+1} \rightarrow \bar{J}^1_u \setminus \bar{\Delta}(\bar{\beta}_u)}, \pi_{\Gamma_2^c} (\nabla^2_{\mathcal{L}}\mathcal{L}^2)^{\bar{L}^2_t\setminus \Psi_{t+1} \rightarrow \bar{J}^2_t \setminus \bar{\Delta}(\bar{\beta}_t)}_{\bar{L}^2_t\setminus \Psi_{t+1} \rightarrow \bar{J}^2_t \setminus \bar{\Delta}(\bar{\beta}_t)}\rangle + \\ 
& + \sum_{t=1}^{q^\prime} (| \{ u < p \ | \ \beta_u^1 \geq \beta_{t+1}^2 \}| + |\{u < p-1 \ | \ \bar{\beta}_u < \bar{\beta}_{t} \}|) \langle (\nabla^2_{\mathcal{L}})^{\bar{K}_t^2}_{\Psi_{t+1}^2} (\mathcal{L}^2)_{\bar{K}_t^2}^{\Psi_{t+1}^2}\rangle;
\end{split}
\end{equation}
\begin{equation}\label{eq:fxir}
\begin{split}
\langle \gamma_r (\xi_R^1)&, \gamma_r(X\nabla_X^2)\rangle = \sum_{\bar{\alpha}_t^2 \leq \bar{\alpha}^1_{p-1}} \bar{B}_t^{\ronum{2}}(p-1) + \sum_{\bar{\alpha}_t^2 \leq \bar{\alpha}_p^1} \bar{B}_t^{\ronum{2}}(p) + \sum_{\alpha_t^2 > \alpha_p^1} B_t^{\ronum{4}}+\\ &+ (\Phi_p^1 = \emptyset) \sum_{\alpha_t^2 = \alpha_p^1} B_t^{\ronum{4}} + \sum_{u=1}^p \sum_{t=1}^{q^\prime} \langle (\mathcal L \nabla^1_{\mathcal{L}})^{\bar{K}_u^1 \rightarrow \bar{I}_u^1}_{\bar{K}_u^1 \rightarrow \bar{I}_u^1}, \gamma_r(\mathcal{L}^2 \nabla^2_{\mathcal{L}})^{K_t^2 \setminus \Phi_t^2 \rightarrow I_t^2 \setminus \Delta(\alpha_t^2)}_{K_t^2 \setminus \Phi_t^2 \rightarrow I_t^2 \setminus \Delta(\alpha_t^2)} \rangle +\\
&+ \sum_{u=1}^p \sum_{t=1}^{q^\prime} \langle (\mathcal{L}^1 \nabla^1_{\mathcal{L}})^{K_u^1 \setminus \Phi_u^1 \rightarrow I_u^1 \setminus \Delta(\alpha_u^1)}_{K_u^1 \setminus \Phi_u^1 \rightarrow I_u^1 \setminus \Delta(\alpha_u^1)}, \pi_{\Gamma_1^r} (\mathcal{L}^2 \nabla^2_{\mathcal{L}})_{K_t^2\setminus \Phi_t^2 \rightarrow I_t^2 \setminus \Delta(\alpha_t^2)}^{K_t^2\setminus \Phi_t^2 \rightarrow I_t^2 \setminus \Delta(\alpha_t^2)} \rangle+ \\
 &+ \sum_{t=1}^{q^\prime} (|\{u < p-1 \ | \ \bar{\alpha}_u^1 \geq \bar{\alpha}_t^2\}| + |\{u < p \ | \ \alpha_u^1 < \alpha_t^2\}|)\langle (\mathcal{L}^2)_{\Phi_t^2}^{L_t^2} (\nabla^2_{\mathcal{L}})_{L_t^2}^{\Phi_t^2}\rangle.
\end{split}
\end{equation}
By $(\Psi_p^1 = \emptyset)$ and $(\Phi_p^1 = \emptyset)$ we mean an indicator that's equal to $1$ if the condition is satisfied and $0$ otherwise. It follows from the construction that $(\Psi_p^1 = \emptyset) = 1$ if and only if $Y_p^1$ is the last block in the alternating path that defines $\mathcal{L}^1$; similarly, $(\Phi_p^1 = \emptyset) = 1$ if and only if $X_p^1$ is the last block in the path (hence it sits in the upper left corner of $\mathcal{L}^1$, for blocks along the path are glued in $\mathcal{L}$ from bottom up). The terms with indicators are not present in the empty block convention from \cite{plethora}, for their contribution is accounted for in other terms. 

Lastly, let us mention the total contribution of $B$-terms to equation \eqref{eq:brack}. If $\psi_1$ is an $h$-function, then the total contribution is
\begin{equation}\label{eq:cbh}\begin{split}
&\sum_{\substack{\bar{\alpha}^2_{t-1} < \bar{\alpha}_{p-1}^1 \\ \bar{\beta}_{t-1}^2 > \bar{\beta}_{p-1}^1}}\left\langle \grll{1}_{\sigma_{p-1}(\Psi_t^2)}^{\sigma_{p-1}(\Psi_t^2)} \grll{2}_{\Psi_t^2}^{\Psi_t^2}\right\rangle + \sum_{\substack{\bar{\alpha}_{t-1}^2 \neq \bar{\alpha}_{p-1}^1\\\bar{\beta}_{t-1}^2 = \bar{\beta}_{p-1}^1}} \left\langle \lgrl{1}_{\Psi_p^1}^{\Psi_p^1} \grll{2}_{\Psi_t^2}^{\Psi_t^2}\right\rangle + \\ +&\sum_{\substack{\bar{\alpha}_{t-1}^2 = \bar{\alpha}_{p-1}^1\\ \bar{\beta}_{t-1}^2 < \bar{\beta}_{p-1}^1}} \left\langle \lm{2}_{\bar{K}_{t-1}^2}^{L^2_{t-1}} \grl{2}_{L_{t-1}^2}^{\bar{K}_{t-1}^2}\right\rangle + \sum_{\substack{\bar{\alpha}_{t-1}^2 = \bar{\alpha}_{p-1}^1\\ \bar{\beta}_{t-1}^2 \geq \bar{\beta}_{p-1}^1 }} \left\langle \lgrl{1}_{\bar{K}_{p-1}^1}^{\bar{K}_{p-1}^1} \lgrl{2}_{\bar{K}_{t-1}^2}^{\bar{K}_{t-1}^2}\right\rangle - \\ - & \sum_{\substack{\bar{\alpha}_{t-1}^2 = \bar{\alpha}_{p-1}^1\\ \bar{\beta}_{t-1}^2 \geq \bar{\beta}_{p-1}^1 }} \left\langle \grll{1}_{\sigma_{p-1}(\bar{L}_{t-1}^2 \setminus \Psi_t^2)}^{\sigma_{p-1}(\bar{L}_{t-1}^2 \setminus \Psi_t^2)} \grll{2}_{\bar{L}_{t-1}^2 \setminus \Psi_t^2}^{\bar{L}_{t-1}^2 \setminus \Psi_t^2} \right\rangle + \sideset{}{^l}\sum_{\substack{\bar{\alpha}_{t-1}^2 = \bar{\alpha}_{p-1}^1\\\bar{\beta}_{t-1}^2 = \bar{\beta}_{p-1}^1 }} \left\langle \lm{2}_{\Phi^2_{t-1}}^{L_{t-1}^2}\grl{2}_{L_{t-1}^2}^{\Phi_{t-1}^2}\right\rangle + \\ + & \sideset{}{^l}\sum_{\substack{\bar{\alpha}_{t-1}^2 = \bar{\alpha}_{p-1}^1\\ \bar{\beta}_{t-1}^2 = \bar{\beta}_{p-1}^1 }}\left\langle \grll{1}_{\bar{L}_{p-1}^1}^{\bar{L}_{p-1}^1} \grll{2}_{\bar{L}_{t-1}^2}^{\bar{L}_{t-1}^2} \right\rangle - \sideset{}{^l}\sum_{\substack{\bar{\alpha}_{t-1}^2 = \bar{\alpha}_{p-1}^1\\ \bar{\beta}_{t-1}^2 = \bar{\beta}_{p-1}^1}} \left\langle \lgrl{1}_{\bar{K}_{p-1}^1}^{\bar{K}_{p-1}^1} \lgrl{2}_{\bar{K}_{t-1}^2}^{\bar{K}_{t-1}^2} \right\rangle,
\end{split}
\end{equation}
where $\sum^l$ means a summation over blocks the $Y_{t-1}^2$ that have their exit point strictly to the left of the exit point of $Y_{p-1^1}$ (for the definition of exit points, see Section \ref{s:lmatr}). If $\psi_1$ is a $g$-function, then the contribution is
\begin{equation}\label{eq:cbg}
\begin{split}
&\sum_{\substack{\beta_t^2<\beta_p^1\\\alpha_t^2>\alpha_p^1}}\left\langle \lgrl{1}_{\rho(\Phi_t^2)}^{\rho(\Phi_t^2)} \lgrl{2}_{\Phi_t^2}^{\Phi_t^2} \right\rangle + \sum_{\substack{\beta_t^2 \neq \beta_p^1\\ \alpha_t^2 = \alpha_p^1}}\left\langle \lgrl{1}_{\Phi_p^1}^{\Phi_p^1} \lgrl{2}_{\Phi_t^2}^{\Phi_t^2}\right\rangle + \\ + &\sum_{\substack{\beta_t^2 = \beta_p^1\\ \alpha_t^2 < \alpha_p^1}} \left\langle \lm{2}_{\bar{K}_{t-1}^2}^{L_t^2} \grl{2}_{L_t^2}^{\bar{K}_{t-1}^2}\right\rangle + \sum_{\substack{\beta_t^2 = \beta_p^1 \\ \alpha_t^2 \geq \alpha_p^1}} \left\langle \grll{1}_{L_p^1}^{L_p^1} \grll{2}_{L_t^2}^{L_t^2} \right\rangle - \\ -&\sum_{\substack{\beta_t^2=\beta_p^1 \\ \alpha_t^2 \geq \alpha_p^1}}\left\langle \lgrl{1}_{\rho(K_t^2\setminus \Phi_t^2)}^{\rho(K_t^2\setminus \Phi_t^2)} \lgrl{2}_{K_t^2 \setminus \Phi_t^2}^{K_t^2 \setminus \Phi_t^2} \right\rangle +\sideset{}{^a}\sum_{\substack{\beta_t^2 = \beta_p^1\\\alpha_t^2=\alpha_p^1}}\mathop{} \left\langle \lm{2}_{\bar{K}_{t-1}^2}^{\Psi_t^2} \grl{2}_{\Psi_t^2}^{\bar{K}_{t-1}^2}\right\rangle + \\ &+ \sideset{}{^a}\sum_{\substack{\beta_t^2 = \beta_p^1\\ \alpha_t^2 = \alpha_p^1}} \left\langle \lgrl{1}_{K_p^1}^{K_p^1} \lgrl{2}_{K_t^2}^{K_t^2}\right\rangle - \sideset{}{^a} \sum_{\substack{\beta_t^2 = \beta_p^1\\ \alpha_t^2 = \alpha_p^1}}\left\langle \grll{1}_{L_p^1}^{L_p^1} \grll{2}_{L_t^2}^{L_t^2}\right\rangle + \\ &+ \sum_{\bar{\beta}_t^2 > \bar{\beta}_{p-1}^1} \left\langle \lm{2}_{\bar{K}_t^2}^{\Psi_{t+1}^2} \grl{2}_{\Psi_{t+1}^2}^{\bar{K}_t^2}\right\rangle + \sum_{\bar{\alpha}_{t}^2 \leq \bar{\alpha}_{p-1}^1} \left\langle \lm{2}_{\Phi_t^2}^{L_t^2} \grl{2}_{L_t^2}^{\Phi_t^2}\right\rangle,
\end{split}
\end{equation}
where $\sum^a$ means that the summation is taken over blocks the $X_t^2$ that have their exit point strictly above the exit point of $X_p^1$.
\subsection{Computation of $\{y(h_{ii}), \psi\}$}
Let $\psi$ be an arbitrary $g$- or $h$-function and let $h_{ii}$ be fixed, $2 \leq i \leq n$. For the shorthand notation from Section~\ref{s:bragh}, the first function in this section is $\log h_{i-1,i} - \log h_{i,i+1}$ and the second function is $\log \psi$, meaning that if an operator has an upper index $1$ or $2$, it's applied to the first or the second function, respectively. We also assume throughout the subsection that a pair $(R_0^r,R_0^c)$ is chosen so that the identities from~\eqref{eq:roid} hold.

\begin{proposition}
The bracket of $y(h_{ii})$ and $\psi$ can be expressed as
\begin{equation}\label{eq:yhpsi}\begin{split}
\{\log y(h_{ii}),\log \psi\} &= -\langle \pi_{<} \eta_L^1, \pi_{>} \eta_L^2 \rangle - \langle \pi_{>} \eta_R^1, \pi_{<} \eta_R^2 \rangle + \\ &+ \langle \gamma_r \xi_R^1, \gamma_r X \nabla_X^2 \rangle + \langle \gamma_c^* \xi_L^1, \gamma_c^* \nabla_Y^2 Y \rangle + \\ &+ \langle \pi_{\hat{\Gamma}_2^c} e_{ii}, \pi_{\hat{\Gamma}_2^c} \nabla_Y^2 Y\rangle - \langle e_{i-1,i-1}, \eta_R^2 \rangle.
\end{split}
\end{equation}
\end{proposition}
\begin{proof}
Recall that the $y$-coordinate of $h_{ii}$ is given by 
\[
y(h_{ii}) = \frac{h_{i-1,i} f_{1,n-i}}{h_{i,i+1} f_{1,n-i+1}}.
\]
Set $f:= f_{1,n-i}/f_{1,n-i+1}$. Using the diagonal derivatives formulas for $f$ from Section~\ref{s:diagders} and formula~\eqref{eq:phipsi}, we can express the bracket $\{\log f, \log \psi\}$ as
\[
\{\log f, \log \psi\} = \langle e_{ii}, \nabla^2_Y Y\rangle - \langle e_{i-1,i-1}, Y\nabla^2_Y \rangle - \langle R_0^c e_{ii}, E_L^2\rangle + \langle R_0^r e_{i-1,i-1}, E_R^2 \rangle.
\]
Combining the latter formula with the expression for $D$ from equation~\eqref{eq:diagpart}, $D+ \{\log f, \log \psi\}$ becomes
\[\begin{split}
-&\langle \pi_{\Gamma_2^c} e_{ii}, \nabla^2_Y Y\rangle - \langle e_{i-1,i-1}, \gamma_r(X\nabla^2_X) \rangle + \langle R_0^c e_{ii}, E_L^2 \rangle - \langle R_0^r e_{i-1,i-1}, E_R^2 \rangle+\\ + &\langle e_{ii}, \nabla^2_Y Y\rangle - \langle e_{i-1,i-1}, Y\nabla^2_Y \rangle - \langle R_0^c e_{ii}, E_L^2\rangle + \langle R_0^r e_{i-1,i-1}, E_R^2 \rangle = \\ &= \langle \pi_{\hat{\Gamma}_2^c} e_{ii}, \nabla^2_Y Y\rangle - \langle e_{i-1,i-1}, \eta_R^2 \rangle.
\end{split}
\]
Now, applying equation~\eqref{eq:brack} to $\{\log h_{i-1,i} - \log h_{i,i+1}, \log \psi\}$ the formula follows.
\end{proof}
\begin{corp}\label{c:yhh}
As a consequence, $\{\log y(h_{ii}), \log h_{jj}\} = \delta_{ij}$ for any $j$.
\end{corp}
\begin{proof}
The first two terms of equation~\eqref{eq:yhpsi} vanish, for $Y\nabla_Y \log h_{jj} \in \mathfrak{b}_+$ and $\nabla_Y \log h_{jj} \cdot Y \in \mathfrak{b}_-$; since $h_{jj}$ doesn't depend on $X$, the third term vanishes as well. Now, recall from Section~\ref{s:diagders} that \[\pi_0( \nabla_Y \log h_{jj} \cdot Y) = \pi_0 (Y \nabla_Y \log h_{jj}) = \Delta(j,n)\] and $\xi_L (\log h_{i-1,i} -\log h_{i,i+1}) = e_{ii}$, where $\Delta(j,n) = \sum_{k=j}^{n} e_{kk}$. Therefore
\[\begin{split}
\{\log y(h_{ii}),\log \psi\} &= \langle \gamma_c^* e_{ii}, \gamma_c^*\Delta(j,n) \rangle + \langle \pi_{\hat{\Gamma}^c_2}e_{ii}, \pi_{\hat{\Gamma}^c_2} \Delta(j,n) \rangle - \langle e_{i-1,i-1}, \Delta(j,n) \rangle =\\ &= \langle e_{ii} - e_{i-1,i-1}, \Delta(j,n) \rangle = \delta_{ij}. \qedhere
\end{split} 
\]
\end{proof}

\begin{lemma}\label{eq:h2last}
The following formulas for the last two terms of equation \eqref{eq:yhpsi} hold:
\[
\langle \pi_{\hat{\Gamma}_2^c} e_{ii}, \pi_{\hat{\Gamma}_2^c} \nabla_Y^2 Y\rangle = \sum_{t=1}^{q^\prime} \langle \pi_{\hat{\Gamma}_2^c} e_{ii}, \begin{bmatrix} 0 & 0\\ 0 & \grll{2}_{\bar{L}_t^2 \setminus \Psi_{t+1}^2}^{\bar{L}_t^2 \setminus \Psi_{t+1}^2} \end{bmatrix} \rangle;
\]
\[
-\langle e_{i-1,i-1}, \eta_R^2 \rangle = - \sum_{t=1}^{q^\prime} \langle e_{i-1,i-1}, \begin{bmatrix} \lgrl{2}_{\bar{K}_t^2}^{\bar{K}_t^2} & 0 \\ 0 & 0 \end{bmatrix} \rangle - \sum_{t=1}^{q^\prime} \langle e_{i-1,i-1}, \gamma_r \begin{bmatrix} 0 & 0 \\ 0 & \lgrl{2}_{K_t^2 \setminus \Phi_t^2}^{K_t^2 \setminus \Phi_t^2} \end{bmatrix} \rangle.
\]
\end{lemma}
\begin{proof}
The gradient $\nabla^2_Y Y$ can be expressed as
\[\begin{split}
\nabla^2_Y Y &= \sum_{t=1}^{q^\prime} \grl{2}_{\bar{L}_t^2 \rightarrow \bar{J}_t^2}^{\bar{K}_t^2 \rightarrow \bar{I}_t^2} \cdot Y = \sum_{t=1}^{q^\prime} \begin{bmatrix} 0 & 0 \\ \grl{2}_{\bar{L}_t^2}^{\bar{K}_t^2} Y_{\hat{\bar{J}}_t^2}^{\bar{I}_t^2} & \grl{2}_{\bar{L}_t^2}^{\bar{K}_t^2} \lm{2}_{\bar{K}_t^2}^{\bar{L}_t^2} \end{bmatrix} = \\ &= \sum_{t=1}^{q^\prime} \begin{bmatrix} 0 & 0 & 0 \\ \grl{2}_{\Psi_{t+1}^2}^{\bar{K}_t^2} Y_{\hat{\bar{J}}_t^2}^{\bar{I}_t^2} & \grl{2}_{\Psi_{t+1}^2}^{\bar{K}_t^2} \lm{2}_{\bar{K}_t^2}^{\Psi_{t+1}^2} & \grl{2}_{\Psi_{t+1}^2}^{\bar{K}_t^2} \lm{2}_{\bar{K}_t^2}^{\bar{K}_t^2 \setminus \Psi_{t+1}^2} \\ \grl{2}_{\bar{L}_t^2 \setminus \Psi_{t+1}^2}^{\bar{K}_t^2} Y_{\hat{\bar{J}}_t^2}^{\bar{I}_t^2} & \grl{2}_{\bar{K}_t^2 \setminus \Psi_{t+1}^2}^{\bar{K}_t^2} \lm{2}_{\bar{K}_t^2}^{\Psi_{t+1}^2} & \grl{2}_{\bar{L}_t^2 \setminus \Psi_{t+1}^2}^{\bar{K}_t^2} \lm{2}_{\bar{K}_t^2}^{\bar{L}_t^2 \setminus \Psi_{t+1}^2} \end{bmatrix},
\end{split}
\]
where $\hat{\bar{J}}_t^2 = [1,n] \setminus \bar{J}_t^2$. Now, if one projects $\nabla^2_Y Y$ onto the diagonal, only the central blocks survive; a further application of $\pi_{\hat{\Gamma}_2^c}$ nullifies the middle block, for it occupies the location $\bar{\Delta}(\bar{\alpha}_t) \times \bar{\Delta}(\bar{\alpha}_t)$, and thus the first formula follows. A block formula for $\eta_R^2$, though easily derivable in a similar fashion, was deduced in \cite{plethora}; hence the second formula follows.
\end{proof}

\begin{proposition}
The following formula holds:
\[
\{\log y(h_{ii}), \log\psi \} = \begin{cases} 1,\   &\psi = h_{ii} \\ 0,\  &\text{otherwise} \end{cases}
\]
\end{proposition}
\begin{proof}
First of all, let us assume that $\psi$ is not equal to some $h_{jj}$, for this case is covered by Corollary~\ref{c:yhh}. Recall that $Y$ blocks have the form $Y_t = Y_{[1, \bar{\alpha}_t]}^{[\bar{\beta}_t,n]}$. The assumption $\psi \neq h_{jj}$ for all $j$ implies that $\bar{\beta}_{p-1}^1 \leq \bar{\beta}_{t-1}^2$ and $\bar{\alpha}_{p-1}^1 \geq \bar{\alpha}_{t-1}^2$ for all $t$. Indeed, if on the contrary $\bar{\beta}_{p-1}^1 > \bar{\beta}_{t-1}^2$, then $\bar{\beta}_{p-1}^1 = 2$ and $\bar{\beta}_{t-1}^2 = 1$; this means that $Y_{t-1}^2 = Y$ and hence $\psi$ must be some $h_{jj}$. A similar reasoning applies to $\bar{\alpha}$, for $\bar{\alpha}^1_{p-1} \in \{n,n-1\}$.

Next, we need to collect block formulas of all terms in equation \eqref{eq:yhpsi}, which are given in Section~\ref{s:blockf} and in Lemma~\ref{eq:h2last}. Under the stated assumption, some of the terms of the block formulas readily vanish. Indeed, observe the following:
\begin{itemize}
\item All terms that do not contain the first function vanish (for instance, $\sum \langle \lm{2}_{\bar{K}_{t-1}^2}^{L_{t-1}^2} \grl{2}_{L_{t-1}^2}^{\bar{K}_{t-1}^2}\rangle = 0$);
\item The sums $\sum^l$ vanish. Indeed, these are sums over blocks $Y_{t}^2$ which have their exit point to the left of $Y_{p-1}^1$, hence the exit point of $Y_{t}^2$ must be $(1,1)$. That's precisely the case $\psi = h_{jj}$, which was considered in Corollary~\ref{c:yhh};
\item $\lgrl{1}_{K_p^1}^{K_p^1} = 0$, for the leading block of the first function is $Y_{p-1}^1$ and $K_p^1$ spans the rows of $X_p^1$;
\item $\grll{1}_{\rho(L_t^2)}^{\rho(L_t^2)} = 0$ under the assumption $\beta_t^2 < \beta_p^1$, for then $\rho(L_t^2) \subseteq L_p^1 \setminus \Psi_p^1$;
\item $\grll{1}_{L_u^1}^{L_u^1} = 0$ for all $u < p$;
\item For $u = p$, even though $\grll{1}_{L_p^1}^{L_p^1}$ might be nonzero, the only term it belongs to vanishes:
\[\begin{split}
\langle \grll{1}^{L_p^1 \rightarrow J_p^1}_{L_p^1 \rightarrow J_p^1}, \gamma_c^* &\grll{2}_{\bar{L}_t^2 \setminus \Psi_{t+1}^2 \rightarrow \bar{J}_t^2 \setminus \bar{\Delta}(\bar{\beta}_t^2)}^{\bar{L}_t^2 \setminus \Psi_{t+1}^2 \rightarrow \bar{J}_t^2 \setminus \bar{\Delta}(\bar{\beta}_t^2)} \rangle =\\= &\langle \grll{1}^{\Psi_p^1 \rightarrow \bar{\Delta}(\bar{\beta}_{p-1}^1)}_{\Psi_p^1 \rightarrow \bar{\Delta}(\bar{\beta}_{p-1}^1)}, \grll{2}_{\bar{L}_t^2 \setminus \Psi_{t+1}^2 \rightarrow \bar{J}_t^2 \setminus \bar{\Delta}(\bar{\beta}_t^2)}^{\bar{L}_t^2 \setminus \Psi_{t+1}^2 \rightarrow \bar{J}_t^2 \setminus \bar{\Delta}(\bar{\beta}_t^2)} \rangle = 0,
\end{split}\]
because if $\bar{\Delta}(\bar{\beta}_{p-1}^1)$ is nontrivial, it's $[1,1_+]$ (which is the leftmost $Y$-run, hence if we cut out a $Y$-run in the second slot, the nonzero blocks in the second matrix never overlap with such in the first one).
\item $\lgrl{1}_{K_{u}^1 \setminus \Phi_u^1}^{K_{u}^1 \setminus \Phi_u^1} = 0$ for all $u$;
\item All terms that involve the blocks of the first function with indices $u < p-1$ vanish.
\end{itemize}
Additionally, one term requires a special care. Observe that
\[
\grll{1}_{\bar{L}_{p-1}^1 \setminus \Psi_p^1 \rightarrow \bar{J}_{p-1}^1 \setminus \bar{\Delta}(\bar{\beta}_{p}^1)}^{\bar{L}_{p-1}^1 \setminus \Psi_p^1 \rightarrow \bar{J}_{p-1}^1 \setminus \bar{\Delta}(\bar{\beta}_{p}^1)} = \begin{cases} e_{ii}, \ &i \notin \bar{\Delta}(\bar{\beta}_{p}^1) \\
0, \  &\text{otherwise}
\end{cases}
\]
Since we can assume $\bar{\beta}_{p-1}^1 \leq \bar{\beta}_{t-1}^2$ for any $t$, we find that
\[
\sum_{t=1}^{q^\prime} \langle \grll{1}_{\bar{L}_{p-1}^1 \setminus \Psi_p^1 \rightarrow \bar{J}_{p-1}^1 \setminus \bar{\Delta}(\bar{\beta}_{p}^1)}^{\bar{L}_{p-1}^1 \setminus \Psi_p^1 \rightarrow \bar{J}_{p-1}^1 \setminus \bar{\Delta}(\bar{\beta}_{p}^1)}, \pi_{\Gamma_2^c} \grll{2}_{\bar{L}_{t}^2 \setminus \Psi_{t+1}^2 \rightarrow \bar{J}_t^2 \setminus \bar{\Delta}(\bar{\beta}_t^2)}^{\bar{L}_{t}^2 \setminus \Psi_{t+1}^2 \rightarrow \bar{J}_t^2 \setminus \bar{\Delta}(\bar{\beta}_t^2)} \rangle = \sum_{t=1}^{q^\prime}\langle \pi_{\Gamma_2^c} e_{ii}, \begin{bmatrix}0 & 0 \\ 0 & \grll{2}_{\bar{L}_t^2 \setminus \Psi_{t+1}^2}^{\bar{L}_t^2 \setminus \Psi_{t+1}^2}\end{bmatrix} \rangle.
\]
Indeed, this formula follows right away if $i \notin \bar{\Delta}(\bar{\beta}_p^1)$. But if $i \in \bar{\Delta}(\bar{\beta}_p^1)$, the LHS is zero; so is the expression on the right, for $\bar{\Delta}(\bar{\beta}_p^1)$ is the leftmost $Y$-run and we cut out a $Y$-run in the second slot.

With all above in mind, formula \eqref{eq:yhpsi} expands in terms of blocks as
\[\begin{split}
\{\log y(h_{ii}), \log \psi\} &= -\sum_{\bar{\alpha}_t^2 < \bar{\alpha}_{p-1}^1}  \langle \grll{1}_{\sigma_{p-1}(\bar{L}_t^2)}^{\sigma_{p-1}(\bar{L}_t^2)} \grll{2}_{\bar{L}_t^2}^{\bar{L}_t^2} \rangle +\sum_{\bar{\alpha}_t^2 < \bar{\alpha}_{p-1}^1} \langle \lgrl{1}_{\sigma_{p-1}(\bar{K}_t^2)}^{\sigma_{p-1}(\bar{K}_t^2)} \lgrl{2}_{\bar{K}_t^2}^{\bar{K}_t^2} \rangle  + \\ &+ \sum_{t=1}^{q^\prime} \langle \pi_{\Gamma_2^c} e_{ii}, \begin{bmatrix}0 & 0 \\ 0 & \grll{2}_{\bar{L}_t^2 \setminus \Psi_{t+1}^2}^{\bar{L}_t^2 \setminus \Psi_{t+1}^2}\end{bmatrix} \rangle + \sum_{\bar{\alpha}_{t-1}^2 < \bar{\alpha}_{p-1}^1 } \langle \grll{1}_{\sigma_{p-1}(\Psi_t^2)}^{\sigma_{p-1}(\Psi_t^2)} \grll{2}_{\Psi_t^2}^{\Psi_t^2} \rangle+\\+ &\sum_{t=1}^{q^\prime} \langle \lgrl{1}_{\bar{K}_{p-1}^1 \rightarrow \bar{I}_{p-1}^1}^{\bar{K}_{p-1}^1 \rightarrow \bar{I}_{p-1}^1}, \gamma_r \lgrl{2}_{K_t^2 \setminus \Phi_t^2 \rightarrow I_t^2 \setminus \Delta(\alpha_t^2)}^{K_t^2 \setminus \Phi_t^2 \rightarrow I_t^2 \setminus \Delta(\alpha_t^2)} \rangle - \\ & - \sum_{\substack{\bar{\alpha}_{t-1}^2 = \bar{\alpha}_{p-1}^1 \\ \bar{\beta}_{t-1}^2 \geq \bar{\beta}_{p-1}^1}} \langle \grll{1}_{\sigma_{p-1}(\bar{L}_{t-1}^2 \setminus \Psi_t^2)}^{\sigma_{p-1}(\bar{L}_{t-1}^2 \setminus \Psi_t^2)} \grll{2}_{\bar{L}_{t-1}^2 \setminus \Psi_t^2}^{\bar{L}_{t-1}^2 \setminus \Psi_t^2} \rangle + \\ &+ \sum_{\substack{\bar{\alpha}_{t-1} = \bar{\alpha}_{p-1} \\ \bar{\beta}^2_{t-1} \geq \bar{\beta}_{p-1}^1}} \langle \lgrl{1}_{\bar{K}_{p-1}}^{\bar{K}_{p-1}} \lgrl{2}_{\bar{K}_{t-1}^2}^{\bar{K}_{t-1}^2} \rangle +  \sum_{t=1}^{q^\prime} \langle \pi_{\hat{\Gamma}_2^c}e_{ii}, \begin{bmatrix} 0 & 0 \\ 0 & \grll{2}_{\bar{L}_t \setminus \Psi_{t+1}}^{\bar{L}_t \setminus \Psi_{t+1}} \end{bmatrix} \rangle - \\ &- \sum_{t=1}^{q^\prime} \langle e_{i-1,i-1}, \begin{bmatrix} \lgrl{2}_{\bar{K}_t^2}^{\bar{K}_t^2} & 0 \\ 0 & 0 \end{bmatrix} \rangle - \sum_{t=1}^{q^\prime} \langle e_{i-1,i-1}, \gamma_r \begin{bmatrix} 0 & 0 \\ 0 & \lgrl{2}_{K_t^2 \setminus \Phi_t^2}^{K_t^2 \setminus \Phi_t^2} \end{bmatrix} \rangle.
\end{split}
\]
Now we combine the remaining terms to obtain zero. The conditions under the second and the seventh sums combine into a condition satisfied for all blocks, and the resulting sum cancels out with the ninth one. The fifth sum cancels out with the last one. The third sum combines with the eighth one, and the resulting sum is 
\begin{equation}\label{eq:theg}
\begin{split}
\sum_{t=1}^{q^\prime} \langle \pi_{\Gamma_2^c} e_{ii}, \begin{bmatrix} 0 & 0 \\ 0 & \grll{2}_{\bar{L}_t^2\setminus \Psi_{t+1}^2}^{\bar{L}_t^2\setminus \Psi_{t+1}^2}\end{bmatrix} \rangle + \sum_{t=1}^{q^\prime} \langle \pi_{\hat{\Gamma}_2^c} e_{ii},& \begin{bmatrix} 0 & 0 \\ 0 & \grll{2}_{\bar{L}_t^2\setminus \Psi_{t+1}^2}^{\bar{L}_t^2\setminus \Psi_{t+1}^2}\end{bmatrix} \rangle = \\ &= \sum_{t=1}^{q^\prime} \langle e_{ii}, \begin{bmatrix} 0 & 0 \\ 0 & \grll{2}_{\bar{L}_t^2\setminus \Psi_{t+1}^2}^{\bar{L}_t^2\setminus \Psi_{t+1}^2}\end{bmatrix} \rangle.
\end{split}
\end{equation}
All the remaining terms (the first, the fourth and the sixth) combine into
\begin{equation}\label{eq:ffs}
\begin{split}
-\sum_{\bar{\alpha}_t^2 < \bar{\alpha}_{p-1}^1} \langle \grll{1}_{\sigma_{p-1}(\bar{L}_t^2)}^{\sigma_{p-1}(\bar{L}_t^2)} \grll{2}_{\bar{L}_t}^{\bar{L}_t}\rangle &+ \sum_{\bar{\alpha}_{t-1}^2 < \bar{\alpha}_{p-1}^1 } \langle \grll{1}_{\sigma_{p-1}(\Psi_t^2)}^{\sigma_{p-1}(\Psi_t^2)} \grll{2}_{\Psi_t^2}^{\Psi_t^2} \rangle -\\& - \sum_{\substack{\bar{\alpha}^2_{t-1} = \bar{\alpha}^1_{p-1} \\ \bar{\beta}^2_{t-1} \geq \bar{\beta}^1_{p-1} }} \langle \grll{1}_{\sigma_{p-1}(\bar{L}_{t-1}^2 \setminus \Psi_t^2)}^{\sigma_{p-1}(\bar{L}_{t-1}^2 \setminus \Psi_t^2)} \grll{2}_{\bar{L}_{t-1}^2 \setminus \Psi_t^2}^{\bar{L}_{t-1}^2 \setminus \Psi_t^2}\rangle = \\ =&- \sum_{\bar{\alpha}_{t-1}^2 \leq \bar{\alpha}_{p-1}^1} \langle \grll{1}_{\sigma_{p-1}(\bar{L}_{t-1}^2 \setminus \Psi_{t}^2)}^{\sigma_{p-1}(\bar{L}_{t-1}^2 \setminus \Psi_{t}^2)} \grll{2}_{\bar{L}_{t-1} \setminus \Psi_{t}^2}^{\bar{L}_{t-1} \setminus \Psi_{t}^2}\rangle.
\end{split}
\end{equation}
Now, the contributions of equations \eqref{eq:theg} and \eqref{eq:ffs} cancel each other out (note that the condition $\bar{\alpha}_{t-1}^2 \leq \bar{\alpha}_{p-1}^1$ is satisfied for all blocks, for this is a consequence of the assumption $\psi \neq h_{jj}$). Thus the result follows.
\end{proof}
\subsection{Computation of $\{y(g_{ii}), \psi\}$}

Let $\psi$ be an arbitrary $g$- or $h$-function and let $g_{ii}$ be fixed, $2 \leq i \leq n$. We employ the shorthand notation from the previous sections, choosing the first function to be $\log g_{i,i-1} - \log g_{i+1,i}$ and the second function to be $\log \psi$. Let us assume throughout the subsection that a pair $(R_0^r,R_0^c)$ is chosen so that the identities from~\eqref{eq:roid} hold.

\begin{proposition}
The bracket of $y(g_{ii})$ and $\psi$ can be expressed as
\begin{equation}\label{eq:ygpsi}\begin{split}
\{\log y(g_{ii}),\log \psi\} &= -\langle \pi_{<} \eta_L^1, \pi_{>} \eta_L^2 \rangle - \langle \pi_{>} \eta_R^1, \pi_{<} \eta_R^2 \rangle + \\ &+ \langle \gamma_r \xi_R^1, \gamma_r X \nabla_X^2 \rangle + \langle \gamma_c^* \xi_L^1, \gamma_c^* \nabla_Y^2 Y \rangle + \\ &+ \langle \pi_{\hat{\Gamma}_1^r} e_{ii}, \pi_{\hat{\Gamma}_1^r} X\nabla_X^2\rangle - \langle e_{i-1,i-1}, \eta_L^2 \rangle.
\end{split}
\end{equation}
\end{proposition}
\begin{proof}
Recall that we employ the conventions from Section~\ref{s:iniclust} when indices are out of range. The $y$-coordinate of $g_{ii}$ is given by
\[
y(g_{ii}) = \frac{g_{i+1,i+1} f_{n-i+1,1} g_{i,i-1}}{g_{i-1,i-1} f_{n-i,1} g_{i+1,i}}.
\]

With formula \eqref{eq:phipsi} and the diagonal derivatives formulas from Section~\ref{s:diagders}, we can write the bracket $\{\log g, \log \psi\}$ as
\[\begin{split}
\{\log g, \log \psi\} &= \langle e_{i-1,i-1} + e_{ii}, \nabla_Y^2 Y\rangle - \langle e_{i-1,i-1} + e_{ii}, Y \nabla_Y^2 \rangle -\\ &-\langle R_0^c (e_{i-1,i-1} + e_{ii}), E_L^2 \rangle + \langle R_0^r (e_{i-1,i-1} + e_{ii}), E_R^2 \rangle
\end{split}
\]
and the bracket $\{\log f, \log \psi\}$ as
\[
\{\log f, \log \psi\} = -\langle e_{ii}, \nabla_Y^2 Y\rangle + \langle e_{i-1,i-1}, Y\nabla_Y^2 \rangle + \langle R_0^c e_{ii}, E_L^2 \rangle - \langle R_0^r e_{i-1,i-1}, E_R^2 \rangle.
\]
Now, the sum $\{\log f, \log\psi\} + \{\log g,\log \psi\}$ becomes
\[
\{\log f + \log g, \log \psi\} = \langle e_{i-1,i-1}, \nabla^2_Y Y\rangle - \langle e_{ii}, Y\nabla^2_Y \rangle - \langle R_0^c e_{i-1,i-1}, E_L^2 \rangle + \langle R_0^r e_{ii}, E_R^2 \rangle.
\]
To derive the formula for $\{\log y(g_{ii}), \log \psi\}$, we use formula \eqref{eq:brack}. Notice that the first four terms are already in the appropriate form, so we only need to deal with the diagonal part $D$. Consequently, we need to show that
\[
D + \{\log f + \log g, \log \psi\} = \langle \pi_{\hat{\Gamma}_1^r} e_{ii}, \pi_{\hat{\Gamma}_1^r} X\nabla_X^2\rangle - \langle e_{i-1,i-1}, \eta_L^2 \rangle.
\]
Notice that $R_0 \gamma = -\pi_{\Gamma_1} + R_0 \pi_{\Gamma_1}$. With this, we rewrite
\[\begin{split}
\langle R_0^c \pi_0 \xi_L^1, E_L^2 \rangle &= \langle R_0^c \gamma_c(e_{i-1,i-1}), E_L^2 \rangle = -\langle \pi_{\Gamma_1^c} e_{i-1,i-1}, E_L^2 \rangle + \langle R_0^c \pi_{\Gamma_1^c} e_{i-1,i-1}, E_L^2 \rangle;\\
-\langle  R_0^r\pi_0 \eta_R^1, E_R^2 \rangle &= -\langle R_0^r \gamma_r(e_{ii}), E_R^2 \rangle =  \langle \pi_{\Gamma_1^r} e_{ii}, E_R^2 \rangle - \langle R_0^r \pi_{\Gamma_1^r} e_{ii}, E_R^2 \rangle.
\end{split}
\]
Now, expressing $D$ as in equation \eqref{eq:diagpart2}, the full expression for $D + \{\log (fg), \log \psi\}$ expands as
\[
\begin{split}
&-\langle \pi_{\Gamma_1^c} e_{i-1,i-1}, \gamma_c^*(\nabla^2_Y Y) \rangle - \langle \pi_{\Gamma_1^r}e_{ii}, X \nabla^2_X \rangle - \langle \pi_{\Gamma_1^c} e_{i-1,i-1}, E_L^2 \rangle + \langle R_0^c \pi_{\Gamma_1^c} e_{i-1,i-1}, E_L^2\rangle - \\
 &- \langle \pi_{\hat{\Gamma}_1^c} e_{i-1,i-1}, E_L^2 \rangle + \langle R_0^c \pi_{\hat{\Gamma}_1^c}e_{i-1,i-1}, E_L^2\rangle + \langle \pi_{\Gamma_1^r} e_{ii}, E_R^2 \rangle - \langle R_0^r \pi_{\Gamma_1^r} e_{ii}, E_R^2 \rangle + \\
 &+ \langle \pi_{\hat{\Gamma}_1^r} e_{ii}, E_R^2\rangle - \langle R_0^r \pi_{\hat{\Gamma}_1^r} e_{ii}, E_R^2 \rangle +\langle e_{i-1,i-1}, \nabla^2_Y Y\rangle - \langle e_{ii}, Y\nabla^2_Y \rangle -\\ &- \langle R_0^c e_{i-1,i-1}, E_L^2 \rangle + \langle R_0^r e_{ii}, E_R^2 \rangle\rangle.
\end{split}
\]
It's easy to see that all terms containing $R_0^r$, as well as all terms containing $R_0^c$, result in zero. The rest can be combined as follows:
\[
-\langle \pi_{\Gamma_1^c} e_{i-1,i-1}, \gamma_c^*(\nabla^2_Y Y)\rangle - \langle \pi_{\Gamma_1^c} e_{i-1,i-1}, E_L^2 \rangle - \langle \pi_{\hat{\Gamma}_1^c}e_{i-1,i-1}, E_L^2 \rangle + \langle e_{i-1,i-1}, \nabla^2_Y Y\rangle = -\langle e_{i-1,i-1},\eta_L^2 \rangle;
\]
\[
-\langle \pi_{\Gamma_1^r} e_{ii}, X\nabla^2_X \rangle + \langle \pi_{\Gamma_1^r} e_{ii}, E_R^2 \rangle + \langle \pi_{\hat{\Gamma}_1^r} e_{ii}, E_R^2 \rangle - \langle e_{ii}, Y\nabla^2_Y \rangle = \langle \pi_{\hat{\Gamma}_1^r} e_{ii}, X\nabla^2_X \rangle.
\]
Thus the formula holds.
\end{proof}

\begin{corp}\label{c:ygg}
As a consequence, $\{\log y(g_{ii}), \log g_{jj} \} = \delta_{ij}$ for any $j$.
\end{corp}
\begin{proof}
Recall that $X \nabla_X g_{jj} \in \mathfrak{b}_+$ and $\nabla_X g_{jj} \cdot X \in \mathfrak{b}_-$, so the first two terms together with the fourth one vanish; since $\pi_0(\xi_R^1) = e_{ii}$ and $\pi_0 (X \nabla^2_X) = \pi_0 (\nabla^2_XX) = \Delta(j,n)$, where $\Delta(j,n) = \sum_{k=j}^n e_{kk}$, we see that
\[\begin{split}
\{\log y(g_{ii}), \log g_{jj} \} &= \langle \pi_{\Gamma_1^r} e_{ii}, \pi_{\Gamma_1^r} \Delta(j,n) \rangle + \langle \pi_{\hat{\Gamma}_1^r} e_{ii}, \pi_{\hat{\Gamma}_1^r} \Delta(j,n)\rangle - \langle e_{i-1,i-1}, \Delta(j,n)\rangle = \\ &= \langle e_{ii} - e_{i-1,i-1}, \Delta(j,n) \rangle = \delta_{ij}.\qedhere
\end{split}
\]
\end{proof}

\begin{lemma}\label{l:blockg}
The following formulas hold for the last two terms in equation \eqref{eq:ygpsi}:
\[
\langle \pi_{\hat{\Gamma}_1^r} e_{ii}, X \nabla^2_X \rangle = \langle \pi_{\hat{\Gamma}_1^r} e_{ii}, \begin{bmatrix} 0 & 0 \\ 0 & \lgrl{2}_{K_t^2 \setminus \Phi_t^2}^{K_t^2 \setminus \Phi_t^2}\end{bmatrix}\rangle;
\]
\[
-\langle e_{i-1,i-1}, \eta_L^2 \rangle = -\sum_{u=1}^{s^2} \langle e_{i-1,i-1}, \begin{bmatrix} \grll{2}_{L_t^2}^{L_t^2} & 0 \\ 0 & 0\end{bmatrix} \rangle - \sum_{t=1}^{s^2} \langle e_{i-1,i-1}, \gamma_c^* \begin{bmatrix} 0 & 0 \\ 0 & \grll{1}_{\bar{L}_t^2 \setminus \Psi_{t+1}^2}^{\bar{L}_t^2 \setminus \Psi_{t+1}^2} \end{bmatrix}\rangle.
\]
\end{lemma}
\begin{proof}
The gradient $X \nabla^2_X$ can be expressed as
\[\begin{split}
X\nabla^2_X &= \sum_{t=1}^{s^2} X (\nabla_{\mathcal{L}}^2)_{L_t^2 \rightarrow J_t^2}^{K_t^2 \rightarrow I_t^2} = \sum_{t=1}^{s^2} \begin{bmatrix} 0 & X_{\hat{I}_t^2}^{J_t^2} \grl{2}_{L_t^2}^{K_t^2} \\ 0 & \lm{2}_{K_t^2}^{L_t^2} \grl{2}_{L_t^2}^{K_t^2} \end{bmatrix} = \\ &= \sum_{t=1}^{s^2} \begin{bmatrix} 0 & X_{\hat{I}_t^2}^{J_t^2} \grl{2}_{L_t^2}^{\Phi_t^2} & X_{\hat{I}_t^2}^{J_t^2} \grl{2}_{L_t^2}^{K_t^2 \setminus \Phi_t^2} \\ 0 & \lm{2}_{\Phi_t^2}^{L_t^2} \grl{2}_{L_t^2}^{\Phi_t^2} & \lm{2}_{\Phi_t^2}^{L_t^2} \grl{2}_{L_t^2}^{K_t^2 \setminus \Phi_t^2} \\ 0 & \lm{2}_{K_t^2 \setminus \Phi_t^2}^{L_t^2} \grl{2}_{L_t^2}^{\Phi_t^2} & \lm{2}_{K_t^2 \setminus \Phi_t^2}^{L_t^2} \grl{2}_{L_t^2}^{K_t^2 \setminus \Phi_t^2} \end{bmatrix}
\end{split},
\]
where $\hat{I}_t^2 = [1,n] \setminus I_t^2$. Now, the projection onto diagonal matrices eliminates all off-diagonal blocks; a further projection onto $\mathfrak{g}_{\hat{\Gamma}_1^r}$ nullifies the middle block, for it occupies the location $\Delta(\alpha_t^2) \times \Delta(\alpha_t^2)$. Thus the first formula hold. For the second formula, one can use a block formula for $\eta_L^2$ derived in \cite{plethora} (which is also easily derivable in a similar manner).
\end{proof}

\begin{proposition}
The following formula holds:
\[
\{\log y(g_{ii}), \log \psi \} = \begin{cases} 1, \ &\psi = g_{ii}\\ 0, \ &\text{otherwise} \end{cases}
\]
\end{proposition}
\begin{proof}
First of all, let us assume that $\psi \neq g_{jj}$ for all $j$, for this case was studied in Corollary~\ref{c:ygg}. As a consequence, we can assume throughout the proof that for any $X_t^2 = X_{[\alpha_t^2,n]}^{[1,\beta^2_t]}$, we have $\alpha_p^1 \leq \alpha_t^2$ and $\beta_p^1 \geq \beta_{t}^2$. Indeed, if\footnote{Note that if $\alpha_p^1=\alpha_t^2 = 1$, it doesn't follow that $\psi$ is a trailing minor of $X$, for in this case $\psi$ can also be a trailing minor of $\mathcal{L}^1$.}  $\alpha_p^1 > \alpha_t^2$, then $\alpha_p^1 = 2$ and $\alpha_t^2 = 1$, which implies that $\psi = g_{jj}$ for some $j$. A similar reasoning applies to $\beta_p^1$, for $\beta_p^1 \in \{n-1,n\}$.

Now, we need to expand every term in equation \eqref{eq:ygpsi} using block formulas from Section~\ref{s:blockf} and Lemma~\ref{l:blockg}. We can say right away that some of the terms in the block formulas vanish via the following observations:
\begin{itemize}
\item All terms that do not contain the first function are zero. For instance, $\sum \langle \lm{2}_{\bar{K}_{t-1}^2}^{L_t^2} \grl{2}_{L_t^2}^{\bar{K}_{t-1}^2} \rangle = 0$;
\item The sums $\sum^{a}$ vanish, for they are taken over the blocks $X_t^2$ that have their exit point above the exit point of $X_p^1$. Since the exit point of the latter is $(2,1)$, the exit point of the former must be $(1,1)$, which is precisely the case $\psi = g_{jj}$, which in turn is excluded from the consideration;
\item $\grll{1}^{\bar{L}_p^1}_{\bar{L}_p^1} = 0$, for the leading block of the first function is $X_p^1$ and $\bar{L}_p^1$ spans the rows of $Y_p^1$;
\item $\lgrl{1}_{\sigma_{p}(\bar{K}_t^2)}^{\sigma_{p}(\bar{K}_t^2)} = 0$ if $\bar{\alpha}_t^2 < \bar{\alpha}_p^1$. This relation holds due to $\sigma_p(\bar{K}_t^2) \subseteq \bar{K}_p^1 \setminus \Phi_p^1$, and the gradient is zero along the latter rows;
\item $\lgrl{1}_{\bar{K}_u^1}^{\bar{K}_u^1} = 0$ for $u < p$;
\item For $u = p$, even though $\lgrl{1}_{\bar{K}_p^1}^{\bar{K}_p^1}$ might be nonzero, the only term it belongs to vanishes:
\[
\langle \lgrl{1}_{\bar{K}_p^1 \rightarrow \bar{I}_p^1}^{\bar{K}_p^1 \rightarrow \bar{I}_p^1}, \gamma_r \lgrl{2}_{K_t^2 \setminus \Phi_t^2 \rightarrow I_t^2 \setminus \Delta(\alpha_t^2)}^{K_t^2 \setminus \Phi_t^2 \rightarrow I_t^2 \setminus \Delta(\alpha_t^2)} \rangle = \langle \lgrl{1}_{\Phi_p \rightarrow \Delta(\alpha_p^1)}^{\Phi_p \rightarrow \Delta(\alpha_p^1)}, \lgrl{2}_{K_t^2 \setminus \Phi_t^2 \rightarrow I_t^2 \setminus \Delta(\alpha_t^2)}^{K_t^2 \setminus \Phi_t^2 \rightarrow I_t^2 \setminus \Delta(\alpha_t^2)} \rangle = 0,
\]
for if $\Delta(\alpha_p^1)$ is nontrivial, it is the leftmost $X$-run, and since we cut out an $X$-run in the second matrix, the result is zero.
\item $\grll{1}_{\bar{L}_u^1 \setminus \Psi_{u+1}^1 \rightarrow \bar{J}_u^1 \setminus \bar{\Delta}(\bar{\beta}_u^1)}^{\bar{L}_u^1 \setminus \Psi_{u+1}^1 \rightarrow \bar{J}_u^1 \setminus \bar{\Delta}(\bar{\beta}_u^1)} = 0$ for all $u$;
\item All terms that involve the blocks of the first function with indices $u < p-1$ vanish; 
\end{itemize}
Additionally, one term requires a special care. Observe that
\[
\lgrl{1}_{K_p^1 \setminus \Phi_p^1 \rightarrow I_p^1 \setminus \Delta(\alpha_p^1)}^{K_p^1 \setminus \Phi_p^1 \rightarrow I_p^1 \setminus \Delta(\alpha_p^1)} = \begin{cases}
e_{ii}, \ \ i \notin \Delta(\alpha_p^1)\\
0, \ \ \text{otherwise}
\end{cases}
\]
We argue that 
\[
\langle \lgrl{1}_{K_p^1 \setminus \Phi_p^1 \rightarrow I_p^1 \setminus \Delta(\alpha_p^1)}^{K_p^1 \setminus \Phi_p^1 \rightarrow I_p^1 \setminus \Delta(\alpha_p^1)}, \pi_{\Gamma_1^r} \lgrl{2}_{K_t^2 \setminus \Phi_t^2 \rightarrow I_t^2 \setminus \Delta(\alpha_t^2)}^{K_t^2 \setminus \Phi_t^2 \rightarrow I_t^2 \setminus \Delta(\alpha_t^2)} \rangle =  \langle \pi_{\Gamma_1^r} e_{ii}, \begin{bmatrix} 0 & 0 \\ 0 & \lgrl{2}_{K_t^2 \setminus \Phi_t^2}^{K_t^2 \setminus \Phi_t^2}\end{bmatrix} \rangle.
\]
Indeed, the formula follows right away if $i \notin \Delta(\alpha_p^1)$. On the contrary, if $i \in \Delta(\alpha_p^1)$, then the LHS is zero; such is the RHS as well, for we remove an $X$-run from the second matrix.

With all the above observations, formula \eqref{eq:ygpsi} expands as
\[\begin{split}
\{\log y(g_{ii}), \log \psi \} &= - \sum_{\beta_t^2 < \beta_p^1} \langle \lgrl{1}_{\rho(K_t^2)}^{\rho(K_t^2)} \lgrl{2}_{K_t^2}^{K_t^2} + \sum_{\beta_t^2 < \beta_p^1} \langle \grll{1}_{\rho(L_t^2)}^{\rho(L_t^2)} \grll{2}_{L_t^2}^{L_t^2} \rangle + \\ & + \sum_{t=1}^{s^2}\langle \pi_{\Gamma_1^r} e_{ii}, \begin{bmatrix} 0 & 0 \\ 0 & \lgrl{2}_{K_t^2 \setminus \Phi_t^2}^{K_t^2 \setminus \Phi_t^2}\end{bmatrix} \rangle + \sum_{\beta_t^2 < \beta_p^1} \langle \lgrl{1}_{\rho(\Phi_t^2)}^{\rho(\Phi_t^2)} \lgrl{2}_{\Phi_t^2}^{\Phi_t^2} \rangle
+ \\+ &\sum_{t=1}^{s^2} \langle \grll{1}_{L_p^1 \rightarrow J_p^1}^{L_p^1 \rightarrow J_p^1}, \gamma_c^* \grll{2}_{\bar{L}_t^2 \setminus \Psi_{t+1}^2 \rightarrow \bar{J}_t^2 \setminus \bar{\Delta}(\bar{\beta}_t^2)}^{\bar{L}_t^2 \setminus \Psi_{t+1}^2 \rightarrow \bar{J}_t^2 \setminus \bar{\Delta}(\bar{\beta}_t^2)} - \\ &  - \sum_{\beta_t^2 = \beta_p^1} \langle \lgrl{1}_{\rho(K_t^2 \setminus \Phi_t^2)}^{\rho(K_t^2 \setminus \Phi_t^2)} \lgrl{2}_{K_t^2 \setminus \Phi_t^2}^{K_t^2 \setminus \Phi_t^2} \rangle + \\ &+ \sum_{\beta_t^2 = \beta_p^1} \langle \grll{1}_{L_p^1}^{L_p^1} \grll{2}_{L_t^2}^{L_t^2} \rangle + \sum_{t=1}^{s^2} \langle \pi_{\hat{\Gamma}_1^r}, \begin{bmatrix} 0 & 0 \\ 0 & \lgrl{2}_{K_t^2 \setminus \Phi_t^2}^{K_t^2 \setminus \Phi_t^2} \end{bmatrix} \rangle - \\ &- \sum_{t=1}^{s^2} \langle e_{i-1,i-1}, \begin{bmatrix} \lgrl{1}_{L_t^2}^{L_t^2} & 0 \\ 0 & 0 \end{bmatrix} \rangle - \sum_{t=1}^{s^2} \langle e_{i-1,i-1}, \gamma_c^* \begin{bmatrix} 0 & 0 \\ 0 & \grll{1}_{\bar{L}_t^2 \setminus \Psi_{t+1}^2}^{\bar{L}_t^2 \setminus \Psi_{t+1}^2} \end{bmatrix} \rangle.
\end{split}
\]
Now let's combine these terms together. The conditions under the second and the seventh sums combine into $\beta_t^2 \leq \beta_p^1$, which is satisfied for all blocks due to our assumption; hence, these terms cancel out with the ninth sum. The fifth term cancels out with the last one. The third and the eighth term combine into
\begin{equation}\label{eq:ygpenlast}\begin{split}
\sum_{t=1}^{s^2} \langle \pi_{\Gamma_1^r} e_{ii}, \begin{bmatrix} 0 & 0 \\ 0 & \lgrl{2}_{K_t^2 \setminus \Phi_t^2}^{K_t^2 \setminus \Phi_t^2} \end{bmatrix} \rangle + \sum_{t=1}^{s^2} \langle \pi_{\Gamma_1^r} e_{ii}&, \begin{bmatrix} 0 & 0 \\ 0 & \lgrl{2}_{K_t^2 \setminus \Phi_t^2}^{K_t^2 \setminus \Phi_t^2} \end{bmatrix} \rangle = \\ &= \sum_{t=1}^{s^2} \langle e_{ii}, \begin{bmatrix} 0 & 0 \\ 0 & \lgrl{2}_{K_t^2 \setminus \Phi_t^2}^{K_t^2 \setminus \Phi_t^2} \end{bmatrix} \rangle.
\end{split}
\end{equation}
The remaining terms (the first, the fourth and the sixth) add up to
\begin{equation}\label{eq:yglast}\begin{split}
- \sum_{\beta_t^2 < \beta_p^1} \langle \lgrl{1}_{\rho(K_t^2)}^{\rho(K_t^2)} \lgrl{2}_{K_t^2}^{K_t^2} \rangle &+ \sum_{\beta_t^2 < \beta_p^1} \langle \lgrl{1}_{\rho(\Phi_t^2)}^{\rho(\Phi_t^2)} \lgrl{2}_{\Phi_t^2}^{\Phi_t^2} \rangle - \\ &- \sum_{\beta_t^2 = \beta_p^1} \langle \lgrl{1}_{\rho(K_t^2 \setminus \Phi_t^2)}^{\rho(K_t^2\setminus \Phi_t^2)} \lgrl{2}_{K_t^2 \setminus \Phi_t^2}^{K_t^2 \setminus \Phi_t^2}\rangle = \\ = &-\sum_{t=1}^{s^2} \langle \lgrl{1}_{\rho(K_t^2 \setminus \Phi_t^2)}^{\rho(K_t^2\setminus \Phi_t^2)} \lgrl{2}_{K_t^2 \setminus \Phi_t^2}^{K_t^2 \setminus \Phi_t^2}\rangle,
\end{split}
\end{equation}
where we used the fact that $\beta_t^2 \leq \beta_p^1$ is satisfied for all blocks under the stated assumption. Now notice that the total contribution of equations \eqref{eq:ygpenlast} and \eqref{eq:yglast} is zero. Thus the result follows.
\end{proof}
\section{Case of $D(\SL_n)$}\label{s:adj_sln}
In this section, we show how to derive Theorem~\ref{t:mainsln} from Theorem~\ref{t:main}. Let us restate Theorem~\ref{t:mainsln} here for convenience:

\begin{theorem*}
Let $\bg = (\bg^r, \bg^c)$ be a pair of aperiodic oriented Belavin-Drinfeld triples. There exists a generalized cluster structure $\gc(\bg)$ on $D(\SL_n) = \SL_n \times \SL_n$ such that
\begin{enumerate}[(i)]
\item The number of stable variables is $k_{\bg^r}+k_{\bg^c} + (n-1)$, and the exchange matrix has full rank;
\item The generalized cluster structure $\gc(\bg)$ is regular, and the ring of regular functions $\mathcal{O}(D(\SL_n))$ is naturally isomorphic to the upper cluster algebra $\bar{\mathcal{A}}_{\mathbb{C}}(\gc(\bg))$;\label{tm:iso2}
\item The global toric action of $(\mathbb{C}^*)^{k_{\bg^r}+k_{\bg^c}}$ on $\gc(\bg)$ is induced by the left action of $\mathcal{H}_{\bg^r}$ and the right action of $\mathcal{H}_{\bg^c}$ on $D(\SL_n)$;\label{tm:tor2}
\item Any Poisson bracket defined by the pair $\bg$ on $D(\SL_n)$ is compatible with $\gc(\bg)$.\label{tm:compb2}
\end{enumerate}
\end{theorem*}

Let $\bg:=(\bg^r,\bg^c)$ be an oriented aperiodic Belavin-Drinfeld pair. Fix a choice of $(R_0^r,R_0^c)$ for the Poisson bracket $\{\cdot,\cdot\}_{D(\SL_n)}$ on $D(\SL_n)$ and extend both $R_0^r$ and $R_0^c$ to the Cartan subalgebra of $\gl_n$ via $R_0^r(I) := R_0^c(I) := (1/2)I$, where $I$ is the identity matrix. Let $\{\cdot,\cdot\}_{D(\GL_n)}$ be the resulting Poisson bracket on $D(\GL_n)$.

\begin{lemma}\label{l:prj}
Under the above choice of $(R_0^r,R_0^c)$, the restriction map $\mathcal{O}(D(\GL_n)) \rightarrow \mathcal{O}(D(\SL_n))$ is Poisson; in other words, for any $f_1,f_2 \in \mathcal{O}(D(\GL_n))$,
\begin{equation}\label{eq:restrpoiss}
\{f_1,f_2\}_{D(\GL_n)}(X,Y) = \{f_1|_{D(\SL_n)},f_2|_{D(\SL_n)}\}_{D(\SL_n)}(X,Y), \ \ (X,Y) \in D(\SL_n).
\end{equation}
\end{lemma}
\begin{proof}
Let $\pi_*$ the projection of $\gl_n$ onto $\sll_n$
\[
\pi_* : \gl_n \rightarrow \sll_n, \ \ \pi_*(A) = A-\frac{1}{n}\tr(A)I,
\]
and let $f_1$ and $f_2$ be regular functions on $D(\GL_n)$. Recall that
\[\begin{split}
\{f_1,f_2\}_{D(\GL_n)}(X,Y) = \langle R_+^c&(E_Lf_1),E_Lf_2\rangle - \langle R_+^r (E_R f_1), E_R f_2\rangle+\\ + &\langle X\nabla_X f_1, Y\nabla_Y f_2\rangle - \langle \nabla_X f_1 X, \nabla_Y f_2 Y\rangle;
\end{split}
\]
\[\begin{split}
\{f_1|_{D(\SL_n)},f_2|_{D(\SL_n)}\}_{D(\SL_n)}(X,Y) = \langle R_+^c&(\pi_*(E_Lf_1)),\pi_*(E_Lf_2)\rangle - \langle R_+^r (\pi_*(E_R f_1)), \pi_*(E_R f_2)\rangle + \\ + &\langle \pi_*(X\nabla_X f_1), \pi_*(Y\nabla_Y f_2)\rangle - \langle \pi_*(\nabla_X f_1 X), \pi_*(\nabla_Y f_2 Y)\rangle.
\end{split}
\]
A simple computation shows that
\[
\langle R_+^c(E_Lf_1),E_Lf_2\rangle - \langle R_+^c(\pi_*(E_Lf_1)),\pi_*(E_Lf_2)\rangle = \frac{1}{2n} \tr(E_L f_1) \tr(E_L f_2);
\]
\[
\langle R_+^r(E_Rf_1),E_Rf_2\rangle - \langle R_+^r(\pi_*(E_Rf_1)),\pi_*(E_Rf_2)\rangle = \frac{1}{2n} \tr(E_R f_1) \tr(E_R f_2);
\]
\[
\langle X\nabla_X f_1, Y\nabla_Y f_2\rangle  - \langle \pi_*(X\nabla_X f_1), \pi_*(Y\nabla_Y f_2)\rangle = \frac{1}{n}\tr(X\nabla_X f_1) \tr(Y\nabla_Y f_2);
\]
\[
\langle \nabla_X f_1 X, \nabla_Y f_2 Y\rangle - \langle \pi_*(\nabla_X f_1 X), \pi_*(\nabla_Y f_2 Y)\rangle = \frac{1}{n}\tr(\nabla_X f_1 X)\tr(\nabla_Y f_2 Y).
\]
Now, since $\tr(AB) = \tr(BA)$, we see that $\tr(E_Lf_1) = \tr(E_R f_1)$, $\tr(X\nabla_X f_1) = \tr(\nabla_X f_1 X)$, and so on. Combining the above identities yields equation~\eqref{eq:restrpoiss}.
\end{proof}

Now let us turn to the proof of Theorem~\ref{t:mainsln}:
\begin{proof}
The initial extended seed on $D(\SL_n)$ is obtained from the initial extended seed on $D(\GL_n)$ via setting $\det X = \det Y = 1$ and removing the frozen variables $g_{11}$ and $h_{11}$. Therefore, the total number of frozen variables on $D(\SL_n)$ is $k_{\bg^r} + k_{\bg^c}$. Part~\ref{tm:iso2} is trivial: Given any regular function on $D(\SL_n)$, one can extend it to a regular function on $D(\GL_n)$, then express the function as a Laurent polynomial in terms of any extended cluster, and finally restrict it back to $D(\SL_n)$. 

Any extended cluster on $D(\SL_n)$ is log-canonical due to Lemma~\ref{l:prj} and the fact that the statement is true for $D(\GL_n)$. Let $\tilde{B}_{D(\GL_n}$ and $\tilde{B}_{D(\SL_n)}$ be the extended exchange matrices for the initial extended seeds, and let $\Omega_{D(\GL_n)}$ and $\Omega_{D(\SL_n)}$ be the coefficient matrices of the brackets. Due to the choice of $(R^r_0,R^c_0)$, $\det X$ and $\det Y$ are Casimirs on $D(\GL_n)$, and therefore the corresponding rows and columns of $\Omega_{D(\GL_n)}$ are zero; furthermore, the log-coefficient of any pair of variables on $D(\SL_n)$ coincides with the log-coefficient on $D(\GL_n)$ due to Lemma~\ref{l:prj}. Since in addition $\tilde{B}_{D(\GL_n)} \Omega_{D(\GL_n)} = \begin{bmatrix} I & 0 \end{bmatrix}$ (compatibility with the Poisson bracket), it follows that $\tilde{B}_{D(\SL_n)} \Omega_{D(\SL_n)} = \begin{bmatrix} I & 0 \end{bmatrix}$, and thus, by Proposition~\ref{p:compb}, the Poisson bracket $\{\cdot,\cdot\}_{D(\SL_n)}$ is compatible with the generalized cluster structure on $D(\SL_n)$. In particular, $\tilde{B}_{D(\SL_n)}$ has full rank.

For part~\ref{tm:tor2}, we use the groups $\mathcal{H}_{\bg^r}$ and $\mathcal{H}_{\bg^c}$ that were defined in Section~\ref{s:tor} (the difference is that we no longer have the two-dimensional action by scalar matrices). Note that $\mathcal{H}_{\bg^r},\mathcal{H}_{\bg^c}\subseteq \SL_n$, hence these groups induce zero weights on the frozen variables $g_{11} = \det X$ and $h_{11} = \det Y$. As a result, if $W_{D(\SL_n)}$ and $W_{D(\GL_n)}$ are the weight matrices for the actions of $\mathcal{H}_{\bg^r}\times \mathcal{H}_{\bg^c}$ upon $D(\GL_n)$ and $D(\SL_n)$, the identity $\tilde{B}_{D(\SL_n)}W_{D(\SL_n)} = 0$ follows from the identity $\tilde{B}_{D(\GL_n)}W_{D(\GL_n)} = 0$, for the weights of $g_{11}$ and $h_{11}$ are zero. We conclude from Proposition~\ref{p:toric} that the local toric action induced by $\mathcal{H}_{\bg^r}\times \mathcal{H}_{\bg^c}$ on $D(\SL_n)$ is $\gc$-extendable.
\end{proof}
\section{Selected examples}\label{s:exs}
In this section, we provide three examples of generalized cluster structures on $\GL_n \times \GL_n$ for $n\in \{3,4,5\}$. Note that some of the arrows in the quivers are dashed only for convenience; their weight is equal to $1$, as the weight of all the other arrows. For more examples, the reader may visit the author's github repository: 

\begin{center}\href{https://github.com/Grabovskii/GSV_Examples}{github.com/Grabovskii/GSV\_Examples}\end{center}

\subsection{Cremmer-Gervais $i\mapsto i-1$, $n = 3$}
The initial quiver is illustrated in Figure~\ref{f:n3nst}. There are two $\mathcal{L}$-matrices:
\[
\mathcal{L}_1(X,Y) = \begin{bmatrix}
x_{21} & x_{22} & x_{23} & \\
x_{31} & x_{32} & x_{33} & \\
 & y_{11} & y_{12} & y_{13}\\
 & y_{21} & y_{22} & y_{23}
\end{bmatrix}, \ \ \mathcal{L}_2(X,Y) = \begin{bmatrix} y_{13} & x_{21}\\ y_{23} & x_{31}\end{bmatrix}.
\]

\begin{figure}[htb]
\begin{center}
\includegraphics[scale=0.45]{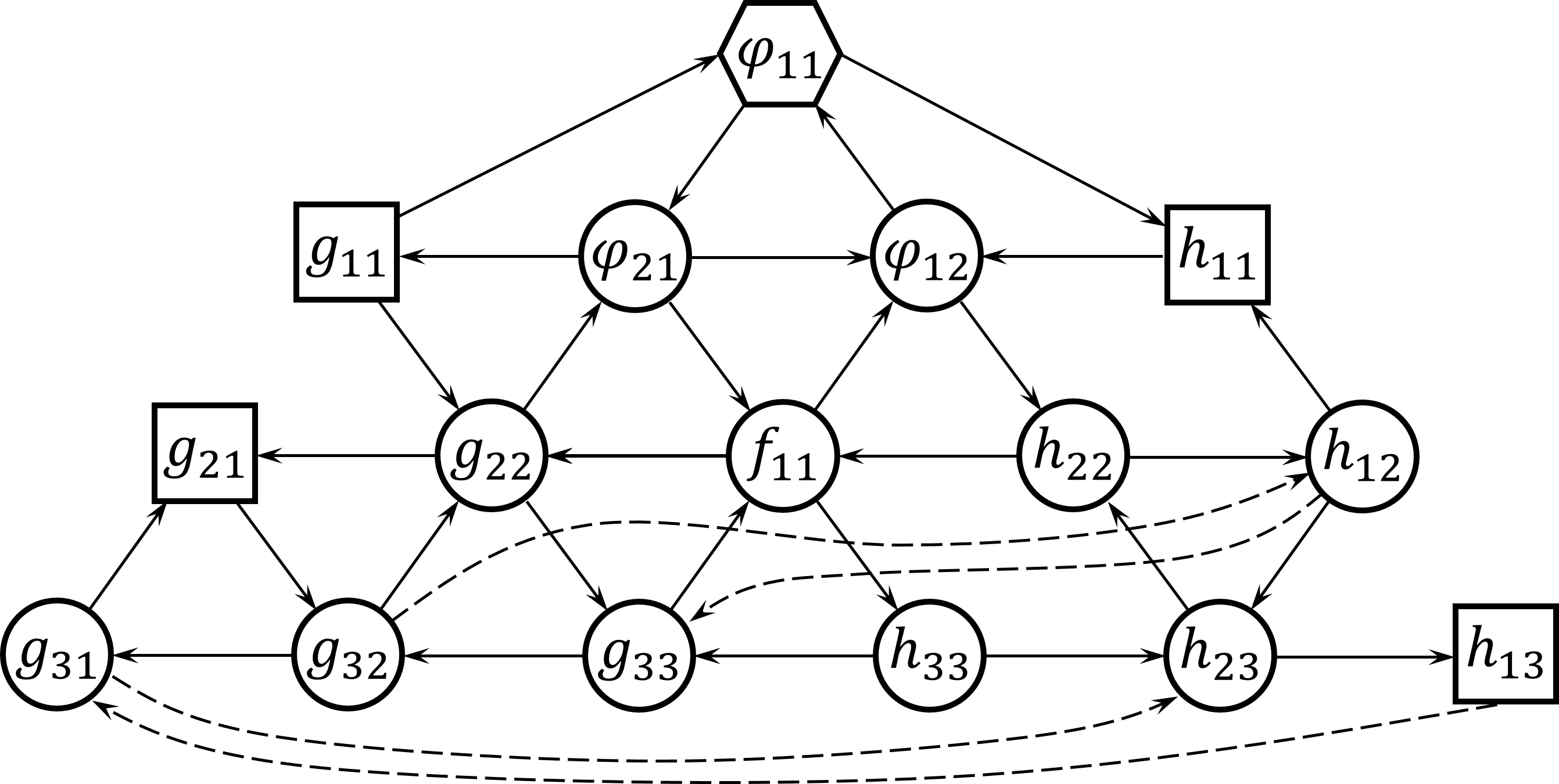}
\end{center}
\caption{The initial quiver for the Cremmer-Gervais structure in $n=3$, $\Gamma_1^r=\Gamma_1^c = \{2\}$,\hspace{0.3in} \mbox{$\Gamma_2^r=\Gamma_2^c = \{1\}$.}}
\label{f:n3nst}
\end{figure}

\subsection{Cremmer-Gervais $i \mapsto i+1$, $n=4$}
The initial quiver is illustrated in Figure~\ref{f:n4nst}. As we showed in Example~\ref{ex:cg4}, there are two $\mathcal{L}$-matrices:
\[
\mathcal{L}_1(X,Y) = \begin{bmatrix}x_{41} & x_{42} & x_{43} & 0 & 0 & 0\\ y_{12} & y_{13} & y_{14} & 0 & 0 & 0\\ y_{22} & y_{23} & y_{24} & x_{11} & x_{12} & x_{13} \\ y_{32} & y_{33} & y_{34} & x_{21} & x_{22} & x_{23}\\ y_{42} & y_{43} & y_{44} & x_{31} & x_{32} & x_{33}\\  0 & 0 & 0 & x_{41} & x_{42} & x_{43}  \end{bmatrix}, \ \ \ 
\mathcal{L}_2(X,Y) = \begin{bmatrix}
y_{12} & y_{13} & y_{14} & 0 & 0 & 0\\
y_{22} & y_{23} & y_{24} & x_{11} & x_{12} & x_{13}\\
y_{32} & y_{33} & y_{34} & x_{21} & x_{22} & x_{23}\\
y_{42} & y_{43} & y_{44} & x_{31} & x_{32} & x_{33}\\
0 & 0 & 0 & x_{41} & x_{42} & x_{43}\\
0 & 0 & 0 & y_{12} & y_{13} & y_{14}
\end{bmatrix}.
\]
\begin{figure}[htb]
\begin{center}
\includegraphics[scale=0.5]{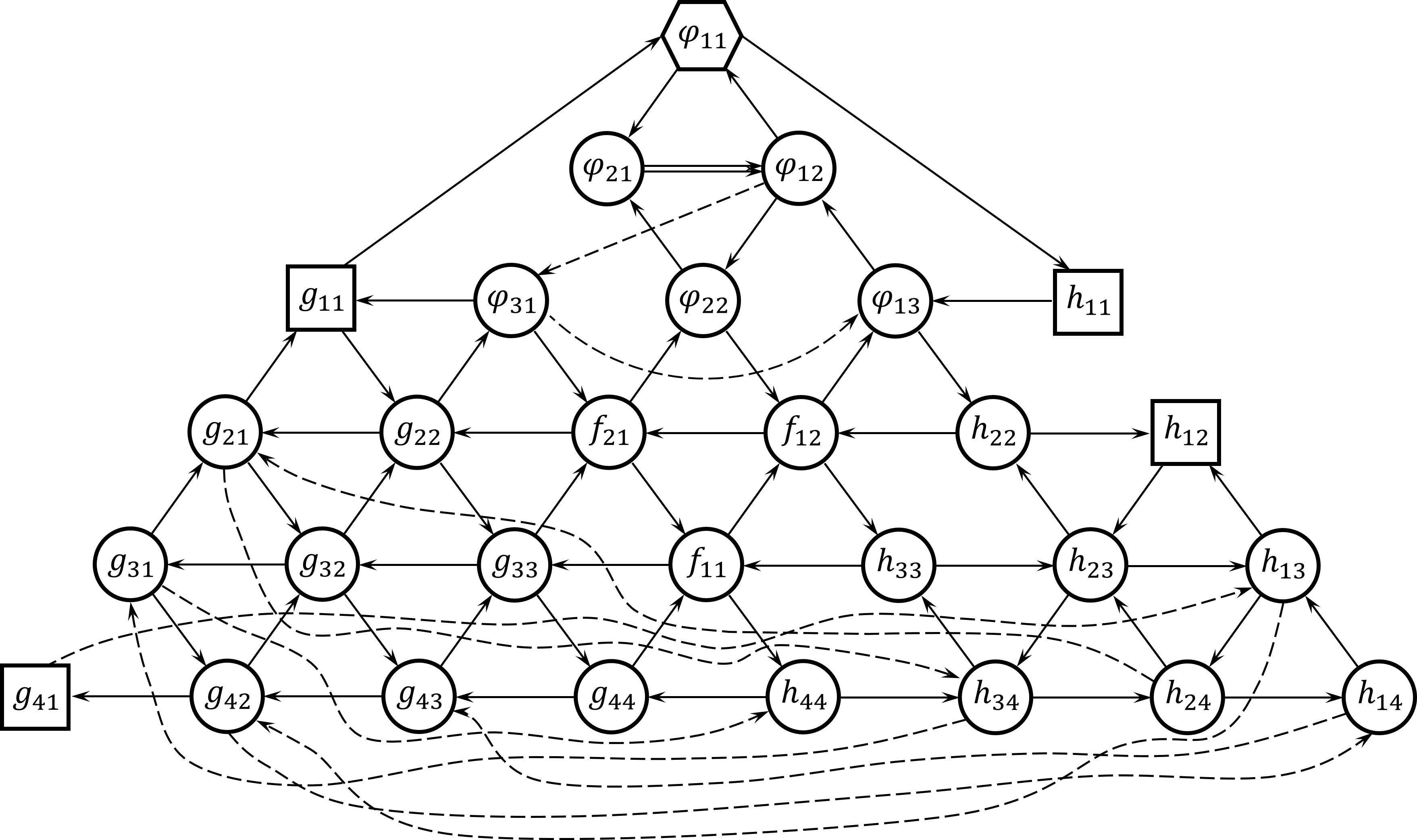}
\end{center}
\caption{The initial quiver for Cremmer-Gervais structure $i \mapsto i+1$, $\bg^r=\bg^c$, $n=4$.}
\label{f:n4nst}
\end{figure}

\subsection{An example with different $\bg^r$ and $\bg^c$, $n=5$}

The initial quiver of the resulting $\gc(\bg^r,\bg^c)$ is illustrated in Figure~\ref{f:n5nst}.

\begin{figure}[htb]
\begin{center}
\includegraphics[scale=0.43]{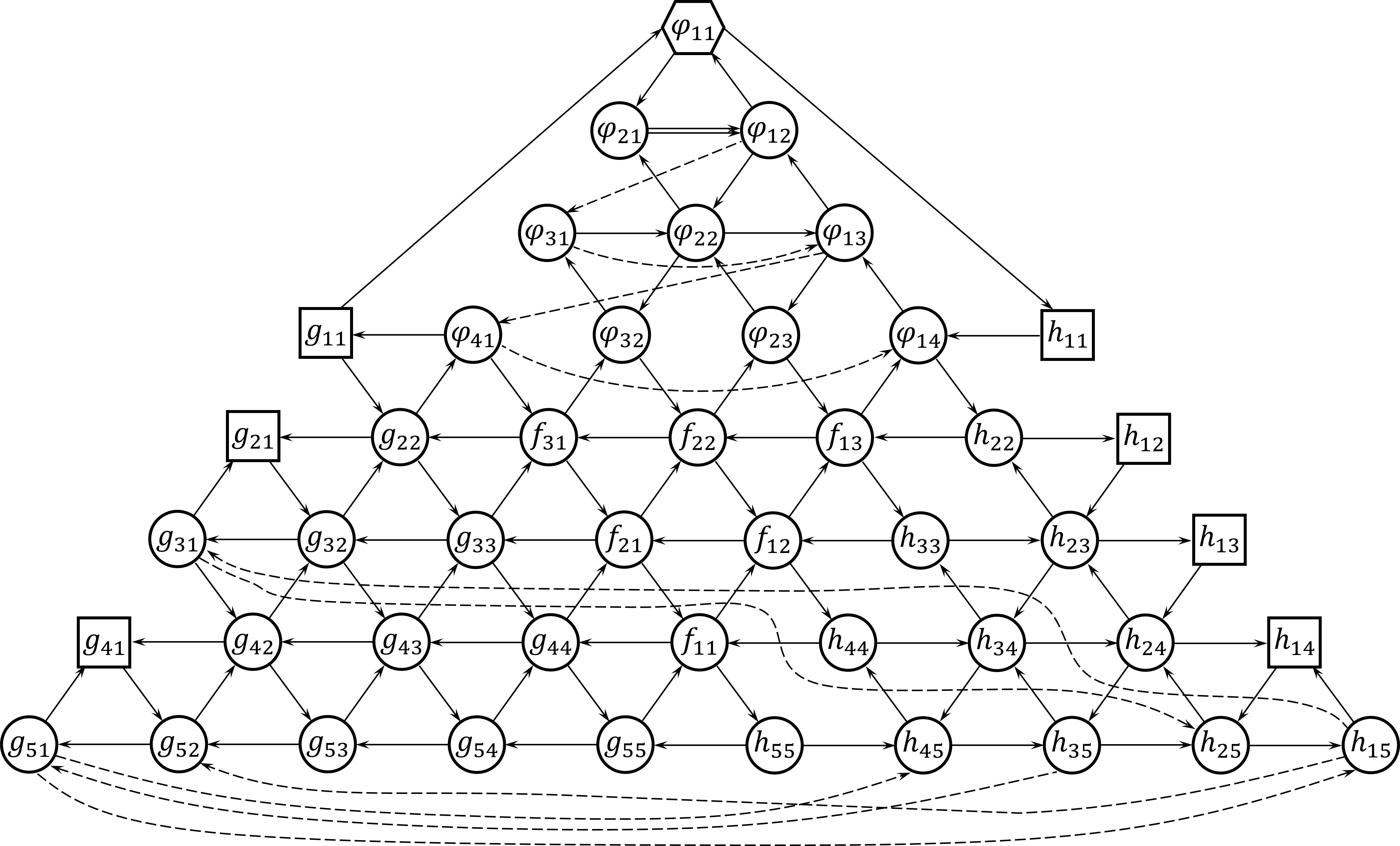}
\end{center}
\caption{The initial quiver for the generalized cluster structure on $\GL_5 \times \GL_5$ induced by the BD pair $\bg=(\bg^r,\bg^c)$ with $\Gamma_1^r = \{2,4\}$, $\Gamma_2^r = \{1,3\}$, $\gamma_r(2) = 1$, $\gamma_r(4) = 3$, $\Gamma_1^c = \{1\}$, $\Gamma_2^c = \{4\}$, $\gamma_c(1) = 4$.}
\label{f:n5nst}
\end{figure}

Set $n:=5$, $\Gamma_1^r := \{2,4\}$, $\Gamma_2^r := \{1,3\}$, $\gamma_r(2): = 1$, $\gamma_r(4): = 3$; $\Gamma_1^c := \{1\}$, $\Gamma_2^c: = \{4\}$, $\gamma_c(1) := 4$. There are six $\mathcal{L}$-matrices, five of which are trivial and one nontrivial:
\[
\mathcal{L}_1(X,Y):= \begin{bmatrix}
y_{13} & y_{14} & y_{15} &  &  &  &  & \\
y_{23} & y_{24} & y_{25} &  &  &  &  & \\
y_{33} & y_{34} & y_{35} & x_{41} & x_{42} &  &  & \\
y_{43} & y_{44} & y_{45} & x_{51} & x_{52} &  &  & \\
 &  &  & y_{14} & y_{15} & x_{21} & x_{22} & x_{23} \\
 &  &  & y_{24} & y_{25} & x_{31} & x_{32} & x_{33}\\
 &  &  &  &  & x_{41} & x_{42} & x_{43}\\
 &  &  &  &  & x_{51} & x_{52} & x_{53}
\end{bmatrix};
\]
\[
\mathcal{L}_2(X,Y) := X^{[1,2]}_{[4,5]}, \ \ \mathcal{L}_3(X,Y):=X^{[1,4]}_{[2,5]}, \ \ \mathcal{L}_4(X,Y):=Y^{[4,5]}_{[1,2]}, \ \ \mathcal{L}_5(X,Y):=Y^{[2,5]}_{[1,4]}.
\]


\end{document}